\documentclass[12 pt, hidelinks, leqno]{amsart}

\usepackage[top=2.3cm, bottom=2cm, left=2.2cm, right=2.2cm]{geometry}
\usepackage[english]{babel}
\usepackage{hyperref}
\usepackage{amsmath}
\usepackage{amsthm}
\usepackage{amssymb}
\usepackage{ucs}
\usepackage{mathtools}
\usepackage{enumerate}
\usepackage[utf8x]{inputenc}
\usepackage[T1]{fontenc}
\usepackage{color}
\usepackage{tikz}
\usepackage{diagbox}
\usepackage{tikz-cd}
\usepackage{tkz-graph}
\usepackage{comment}
\usepackage{relsize,exscale}

\AtBeginDocument{%
   \def\MR#1{}
}

\theoremstyle{plain}
\newtheorem{theorem}{Theorem}[section]
\newtheorem{theoremA}{Theorem}

\newtheorem{lemma}[theorem]{Lemma}
\newtheorem{proposition}[theorem]{Proposition}
\newtheorem{corollary}[theorem]{Corollary}
\theoremstyle{definition}
\newtheorem{definition}[theorem]{Definition}
\newtheorem{example}[theorem]{Example}

\newtheorem{remark}[theorem]{Remark}

\newcommand{\R}{\mathbb{R}}

\newcommand{\C}{\mathbb{C}}
\newcommand{\N}{\mathbb{N}}

\newcommand{\G}{\mathcal{G}}
\newcommand{\T}{\mathcal{T}}

\newcommand{\Z}{\mathbb{Z}}
\newcommand{\Lcal}{\mathcal{L}}

\newcommand{\Acal}{\mathcal{A}}

\newcommand{\Ncal}{\mathcal{N}}

\newcommand{\w}{\wr_{*}}

\newcommand{\ebar}{\overline{e}}

\newcommand{\ot}{\otimes}
\newcommand{\id}{{\rm id}}

\newcommand{\Tr}{{\rm Tr}}

\newcommand{\Pol}{{\rm Pol}}
\newcommand{\Irr}{{\rm Irr}}
\newcommand{\Linf}{{\rm L}^\infty}

\newcommand{\F}{\mathbb{F}}
\newcommand{\Aut}{{\rm Aut}}
\newcommand{\Ad}{{\rm Ad}}

\newcommand{\Int}{{\rm Int}}

\newcommand{\Mor}{\mathrm{Mor}}

\newcommand{\Rep}{\mathrm{Rep}}
\newcommand{\GB}{{\rm Aut}^+(B,\psi)}
\newcommand{\GG}{\mathbb{G}}

\newcommand{\E}{\mathbb{E}}

\newcommand{\Ccal}{\mathcal{C}}
\newcommand{\Zcal}{\mathcal{Z}}
\newcommand{\Ucal}{\mathcal{U}}

\newcommand{\Gtilde}{\widetilde{\GG}}
\newcommand{\Etilde}{\widetilde{\E}}

\newcommand{\Dtilde}{\widetilde{\Delta}}

\newcommand{\otm}{\ot_{{\rm max}}}
\newcommand{\ubar}{\overline{u}}
\newcommand\scalemath[2]{\scalebox{#1}{\mbox{\ensuremath{\displaystyle #2}}}}

\title{Generalized free wreath products and their operator algebras}
\author{Pierre Fima}
\address{Pierre Fima
\newline
Universit\'e Paris Cit\'e and Sorbonne Universit\'e, CNRS, IMJ-PRG, F-75013 Paris, France.}
\email{pierre.fima@imj-prg.fr}
\thanks{P.F. is partially supported by the CNRS IEA GAOA}
 
 \author{Arthur Troupel}
\address{Arthur Troupel
\newline
Universit\'e Paris Cit\'e and Sorbonne Universit\'e, CNRS, IMJ-PRG, F-75013 Paris, France.}
\email{arthur.troupel@imj-prg.fr}

\begin{document}

\begin{abstract}
   We develop a new approach on free wreath products, generalizing the constructions of Bichon and of Fima-Pittau. We show stability properties for certain approximation properties such as exactness, Haagerup property, hyperlinearity and K-amenability. We study qualitative properties of the associated von Neumann algebra: factoriality, primeness and absence of Cartan subalgebra and we give a formula for Connes' $T$-invariant. Finally, we give some  explicit computations of K-theory groups for C*-algebras of generalized free wreath products.
\end{abstract}

\maketitle

\section{Introduction}

\noindent Quantum information theory has now emerged as a powerful framework for processing information, leveraging the rich structures of quantum mechanics. Over the past decade, there has been a growing interest in understanding probability distributions arising from models of entangled quantum systems. In this context, compact quantum groups (CQG), as introduced by Woronowicz \cite{wor87,wor88,wor98}, play a fundamental role in capturing the natural notion of symmetry. More recently, deep connections have been established with long-standing problems in operator algebras. A major breakthrough in this area is the proposed resolution of the Connes Embedding Problem by Ji, Natarajan, Vidick, Wright, and Yuen in 2020 \cite{JNVWY22}, which employs techniques from quantum information theory to address a fundamental question regarding the approximation of operator algebras using finite structures. Simultaneously, intriguing links have emerged between quantum information theory and compact quantum groups. In particular, the well-established connection between quantum automorphisms and quantum information theory \cite{MR20} provides insight into non-local games and physically observable quantum behaviors.

\vspace{0.2cm}

\noindent We believe that a detailed study of operator algebras associated with quantum automorphism groups could lead to concrete applications in quantum information theory. The present paper is dedicated to such an investigation.

\vspace{0.2cm}

\noindent Following Woronowicz’s development of compact quantum group theory—which notably includes Drinfeld and Jimbo’s $q$-deformations of compact Lie groups \cite{dri86,dri87,jim85}—new examples of compact quantum groups emerged from Wang’s work \cite{Wa95,wan98}. Among them are the quantum permutation group $S_N^+$ and, more generally, the quantum automorphism group ${\rm Aut}^+(B,\psi)$ of a finite-dimensional C*-algebra equipped with a faithful state $\psi$ \cite{Ba99, Ba02}. The operator algebras associated with ${\rm Aut}^+(B,\psi)$ have been extensively studied \cite{Br13,Vo17}. Additionally, the quantum automorphism group of a finite graph was introduced by Bichon in \cite{bic03}. It is well-known that the classical automorphism group of the graph obtained by taking the $N$-fold disjoint union of a finite connected graph $\G$ is given by the wreath product of the classical automorphism group of $\G$ with $S_N$. Motivated by this, Bichon introduced a free version of the wreath product \cite{Bi04}, which satisfies a similar property for quantum automorphisms, when replacing the group $S_N$ by the quantum permutation group $S_N^+$, offering a parallel geometric interpretation in the quantum setting.

\vspace{0.2cm}

\noindent This \emph{free wreath product} introduced by Bichon has been subsequently generalized in various directions \cite{pit14,FP16,TW18,FS18,Fr22}, and is the central object of study in this paper. The operator algebras associated with Bichon’s free wreath product $G\wr_*S_N^+$, where $G$ is a compact quantum group, are now well understood \cite{FT24}. The primary goal of this work is to extend this understanding to a more general free wreath product, namely $G\wr_*{\rm Aut}^+(B,\psi)$ as defined in \cite{FP16}.

\vspace{0.2cm}

\noindent During our study, we found it beneficial to develop an alternative approach to defining the quantum group $G\wr_*{\rm Aut}^+(B,\psi)$, diverging from the existing methods in \cite{FP16} and \cite{TW18}. Compact quantum groups, are naturally categorical objects, and both previous approaches define the free wreath product via Tannaka-Krein reconstruction. While this method is well-adapted to grasp the representation theory of the free wreath product, our prior research \cite{FT24} on operator algebras associated with Bichon’s free wreath product $G\wr_*S_N^+$ revealed that the representation theory is not a mandatory step toward understanding the operator algebras. Instead, the key ingredient was the construction of a specific representation of the C*-algebra of the free wreath product, through its universal property as originally defined by Bichon rather than through Tannaka-Krein reconstruction.

\vspace{0.2cm}

\noindent This realization led us to seek a new definition of the C*-algebra of the free wreath product $G\wr_*{\rm Aut}^+(B,\psi)$, aiming for a universal property simpler than that of the Tannaka-Krein approach. In doing so, we were able to unify and generalize various constructions from \cite{Bi04,pit14,FP16,Fr22}. Specifically, given any compact quantum group $G$ and any $\psi$-preserving action $H\curvearrowright B$ of a compact quantum group $H$ on a finite-dimensional C*-algebra $B$, we construct in Theorem \ref{ThmDefFreeWr} a compact quantum group $G\wr_{*,\beta}H$. This construction coincides with $G\wr_*{\rm Aut}^+(B,\psi)$ when considering the universal action $\beta: {\rm Aut}^+(B,\psi) \curvearrowright B$. Furthermore, we introduce an amalgamated version, $G\wr_{*,\beta, F}H$, incorporating a dual quantum subgroup $F$ of $G$, thereby unifying the constructions in \cite{pit14,FP16} with those in \cite{Fr22}.

\vspace{0.2cm}

\noindent In the preliminary section, Section \ref{preliminaries} after introducing the necessary background on amalgamated free products and compact quantum groups, we provide a detailed review of compact quantum group actions on finite-dimensional C*-algebras and establish key technical lemmas that will be used throughout the paper.

\vspace{0.2cm}

\noindent In Section \ref{sectiongen}, we construct the (amalgamated) generalized free wreath product $\GG := G\wr_{*,\beta, F}H$ and present several examples.

\vspace{0.2cm}

\noindent Section \ref{sectionblock} develops our main analytical tool: the block decomposition technique, which we use to compute the Haar measure.

\vspace{0.2cm}

\noindent In Section \ref{section reduced}, we use the Haar measure to analyze the reduced C*-algebra, the von Neumann algebra, and the modular structure of $\GG$. This, in turn, allows us to investigate approximation properties of the associated operator algebras. To state our first result, we will use the following non-standard terminology: we will always denote by $\widehat{G}$ the discrete dual of a compact quantum group $G$, and we say that $\widehat{G}$ has the \textit{Haagerup property} if $\Linf(G)$ has the Haagerup property (with respect to the Haar state, see \cite{CS15}) and that $\widehat{G}$ is \textit{exact} when the reduced C*-algebra $C_r(G)$ is an exact C*-algebra. When $G$ is Kac, we say that its dual $\widehat{G}$ is \textit{hyperlinear} if the finite von Neumann algebra $(\Linf(G),h_G)$ embeds in an ultraproduct of the hyperfinite ${\rm II}_1$-factor in a trace preserving way, where $h_G$ denotes the Haar state on $\Linf(G)$. Finally, $G$ is called \textit{coamenable} when the canonical surjection from its full C*-algebra $C(G)$ to $C_r(G)$ is an isomorphism. The main result is the following, see Theorem \ref{thmAtxt}.

\begin{theoremA}\label{ThmA}
Let $\beta\,:\, H\curvearrowright B$ be an ergodic action, $G$ any compact quantum group with dual quantum subgroup $F$ and define $\GG:=G\wr_{*,\beta,F}H$. The following holds.
\begin{enumerate}
    \item $\widehat{\GG}$ is exact if and only if both $\widehat{G}$ and $\widehat{H}$ are exact.
    \item If $\Irr(F)$ is finite, $\widehat{\GG}$ has the Haagerup property if and only if both $\widehat{G}$ and $\widehat{H}$ have the Haagerup property.
    \item $\GG$ is Kac if and only if both $G$ and $H$ are Kac and in this case, assuming that $F$ is coamenable, $\widehat{\GG}$ is hyperlinear if and only if both $\widehat{G}$ and $\widehat{H}$ are hyperlinear.
    \item If $F$ is trivial and both $G$ and $H$ are non-trivial then $\GG$ is coamenable if and only if $G=H=\Z_2$ and $B\in\{\C,\C^2\}$.
\end{enumerate}
\end{theoremA}
\noindent Note that this result is particularly applicable to the universal action $\beta : {\rm Aut}^+(B, \psi) \curvearrowright B$, which is ergodic so that Theorem \ref{ThmA} covers in particular the case of Fima-Pittau free wreath products. The statement of Theorem \ref{ThmA} concerning exactness is only known for Fima-Pittau free wreath products \cite{FP16} (see also \cite{lt14} for the more specific case of Bichon's free wreath-products). The other statements of Theorem \ref{ThmA} are only known for Bichon's free wreath products \cite{FT24} i.e. when considering the universal action $\beta\,:\,S_N^+\curvearrowright\C^N$. Let us however mention the works \cite{lem14,lt14,FP16,TW18} containing results in relation with the stability of the Haagerup property (actually the stability of the \textit{central ACPAP}) for some specific free wreath products and the work \cite{BCF20}, based on the topological generation method developed in \cite{BCV17}, concerning the hyperlinearity of quantum reflection groups $\widehat{\Z_s}\wr_*S_N^+$.

\vspace{0.2cm}

\noindent In Section \ref{section quali}, we delve into the qualitative properties of the von Neumann algebra $\text{L}^\infty(\GG)$. Following a preliminary section on intertwining theory, we examine the factoriality and Connes' $T$ invariant of $\text{L}^\infty(\GG)$.

\begin{theoremA}\label{ThmB}
Suppose that $\beta$ is $2$-ergodic, both $G$ and $H$ are infinite quantum groups and $F$ is trivial then $\Linf(\GG)$ is a non amenable factor of type ${\rm II}_1$ or ${\rm III}$ and its $T$ invariant is:
$$T(\Linf(\GG))=\{t\in\R\,:\,\sigma_t^{h_G}=\id\text{ and }\sigma_t^{h_H}=\id\}.$$
where $\sigma_t^{h_G}$ is the modular group of the Haar state of $G$ (and the same for $H$). Moreover $\Linf(\GG)$ is of type ${\rm II}_1$ if and only if both $G$ and $H$ are Kac.
\end{theoremA}

\noindent We actually prove a more general result in Theorem \ref{ThmFactor}, where we allow the amalgam $F$ to be non-trivial. In this case, $\text{L}^\infty(\GG)$ is not always a factor, and we explicitly compute the center. Note that Theorem \ref{ThmB} particularly applies to the universal action $\beta : {\rm Aut}^+(B, \psi) \curvearrowright B$, which is 2-ergodic (see Section \ref{preliminaries} for the definition). The statement of Theorem \ref{ThmB} is known only in the case of Bichon's free wreath product $G \wr_* S_N^+$, as seen in \cite{FT24}. Let us also mention \cite{lem14,Wa14} where factoriality is proved for some specific cases of Bichon's free wreath products.

\vspace{0.2cm}

\noindent In Section \ref{section quali}, we also study primeness and absence of Cartan subalgebras. Recall that a factor $M$ is called \textit{prime} if it is not isomorphic to the tensor product $P\ot Q$ of two diffuse factors $P$ and $Q$. Also, we say that a von Neumann subalgebra $A\subset M$ of a von Neumann algebra is a \textit{Cartan subalgebra} if $A$ is maximal abelian i.e. $A'\cap M=A$, \textit{regular} i.e. $M$ is generated, as a von Neumann algebra, by the normalizer $\{u\in\mathcal{U}(M)\,:\, uAu^*=A\}$ of $A$ and $A$ \textit{has expectation} i.e. there is a faithful normal conditional expectation from $M$ to $A$. The result is the following, see Theorem \ref{ThmPrimeCartan}.

\begin{theoremA}\label{ThmC}
Suppose that ${\rm dim}(B)\geq 3$, $\beta$ is $2$-ergodic, both $G$ and $H$ are infinite quantum groups and $F$ is trivial then the following holds.
\begin{enumerate}
\item $\Linf(\GG)$ is prime.
\item $\Linf(\GG)$ has no Cartan subalgebras.
\end{enumerate}
\end{theoremA}

\noindent Once again, the previous theorem applies to the particular case of the universal action $\beta : \text{Aut}^+(B, \psi) \curvearrowright B$, thus significantly generalizing the results of \cite{FT24}.

\vspace{0.2cm}

\noindent In Section \ref{section KK}, we study the $K$-theory of C*-algebras of free wreath products. We begin by characterizing the $K$-amenability of a free wreath product. Recall that $\widehat{G}$ is \textit{K-amenable} if the canonical surjection from the full C*-algebra $C(G)$ to the reduced C*-algebra $C_r(G)$ is a KK-equivalence. It is known that $\widehat{S_N^+}$ is K-amenable \cite{Vo17} and, by using a monoidal equivalence argument, it is known \cite{FM20} that if $G$ is torsion-free and satisfies the strong Baum-Connes conjecture then $\widehat{G\wr_*S_N^+}$ is K-amenable. The following result, corresponding to Theorem \ref{thmDtxt}, was only known in the case of Bichon's free wreath products, see \cite[Theorem A]{FT24}.

\begin{theoremA}\label{thmD}
Let $\beta\,:\,H\curvearrowright B$ be any ergodic action, $G$ any compact quantum group with dual quantum subgroup and $\GG:=G\wr_{*,\beta,F}H$ the free wreath product. The following are equivalent.
\begin{enumerate}
\item $\widehat{G}$ and $\widehat{H}$ are $K$-amenable.
\item $\widehat{\GG}$ is $K$-amenable.
\end{enumerate}
\end{theoremA}

\noindent We then provide long exact sequences to compute the $K$-theoretic groups of $C_\bullet(\GG)$ (which denotes either the full or the reduced C*-algebra of $\GG$). However, to derive explicit computations from these long exact sequences, one needs to know the $K$-theoretic groups $K_i(C_\bullet(H))$ as well as explicit generators of $K_0(C_\bullet(H))$. Note that the $K$-theory of $C_\bullet({\Aut}^+(B, \psi))$ has been computed by Voigt \cite{Vo17}, but explicit generators of the $K_0$-group are known only for $S_N^+$. Nevertheless, we find explicit generators of the $K_0$-group in the case $B = M_N(\C)$, which allows us to obtain some explicit computations. For $s \in \mathbb{N}^*$, we denote by $\F_s$ the free group with $s$ generators and $\Z_s := \Z / s\Z$. The following result is Theorem \ref{thmEtxt}.

\begin{theoremA}\label{ThmE}
For all $s,N\in\N^*$ and for any faithful state $\psi\in M_N(\C)^*$ one has:
$$K_i(C_\bullet(\Z_s\wr_*{\rm Aut}^+(M_N(\C),\psi)))=\left\{\begin{array}{lcl}\Z^s\oplus\Z_N&\text{if}&i=0,\\\Z&\text{if}&i=1.\end{array}\right.\quad\text{and,}$$
$$K_i(C_\bullet(\widehat{\F}_s\wr_*{\rm Aut}^+(M_N(\C),\psi)))=\left\{\begin{array}{lcl}\Z\oplus\Z_N&\text{if}&i=0,\\\Z^{s+1}&\text{if}&i=1.\end{array}\right.$$
\end{theoremA}

\section*{Acknowledgements}

\noindent The authors would like to thank Cyril Houdayer for helpful conversations.

\tableofcontents

\section{Preliminaries}\label{preliminaries}

\noindent All von Neumann algebras are supposed to have separable predual and all C*-algebras and Hilbert spaces are supposed to be separable. The scalar product on an Hilbert space, or an Hilbert C*-module, is always right linear. We use the same symbol $\ot$ to denote the algebraic tensor product of unital $*$-algebras, the minimal tensor product of C*-algebras, the tensor product of Hilbert spaces and the tensor product of von Neumann algebras. We use the symbol $\otm$ to denote the maximal tensor product of C*-algebras. When $H$ is a Hilbert space (resp. Hilbert $A$-module), we denote by $\Lcal(H)$ (resp. $\Lcal_A(H)$) the C*-algebra of bounded linear operators on $H$ (resp. adjointable operators on $H$).

\subsection{Operator algebras} Let $A$, $B$ be unital C*-algebras and $\varphi\,:\,A\rightarrow B$ a unital completely positive (ucp) map. The \textit{GNS-construction} of $\varphi$ is the unique, up to a canonical isomorphism, triple $(H,\pi,\Omega)$, where $H$ is a Hilbert C*-module over $B$, $\pi\,:\,A\rightarrow\Lcal_B(H)$ is a unital $*$-homomorphism and $\Omega\in H$ is a unit vector such that $\varphi(x)=\langle\Omega,\pi(x)\Omega\rangle$ for all $x\in A$ and $\pi(A)\Omega\cdot B$ is dense in $H$. \textit{The reduced C*-algebra of $\varphi$} is defined by $C^*_r(\varphi):=\pi(A)$, where $\pi$ is the GNS-morphism of $\varphi$. By uniqueness of the GNS-construction, it does not depends on a particular choice of the GNS construction $(H,\pi,\Omega)$ of $\varphi$. Moreover, $\varphi$ is called \textit{GNS-faithful} when its GNS-morphism $\pi$ is faithful. In that case, $A$ itself is the reduced C*-algebra of $\varphi$. Note that the reduced C*-algebra of $\varphi$ is, up to a canonical isomorphism, the unique quotient $\pi\,:\,A\twoheadrightarrow C^*_r(\varphi)$ of $A$ such that there is a GNS-faithful ucp map $\varphi_r\,:\, C^*_r(\varphi)\rightarrow B$ satisfying $\varphi_r\circ\pi=\varphi$. To be more precise, the uniqueness means that if $C$ is any unital C*-algebra with a surjective unital $*$-homomorphism $\rho\,:\,A\twoheadrightarrow C$ and a GNS-faithful ucp map $\psi\,:\,C\rightarrow B$ such that $\psi\circ\rho=\varphi$ then there exists a unique unital $*$-isomorphism $\widetilde{\rho}\,:\,C^*_r(\varphi)\rightarrow C$ such that $\widetilde{\rho}\circ\pi=\rho$. Note that $\widetilde{\rho}$ automatically intertwines the ucp maps i.e. $\psi\circ\widetilde{\rho}=\varphi_r$. The following elementary Remark will be useful.

\begin{remark}\label{RmkReducedUcp}
Let $\varphi\,:\, A\rightarrow B$ be a ucp map with reduced C*-algebra and GNS-morphism $\pi\,:\,A\twoheadrightarrow C^*_r(\varphi)$ and let $\varphi_r\,:\,C^*_r(\varphi)\rightarrow B$ be the unique GNS-faithful ucp map such that $\varphi_r\circ\pi=\varphi$. If $\varphi_r$ is faithful then, for any unital C*-subalgebra $A_0\subset A$ one has $C^*_r(\varphi\vert_{A_0})=\pi(A_0)$. It is a direct consequence of the uniqueness of the reduced C*-algebra of a ucp map since $\varphi_r\vert_{A_0}$ is still faithful hence GNS-faithful.
\end{remark}

\vspace{0.2cm}

\noindent Let $B\subset A$ be a unital C*-subalgebra. A \textit{conditional expectation} from $A$ to $B$ is a ucp map $E\,:\,A\rightarrow B$ such that $E(b)=b$ for all $b\in B$. In particular, since $B$ is in the multiplicative domain of the ucp map $E$, we have $E(xay)=xE(a)y$ for all $x,y\in B$ and $a\in A$.

\vspace{0.2cm}

\noindent We now quickly review some of the results of \cite{FG20}. Let $A_1,A_2$ be unital C*-algebras with a common subalgebra $B\subset A_k$ and write $A=A_1\underset{B}{*}A_2$ the full amalgamated free product. For $k=1,2$, it is well known that the canonical maps $A_k\rightarrow A$ are faithful so we will view $A_k\subset A$. Given two conditional expectations $E_k\,:\, A_k\rightarrow B$ (not necessarily GNS faithful), an operator $a\in A$ is called \textit{reduced} (relatively to $E_1$ and $E_2$) if it is of the form $a=a_1\dots a_n$ with, $n\in \N^*$ and for $1\leq l\leq n$, $a_l\in \ker(E_{i_l})\subset A_{i_l}$ and $i_{l}\neq i_{l+1}$. It is shown in \cite{FG20} that there exist unique conditional expectations $E_{A_k}\,:\,A\rightarrow A_k$ satisfying:
$$E_{A_k}(a)=0\text{ for }a=a_1\dots a_n\text{ reduced with }n\geq 2\text{ and }E_{A_k}\vert_{A_{\overline{k}}}=E_{\overline{k}},\text {where }\overline{k}\in\{1,2\}\setminus\{k\}.$$
It is also shown that there exists a unique conditional expectation $E_B\,:\, A\rightarrow B$ such that $E_B(a)=0$ for any reduced operator $a\in A$.

\vspace{0.2cm}

\noindent When both $E_1$ and $E_2$ are GNS-faithful then the reduced C*-algebras of $E_{A_1}$, $E_{A_2}$ and $E_B$ are all canonically isomorphic to the Voiculescu's reduced amalgamated free product denoted by $A_r:=(A_1,E_1)\underset{B}{*}(A_2,E_2)$. We will still denote by $E_{A_k}$ and $E_B$ the GNS-faithful conditional expectations $A_r\rightarrow A_k$ and $A_r\rightarrow B$. An easy adaptation of the proof of the main result in \cite{Dy98} (see also \cite[Exercise 4.5.5]{BO08}), shows that if both $E_1$ and $E_2$ are faithful then $E_B$ is faithful on $A_r$. This implies that both $E_{A_1}$ and $E_{A_2}$ are faithful on $A_r$ since we have $E_B=E_k\circ A_k$, for $k=1,2$.

\vspace{0.2cm}

\noindent Let now $M_1,M_2$ be two von Neumann algebras with a common subalgebra $N\subset M_k$ and faithful normal conditional expectation $E_k\,:\, M_k\rightarrow N$. There exists a unique, up to a canonical isomorphism, von Neumann algebra $M:=M_1\underset{N}{*}M_2$ generated by $M_1\cup M_2\subset M$, where the two copies of $N$ are identified and with a faithful normal conditional expectation $E_N\,:\, M\rightarrow N$ such that $E_N(x)=0$ for any $x\in M$ reduced (relative to $E_1$ and $E_2$). Moreover, $M$ comes equipped with two normal faithful conditional expectations $E_{M_k}\,:\, M\rightarrow M_k$ such that $E_{M_k}(x)=0$ for any $x$ reduced of length at least $2$ and $E_{M_k}\vert_{M_{\overline{k}}}=E_{\overline{k}}$, where $\overline{k}\in\{1,2\}\setminus\{k\}$. To specify the conditional expectations, we sometimes write $M=(M_1,E_1)\underset{N}{*}(M_2,E_2)$.

\vspace{0.2cm}

\noindent Suppose now that we have strongly continuous one parameter groups $\alpha_k\,:\,\R\rightarrow {\rm Aut}(M_k)$ for $k\in\lbrace 1,2\rbrace$, and $\gamma\,:\,\R\rightarrow {\rm Aut}(N)$ such that $\alpha^{(k)}_t\vert_N=\gamma_t$ and $E_k\circ\alpha^{(k)}_t=\gamma_t\circ E_1$ for all $t\in\R$, $k\in\{1,2\}$. By the universal property of the von Neumann algebraic free product, there exists, for each $t\in\R$, a unique $\tau_t\in{\rm Aut}(M)$ such that $\tau_t\vert_{M_k}=\alpha^{(k)}_t$ for $k\in\lbrace 1,2\rbrace$. We sometimes write $\tau:=\alpha_1*\alpha_2$. Using the uniqueness, we see that $(t\mapsto\tau_t)$ is a group homomorphism. We will use the following well known Lemma and we include a proof for the convenience of the reader.

\begin{lemma}\label{LemOneParameter}
Fix a faithful normal state (f.n.s.) $\mu\in N_*$ and assume that $\mu\circ\gamma_t=\mu$ for all $t\in\R$. Then $\tau$ is a strongly continuous one parameter group of $M$ preserving $\mu\circ E_N$.
\end{lemma}

\begin{proof}
Define the f.n.s. $\omega=\mu\circ E_N\in M_*$ and note that, for all $t\in\R$ and $x\in M$, $\omega\circ\tau_t(x)=\omega(x)$. Indeed, the equality is clear when $x$ is reduced and for $x\in N$ it follows from the hypothesis $\mu\circ\gamma_t=\mu$ for all $t\in\R$. Define $X:=\{x\in M\,:\,(t\mapsto\tau_t(x))\text{ is strongly continuous}\}$ and note that $X$ is a vector subspace, stable by involution and contains $M_1\cup M_2$. Let $x,y\in X$ and let us show that $xy\in X$. View $M\subset\mathcal{L}({\rm L}^2(M,\omega))$, where $\xi_\omega$ is the cyclic vector in the GNS construction ${\rm L}^2(M,\omega)$ of $\omega$. One has, for all $s,t\in\R$, $\tau_t(xy)-\tau_s(xy)=\tau_t(x)(\tau_t(y)-\tau_s(y))+(\tau_t(x)-\tau_s(y))\tau_s(y)$. Hence, for all $\xi\in H$, and all sequence $(t_n)_n$ converging to $t$ one has:
$$\Vert(\tau_{t_n}(xy)-\tau_{t}(xy))\xi\Vert\leq\Vert x\Vert\,\Vert(\tau_{t_n}(y)-\tau_t(y))\xi\Vert+\Vert(\tau_{t_n}(x)-\tau_t(x))\tau_t(y)\xi\Vert\rightarrow_n0$$
since both $x,y\in X$. It follows that $X$ is a subalgebra of $M$ containing both $M_1$ and $M_2$. Hence, it suffices to show that $X$ is strong* closed in $M$. Let $(x_n)_n\in X$ be a sequence converging strongly to $x\in M$ (in particular, $(x_n)$ is norm bounded). It suffices to show that the set $$Z:=\{\xi\in {\rm L}^2(M,\omega)\,:\,(t\mapsto\tau_t(x)\xi)\text{ is continuous }\}$$
is equal to ${\rm L}^2(M,\omega)$. Let $\xi=a\xi_\omega$, where $a\in M$ is analytic with respect to the modular group $(\sigma_t)_t$ of $\omega$. Denoting by $J$ the anti-unitary modular operator of $\omega$ we have,
\begin{eqnarray*}
\Vert\tau_t(x_n)\xi-\tau_t(x)\xi\Vert&=&\Vert\tau_t(x_n-x)a\xi_\omega\Vert=\Vert\tau_t(x_n-x)J\sigma_{i/2}(a)^*J\xi_\omega\Vert\\
&=&\Vert J\sigma_{i/2}(a)^*J\tau_t(x_n-x)\xi_\omega\Vert\leq\Vert\sigma_{i/2}(a)\Vert\Vert\tau_t(x_n-x)\xi_\omega\Vert\\
&=&\Vert\sigma_{i/2}(a)\Vert\cdot\omega(\tau_t((x_n-x)^*(x_n-x)))^{1/2}\\
&=&\Vert\sigma_{i/2}(a)\Vert\cdot\omega((x_n-x)^*(x_n-x))^{1/2}.
\end{eqnarray*}
It follows that the sequence of continuous maps $(t\mapsto\sigma_t(x_n)\xi)_n$ converges uniformly on $\R$ to the map $t\mapsto\sigma_t(x)\xi$. Hence, $a\xi_\omega\in Z$ for all $a\in M$ analytic with respect to $(\sigma_t)_t$. Since this set is dense in ${\rm L}^2(M,\omega)$, it suffices to show that $Z$ is norm-closed in ${\rm L}^2(M,\omega)$. Let $\xi_n\in Z$ be a sequence converging to $\xi\in{\rm L}^2(M,\omega)$. Writing $\Vert\tau_t(x)\xi_n-\tau_t(x)\xi\Vert\leq\Vert x\Vert\Vert\xi_n-\xi\Vert$, we can argue as before to deduce that $\xi\in Z$. It concludes the proof.\end{proof}

\begin{remark}\label{RmkModularFree}
Let $\mu\in N_*$ be any f.n.s. and define $\varphi_k:=\mu\circ E_k\in (M_k)_*$ and $\varphi:=\mu\circ E_N\in M_*$. The following relation between the modular groups of $\varphi$, $\varphi_1$ and $\varphi_2$ is shown in \cite{Ue99}: $\sigma^\varphi=\sigma^{\varphi_1}*\sigma^{\varphi_2}$.
\end{remark}

\noindent Let $\omega$ be a f.n.s. on a von Neumann algebra $M$ with modular group $(\sigma_t)_{t\in\R}$. For $\lambda\in\C$, let $V_\lambda:=\{x\in M\,:\,\sigma_t(x)=\lambda^{it}x\text{ for all }t\in\R\}$ and note that $V_\lambda V_\mu\subset V_{\lambda\mu}$ and $V_\lambda^*=V_{\lambda^{-1}}$. In particular $V_1$, also known as the centralizer of $\omega$ and denoted by $M^\omega$, is a unital von Neumann subalgebra of $M$ and $V_\omega:={\rm Span}\left\{V_\lambda\,:\,\lambda\in\C\right\}$ is a unital $*$-subalgebra of $M$.

\vspace{0.2cm}

\noindent The f.n.s. $\omega$ is called \textit{almost periodic} if $V_\omega''=M$. It is easy to check that any state on a finite dimensional von Neumann algebra is almost periodic and the tensor product as well as the amalgamated free product of almost periodic states is also almost periodic. Finally, we also note that if $G$ is a compact quantum group and $M={\rm L}^\infty(G)$ then the Haar state on $M$ is almost periodic, with $V_\omega={\rm Pol}(G)$. The following well known remark will be useful.

\begin{remark}
If $\omega$ is almost periodic on a diffuse von Neumann algebra $M$ then its centralizer is diffuse. Indeed, if $p\in M^\omega$ is a (non-zero) minimal projection, then the centralizer of the f.n.s. on $pMp$ given $\varphi:=\omega(p)^{-1}\omega(p\cdot p)$ is $M^\varphi=pM^\omega p=\C p$. It follows that, for all $\lambda\in\C$ and all $x\in V^\varphi_\lambda$, $x^*x\in V^\varphi_\lambda V^\varphi_{\lambda^{-1}}\subset M^\varphi=\C p$. Hence, $V_\varphi\subset \C p$ and, since $\omega$ is almost periodic, $\varphi$ is almost periodic on $pMp$ so $pMp=\C p$ and $p$ is minimal in $M$.
\end{remark}

\subsection{Compact quantum groups}\label{SectionCQG}

\noindent For a CQG $G$, we denote by $C(G)$ its \textit{maximal} C*-algebra, which is the enveloping C*-algebra of the unital $*$-algebra $\Pol(G)$ given by the linear span of coefficients of irreducible unitary representations of $G$. The collection of finite dimensional unitary representations is denoted by $\Rep(G)$ and the set of equivalence classes of irreducible unitary representations is denoted by $\Irr(G)$. We will denote by $\varepsilon_G\,:\, C(G)\rightarrow\C$ the counit of $G$ which satisfies ($\id\ot\varepsilon_G)(u)=\id_H$ for all finite dimensional unitary representations $u\in\Lcal(H)\ot C(G)$.

\vspace{0.2cm}

\noindent Let us recall below the modular ingredients of a CQG. Let $u\in\Lcal(H_u)\ot C(G)$ be a finite dimensional unitary representation (not necessarily irreducible). The contragredient representation $\overline{u}$ is the unique, up to isomorphism, finite dimensional unitary representation such that there exists a pair $(s_u,s_{\overline{u}})$ with $s_u\in\Mor(1,u\ot\overline{u})$ and $s_{\overline{u}}\in\Mor(1,\overline{u}\ot u)$ such that: $$(s_{\overline{u}}^*\ot\id_{H_{\ubar}})(\id_{H_{\overline{u}}}\ot s_u)=\id_{H_{\overline{u}}}\text{ and }(s_{u}^*\ot\id_{H_{u}})(\id_{H_{u}}\ot s_{\overline{u}})=\id_{H_{u}}.$$
We say that such a pair $(s_u,s_{\ubar})$ solves the conjugate equation for $(u,\ubar)$. Let $(s_u,s_{\ubar})$ be a pair solving the conjugate equation for $(u,\ubar)$ and define the positive functionals $\varphi_u,\psi_u$ on the C*-algebra ${\rm End}(u)\subset\Lcal(H_u)$ by $\varphi_u(T):=s_{\ubar}^*(\id_{H_{\ubar}}\ot T)s_{\ubar}$ and $\psi_u(T):=s_{u}^*( T\ot\id_{H_{\ubar}})s_{u}$. Note that $\psi_u={\rm Tr}(Q_u\cdot )$ and $\varphi_u={\rm Tr}(Q_{\ubar}\cdot)$ (where ${\rm Tr}$ is the unique trace on $\Lcal(H_u)$ such that ${\rm Tr}(1)={\rm dim}(H_u)$), where $Q_u=J_u^*J_u\in\Lcal(H_u)$ is positive invertible and $J_u\,:\, H_u\rightarrow H_{\overline{u}}$ is the unique invertible antilinear map satisfying $\langle J_u\xi,\eta\rangle=\langle s_u,\xi\ot\eta\rangle$ for all $\xi\in H_u$, $\eta\in H_{\overline{u}}$. $Q_{\ubar}=J_{\ubar}J_{\ubar}^*$ is defined similarly, using $s_{\ubar}$ instead of $s_u$. The pair $(s_u,s_{\ubar})$ is called \textit{standard} whenever $\varphi_u=\psi_u$ and in the case of a standard pair one has $J_{\ubar}=J_u^{-1}$, $Q_{\ubar}=Q_u^{-1}$.
 
\vspace{0.2cm}

\noindent It is known that a standard pair $(s_u,s_{\ubar})$ always exists and for any standard pair $(s_u,s_{\ubar})$ one has $\Vert s_{\overline{u}}\Vert=\Vert s_u\Vert$ and this number does not depend on the choice of a standard pair. We then define \textit{the quantum dimension of $u$} as ${\rm dim}_q(u):=\Vert s_u\Vert^2={\rm Tr}(Q_u)={\rm Tr}(Q_{u}^{-1})=\Vert s_{\ubar}\Vert^2$ for a standard pair $(s_u,s_{\ubar})$. In what follows, $Q_u\in\Lcal(H_u)$ will always denote the unique positive invertible operator in $\Lcal(H_u)$ defined above. Such operators satisfy $SQ_u=Q_vS$ for all $u,v\in\Rep(G)$ and all $S\in\Mor(u,v)$.

\vspace{0.2cm}

\noindent We will always choose an orthonormal basis $(e^u_i)_i$ of $H_u$ diagonalizing $Q_u$ i.e. $Q_ue^u_i=Q_{u,i}e^u_i$ ($Q_{u,i}>0$) then, $e^{\overline{u}}_i:=Q_{u,i}^{-\frac{1}{2}}J_ue^u_i$ is an orthonormal basis diagonalizing $Q_{\overline{u}}$ and we have $Q_{\overline{u}}e^{\overline{u}}_i=Q^{\frac{1}{2}}_{u,i}e^{\overline{u}}_i$ and $J_{\overline{u}}e^{\overline{u}}_i=Q_{u,i}^{-\frac{1}{2}}e^u_i$ so that our notations are coherent with $Q_{\overline{u},i}:=Q_{u,i}^{-1}$. Note that, in this basis we have $s_u=\sum_iQ_{u,i}^{\frac{1}{2}}e^{u}_{i}\ot e^{\overline{u}}_i$ (and $s_{\overline{u}}=\sum_iQ_{u,i}^{-\frac{1}{2}}e^{\overline{u}}_i\ot e^u_i$).

\vspace{0.2cm}

\noindent Let $h_G$ be the Haar state on $C(G)$ and $({\rm L}^2(G),\lambda_G,\xi_G)$ its GNS construction. Let $$C_r(G):=\lambda_G(C(G))\subset\Lcal({\rm L}^2(G))$$ be the \textit{reduced C*-algebra} of $G$. The surjective unital $*$-homomorphism $\lambda_G\,:\,C(G)\rightarrow C_r(G)$ is called the \textit{canonical surjection}. Recall that $G$ is called \textit{coamenable} whenever $\lambda_G$ is injective. The \textit{von Neumann algebra of} $G$ is denoted by $\Linf(G):=C_r(G)''\subset\Lcal({\rm L}^2(G))$. We will still denote by $h_G$ the Haar state of $G$ on the C*-algebra $C_r(G)$ as well as on the von Neumann algebra ${\rm L}^\infty(G)$ i.e. $h_G=\langle\xi_G,\cdot\xi_G\rangle$ when viewed as a state on $C_r(G)$ or a normal state on $\Linf(G)$. We recall that $h_G$ is faithful on both $C_r(G)$ and $\Linf(G)$. The modular group $\sigma_t^G$ of $h_G$ and the scaling group $\tau_t^G$ of $G$ are given by $(\id\ot\sigma^G_t)(u)=(Q_u^{it}\ot 1)u(Q_u^{it}\ot 1)$ and $(\id\ot\tau^G_t)(u)=(Q_u^{it}\ot 1)u(Q_u^{-it}\ot 1)$ for all $t\in\R$. It is well known that the scaling group of $G$ is the unique one parameter group $(\tau_t^G)_t$ of $\Linf(G)$ such that $\Delta\circ\sigma_t^G=(\tau_t^G\ot\sigma_t^G)\circ\Delta$  for all $u\in\Rep(G)$ and $t\in\R$. Moreover, the scaling group preserves that Haar state: $h_G\circ\tau_t^G=h_G$. The compact quantum group $G$ is said to be of \textit{Kac type} whenever $h_G$ is a trace which happens if and only if $Q_u=1$ for all $u\in\Rep(G)$.

\vspace{0.2cm}

\noindent We say that a compact quantum group $F$ is a \textit{dual quantum subgroup} of a compact quantum group $G$ if there is a faithful *-morphism $C(F)\hookrightarrow C(G)$ intertwining the comultiplications. We often see $C(F)$ as a C*-subalgebra of $C(G)$. Recall the following proposition from \cite{FT24}.
\begin{proposition}\label{dualqsg}
Let $G$ and $F$ be CQG. The following data are equivalent.
\begin{itemize}
    \item $\iota\,:\, C(F)\rightarrow C(G)$ is a faithful unital $*$-homomorphism intertwining the comultiplications.
    \item $\iota\,:\,\Pol(F)\rightarrow\Pol(G)$ is a faithful unital $*$-homomorphism intertwining the comultiplications.
    \item $\iota\,:\,C_r(F)\rightarrow C_r(G)$ is a faithful unital $*$-homomorphism intertwining the comultiplications.
\end{itemize} If one of the following equivalent conditions is satisfied, we view $\Pol(F)\subset\Pol(G)$, $C(F)\subset C(G)$ and $C_r(F)\subset C_r(G)$. Then, the unique linear map $E_F\,:\,\Pol(G)\rightarrow \Pol(F)$ such that
$$(\id\ot E_F)(u^x)=\left\{\begin{array}{lcl}u^x&\text{if}&x\in\Irr(F),\\0&\text{if}&x\in\Irr(G)\setminus\Irr(F).\end{array}\right.$$
has a unique ucp extension to a map $E_F\,:\,C(G)\rightarrow C(F)$ which is a Haar-state-preserving conditional expectation onto the subalgebra $C(F)\subset C(G)$. Hence, there exists a unique conditional expectation $E_F^r\,:\, C_r(G)\rightarrow C_r(F)$ such that $E_F^r\circ\lambda_G=\lambda_F\circ E_F$. Moreover, $E^r_F$ is faithful and Haar-state-preserving and extends to a Haar-state preserving normal faithful conditional expectation $E_F''\,:\,\Linf(G)\rightarrow\Linf(F)$.
\end{proposition}

\begin{remark}\label{RmkEF} The following holds.
\begin{enumerate}
\item Note that, since $\lambda_G$ is surjective, $E_F^r$ is faithful and $E^r_F\circ\lambda_G=\lambda_F\circ E_F$, $C_r(\GG)$ is actually the reduced C*-algebra of the ucp map $\lambda_F\circ E_F$ i.e. $C_r(\GG)=C^*_r(\lambda_F\circ E_F)$.
\item If $a\in\Pol(G)$ and $E_F(a)=0$ then we can write $\Delta(a)$ as a finite sum $\Delta(a)=\sum_k a_k\ot b_k$ with $a_k,b_k\in\Pol(G)$ and $E_F(a_k)=0=E_F(b_k)$ for all $k$. Indeed, assuming that $a\neq 0$ and since $a\in\Pol(G)$, there exists a finite subset $S\subset\Irr(G)$ and non-zero $\lambda^x_{ij}\in\C$, $1\leq i,j\leq{\rm dim}(x)$ such that $a=\sum_{x\in S,i,j}\lambda^x_{ij}u^x_{ij}$. Then, $\Delta(a)=\sum_{x\in S,i,j}\lambda^x_{i,j,s} u^x_{is}\ot u^x_{sj}$ and $E_F(a)=\sum_{x\in S\cap\Irr(F),i,j}\lambda^x_{ij}u^x_{ij}$. Since $E_F(a)=0$ and coefficients of representations $u^x$ for $x\in\Irr(G)$ form a basis of $\Pol(G)$, we deduce that $S\subset\Irr(G)\setminus\Irr(F)$.
\item $E_F$ is the unique Haar state preserving conditional expectation, at the full, reduced, von Neumann or algebraic level. Indeed, if $E\,:\, C(G)\rightarrow C(F)$ is a conditional expectation preserving the Haar states then there exists a unique conditional expectation $E^r\,:\, C_r(G)\rightarrow C_r(F)$ such that $E^r\circ\lambda_G=\lambda_F\circ E$ and $E^r$ still preserves the Haar states. By faithfulness of the Haar states at the reduced level, we deduce that $E^r=E^r_F$. Indeed, $h_F((E^r_F(x)-E^r(x))b)=h_F(E_F^r(xb))-h_F(E^r(xb))=h_G(xb)-h_G(xb)=0$ for all $b\in C_r(F)$ and all $x\in C_r(G)$. Since $h_F$ is faithful on $C_r(F)$ it gives the result. It implies that $E=E_F$ on $\Pol(G)$ and so $E=E_F$ on $C(G)$. For the reduced, von Neumann and algebraic level, the statement is clear since the Haar states are faithful in those cases.
\end{enumerate}
\end{remark}

\subsection{Actions on finite dimensional C*-algebras}\label{SectionQAut}

\noindent A (right) \textit{action} of a CQG $H$ on a finite dimensional unital C*-algebra $B$ is a unital $*$-homomorphism $\beta\,:\, B\rightarrow B\ot \Pol(H)$ such that $(\id_B\ot\Delta_H)\beta=(\beta\ot\id)\beta$ and $(\id\ot\varepsilon_H)\beta=\id_B$. In particular, $\beta$ has to be injective. The action is called \textit{ergodic} if the fixed point C*-subalgebra is trivial i.e. $B^H:=\{b\in B\,:\,\beta(b)=b\ot 1\}=\C1_B$. Note that, by invariance of Haar measure, the map $(\id_B\ot h_H)\beta$ is a conditional expectation onto $B^H$.
The action is called \textit{faithful} if $C(H)$ is generated, as a C*-algebra, by $\{(\omega\ot\id)\beta(b)\,:\,b\in B,\,\omega\in B^*\}$. Given a state $\psi$ on $B$, we say that the action $\beta$ is $\psi$-\textit{preserving}, or that $\psi$ is a $\beta$-\textit{invariant} state, if $(\psi\ot\id)\beta=\psi(\cdot)1$. Note that if $\beta$ is ergodic then there exists a unique $\beta$-invariant state $\psi\in B^*$, given by $\psi(b)1=(\id\ot h_H)\beta(b)$. In particular, since $h_H$ is faithful on $\Pol(H)$ and $\beta$ is injective, if $\beta$ is ergodic then its unique invariant state $\psi\in B^*$ is faithful.

\vspace{0.2cm}

\noindent For the remainder of this paper, $\psi$ will be a faithful state on the finite dimensional C*-algebra $B$. View $B$ as a Hilbert space with scalar product $\langle a,b\rangle_{\psi}:=\psi(a^*b)$ for all $a,b\in B$. $B$ acts on the Hilbert space $B$ by left multiplication: the map $L_B\,:\,B\rightarrow\mathcal{L}(B)$, $L_B(b)(a)=ba$ is a faithful unital $*$-homomorphism.

\vspace{0.2cm}

\noindent Let $\beta\,:\,B\rightarrow B\ot C(H)$ be a $\psi$-preserving action of the CQG $H$ on $B$. From the action $\beta\,:\,B\rightarrow B\ot C(H)$ we consider the unique operator $u\in\Lcal(B)\ot C(H)=\mathcal{K}_{C(H)}(B\ot C(H))$ satisfying $u(b\ot a):=\beta(b)(1_B\ot a)$ for all $b\in B$, $a\in C(H)$. Let $m\,:\, B\ot B\rightarrow B$ and $\eta\,:\,\C\rightarrow B$ be the multiplication and the unit of $B$ respectively. The following simple lemma is well known \cite{Ba99}.

\begin{lemma}\label{LemRepAction}
Given any action $\psi$-preserving action $\beta\,:\, B\rightarrow B\ot C(H)$ the associated unitary representation $u\in\Lcal(B)\ot C(H)$ satisfies:
\begin{itemize}
\item $u$ is unitary representation of $H$ on $B$.
\item $m\in\Mor_H(u\ot u,u)$.
\item $\eta\in\Mor_H(1,u)$.
\item $u(L_B(b)\ot 1)u^*=(L_B\ot\id)(\beta(b))$ for all $b\in B$.
\item The action $\beta$ is faithful if and only if $C(H)$ is generated by the coefficients of $u$.
\item The action $\beta$ is ergodic if and only if $\Mor_H(1,u)=\C\eta$.
\end{itemize}
\end{lemma}

\begin{remark}\label{RmkConjugate}
Let $u\in\Lcal(B)\ot C(H)$ be the representation associated to a $\psi$-preserving action $\beta\,:\, H\curvearrowright B$ then $u\simeq\ubar$. Indeed, it is easy to check that $s_u=s_{\ubar}:=m^\star\eta\in{\rm Mor}(1,u\ot u)$ is a pair solving the conjugate equation of $(u,u)$, where $m^\star\,:\, B\rightarrow B\ot B$ is the adjoint of $m$ with respect to the $\psi$-scalar product on $B$.
\end{remark}

\noindent Denote by ${\rm Aut}^+(B,\psi)$ the quantum automorphism group of $(B,\psi)$. It is the universal CQG acting on $(B,\psi)$ in the sense that there is a $\psi$-preserving action $\beta\,:\, B\rightarrow B\otimes C(\GB)$ and, if $G$ is any CQG acting on $B$ via $\alpha\,:\,B\rightarrow B\ot C(G)$ in a $\psi$-preserving way then, there exists a unique unital $*$-homomorphism $\pi\,:\,C({\rm Aut}^+(B,\psi))\rightarrow C(G)$ such that $\pi$ intertwines the comultiplications and the actions i.e. $(\id\ot\pi)\beta=\alpha$. The uniqueness, up to a canonical isomorphism, of $\GB$ is clear, and the action $\beta$ is faithful. The unitary $u\in\Lcal(B)\ot C(\GB)$ associated with $\beta$ is called the \textit{fundamental representation} of $\GB$. The existence of $\GB$ is also clear (see \cite{Ba99,Ba02}), it suffices to define $C({\rm Aut}^+(B,\psi))$ as the universal unital C*-algebra generated by coefficients of $u\in\mathcal{L}(B)\ot C({\rm Aut}^+(B,\psi))$ with relations:

\begin{itemize}
\item $u$ is unitary.
\item $m\in\Mor(u\ot u,u)$.
\item $\eta\in\Mor(1,u)$.
\end{itemize}

\noindent The comultiplication $\Delta_B\,:\,C({\rm Aut}^+(B,\psi))\rightarrow C({\rm Aut}^+(B,\psi))\ot C({\rm Aut}^+(B,\psi))$ is then defined, by the universal property of $C(\GB)$, as the unique unital $*$-homomorphism such that $u$ is a representation i.e. $(\id\ot\Delta_B)(u)=u_{12} u_{13}$. We sometimes write $\beta_B\,:\,\GB\curvearrowright B$ for the universal $\psi$-preserving  action of $\GB$ on $B$.

\vspace{0.2cm}

\noindent Write $B=\bigoplus_{\kappa=1}^KM_{N_\kappa}(\C)$ and $\psi=\sum_{\kappa}\Tr(Q_\kappa\cdot)$, where $Q_\kappa\in M_{N_\kappa}(\C)$ is a positive definite matrix with $\sum_\kappa\Tr(Q_\kappa)=1$. Diagonalising $Q_\kappa$, we find matrix units $(e^\kappa_{ij})_{1\leq i,j\leq N_\kappa}$ in $M_{N_\kappa}(\C)$ such that $Q_\kappa e^\kappa_{ij}=\lambda_{\kappa,i}e^\kappa_{ij}$ for all $1\leq i,j\leq N_\kappa$ and all $1\leq \kappa\leq K$. Since $\psi((e^{\kappa}_{ij})^*e^\zeta_{kl})=\delta_{\kappa,\zeta}\delta_{ik}\Tr(Q_\kappa e^\kappa_{jl})=\delta_{\kappa,\zeta}\delta_{ik}\delta_{jl}\lambda_{\kappa,j}$, we deduce that $(b^\kappa_{ij})_{\kappa,i,j}$ is an orthonormal basis of the Hilbert space $B$, where $b^\kappa_{ij}:=\lambda^{-\frac{1}{2}}_{\kappa,j}e^\kappa_{ij}$. Note that $b^\kappa_{ip}b^{\zeta}_{qj}=\delta_{\kappa\zeta}\delta_{pq}\lambda^{-\frac{1}{2}}_{\kappa,p}b^\kappa_{ij}$ and $(b^\kappa_{ij})^*=\left(\frac{\lambda_{\kappa,i}}{\lambda_{\kappa,j}}\right)^{\frac{1}{2}}b^\kappa_{ji}$.

\begin{remark}\label{RmkSpectralDecomposition}
 The positive operator $mm^\star$ is actually diagonal in the orthonormal basis $(b^\kappa_{rs})$ of $B$. Indeed, a direct computation shows that $mm^\star(b^\kappa_{rs})={\rm Tr}(Q_\kappa^{-1})b^\kappa_{rs}$ for all $\kappa,r,s$. Writing $\delta_1,\dots\delta_n>0$ the pairwise distinct values taken by ${\rm Tr}(Q_\kappa^{-1})$ we see that ${\rm Sp}(mm^\star)=\{\delta_1,\dots,\delta_n\}$ and, as C*-algebras, $B=\bigoplus_i B_i$, where
$$B_i:={\rm Span}\{b^\kappa_{rs}\,:\,1\leq\kappa\leq K\text{ such that }{\rm Tr}(Q_\kappa^{-1})=\delta_i,\,1\leq r,s\leq N_\kappa\}=\bigoplus_{\kappa,{\rm Tr} (Q_\kappa^{-1})=\delta_i}M_{N_\kappa}(\C)$$
is the $\delta_i$-eigenspace of $mm^\star$.
\end{remark}

\noindent Recall that a faithful state $\psi\in B^*$ is called a \textit{$\delta$-form} ($\delta>0$) if $mm^\star=\delta\id_B$. Note that our convention is a little bit different from the usual one \cite{Ba99,Ba02,Br13}.

\begin{remark}\label{remadelta}\label{ergodelta}
The following holds.
\begin{enumerate}
\item \textit{$\psi$ is a $\delta$-form $\Leftrightarrow$  $\kappa\mapsto{\rm Tr}(Q_\kappa^{-1})$ is constant and in that case $\delta={\rm Tr}(Q_\kappa^{-1})$ $\forall\kappa$. In general, writing $B=\bigoplus B_i$, each $\psi_i:=\frac{1}{\psi(1_{B_i})}\psi\vert_{B_i}$ is a $\delta_i$-form on $B_i$.}

\noindent It is a direct consequence of Remark \ref{RmkSpectralDecomposition}.

\item \textit{$\beta_B\,:\,\GB\curvearrowright B$ is always faithful and $\beta_B$ is ergodic if and only if  $\psi$ is a $\delta$-form.}

\noindent Indeed, if $\psi$ is a $\delta$-form then it is known \cite{Ba99,Ba02,Br13} that ${\Mor}(1,u)=\C\eta$ so $\beta_B$ is ergodic. If $\psi$ is not a $\delta$-form then, by the spectral decomposition from Remark \ref{RmkSpectralDecomposition}, $n\geq 2$ and, for $1\leq i\neq j\leq n$, one has $\delta_i\neq \delta_j$ so that $p_i\eta$ and $p_j\eta\in B$ are linearly independent since $1_B=\sum_{\kappa,r}\lambda_\kappa^{-1/2}b^\kappa_{rr}$ which gives $p_k(1_B)=\sum_{\kappa,\,{\rm Tr}(Q_\kappa^{-1})=\delta_k}\lambda_\kappa^{-1/2}b^\kappa_{rr}$ for all $1\leq k\leq n$. Note that, since $mm^\star$ is in the C*-algebra ${\rm End}(u)\subset\Lcal(B)$, we also have $p_k\in{\rm End}(u)$ for all $k$. Hence, $p_i\eta,p_j\eta\in{\rm Mor}(1,u)$ which implies that $\beta_B$ is not ergodic.

\item \textit{For any $\psi$-preserving action $\beta\,:\,H\curvearrowright B$, if $\beta$ is ergodic then $\psi$ is a $\delta$-form.}

\noindent Indeed, assume that $\psi$ is not a $\delta$-form and write $\pi\,:\,C(\GB)\rightarrow C(H)$ the canonical map intertwining the comultiplications and such that $(\id\ot\pi)\beta_B=\beta$. In particular, $B^{\GB}\subseteq B^H$. By $(2)$, $\beta_B$ is not ergodic so ${\rm dim}(B^{\GB})\geq 2$ which implies that $\beta$ is not ergodic either.
\end{enumerate}
\end{remark}

\noindent Given any $\psi$-preserving action $\beta\,:\, B\rightarrow B\ot C(H)$ with associated representation $u\in\Lcal(B)\ot C(H)$ we write:
$$u=\sum_{\kappa,\zeta,i,j,k,l}b^{\kappa}_{ij}\circ (b^{\zeta}_{kl})^{\star}\ot u^{ij,\kappa}_{kl,\zeta},$$
where we view $b\in B$ as an operator in $\Lcal(\C,B)$ and $b^\star\in\Lcal(B,\C)$ is the adjoint operator, the action $\beta$ is then given by  $\beta(b^{\zeta}_{kl})=\sum_{i,j,\kappa} b^{\kappa}_{ij}\ot u^{ij,\kappa}_{kl,\zeta}$. The relations $u$ unitary, $m\in\Mor(u\ot u,u)$ and $\eta\in\Mor(1,u)$ are equivalent to the following relations on the coefficients of $u$.

\begin{itemize}
\item $\sum_{r=1}^{N_\gamma}\lambda_{\gamma,r}^{-\frac{1}{2}}u^{ir,\gamma}_{kp,\kappa}u^{rj,\gamma}_{ql,\zeta}=\delta_{\kappa\zeta}\delta_{pq}\lambda_{\kappa,p}^{-\frac{1}{2}}u^{ij,\gamma}_{kl,\kappa}$ and $\sum_{r=1}^{N_\gamma}\lambda_{\gamma,r}^{-\frac{1}{2}}u^{ip,\kappa}_{kr,\gamma}u^{qj,\zeta}_{rl,\gamma}=\delta_{\kappa\zeta}\delta_{pq}\lambda_{\kappa,p}^{-\frac{1}{2}}u^{ij,\kappa}_{kl,\gamma}$.
\item $\sum_{\gamma,r}\lambda_{\gamma,r}^{\frac{1}{2}}u^{ij,\kappa}_{rr,\gamma}=\delta_{ij}\lambda_{\kappa,i}^{\frac{1}{2}}$ and $\sum_{\gamma,r}\lambda_{\gamma,r}^{\frac{1}{2}}u^{rr,\gamma}_{kl,\zeta}=\delta_{kl}\lambda_{\zeta,k}^{\frac{1}{2}}$.
\item $\left(u^{ij,\kappa}_{kl,\zeta}\right)^*=\left(\frac{\lambda_{\kappa,i}}{\lambda_{\zeta,k}}\right)^{-\frac{1}{2}}\left(\frac{\lambda_{\kappa,j}}{\lambda_{\zeta,l}}\right)^{\frac{1}{2}}u^{ji,\kappa}_{lk,\zeta}$.
 \end{itemize}
And the comultiplication is given on coefficients by $\Delta(u^{ij,\kappa}_{kl,\zeta})=\sum_{r,s,\gamma}u^{ij,\kappa}_{rs,\gamma}\ot u^{rs,\gamma}_{kl,\zeta}$. In particular $C(\GB)$ is the universal unital C*-algebra generated by elements $u^{ij,\kappa}_{kl,\zeta}\in C(\GB)$ for $1\leq\kappa,\zeta\leq K$, $1\leq i,j\leq N_\kappa$, $1\leq k,l\leq N_\zeta$ satisfying the above relations.

\begin{proposition}\label{PropQAut}
Let $\beta\,:\, H\curvearrowright B$ be an ergodic action on the finite dimensional C*-algebra $B$ with unique invariant state $\psi\in B^*$ and associated representation $u\in\mathcal{L}(B)\ot C(H)$. The following holds.
\begin{enumerate}
\item $h(u^{ij,\kappa}_{kl,\gamma})=\delta_{ij}\delta_{kl}(\lambda_{\kappa,i}\lambda_{\gamma,k})^{\frac{1}{2}}$.
\item $Q_u$ is the positive modular operator $\nabla_\psi$ of $\psi$. In particular,
\begin{enumerate}
\item ${\rm dim}_q(u)=\delta$.
\item $\sigma^h_t(u^{ij,\kappa}_{kl,\gamma})=\left(\frac{\lambda_{\gamma,k}\lambda_{\kappa,i}}{\lambda_{\gamma,l}\lambda_{\kappa,j}}\right)^{it}u^{ij,\kappa}_{kl,\gamma}$ and
$(\id\ot\sigma^h_t)\circ\beta=(\sigma_t^\psi\ot\id)\circ\beta\circ\sigma_t^\psi$.
\item $\tau^h_t(u^{ij,\kappa}_{kl,\gamma}) = \left(\frac{\lambda_{\gamma,l}\lambda_{\kappa,i}}{\lambda_{\gamma,k}\lambda_{\kappa,j}}\right)^{it}u^{ij,\kappa}_{kl,\gamma}$ and
$(\id\ot\tau^h_t)\circ\beta=(\sigma_t^\psi\ot\id)\circ\beta\circ\sigma_{-t}^{\psi}$.
\end{enumerate}
\end{enumerate}
\end{proposition}

\begin{proof}\noindent$(1)$. Since $(\id\ot h)(u)$ is the orthogonal projection onto ${\rm Mor}(1,u)=\C\eta$ and $\Vert\eta\Vert^2=\psi(1)=1$ one has $(\id\ot h)(u)=\eta\eta^\star$. Writing $\eta=\sum_{\kappa,i}\lambda_{\kappa,i}^{\frac{1}{2}}b^\kappa_{ii}$ we have
$$\eta\eta^\star=\sum_{\kappa,\gamma,i,k}b^\kappa_{ii}\circ(b^\gamma_{kk})^\star(\lambda_{\kappa,i}\lambda_{\gamma,k})^{\frac{1}{2}}=(\id\ot h)(u).$$

\vspace{0.2cm}

\noindent$(2)$. Recall from Remark \ref{RmkConjugate} that the pair $(m^\star\eta,m^\star\eta)$ solves the conjugate equation for $(u,u)$. Let us show, under the hypothesis of ergodicity, that the pair $(m^\star\eta,m^\star\eta)$ is standard. It suffices to show that, for all $T\in{\rm End}(u)$, $\varphi_u(T)=\psi_u(T)$, where $\varphi_u(T):=\eta^\star m(\id_B\ot T)m^\star\eta$ and $\psi_u(T):=\eta^\star m(T\ot \id_B)m^\star\eta$. Note that both $m(T\ot\id_B)m^\star\eta$ and $m(\id_B\ot T)m^\star\eta$ are in $\Mor(1,u)$. Since the action is ergodic, there exist $\lambda,\mu\in\C$ such that $m(T\ot\id_B)m^\star\eta=\lambda\eta$ and $m(\id_B\ot T)m^\star\eta=\mu\eta$. It suffices to show that $\lambda=\mu$. Write, in the basis of matrix units $(e^\kappa_{ij})_{\kappa,i,j}$ in $\Lcal(B)$ such that $Q_\kappa e^{\kappa}_{ij}=\lambda_{\kappa,i}e^{\kappa}_{ij}$, $T(e^\kappa_{ij})=\sum_{\zeta,k,l}T^{ij,\kappa}_{kl,\zeta}e^\zeta_{kl}$, where $T^{ij,\kappa}_{kl,\zeta}\in\C$ and note that $m^\star\eta=\sum_{\kappa,i,j}\lambda_{\kappa,j}^{-1}e^\kappa_{ij}\ot e^\kappa_{ji}$. A direct computation gives:
$$m(T\ot\id_B)m^\star\eta=\sum_{\kappa,i,j}\left(\sum_k\lambda_{\kappa,k}^{-1}T^{jk,\kappa}_{ik,\kappa}\right)e^\kappa_{ij}\quad\text{and}\quad m(\id_B\ot T)m^\star\eta=\sum_{\kappa,i,j}\left(\sum_k\lambda_{\kappa,k}^{-1}T^{ki,\kappa}_{kj,\kappa}\right)e^\kappa_{ij}.$$
Since $m(T\ot\id_B)m^\star\eta=\lambda \sum_{\kappa,i}e^\kappa_{ii}$ we find that the map $\left[(\kappa,i)\mapsto\sum_k\lambda_{\kappa,k}^{-1}T^{ik,\kappa}_{ik,\kappa}\right]$ is constant, equals to $\lambda$. Similarly, the map $\left[(\kappa,i)\mapsto\sum_k\lambda_{\kappa,k}^{-1}T^{ki,\kappa}_{ki,\kappa}\right]$ is constant, equals to $\mu$. Hence, for all $1\leq\kappa\leq K$ and all $1\leq k\leq N_\kappa$, one has $\mu=\lambda_{\kappa,1}^{-1}T^{1k,\kappa}_{1k,\kappa}+\sum_{i=2}^{N_\kappa}\lambda_{\kappa,i}^{-1}T^{ik,\kappa}_{ik,\kappa}$ which implies that $T^{1k,\kappa}_{1k,\kappa}=\lambda_{\kappa,1}(\mu-\sum_{i=2}^{N_\kappa}\lambda_{\kappa,i}^{-1} T^{ik,\kappa}_{ik,\kappa})$. It follows that:

\begin{eqnarray*}
\lambda&=&\sum_{k=1}^{N_\kappa}\lambda_{\kappa,k}^{-1}T^{1k,\kappa}_{1k,\kappa}=\lambda_{\kappa,1}\sum_{k=1}^{N_\kappa}\lambda_{\kappa,k}^{-1}(\mu-\sum_{i=2}^{N_\kappa}\lambda_{\kappa,i}^{-1} T^{ik,\kappa}_{ik,\kappa})=\lambda_{\kappa,1}\delta\mu-\lambda_{\kappa,1}\sum_{i=2}^{N_\kappa}\lambda_{\kappa,i}^{-1}\sum_{k=1}^{N_\kappa}\lambda_{\kappa,k}^{-1}T^{ik,\kappa}_{ik,\kappa}\\
&=&\lambda_{\kappa,1}\delta\mu-\lambda_{\kappa,1}\sum_{i=2}^{N_\kappa}\lambda_{\kappa,i}^{-1}\lambda
=\lambda_{\kappa,1}\delta\mu-\lambda_{\kappa,1}(\delta-\lambda_{\kappa,1}^{-1})\lambda=\lambda_{\kappa,1}\delta(\mu-\lambda)+\lambda.
\end{eqnarray*}

We deduce that $\lambda=\mu$ so $(m^\star\eta,m^\star\eta)$ is standard. It follows that $Q_u=J_u^*J_u$, where $\langle J_u(a),b\rangle=\langle m^\star\eta,a\ot b\rangle=\langle 1_B,ab\rangle=\psi(ab)=\langle a^*,b\rangle$ for all $a,b\in B$. Hence, $J_u(a)=a^*=S_\psi(a)$, the modular operator of $\psi$. It follows, by definition, that $Q_u=S_\psi^*S_\psi=\nabla_\psi$. Write $S_\psi=J_\psi\nabla^{\frac{1}{2}}_\psi$. A direct computation shows that $J_\psi$ is the unique anti-unitary on $B$ such that $J_\psi(b^\gamma_{kl})=b^\gamma_{lk}$ and $\nabla_\psi(b^\gamma_{kl})=\frac{\lambda_{\gamma,k}}{\lambda_{\gamma,l}}b^\gamma_{kl}$. Hence, ${\rm dim}_q(u)={\rm Tr}(\nabla_\psi)=\sum_\kappa{\rm Tr}(Q_\kappa){\rm Tr}(Q_\kappa^{-1})=\delta\sum_\kappa{\rm Tr}(Q_\kappa)=\delta$. Since $(\id\ot\sigma_t^h)(u)=(\nabla_\psi^{it}\ot 1)u(\nabla_\psi^{it}\ot 1)$, it follows that $(\id\ot\sigma^h_t)(u)=\sum\left(\frac{\lambda_{\gamma,k}\lambda_{\kappa,i}}{\lambda_{\gamma,l}\lambda_{\kappa,j}}\right)^{it}b^\kappa_{ij}(b^\gamma_{kl})^\star\ot u^{ij,\kappa}_{kl,\gamma}$ which proves the first part of $(b)$. Moreover, since $\sigma_t^\psi(b)=\nabla_\psi^{it}b\nabla_\psi^{-it}$ for all $b\in B$ we find $\sigma_t^\psi(b^\gamma_{kl})=\left(\frac{\lambda_{\gamma,k}}{\lambda_{\gamma,l}}\right)^{it}b^\gamma_{kl}$ from which we can easily check the last part of $(b)$. $(c)$ is proven the same way using the relation:
\begin{equation*}
    (\id\ot \tau_t^h)(u) = (\nabla_\psi^{it}\ot 1)u(\nabla_\psi^{-it}\ot 1).\qedhere
\end{equation*}
\end{proof}

\begin{remark}
Let $\beta\,:\,H\curvearrowright B$ be any ergodic $\psi$-preserving action of the CQG $H$ on the finite dimensional C*-algebra $B$. It follows from statement $(b)$ of Proposition \ref{PropQAut} that if $H$ is Kac then $\psi$ is tracial. Moreover, if $\beta$ is a faithful action and $\psi$ is tracial then $H$ is Kac.
\end{remark}

\noindent For the remainder of this section, we fix an ergodic action $\beta\,:\,H\curvearrowright B$ on the finite dimensional C*- algebra $B$ with unique invariant state $\psi\in B^*$ and associated representation $u\in\Lcal(B)\ot C(H)$.

\vspace{0.2cm}

\noindent For $1\leq\kappa\leq K$ consider the unique unital $*$-homomorphism $\chi_\kappa\,:\, B\rightarrow M_{N_\kappa}(\C)$ such that $\chi_\kappa(b^{\zeta}_{ij})=\delta_{\kappa\zeta}b^{\kappa}_{ij}$
and define the unital $*$-homomorphism 
$$\beta_\kappa:=(\chi_\kappa\ot\id)\beta\,:\, B\rightarrow M_{N_\kappa}(\C)\ot C(H)$$
so that $\beta_\kappa(b^\gamma_{kl})=\sum_{i,j} b^{\kappa}_{ij}\ot u^{ij,\kappa}_{kl,\gamma}$. We also define the state $\psi_\kappa:=\frac{{\rm Tr}(Q_\kappa\cdot)}{{\rm Tr}(Q_\kappa)}\in M_{N_\kappa}(\C)^*$. Note that $\sigma_t^{\psi_\kappa}\circ\chi_\kappa=\chi_\kappa\circ\sigma_t^\psi$ for all $t\in\R$. Note also that for all $1\leq \kappa \leq K$ the fact that $\psi$ is a $\delta$-form implies that $\Tr(Q_\kappa^{-1})=\delta$.

\vspace{0.2cm}

\noindent Given a faithful state $\omega\in M_N(\C)^*$, $\omega={\rm Tr}(Q\cdot)$ where $Q\in M_N(\C)$, $Q>0$ of trace $1$, we define its \textit{inverse state} $\omega^{-1}:=\frac{{\rm Tr}(Q^{-1}\cdot)}{{\rm Tr}(Q^{-1})}\in M_{N}(\C)^*$. Note that $\omega^{-1}$ is faithful and $\sigma_t^{\omega^{-1}}=\sigma_{-t}^\omega$ for all $t\in\R$.

\begin{proposition}\label{PropUcpSection}
For all $1\leq\kappa\leq K$, the following holds.
\begin{enumerate}
\item $(\id\ot h)\beta_\kappa=\psi(\cdot)1_\kappa$. In particular, $\beta_\kappa$ is a faithful homomorphism.
\item For all $t\in\R$, $(\sigma_{-t}^{\psi_\kappa}\ot\sigma_t^h)\beta_\kappa=\beta_\kappa\sigma_t^\psi$  and $(\sigma_{-t}^{\psi_\kappa}\ot\tau_t^h)\beta_\kappa=\beta_\kappa\sigma_{-t}^\psi$.
\item There exists a unique ucp map $E_\kappa\,:\, M_{N_\kappa}(\C)\ot\Linf(H)\rightarrow B$ such that
$$\psi\circ E_\kappa=\psi_\kappa^{-1}\ot h\quad\text{and}\quad E_\kappa\circ\beta_\kappa=\id_B.$$
Moreover, $E_\kappa$ is normal, faithful and, for all $1\leq r,s\leq N_\kappa$ and all $a\in\Linf(H)$,
$$E_\kappa(b^\kappa_{rs}\ot a)=\frac{1}{\delta\lambda^2_{\kappa,s}}\sum_{\gamma,k,l}h\left(\left(u^{rs,\kappa}_{kl,\gamma}\right)^*a\right)b^\gamma_{kl}.$$
In particular, $E_\kappa(e^\kappa_{rs}\ot 1)=\frac{\delta_{rs}}{\delta\lambda_{\kappa,r}}1_B$.

\end{enumerate}
\end{proposition}

\begin{proof}
$(1)$. The unique $\beta$-invariant state $\psi$ is defined by $(\id\ot h)\beta=\psi(\cdot)1$. Applying $\chi_\kappa$ gives $(1)$. Since $\psi$ is faithful, the maps $\beta_\kappa$ also are.

\vspace{0.2cm}

\noindent$(2)$. Follows from Assertion $(2.b)$ and $(2.c)$ in Proposition \ref{PropQAut}.

\vspace{0.2cm}

\noindent$(3)$. Since $\sigma_t^{\psi_\kappa^{-1}}=\sigma_{-t}^{\psi_\kappa}$ it follows from $(2)$ that the subalgebra $\beta_\kappa(B)\subset M_{N_\kappa}(\C)\ot \Linf(H)$ is globally invariant under the modular group of $\psi_\kappa^{-1}\ot h$. By Takesaki's Theorem there exists a unique $\psi_\kappa^{-1}\ot h$-preserving conditional expectation $\E_\kappa\,:\,M_{N_\kappa}(\C)\ot\Linf(H)\rightarrow\beta_\kappa(B)$, which is normal and faithful. Then, $E_\kappa:=\beta_\kappa^{-1}\circ \E_\kappa$ satisfies $E_\kappa\circ\beta_\kappa=\id_B$ and, using $(1)$ and the fact that $\E_\kappa$ is state preserving, we deduce that $\psi\circ E_\kappa=\psi_\kappa^{-1}\ot h$. From this property and since $\psi_\kappa^{-1}\ot h$ is faithful, we deduce that $E_\kappa$ is faithful. By $(1)$, $(\beta_\kappa(b^\gamma_{kl}))_{\gamma,k,l}$ is a $\psi_\kappa^{-1}\ot h$-orthonormal basis of the finite dimensional subalgebra $\beta_\kappa(B)$ which easily implies that
$$\E_\kappa(X)=\sum_{\gamma,k,l}(\psi_\kappa^{-1}\ot h)(\beta_\kappa(b^\gamma_{kl})^*X)\beta_\kappa(b^\gamma_{kl})\text{ for all }X\in M_{N_\kappa}(\C)\ot\Linf(H).$$
Using the formulas $(b^\gamma_{kl})^*=\left(\frac{\lambda_{\gamma,k}}{\lambda_{\gamma,l}}\right)^{\frac{1}{2}}b^\gamma_{lk}$ and $b^\kappa_{ip}b^\kappa_{qj}=\lambda_{\kappa,p}^{-\frac{1}{2}}\delta_{pq}b^\kappa_{ij}$ we get:
$$\beta_\kappa(b^\gamma_{kl})^*(b^\kappa_{rs}\ot a)=\left(\frac{\lambda_{\gamma,k}}{\lambda_{\gamma,l}}\right)^{\frac{1}{2}}\lambda_{\kappa,r}^{-\frac{1}{2}}\sum_ib^\kappa_{is}\ot u^{ir,\kappa}_{lk,\gamma}a.$$
Since $\psi_\kappa^{-1}(b^\kappa_{is})={\rm Tr}(Q_\kappa^{-1})^{-1}\lambda_{\kappa,s}^{-\frac{3}{2}}\delta_{is}$  and $\left(u^{rs,\kappa}_{kl,\gamma}\right)^*=\left(\frac{\lambda_{\kappa,s}\lambda_{\gamma,k}}{\lambda_{\kappa,r}\lambda_{\gamma,l}}\right)^{\frac{1}{2}}u^{sr,\kappa}_{lk,\gamma}$ we deduce that:
\begin{align*}
\E_\kappa(b^\kappa_{rs}\ot a)&=\sum_{\gamma,k,l}(\psi_\kappa^{-1}\ot h)(\beta_\kappa(b^\gamma_{kl})^*(b^\kappa_{rs}\ot a))\beta_\kappa(b^\gamma_{kl})\\
&={\rm Tr}(Q_\kappa^{-1})^{-1}\lambda_{\kappa,r}^{-\frac{1}{2}}\lambda_{\kappa,s}^{-\frac{3}{2}}\sum_{\gamma,k,l}\left(\frac{\lambda_{\gamma,k}}{\lambda_{\gamma,l}}\right)^{\frac{1}{2}}h(u^{sr,\kappa}_{lk,\gamma}a)\beta_\kappa(b^\gamma_{kl})\\
&={\rm Tr}(Q_\kappa^{-1})^{-1}\lambda_{\kappa,s}^{-2}\sum_{\gamma,k,l}h\left(\left(u^{rs,\kappa}_{kl,\gamma}\right)^*a\right)\beta_\kappa(b^\gamma_{kl}).
\qedhere
\end{align*}
\end{proof}

\begin{remark}
We will use the same symbol to denote the restriction $$E_\kappa\,:\,M_{N_\kappa}(\C)\ot C_r(H)\rightarrow B$$ which is still ucp and faithful, as well as $E_\kappa:=E_\kappa\circ(\id\ot\lambda_{H})\,:\,M_{N_\kappa}(\C)\ot C(H)\rightarrow B$, still ucp but not even GNS-faithful if $H$ is not coamenable. Those maps satisfy the same algebraic properties as the initial map.
\end{remark}

\noindent We will need the following Lemma.

\begin{lemma}\label{LemE} The following holds.
\begin{enumerate}
\item For all $x\in \Pol(H)$, $1\leq\kappa,\zeta,\gamma\leq K$, $1\leq i,j\leq N_\kappa$, $1\leq k,l\leq N_\gamma$, $1\leq r,p\leq N_\zeta$,
$$\sum_q\lambda_{\gamma,q}^{-1/2}u^{pr,\zeta}_{kq,\gamma}h(u^{ij,\kappa}_{ql,\gamma}x)=\sum_{s,(1),(2)}\lambda_{\zeta,r}^{-1/2}h(u^{ij,\kappa}_{rs,\zeta}x_{(1)})u^{ps,\zeta}_{kl,\gamma}x_{(2)}$$.
\item For all $X=\sum_{ij}b^\kappa_{ij}\ot x_{ij}\in M_{N_\kappa}(\C)\ot \Pol(H)$ if $E_\kappa(X)=0$ then
$$\sum_{i,j,(1),(2)}\lambda_{\kappa,i}^{-1/2}\lambda_{\kappa,j}^{-3/2}h(u^{ji,\kappa}_{kl,\gamma}(x_{ij})_{(1)})(x_{ij})_{(2)}=0,\quad\text{for all }\gamma,k,l.$$
\end{enumerate}
\end{lemma}

\begin{proof}
$(1)$. Fix $x\in \Pol(H)$, $1\leq\kappa,\zeta,\gamma \leq K$, $1\leq i,j\leq N_\kappa$, $1\leq k,l\leq N_\gamma$ and $1\leq p,r\leq N_\zeta$ and define $Z:=\sum_q\lambda_{\gamma,q}^{-1/2}(1\ot u^{pr,\zeta}_{kq,\gamma})\Delta(u^{ij,\kappa}_{ql,\gamma}x)$. Using the relation
$$\Delta(u^{ij,\kappa}_{ql,\gamma}x)=\sum_{\zeta',r',s}u^{ij,\kappa}_{r's,\zeta'}x_{(1)}\ot u^{r's,\zeta'}_{ql,\gamma}x_{(2)},\,\text{we get:}$$
\begin{eqnarray*}
Z&=&\sum_{\zeta',r',s}u^{ij,\kappa}_{r's,\zeta'}x_{(1)}\ot\left(\sum_q\lambda_{\gamma,q}^{-1/2}u^{pr,\zeta}_{kq,\gamma}u^{r's,\zeta'}_{ql,\gamma}x_{(2)}\right)
=\sum_s\lambda_{\zeta,r}^{-1/2}u^{ij,\kappa}_{rs,\zeta}x_{(1)}\ot u^{ps,\zeta}_{kl,\gamma}x_{(2)}
\end{eqnarray*}
Hence, using invariance of the Haar state, and the previous equation we get:
\begin{eqnarray*}
(h\ot\id)(Z)&=&\sum_q\lambda_{\gamma,q}^{-1/2}u^{pr,\zeta}_{kq,\gamma}(h\ot\id)\Delta(u^{ij,\kappa}_{ql,\gamma}x)=\sum_q\lambda_{\gamma,q}^{-1/2}u^{pr,\zeta}_{kq,\gamma}h(u^{ij,\kappa}_{ql,\gamma}x)\\
&=&\sum_s\lambda_{\zeta,r}^{-1/2}h(u^{ij,\kappa}_{rs,\zeta}x_{(1)})u^{ps,\zeta}_{kl,\gamma}x_{(2)}.
\end{eqnarray*}

\vspace{0.2cm}

\noindent$(2)$. Using the formula in Proposition \ref{PropUcpSection}, we get:
$$0=E_\kappa(X)=\delta^{-1}\sum_{\gamma,k,l}\left(\frac{\lambda_{\gamma,l}}{\lambda_{\gamma,k}}\right)^{\frac{1}{2}}\left(\sum_{i,j}\lambda_{\kappa,i}^{-\frac{1}{2}}\lambda_{\kappa,j}^{-\frac{3}{2}}h(u^{ji,\kappa}_{kl,\gamma}x_{ij})\right)b^\gamma_{lk}.$$
Hence, for all $\gamma,k,l$ we have $\sum_{i,j}\lambda_{\kappa,i}^{-\frac{1}{2}}\lambda_{\kappa,j}^{-\frac{3}{2}}h(u^{ji,\kappa}_{kl,\gamma}x_{ij})=0$. Using $(1)$ we find, for all $\gamma,\zeta,k,p$,
\begin{eqnarray*}
0&=&\sum_q\lambda_{\gamma,q}^{-1/2}u^{pr,\zeta}_{kq,\gamma}\left(\sum_{i,j}\lambda_{\kappa,i}^{-\frac{1}{2}}\lambda_{\kappa,j}^{-\frac{3}{2}}h(u^{ji,\kappa}_{qk,\gamma}x_{ij})\right)
=\sum_{i,j}\lambda_{\kappa,i}^{-\frac{1}{2}}\lambda_{\kappa,j}^{-\frac{3}{2}}\left(\sum_q\lambda_{\gamma,q}^{-1/2}u^{pr,\zeta}_{kq,\gamma}h(u^{ji,\kappa}_{qk,\gamma}x_{ij})\right)\\
&=&\sum_{i,j}\lambda_{\kappa,i}^{-\frac{1}{2}}\lambda_{\kappa,j}^{-\frac{3}{2}}\left(\sum_{s,(1),(2)}\lambda_{\zeta,r}^{-1/2}h(u^{ij,\kappa}_{rs,\zeta}(x_{ij})_{(1)})u^{ps,\zeta}_{kk,\gamma}(x_{ij})_{(2)}\right).
\end{eqnarray*}
Multiplying by $\lambda_{\gamma,k}^{1/2}$ and summing over $k,\gamma$ we get:
\begin{eqnarray*}
0&=&\sum_k\lambda_{\gamma,k}^{1/2}\left(\sum_{i,j}\lambda_{\kappa,i}^{-\frac{1}{2}}\lambda_{\kappa,j}^{-\frac{3}{2}}\left(\sum_{s,(1),(2)}\lambda_{\zeta,r}^{-1/2}h(u^{ij,\kappa}_{rs,\zeta}(x_{ij})_{(1)})u^{ps,\zeta}_{kk,\gamma}(x_{ij})_{(2)}\right)\right)\\
&=&\sum_{i,j}\lambda_{\kappa,i}^{-\frac{1}{2}}\lambda_{\kappa,j}^{-\frac{3}{2}}\left(\sum_{s,(1),(2)}\lambda_{\zeta,r}^{-1/2}h(u^{ij,\kappa}_{rs,\zeta}(x_{ij})_{(1)})\left(\sum_k\lambda_{\gamma,k}^{1/2}u^{ps,\zeta}_{kk,\gamma}\right)(x_{ij})_{(2)}\right)\\
&=&\left(\frac{\lambda_{\zeta,p}}{\lambda_{\zeta,r}}\right)^{1/2}\sum_{i,j,(1),(2)}\lambda_{\kappa,i}^{-\frac{1}{2}}\lambda_{\kappa,j}^{-\frac{3}{2}}h(u^{ij,\kappa}_{rp,\zeta}(x_{ij})_{(1)})(x_{ij})_{(2)}.
\end{eqnarray*}
This concludes the proof.\end{proof}

\begin{definition}\label{2-ergodic}
A $\psi$-preserving action $\beta\,:\, B\rightarrow B\ot C(H)$ with associated unitary $u\in\Lcal(B)\ot C(H)$ is called $2$-\textit{ergodic} if ${\rm dim}({\rm Mor}(u,u))=2$.
\end{definition}

\begin{remark}
Note that if $\beta$ is $2$-ergodic then ${\rm dim}(B)\geq 2$. Moreover, the following holds.
\begin{enumerate}
\item $\beta$ is $2$-ergodic if and only if $\Mor(u,u)=\C\eta\eta^\star\oplus\C(\id_B-\eta\eta^\star)$. Indeed, the two vectors $\eta\eta^\star,\id_B-\eta\eta^\star\in\Mor(u,u)$ are linearly independent whenever ${\rm dim}(B)\geq 2$.
    \item As in $(1)$, we see that $\beta$ is $2$-ergodic if and only if $\Mor(1,u\ot u) = \C \eta\ot \eta\oplus \C m^\star\eta$.
    \item A $2$-ergodic action $\beta$ is ergodic. Indeed, if ${\rm dim}({\rm Mor}(u,u))=2$ then $u$ is isomorphic to the sum of two non equivalent irreducible representations which implies that ${\rm dim}(1,u)\leq 1$. Since $\eta\in{\rm Mor}(1,u)$, $\beta$ is ergodic.
\end{enumerate}
\end{remark}

\begin{example}\label{ExAction1} The following holds.
\begin{enumerate}
    \item The canonical $\psi$-preserving action of $\Aut^+(B,\psi)$ on $B$ (${\rm dim}(B)\geq 2$) is $2$-ergodic if and only if $\psi$ is a $\delta$-form, in particular the action of $S_N^+$ on $\C^N$ for any $N\geq 2$ is $2$-ergodic. Similarly, the action of $S_N$ on $\C^N$ is also $2$-ergodic for any $N\geq 2$.
    \item  For a finite group $\Lambda$, the action $\widehat{\Lambda}\curvearrowright C^*(\Lambda)$ given by the comultiplication is ergodic and $\tau_\Lambda$-preserving, where $\tau_\Lambda\in C^*(\Lambda)^*$ is the canonical (faithful) trace. This action is not $2$-ergodic whenever $\vert\Lambda\vert\neq 2$ since the corresponding representation $u$ is such that ${\rm dim}(\Mor(\varepsilon,u\ot u))=\vert \Lambda \vert$.
    \end{enumerate}
\end{example}

\noindent In the next Lemma we show how the $2$-ergodic assumption provide a computation of the Haar state of a product of two coefficients of the representation $u$.

\begin{lemma}\label{LemmaHaarProduct}
If $\beta$ is $2$-ergodic then the following holds.
\begin{eqnarray*}
h(u^{ij,\alpha}_{kl,\gamma}u^{rs,\kappa}_{vw,\zeta})&=&\dfrac{\delta}{\delta-1}\delta_{ij}\delta_{kl}\delta_{rs}\delta_{vw} (\lambda_{\alpha,i}\lambda_{\gamma,k}\lambda_{\kappa,r}\lambda_{\zeta,v})^{\frac{1}{2}}+\frac{\delta_{\alpha \kappa}\delta_{jr}\delta_{is}\delta_{\gamma \zeta}\delta_{kw}\delta_{lv}}{\delta-1} \left(\frac{\lambda_{\alpha,i}\lambda_{\gamma,k}}{\lambda_{\alpha,j}\lambda_{\gamma,l}}\right)^{\frac{1}{2}}\\
  &-& \delta_{ij}\delta_{rs}\delta_{\gamma \zeta}\delta_{kw}\delta_{lv}\frac{1}{\delta-1}\lambda_{\alpha,i}^{\frac{1}{2}}\lambda_{\kappa,r}^{\frac{1}{2}}\lambda_{\gamma,k}^{\frac{1}{2}}\lambda_{\gamma,l}^{-\frac{1}{2}}
  -  \delta_{\alpha \kappa}\delta_{is}\delta_{jr}\delta_{kl}\delta_{vw} \frac{1}{\delta-1}\lambda_{\alpha,j}^{-\frac{1}{2}}\lambda_{\alpha,i}^{\frac{1}{2}}\lambda_{\gamma,k}^{\frac{1}{2}}\lambda_{\zeta,v}^{\frac{1}{2}}.
\end{eqnarray*}
\end{lemma}

\begin{proof}
The action $\beta$ being $2$-ergodic, it is ergodic so its unique invariant state $\psi$ is a $\delta$-form. Moreover, $(\id\ot h)(u\ot u)$ is the orthogonal projection onto $\Mor(1,u\ot u) = \C \eta\ot \eta\oplus \C m^\star\eta$. Write $\eta=\sum_{\kappa,i}\lambda_{\kappa,i}^{\frac{1}{2}}b^\kappa_{ii}$, $m^\star\eta = \sum_{\kappa,i,t}\lambda_{\kappa,i}^{\frac{1}{2}}\lambda_{\kappa,t}^{-\frac{1}{2}} b^\kappa_{it}\ot b^\kappa_{ti}$, and therefore $\Vert m^\star\eta \Vert^2 = \sum_\kappa \Tr(Q_\kappa^{-1})\Tr(Q_\kappa)= \sum_\kappa \delta \Tr(Q_\kappa) = \delta$, because $\psi$ is a $\delta$-form.
Since $(\eta^\star\ot \eta^\star) m^\star\eta = 1$ we have $\Vert m^\star\eta - \eta \ot \eta \Vert^2 = \Vert m^\star\eta \Vert^2-1 =: \delta-1 $. The orthogonal projection on the space is then the orthogonal projection that we obtain taking $(\eta\ot \eta, \frac{1}{\sqrt{\delta-1}}(m^\star\eta - \eta\ot \eta))$ as an orthogonal base. The projection is therefore given by $p = \eta\eta^\star\ot \eta\eta^\star + \frac{1}{\delta-1}(m^\star\eta  - \eta\ot \eta)(m^\star\eta  - \eta\ot \eta)^\star$. Note that:
$$
            p = \left(1+\dfrac{1}{\delta-1}\right) (\eta\eta^\star\ot \eta\eta^\star )+ \frac{1}{\delta-1} m^\star\eta \eta^\star m - \frac{1}{\delta-1}(\eta\ot \eta)\eta^\star m  - \frac{1}{\delta-1}m^\star \eta (\eta^\star\ot \eta^\star).$$
            We have:
\begin{eqnarray*}
    \eta\eta^\star\ot \eta\eta^\star = \sum_{\substack{\alpha,i\\ \gamma,k}} \sum_{\substack{\kappa,r\\ \zeta,v}} (b_{ii}^\alpha \ot b_{rr}^\kappa)\circ (b_{kk}^\gamma \ot b_{vv}^\zeta)^\star  (\lambda_{\alpha,i}\lambda_{\gamma,k}\lambda_{\kappa,r}\lambda_{\zeta,v})^{\frac{1}{2}}\\
    m^\star\eta \eta^\star m = \sum_{\substack{\alpha,i,x\\ \gamma,k,y}} (b_{ix}^\alpha \ot b_{xi}^\alpha)\circ (b_{ky}^\gamma \ot b_{yk}^\gamma)^\star (\lambda_{\alpha,i}^{\frac{1}{2}}\lambda_{\alpha,x}^{-\frac{1}{2}}\lambda_{\gamma,k}^{\frac{1}{2}}\lambda_{\gamma,y}^{-\frac{1}{2}})\\
    (\eta\ot \eta)\eta^\star m = \sum_{\substack{\alpha,i\\ \kappa,r}} \sum_{\substack{\gamma,k,t}} (b_{ii}^\alpha \ot b_{rr}^\kappa)\circ (b_{kt}^\gamma \ot b_{tk}^\gamma)^\star \lambda_{\alpha,i}^{\frac{1}{2}}\lambda_{\kappa,r}^{\frac{1}{2}}\lambda_{\gamma,k}^{\frac{1}{2}}\lambda_{\gamma,t}^{-\frac{1}{2}}\\
    m^\star \eta (\eta^\star\ot \eta^\star)= \sum_{\substack{\alpha,i,t}} \sum_{\substack{\gamma,k\\ \zeta,v}} (b_{it}^\alpha \ot b_{ti}^\alpha)\circ (b_{kk}^\gamma \ot b_{vv}^\zeta)^\star \lambda_{\alpha,t}^{-\frac{1}{2}}\lambda_{\alpha,i}^{\frac{1}{2}}\lambda_{\gamma,k}^{\frac{1}{2}}\lambda_{\zeta,v}^{\frac{1}{2}}\\
\end{eqnarray*}
We deduce that the coefficient associated to $(b_{ij}^\alpha \ot b_{rs}^\kappa)\circ (b_{kl}^\gamma \ot b_{vw}^\zeta)^\star$ is:
\begin{eqnarray*}
  &\left(1+\dfrac{1}{\delta-1}\right)\delta_{ij}\delta_{kl}\delta_{rs}\delta_{vw} (\lambda_{\alpha,i}\lambda_{\gamma,k}\lambda_{\kappa,r}\lambda_{\zeta,v})^{\frac{1}{2}}
  + \delta_{\alpha \kappa}\delta_{jr}\delta_{is}\delta_{\gamma \zeta}\delta_{kw}\delta_{lv}\frac{1}{\delta-1} (\lambda_{\alpha,i}^{\frac{1}{2}}\lambda_{\alpha,j}^{-\frac{1}{2}}\lambda_{\gamma,k}^{\frac{1}{2}}\lambda_{\gamma,l}^{-\frac{1}{2}})\\
  &- \delta_{ij}\delta_{rs}\delta_{\gamma \zeta}\delta_{kw}\delta_{lv}\frac{1}{\delta-1}\lambda_{\alpha,i}^{\frac{1}{2}}\lambda_{\kappa,r}^{\frac{1}{2}}\lambda_{\gamma,k}^{\frac{1}{2}}\lambda_{\gamma,l}^{-\frac{1}{2}}
  -  \delta_{\alpha \kappa}\delta_{is}\delta_{jr}\delta_{kl}\delta_{vw} \frac{1}{\delta-1}\lambda_{\alpha,j}^{-\frac{1}{2}}\lambda_{\alpha,i}^{\frac{1}{2}}\lambda_{\gamma,k}^{\frac{1}{2}}\lambda_{\zeta,v}^{\frac{1}{2}}.
\end{eqnarray*}
This coefficient is equal to $h(u^{ij,\alpha}_{kl,\gamma}u^{rs,\kappa}_{vw,\zeta}).$
\end{proof}

\noindent We record the following Lemma for later use.

\begin{lemma}\label{LemmaT}
If $\beta$ is $2$-ergodic, $\kappa$ is such that $N_\kappa\geq 2$ then, for any unital C*-algebra $C$,
\begin{enumerate}
    \item If $t\in\R$ and $u\in\mathcal{U}(B\ot C)$ are such that $\sigma_{-t}^{\psi_\kappa}\ot\sigma^h_t\ot\id_C={\rm Ad}((\beta_\kappa\ot\id)(u))$ then $u\in\C1\ot\mathcal{Z}(C)$ and $\sigma_t^h=\id$.
    \item $\{u\in\mathcal{U}(B\ot C)\,:\,{\rm Ad}((\beta_\kappa\ot\id)(u))=\id\}\subset\C1_B\ot\mathcal{Z}(C)$.
        
    \item If $t\in\R$ and $u\in\mathcal{U}(B\ot C)$ are such that $\sigma_{-t}^{\psi_\kappa}\ot\tau^H_t\ot\id_C={\rm Ad}((\beta_\kappa\ot\id)(u))$ then $u\in\C1\ot\mathcal{Z}(C)$ and $\tau^H_t=\id$.
\end{enumerate}
\end{lemma}

\begin{proof}
$(1)$. Let $t\in\R$ and $u\in\mathcal{U}(B\ot C)$ be such that $\sigma_{-t}^{\psi_\kappa}\ot\sigma^h_t\ot\id_C={\rm Ad}((\beta_\kappa\ot\id)(u))$. Recall that $\sigma_t^\psi={\rm Ad}(b_t)$, where $b_t=\sum_{\gamma,r}\lambda_{\gamma,r}^{it}e^\gamma_{rr}=\sum_{\gamma,r}\lambda_{\gamma,r}^{it+1/2}b^\gamma_{rr}\in B$ is unitary and $b_t^*=b_{-t}$. Then, for all $b\in B$ and $c\in C$,
\begin{eqnarray*}
(\sigma_{-t}^{\psi_\kappa}\ot\sigma^h_t\ot\id_C)(\beta_\kappa\ot\id)(b\ot c)&=&\beta_\kappa(\sigma_t^\psi(b))\ot c=\beta_\kappa(b_tbb_{-t})\ot c\\
&=&(\beta_\kappa\ot\id)((b_t\ot 1)(b\ot c)(b_{-t}\ot 1))\\
&=&(\beta_\kappa\ot\id)(u)(\beta_\kappa\ot\id)(b\ot c)(\beta_\kappa\ot\id)(u^*)\\
&=&(\beta_\kappa\ot\id)(u(b\ot c)u^*).
\end{eqnarray*}
Since $\beta_\kappa$ is faithful, we deduce that $(b_t\ot 1)(b\ot c)(b_{-t}\ot 1)=u(b\ot c)u^*$ for all $b\in B$, $c\in C$. Hence, $v:=(b_{-t}\ot 1)u\in\mathcal{Z}(B\ot C)$. Since $v$ is unitary and central in $B\ot C$, we may write $v=\sum_{\gamma,r}e^\gamma_{rr}\ot\alpha_\gamma=\sum\lambda_{\gamma,r}^{1/2} b^\gamma_{rr}\ot\alpha_\gamma$ where $\alpha_\gamma\in \mathcal{U}(\mathcal{Z}(C))$. Then, $u=(b_t\ot 1)v=\sum_{\gamma,r}\lambda_{\gamma,r}^{it+1/2} b^\gamma_{rr}\ot \alpha_\gamma$ and $u^*=\sum_{\zeta,s}\lambda_{\zeta,s}^{-it+1/2}b^\zeta_{ss}\ot \alpha^*_\zeta $. It follows that:
$$(\beta_\kappa\ot\id)(u)=\sum_{\gamma,r,k,t}\lambda_{\gamma,r}^{it+1/2} b^\kappa_{kt}\ot u^{kt,\kappa}_{rr,\gamma}\ot \alpha_\gamma\,\,\text{and}\,\,(\beta_\kappa\ot\id)(u^*)=\sum_{\zeta,s,t',l}\lambda_{\zeta,s}^{-it+1/2} b^\kappa_{t'l}\ot u^{t'l,\kappa}_{ss,\zeta}\ot \alpha_\zeta^*.$$
Fix $1\leq p\leq N_\kappa$. Using the formula $b^\kappa_{kt}b^\kappa_{pp}b^\kappa_{t'l}=\frac{\delta_{tp}\delta_{t'p}}{\lambda_{\kappa,p}}b^\kappa_{kl}$, we find:
$$
(\beta_\kappa\ot\id)(u)(b^\kappa_{pp}\ot 1\ot 1)(\beta_\kappa\ot\id)(u^*)=\sum_{\gamma,\zeta,r,s,k,l}\left(\frac{\lambda_{\gamma,r}}{\lambda_{\zeta,s}}\right)^{it}\frac{(\lambda_{\gamma,r}\lambda_{\zeta,s})^{1/2}}{\lambda_{\kappa,p}}b^\kappa_{kl}\ot u^{kp,\kappa}_{rr,\gamma}u^{pl,\kappa}_{ss,\zeta}\ot \alpha_\gamma\alpha_{\zeta}^*.
$$
Note that $(\sigma_{-t}^{\psi_\kappa}\ot\sigma_t^h\ot\id_C)(b^\kappa_{pp}\ot 1\ot 1)=b^\kappa_{pp}\ot 1\ot 1$ hence,
\begin{equation}\label{EqT}\sum_{\gamma,\zeta,r,s}\left(\frac{\lambda_{\gamma,r}}{\lambda_{\zeta,s}}\right)^{it}\frac{(\lambda_{\gamma,r}\lambda_{\zeta,s})^{1/2}}{\lambda_{\kappa,p}} u^{pp,\kappa}_{rr,\gamma}u^{pp,\kappa}_{ss,\zeta}\ot \alpha_\gamma\alpha_{\zeta}^*=1\ot 1.\end{equation}
By the Haar state formula given in Lemma \ref{LemmaHaarProduct} we have:
$$h(u^{pp,\kappa}_{rr,\gamma}u^{pp,\kappa}_{ss,\zeta})=\left(\frac{\delta}{\delta-1}\lambda_{\kappa,p}-\frac{1}{\delta-1}\right)(\lambda_{\gamma,r}\lambda_{\zeta,s})^{1/2}+\delta_{\gamma,\zeta}\delta_{rs}\frac{1-\lambda_{\kappa,p}}{\delta-1}.$$
Applying $h\ot\id_C$ to Equation $(\ref{EqT})$ we get:
$$
\left(\frac{\delta}{\delta-1}-\frac{1}{\delta-1}\lambda_{\kappa,p}^{-1}\right)\left[\sum_{\gamma,\zeta,r,s}\left(\frac{\lambda_{\gamma,r}}{\lambda_{\zeta,s}}\right)^{it}\lambda_{\gamma,r}\lambda_{\zeta,s}\alpha_\gamma\alpha_{\zeta}^*\right]+\left(\frac{1}{\delta-1}\lambda_{\kappa,p}^{-1}-\frac{1}{\delta-1}\right)1=1.
$$
Define $w:=\sum_{\gamma,r}\lambda_{\gamma,r}\lambda_{\gamma,r}^{it}\alpha_\gamma\in \mathcal{Z}(C)$. Note that, since $N_\kappa\geq 2$, $\delta=\sum_{t}\lambda_{\kappa,t}^{-1}>\lambda_{\kappa,p}^{-1}$ hence,
$$\left(\frac{\delta}{\delta-1}-\frac{1}{\delta-1}\lambda_{\kappa,p}^{-1}\right)\neq 0$$
and $ww^*=\sum_{\gamma,\zeta,r,s}\left(\frac{\lambda_{\gamma,r}}{\lambda_{\zeta,s}}\right)^{it}\lambda_{\gamma,r}\lambda_{\zeta,s}\alpha_\gamma\alpha_{\zeta}^*=1$. We conclude that $w$ is unitary so the convex combination $w=\sum_{\gamma,r}\lambda_{\gamma,r}\lambda_{\gamma,r}^{it}\alpha_\gamma$ is an extreme point in the unit ball of $C$ which implies that $\lambda_{\gamma,r}^{it}\alpha_\gamma=\lambda_{\zeta,s}^{it}\alpha_\zeta$ for all $\gamma,\zeta,r,s$.
Writing $a\in\mathcal{U}(\mathcal{Z}(C))$ the unitary satisfying $a=\lambda_{\gamma,r}^{it}\alpha_\gamma$ for all $1\leq\gamma\leq K$ and all $1\leq r\leq N_\gamma$ one has $$u=\sum_{\gamma,r}\lambda_{\gamma,r}^{it+1/2} b^\gamma_{rr}\ot \alpha_\gamma=\sum\lambda_{\gamma,r}^{1/2}b^\gamma_{rr}\ot a=1\ot a.$$ Hence, for all $x\in\Linf(H)$, one has:
\begin{eqnarray*}
1\ot\sigma_t^h(x)\ot 1&=&(\sigma_{-t}^{\psi_\kappa}\ot\sigma_t^h\ot\id)(1\ot x\ot 1)={\rm Ad}((\beta_\kappa\ot\id)(u))(1\ot x\ot 1)\\
&=&(1\ot x\ot aa^*)
=1\ot x\ot 1.
\end{eqnarray*}
It follows that $\sigma_t^h=\id$.

\vspace{0.2cm}

\noindent$(2)$ follows from $(1)$ (with $t=0$).
\vspace{0.2cm}

\noindent$(3)$ can be obtained through the same reasoning as for $(1)$, replacing the modular group relation $(\sigma_{-t}^{\psi_\kappa}\ot \sigma_t^h)\beta_\kappa = \beta_\kappa\circ \sigma_t^\psi$ by the scaling group relation $(\sigma_{-t}^{\psi_\kappa}\ot \tau_t^H)\beta_\kappa = \beta_\kappa\circ \sigma_{-t}^\psi$.\end{proof}

\noindent When $N_\kappa=1$ for all $\kappa$, so that $B=\C^K$, we obtain the following Lemma.

\begin{lemma}\label{LemmaTdim1}
If $\beta$ is $2$-ergodic and $N_\kappa=1$ for all $\kappa$ (so that $B=\C^K$) and $K\geq 3$ then, for any unital C*-algebra $C$ and any $1\leq \gamma,\kappa\leq K$ with $\kappa\neq\gamma$,
$$\C1\ot C=\beta_\gamma(\C^K)\ot C\cap\beta_\kappa(\C^K)\ot C\subset C(H)\ot C.$$
\end{lemma}

\begin{proof}
Since $\beta$ is ergodic, $\psi$ is a $\delta$-form. It is easy to check that the unique $\delta$-form on $\C^K$ is the uniform normalized trace and $\delta=K$. Write $u=\sum_{i,j=1}^Ke_{ij}\ot u_{ij}$, where $e_{ij}$ are the canonical matrix units, the unitary representation of $H$ associated with $\beta$. Note that $\beta_i(\C^K)={\rm Span}\{u_{ij}\,:\,1\leq j\leq K\}$. By Lemma \ref{LemmaHaarProduct} and since $\delta=K$ we find:

\begin{eqnarray*}
h(u_{ij} u_{kl}) &=& \dfrac{1}{K(K-1)} +\dfrac{1}{K-1}\delta_{ik}\delta_{jl} - \dfrac{1}{K(K-1)}\delta_{jl} -  \dfrac{1}{K(K-1)}\delta_{ik}\\
  &=& \begin{cases}
    \dfrac{1}{K} & \text{if $i=k$ and $j=l$}.\\
    0 & \text{if $i=k$ and $j\neq l$, or $i\neq k$ and $j= l$}.\\
    \dfrac{1}{K(K-1)} & \text{otherwise}.\\
  \end{cases}
\end{eqnarray*} 
Let $E_i\,:\,\Linf(H)\rightarrow\C^K$ be the ucp map from Proposition \ref{PropUcpSection}. Using the formula above for the Haar state as well as the formula given in Proposition \ref{PropUcpSection} we obtain:
$$E_i(u_{jk})=\left\{\begin{array}{lcl}\frac{1}{K-1}\sum_{l\neq k}e_{l}&\text{if}& i\neq j,\\e_k&\text{if}&i=j.\end{array}\right.$$
Let $i\neq j$ and $w\in \beta_i(\C^K)\ot C\cap \beta_j(\C^K)\ot C$ be a unitary ($i\neq j$), then we can write $$w  = \sum_k u_{jk}\ot \mu_k = \sum_l u_{il}\ot \nu_l\text{ where, for all }k,l,\,\,\mu_k,\nu_l\in\mathcal{U}(C).$$ Applying $E_i\ot\id_C$ to $w$, we get that for any $l$, $ \nu_l = \frac{1}{K-1} \sum_{k\neq l} \mu_k$. In particular, $\frac{1}{K-1}\sum_{k\neq l} \mu_k$ is a unitary so it is an extreme point in the unit ball of $C$. It follows that $\mu_k=\mu_{k'}$ for all $k,k'\neq l$. Since this holds for all $l$ and $K\geq 3$, we deduce that $\mu_k$ does not depend on $k$ hence $w$ is of the form $w=1\ot v$, with $v\in\mathcal{U}(C)$. The C*-algebra $\beta_i(\C^K)\ot C\cap\beta_j(\C^K)\ot C$ being the linear span of its unitaries, we get $\beta_i(\C^K)\ot C\cap\beta_j(\C^K)\ot C=\C1\ot C$.\end{proof}

\section{Generalized Free Wreath Products}\label{sectiongen}

\noindent In this section, we introduce a novel approach to free wreath products. Recall that for any groups $\Gamma$ and $\Lambda$, and any action $\beta : \Lambda \curvearrowright X$ on a set $X$, the generalized wreath product is defined as $\Gamma \wr_\beta \Lambda := \Gamma^X \rtimes \Lambda$. In this definition, the group action $\Lambda \curvearrowright \Gamma^X$ is by shifting the coordinates according to the action $\beta$. Our aim is to extend this construction, in a free manner, to any compact quantum groups $G$ and $H$ with an action $\beta : H \curvearrowright B$ on a unital C*-algebra $B$. Below, we outline some motivations for our construction and address certain limitations inherent to the C*-algebra $B$.

\subsection{Motivations}
Let's begin by reviewing Bichon's original construction \cite{Bi04}. Let $G$ be a compact quantum group and $u\in M_N(\C)\ot C(S_N^+)$ the generating magic unitary. The C*-algebra of the free wreath product is defined by $C(G\wr_{*}S_N^+):=C(G)^{*N}*C(S_N^+)/I$, where $I$ is the closed two sided ideal generated by $\{\nu_i(a)u_{ij}-u_{ij}\nu_i(a)\,:\,a\in A,\,1\leq i,j\leq N\}$, with $\nu_i\,:\, C(G)\rightarrow C(G)^{*N}$ being the $i^{th}$-copy of $C(G)$ in the free product. 

\vspace{0.2cm}

\noindent We consider now a new way to describe this C*-algebra. Let $\beta\,:\,\C^N\rightarrow \C^N\ot C(S_N^+)$ be the universal action of $S_N^+$ on $\C^N$ i.e. $\beta(e_j):=\sum_ie_i\ot u_{ij}$ and $\rho\,:\,C(G)\rightarrow \C^N\ot C(G)^{*N}$ be the unital $*$-homomorphism defined by $\rho(a):=\sum_ie_i\ot\nu_i(a)$. Viewing both $C(S_N^+),C(G)^{*N}\subset C(G\wr_*S_N^+)$ we have $\rho(a)\beta(e_j)=\sum_ie_i\ot \nu_i(a)u_{ij}=\sum_ie_i\ot u_{ij}\nu_i(a)=\beta(e_j)\rho(a)$. It follows that for all $a\in C(G)$ and $b\in\C^N$ one has $\rho(a)\beta(b)=\beta(b)\rho(a)$. The C*-algebra $C(G\wr_*S_N^+)$ is actually the universal unital C*-algebra generated by $C(S_N^+)$ and the coefficients of a unital $*$-homomorphism $\rho\,:\,C(G)\rightarrow\C^N\ot C(G\wr_*S_N^+)$ satisfying $[\rho(a),\beta(b)]=0$ for all $a\in C(G)$ and $b\in\C^N$ in the following precise sense.

\begin{proposition}\label{PropMotivation}
If $A$ is a unital C*-algebra with unital $*$-homomorphisms $\iota\,:\, C(S_N^+)\rightarrow A$ and $\pi\,:\,C(G)\rightarrow\C^N\ot A$ such that $[\pi(a),(\id\ot\iota)(\beta(b))]=0$ for all $a\in C(G)$ and $b\in\C^N$ then, there exists a unique unital $*$-homomorphism $\widetilde{\pi}\,:\,C(G\wr_*S_N^+)\rightarrow A$ such that $$(\id\ot\widetilde{\pi})\rho=\pi\text{ and }\widetilde{\pi}\vert_{C(S_N^+)}=\iota.$$
\end{proposition}

\begin{proof}
The uniqueness being obvious, let us show the existence. Write $\pi(a)=\sum_ie_i\ot\pi_i(a)$ and note that, since $\pi$ is a unital $*$-homomorphism, each $\pi_i\,:\, C(G)\rightarrow A$ is a unital $*$-homomorphism. Define $\pi_0\,:\, C(G)^{*N}*C(S_N^+)\rightarrow A$ by $\pi_0\circ\nu_i:=\pi_i$ and $\pi_0\vert_{C(S_N^+)}:=\iota$. Note that the commutations relations $[\pi(a),(\id\ot\iota)(\beta(b))]=0$ imply that, for all $a\in C(G)$ and all $1\leq i,j\leq N$ one has $\pi_0(\nu_i(a)u_{ij}-u_{ij}\nu_i(a))=0$. Hence $I\subseteq\ker(\pi_0)$ which implies that there exists a unital $*$-homomorphism $\widetilde{\pi}\,:\,C(G\wr_* S_N^+)\rightarrow A$ satisfying the properties of the Proposition.\end{proof}

\noindent Let now $G, H$ be any compact quantum groups and $\beta\,:\,B\rightarrow B\ot C(H)$ be any action of $H$ on the unital C*-algebra $B$. Motivated by Proposition \ref{PropMotivation} we want to define the C*-algebra $C(\GG)$ of the generalized free wreath product $\GG=G\wr_{*,\beta}H$ as the universal unital C*-algebra generated by $C(H)$ and the coefficients of a unital $*$-homomorphism $\rho\,:\, C(G)\rightarrow B\ot C(\GG)$ with the relation $[\rho(a),\beta(b)]=0$ for all $a\in C(G)$ and $b\in B$. It is not clear however that such a C*-algebra even exists and we will elaborate on that in the next subsection.

\subsection{The amalgamated free wreath product quantum group} \noindent For the remainder of this section, we fix an action $\beta\,:\, B\rightarrow B\ot C(H)$ of a CQG $H$ on finite dimensional C*-algebra $B$ and a CQG $G$ with a dual quantum subgroup $F$, $C(F)\hookrightarrow C(G)$. We denote by $E_F: C(G)\rightarrow C(F)$ the canonical unital ucp map defined in Proposition \ref{dualqsg}.

\vspace{0.2cm}

\noindent Recall that the direct product quantum group $H\times F$ is defined as the compact quantum group such that $C(H\times F) = C(H)\otm C(F)$ and $\Delta_{H\times F} = \Sigma_{23}(\Delta_H \ot \Delta_F)$. Note that, by universal property of the maximal tensor product, $C(H)\otm C(F)$ is the universal unital C*-algebra generated $C(H)$ and $C(F)$ with the relations making $C(H)$ and $C(F)$ commuting.

\vspace{0.2cm}

\noindent We now define \textit{the amalgamated free wreath product quantum group} C*-algebra $C(\GG)$ as the universal unital C*-algebra generated by $C(H)$,  $C(F)$ and the coefficients of a unital $*$-homomorphism $\rho\,:\,C(G)\rightarrow B\ot C(\GG)$ with the relations  $[\rho(a),\beta(b)]=0$ for all $a\in C(G)$ and $b\in B$, $C(H)$ and $C(F)$ commute in $C(\GG)$ and $\rho(f)=1_B\ot f$ for all $f\in C(F)\subset C(\GG)$. More precisely the C*-algebra $C(\GG)$, if it exists, satisfies the following properties:

\begin{itemize}
\item There exists three unital $*$-homomorphisms $\iota_H\,:\,C(H)\rightarrow C(\GG)$, $\iota_F\,:\,C(F)\rightarrow C(\GG)$ and $\rho\,:\,C(G)\rightarrow B\ot C(\GG)$ such that:
\begin{itemize}
\item $[\iota_H(x),\iota_F(f)]=0$ for all $x\in C(H)$ and all $f\in C(F)$.
\item $[(\id\ot\iota_H)(\beta(b)),\rho(a)]=0$ for all $a\in C(G)$ and all $b\in B$.
\item $\rho(f)=1_B\ot\iota_F(f)$ for all $f\in C(F)$.
\end{itemize}
\item If $C$ is any unital C*-algebra with unital $*$-homomorphisms $\iota'_H\,:\,C(H)\rightarrow C$, $\iota'_F\,:\,C(F)\rightarrow C$ and $\rho'\,:\,C(G)\rightarrow B\ot C$ such that:
\begin{itemize}
    \item $[\iota'_H(x),\iota'_F(f)]=0$ for all $x\in C(H)$, $f\in C(F)$,
    \item $[(\id\ot\iota'_H)(\beta(b)),\rho'(a)]=0$ for all $a\in C(G)$, $b\in B$,
    \item $\rho'(f)=1_B\ot\iota'_F(f)$ for all $f\in C(F)$,
\end{itemize}
then, there exists a unique unital $*$-homomorphism $\nu\,:\, C(\GG)\rightarrow C$ such that:
$$\nu\circ\iota_H=\iota'_H,\,\,\nu\circ\iota_F=\iota'_F\text{ and }(\id\ot\nu)\circ\rho=\rho'.$$
\end{itemize}

\noindent While it is not clear that $C(\GG)$ even exists, its uniqueness, up to a canonical isomorphism, is an obvious consequence of the universal property defining $C(\GG)$. When $C(\GG)$ exists, we will sometimes write $\iota\,:\, C(H)\otm C(F)\rightarrow C(\GG)$ the unique unital $*$-homomorphism such that $\iota(x\ot f)=\iota_H(x)\iota_F(f)$ for all $x\in C(H)$ and all $f\in C(F)$.

\begin{remark}\label{RmkIota} Suppose that $C(\GG)$ exists then $\iota$ and $\rho$ are faithful (in particular, both $\iota_H$ and $\iota_F$ are faithful). Indeed, it follows from the universal property, taking $C=C(H)\otm C(G)$, $\iota'_H(x)=x\ot 1_{C(F)}$, $\iota'_F(f)=1_{C(H)}\ot f$ and $\rho'(a)=1_B\ot 1_{C(H)}\ot a$, that there exists a unique unital $*$-homomorphism $\pi\,:\,C(\GG)\rightarrow C(H)\otm C(G)$ such that
    $$(\id\ot\pi)\rho(a)=1_B\ot 1_{C(H)}\ot a\text{ for all }a\in C(G)\text{ and }\pi\circ\iota=\id.$$
We easily deduce that $\iota$ and $\rho$ are faithful.
\end{remark}

\begin{example} When $F$ equals $G$ i.e.  the inclusion $C(F)\subset C(G)$ is surjective then $C(\GG)$ exists: it is easy to see that $C(\GG):=C(H\times G)=C(H)\otm C(G)$ with $\iota_H(x)=x\ot 1_{C(G)}$, $\iota_F(a)=1_{C(H)}\ot a$ and $\rho(a)=1_B\ot 1_{C(H)}\ot a$ for all $x\in C(H)$ and $a\in C(G)$ satisfy the universal property.
\end{example}

\noindent We discuss the obvious limitations of such a construction and some of its properties in the following Proposition.

\begin{proposition}The following holds.
\begin{enumerate}
\item If $B$ is finite dimensional then $C(\GG)$ exists.
\item If $F$ is trivial, $G$ is non-trivial and $B$ is abelian then $C(\GG)$ exists if and only if $B$ is finite dimensional.
\item If $C(\GG)$ exists then $C(\GG)$ is generated, as a C*-algebra, by 
$$\iota_H(C(H))\cup\{(\omega\ot\id)(\rho(a))\,:\,a\in C(G) ,\,\omega\in B^*\}.$$
\end{enumerate}
\end{proposition}

\begin{proof}$(1)$. If $B$ is finite dimensional, it is known \cite{FMP24} that the universal quantum homomorphism $\rho\,:\, C(G)\rightarrow B\ot \Ucal$ exists. Then, the C*-algebra
$$C(\GG):=\Ucal*(C(H)\otm C(F))/I,$$
where $I$ is the closed two sided ideal generated by $$\{(\omega\ot\id)(\rho(a)\beta(b)-\beta(b)\rho(a))\,:\,\omega\in B^*,\,a\in C(G),\,b\in B\}$$
and by $\{(\omega\ot\id)(\rho(f)-1_B\ot f)\,:\,\omega\in B^*,\,f\in C(F)\}$ does the job.

\vspace{0.2cm}

\noindent$(2)$. Suppose that $C(\GG)$ exists and define $\Ucal\subset C(\GG)$ to be the C*-subalgebra of $C(\GG)$ generated by $\{(\omega\ot\id)(\rho(a))\,:\,a\in C(G) ,\,\omega\in B^*\}$ so that $\rho(C(G))\subset B\ot\Ucal$. We show below that $\rho\,:\, C(G)\rightarrow B\ot\Ucal$ is the universal quantum homomorphism from $C(G)$ to $B$ in the sense of \cite{FMP24}. Since ${\rm dim}(C(G))\geq 2$ \cite[Theorem D]{FMP24} will imply that $B$ is finite dimensional. Let $\pi\,:\,C(G)\rightarrow B\ot C$ be a unital $*$-homomorphism and define $\iota'_H:=\varepsilon_H(\cdot)1_C\,:\,C(H)\rightarrow C$. Note that $(\id\ot\iota'_H)(\beta(b))=b\ot 1_C$ and since $B$ is abelian we have $[(\id\ot\iota'_H)(\beta(b)),\pi(a)]=0$ for all $a\in C(G)$ and $b\in B$. It follows that there exists a unique unital $*$-homomorphism $\nu\,:\,C(\GG)\rightarrow C$ such that $(\id\ot\nu)\circ\rho=\pi$ (and $\nu\circ\iota_H=\iota'_H$). Defining $\widetilde{\pi}:=\nu\vert_{\Ucal}$ we have $(\id\ot\widetilde{\pi})\rho=\pi$ and the uniqueness of $\widetilde{\pi}$ is clear.

\vspace{0.2cm}

\noindent$(3)$. It is a direct consequence of the uniqueness in the universal property.\end{proof}

\noindent Motivated by the previous Proposition, we assume for the rest of this paper that $B$ is finite dimensional so that $C(\GG)$ exists. From Remark \ref{RmkIota} we may and will view $C(H),C(F)\subset C(\GG)$ with $\iota_H$, $\iota_F$ being the inclusions. Also, since the map $\iota\,:\,C(H)\otm C(F)\rightarrow C(\GG)$ is faithful we view $C(H)\otm C(F)=\overline{{\rm Span}}(C(H)C(F))\subset C(\GG)$. Hence, the relations become $[\rho(a),\beta(b)]=0$, $[x,f]=0$ and $\rho(f)=1_B\ot f$ for all $a\in C(G)$, $b\in B$, $x\in C(H)$ and $f\in C(F)$. Note that there is a unique unital $*$-homomorphism $\beta_G\,:\,C(G)\ot B\rightarrow B\ot C(\GG)$ satisfying $\beta_G(a\ot b)=\rho(a)\beta(b)$, for all $a\in C(G)$, $b\in B$.

\subsubsection{Comultiplication}

\noindent In the sequel, we assume that $\beta$ is a $\psi$-preserving action of $H$ on the finite dimensional C*-algebra $B$, where $\psi\in B^*$ is a faithful state and we denote by $u\in\Lcal(B)\ot C(H)\subset\Lcal(B)\ot C(\GG)$ the unitary representation of $H$ associated to the action $\beta\,:\, B\rightarrow B\ot C(H)$ as explained before Lemma \ref{LemRepAction}.

\begin{theorem}\label{ThmDefFreeWr}
There exists a unique unital $*$-homomorphism $\Delta\,:\,C(\GG)\rightarrow C(\GG)\ot C(\GG)$ such that $\Delta\vert_{C(H)\otm C(F)}=\Delta_{H\times F}$ and $(\id\ot\Delta)\rho=(\beta_G\ot\id)(\id\ot\rho)\Delta_G$. Moreover,
\begin{enumerate}
    \item $(C(\GG),\Delta)$ is a CQG and $C(H\times F) = C(H)\otm C(F) \subset C(\GG)$ is a dual quantum subgroup.
    \item The counit of $\GG$ is the unique character of $C(\GG)$ satisfying $$(\id\ot\varepsilon_\GG)\rho=\varepsilon_G(\cdot)1_B\text{ and }\varepsilon_\GG\vert_{C(H\times F)}=\varepsilon_{H\times F}.$$
    \item If $v\in\Lcal(H_v)\ot C(G)$ is a f.d. unitary representation of $G$ then,
    $$a(v):=(\id\ot(L_B\ot\id)\rho)(v)u_{23}\in\Lcal(H_v\ot B)\ot C(\GG)$$
    is a unitary representation of $\GG$ on $H_v\ot B$ and $\overline{a(v)}\simeq a(\overline{v})$.
    \item If $v\in \Lcal(H_v)\ot C(F)$ is a f.d. representation of $F$, seen as a representation of $G$ by the fact that $\Rep(F)$ is a full subcategory of $\Rep(G)$, then the representation $a(v)$ is the representation $v\ot u\in \Lcal(H_v\ot B)\ot C(H\times F),$ seen as a representation of $C(\GG)$ because we can see $\Rep(H\times F)$ as a full subcategory of $\Rep(\GG)$.
    \item Every irreducible representation of $\GG$ is equivalent to a subrepresentation of a tensor product of representations of the form $a(v)$, for $v\in\Irr(G)$, and irreducible representations of $H$. If the action $\beta$ is faithful, any irreducible of $G$ is equivalent to a subrepresentation of a tensor product of $a(v)$, for $v\in\Irr(G)$.
    \item For all $v,w,t\in\Rep(G)$ and all $S\in\Mor_G(v\ot w,t)$, the map
    $$(S\ot m)\Sigma_{23}\,:\,H_v\ot B\ot H_w\ot B\rightarrow H_t\ot B$$
    is in $\Mor_\GG(a(v)\ot a(w),a(t))$, where $\Sigma_{23}:H_v\ot B\ot H_w\ot B\rightarrow H_v\ot H_w\ot B\ot B$ is the flip map.
    \end{enumerate}
\end{theorem}

\begin{proof}
Let $\iota':=\Delta_{H\times F}\,:\, C(H\times F)\rightarrow C(\GG)\ot C(\GG)$, from which we get $\iota'_H=\Delta_H$ and $\iota'_F=\Delta_F$ with commuting images and $\rho':=(\beta_G\ot\id)(\id\ot\rho)\Delta_G\,:\,C(G)\rightarrow B\ot C(\GG)\ot C(\GG)$. For $a\in \Pol(G)$ one has $\rho'(a)=\sum (\beta_G\ot\id)(a_{(1)}\ot\rho(a_{(2)}))=\sum(\rho(a_{(1)})\ot 1)(\beta\ot\id)(\rho(a_{(2)}))$, where $\Delta_G(a)=\sum a_{(1)}\ot a_{(2)}$. Hence,
\begin{eqnarray*}
\rho'(a)(\id\ot\iota'_H)(\beta(b))&=&\rho'(a)(\id\ot\Delta_{H})(\beta(b))=\rho'(a)(\beta\ot\id)(\beta(b))\\
&=&\sum(\rho(a_{(1)})\ot 1)(\beta\ot\id)(\rho(a_{(2)}))(\beta\ot\id)(\beta(b))\\
&=&\sum(\rho(a_{(1)})\ot 1)(\beta\ot\id)(\rho(a_{(2)})\beta(b))\\
&=&\sum(\rho(a_{(1)})\ot 1)(\beta\ot\id)(\beta(b))(\beta\ot\id)(\rho(a_{(2)}))\\
&=&(\beta\ot\id)(\beta(b))\sum(\rho(a_{(1)})\ot 1)(\beta\ot\id)(\rho(a_{(2)}))\\
&=&(\id\ot\iota'_H)(\beta(b))\rho'(a).
\end{eqnarray*}

\noindent We also have that for $f\in \Pol(F)$, $\Delta_G(f)=\Delta_F(f)=\sum f_{(1)}\ot f_{(2)}\in\Pol(F)\ot\Pol(F)$ and:
\begin{eqnarray*}
\rho'(f) &=&(\beta_G\ot\id)(\id\ot\rho)\Delta_G(f)
        = \sum(\beta_G\ot\id)\left(f_{(1)}\ot\rho(f_{(2)})\right)\\
        &=&\sum(\beta_G\ot\id)(f_{(1)}\ot 1_B\ot f_{(2)})=\sum\rho(f_{(1)})\ot f_{(2)}=\sum 1_B\ot f_{(1)}\ot f_{(2)}\\
        &=&1_B\ot\Delta_F(f)=1_B\ot\iota'_F(f).
\end{eqnarray*}

\noindent The existence of $\Delta\,:\, C(\GG)\rightarrow C(\GG)\ot C(\GG)$ now follows from the universal property.

\vspace{0.2cm}

\noindent Let us show that $(\id_{\Lcal(H_v\ot B)}\ot\Delta)(a(v))=a(v)_{12}a(v)_{13}$ $\forall v\in\Rep(G)$.
\begin{eqnarray*}
&&(\id\ot\Delta)[(\id\ot(L_B\ot\id)\rho)(v)]\\
&=&(\id_{\Lcal(H_v)}\ot L_B\ot\id_{C(\GG)}\ot\id_{C(\GG)})(\id_{\Lcal(H_v)}\ot(\id_B\ot\Delta)\rho)(v)\\
&=&(\id\ot L_B\ot\id\ot\id)(\id\ot\beta_G\ot\id)(\id\ot\id\ot\rho)(\id\ot\Delta_G)(v)\\
&=&(\id\ot L_B\ot\id\ot\id)(\id\ot\beta_G\ot\id)(\id\ot\id\ot\rho)(v_{12}v_{13})\\
&=&(\id\ot L_B\ot\id\ot\id)(\id\ot\beta_G\ot\id)(v_{12}[(\id\ot\rho)(v)]_{134})\\
&=&(\id\ot L_B\ot\id\ot\id)\bigg([(\id\ot\rho)(v)]_{123}(\id\ot\beta\ot\id)((\id\ot\rho)(v))\bigg).
\end{eqnarray*}
Since $u$ is a representation of $H$ one has $(\id_{\Lcal(B)}\ot\Delta)(u)=(\id_{\Lcal(B)}\ot\Delta_H)(u)=u_{12}u_{13}$ hence,
$$
(\id_{\Lcal(H_v\ot B)}\ot\Delta)(a(v))=(\id_{\Lcal(H_v\ot B)}\ot\Delta)[(\id\ot(L_B\ot\id)\rho)(v)](\id_{\Lcal(H_v\ot B)}\ot\Delta)(u_{23})$$
$$
=(\id\ot L_B\ot\id\ot\id)\bigg([(\id\ot\rho)(v)]_{123}(\id\ot\beta\ot\id)((\id\ot\rho)(v))\bigg)u_{23}u_{24}.$$
\noindent On the other hand, viewing $a(v)_{12}a(v)_{13}$ as a $4$ legged object in $\Lcal(H_v)\ot\Lcal(B)\ot C(\GG)\ot C(\GG)$, and using the relation $u(L_B(b)\ot 1)u^*=(L_B\ot\id)(\beta(b))$, we get:
\begin{eqnarray*}
&&a(v)_{12}a(v)_{13}\\
&=&(\id\ot L_B\ot\id\ot\id)\left([(\id\ot\rho)(v)]_{123}\right)u_{23}(\id\ot L_B\ot\id\ot\id)\left([(\id\ot\rho)(v)]_{124}\right)u_{24}\\
&=&(\id\ot L_B\ot\id\ot\id)\bigg([(\id\ot\rho)(v)]_{123}(\id\ot\beta\ot\id)((\id\ot\rho)(v))\bigg)u_{23}u_{24}.
\end{eqnarray*}
This shows that $(\id_{\Lcal(H_v\ot B)}\ot\Delta)(a(v))=a(v)_{12}a(v)_{13}$, $\forall v\in\Rep(G)$.

\vspace{0.2cm}

\noindent We now prove $(6)$. Let $v,w\in\Rep(G)$. Viewing $(\Sigma_{23}\ot 1)a(v)_{13}a(w)_{23}(\Sigma_{23}\ot 1)$ as a $5$ legged object in $\Lcal(H_v)\ot\Lcal(H_w)\ot\Lcal(B)\ot\Lcal(B)\ot C(\GG)$ one has:
$$(\Sigma_{23}\ot 1)a(v)_{13}a(w)_{23}(\Sigma_{23}\ot 1)=[(\id\ot(L_B\ot\id)\rho)(v)]_{135}u_{35}[(\id\ot(L_B\ot\id)\rho)(v)]_{245}u_{45}$$

\noindent Hence, for $a,b\in B$, we get, with $Y:=(\Sigma_{23}\ot 1)a(v)_{13}a(w)_{23}(1\ot a\ot 1\ot b\ot 1)$,
\begin{eqnarray*}
Y&=&[(\id\ot(L_B\ot\id)\rho)(v)]_{135}u_{35}[(\id\ot\rho)(v)]_{245}(1\ot 1\ot a\ot\beta(b))\\
&=&[(\id\ot(L_B\ot\id)\rho)(v)]_{135}u_{35}(1\ot 1\ot a\ot\beta(b))[(\id\ot\rho)(v)]_{245}\\
&=&[(\id\ot\rho)(v)]_{135}\beta(a)_{35}\beta(b)_{45}[(\id\ot\rho)(v)]_{245}.
\end{eqnarray*}

\noindent It follows that, for any $S\in\Mor_G(v\ot w,t)$ one has:
\begin{eqnarray*}
(S\ot m\ot 1)Y&=&(S\ot\id_B\ot 1)[(\id\ot\rho)(v)]_{134}(1\ot1\ot\beta(a)\beta(b))[(\id\ot\rho)(w)]_{234}\\
&=&(S\ot\id_B\ot 1)[(\id\ot\rho)(v)]_{134}[(\id\ot\rho)(w)]_{234}(1\ot1\ot\beta(ab))\\
&=&(S\ot\id_B\ot 1)(\id\ot\rho)(v_{13}w_{23})(1\ot1\ot\beta(ab))\\
&=&(\id\ot\rho)(t)(S\ot\id_B\ot1)(1\ot1\ot\beta(ab))\\
&=&(\id\ot\rho)(t)(S\ot\id_B\ot1)u_{34}(1\ot 1\ot ab\ot 1)\\
&=&(\id\ot\rho)(t)u_{23}(S\ot m\ot1)(1\ot 1\ot a\ot b\ot 1)\\
&=&a(t)((S\ot m)\Sigma_{23})\ot 1)(1\ot a\ot 1\ot b\ot 1)
\end{eqnarray*}
\noindent Since this is true for all $a,b\in B$, it follows that
$$((S\ot m)\Sigma_{23}\ot 1)a(v)_{13}a(w)_{23}=a(t)((S\ot m)\Sigma_{23}\ot 1).$$

\noindent To show that $(C(\GG),\Delta)$ is a CQG, we will use the version \cite[Definition 1.1]{FP16} of the definition of a CQG. We already checked the second condition of \cite[Definition 1.1]{FP16}. Let us now check the first condition. Let $\mathcal{X}$ be the set of coefficients of $a(v)$, for $v\in\Irr(G)$. It suffices to show that the unital $*$-subalgebra of $C(\GG)$ generated by $\mathcal{X}$ and $\Pol(H)$ is dense in $C(\GG)$. For $\omega\in B^*$ define $\omega_0\in\Lcal(B)^*$, $\omega_0(T):=\omega(T1_B)$. Since $u(1_B\ot1)=1_B\ot 1$ we have,
\begin{eqnarray*}
(\id\ot\omega_0\ot\id)(a(v))&=&(\id\ot\omega\ot\id)\left((\id\ot(L_B\ot\id)\rho)(v)u_{23}(1\ot 1_B\ot 1)\right)\\
&=&(\id\ot\omega\ot\id)\left((\id\ot(L_B\ot\id)\rho)(v)(1\ot 1_B\ot 1)\right)\\
&=&(\id\ot\omega\ot\id)\left((\id\ot\rho)(v)\right)=(\id\ot(\omega\ot\id)\rho)(v).
\end{eqnarray*}
It follows that, for any $v\in\Rep(G)$, $(\omega\ot\id)\rho(v_{ij})\in\mathcal{X}$, for any coefficient $v_{ij}$ of $v$ and any $\omega\in B^*$. Hence, $\{(\omega\ot\id)\rho(a)\,:\,a\in\Pol(G),\,\omega\in B^*\}\subset{\rm Span}(\mathcal{X})$. Moreover, since $a(1)=u$, $\mathcal{X}$ also contains the coefficients of $u$. So the $*$-algebra generated by $\mathcal{X}$ is dense in $C(\GG)$ whenever the action $\beta$ is faithful. In general, the $*$-algebra generated by $\mathcal{X}$ and $\Pol(H)$ is dense in $C(\GG)$. It also implies $(5)$, by the general theory of CQG, once we know that $(C(\GG),\Delta)$ is a CQG. To conclude the proof of $(1)$ we only have to check the last condition of \cite[Definition 1.1]{FP16}. To do so, we use \cite[Remark 1]{FP16}. Let, for $v\in\Rep(G)$, $s_v\in\Mor_G(1,v\ot\overline{v})$ and $s_{\overline{v}}\in\Mor_G(1,\overline{v}\ot v)$ be a pair solving the conjugate equation for $(v,\overline{v})$. Since $\eta\in\Mor_H(1,u)$, $a(1)=u$, assertion $(6)$ gives:
$$S_v:=\Sigma_{23}(s_v\ot m^\star)\eta\in\Mor(1,a(v)\ot a(\overline{v}))\quad\text{and}\quad S_{\overline{v}}:=\Sigma_{23}(s_{\overline{v}}\ot m^\star)\eta\in\Mor(1,a({\overline{v}})\ot a(v)).$$
Suppose that the pair $(s_v,s_{\overline{v}})$ is standard. To show that $(C(\GG),\Delta)$ is a CQG it remains to prove that $S_{\overline{v}}$ is non-degenerated in the sense of \cite[Remark 1]{FP16}. Writing, as in Section \ref{SectionCQG}, $s_{\overline{v}}=\sum_iQ_{v,i}^{-\frac{1}{2}}e_i^{\overline{v}}\ot e_i^{v}$ we find, as in the proof of \cite[Proposition 2.7]{FP16}, and using the notations introduced at the end of Section \ref{SectionQAut},
$$S_{\overline{v}}=\sum_{\kappa,i,j,k}\sqrt{\frac{\lambda_{\kappa,j}}{Q_{v,i}\lambda_{\kappa,k}}}e^{\overline{v}}_i\ot b^\kappa_{jk}\ot e^v_i\ot b^\kappa_{kj}.$$
From this formula, it is clear that $S_{\overline{v}}$ is non-degenerated so $(C(\GG),\Delta)$ is a CQG. The same computation gives $S_v=\sum_{\kappa,i,j,k}\sqrt{\frac{Q_{v,i}\lambda_{\kappa,j}}{\lambda_{\kappa,k}}}e^v_i\ot b^\kappa_{jk}\ot e^{\overline{v}}_i\ot b^\kappa_{kj}$. From these formulas, one checks easily that the pair $(S_v,S_{\overline{v}})$ solves the conjugate equation for $(a(v),a(\overline{v}))$ so $\overline{a(v)}\simeq a(\overline{v})$ for all $v\in\Rep(G)$. 

\vspace{0.2cm}

\noindent Note that statement  $(4)$ is an easy consequence of the fact that $\rho(f)=1_B\ot f$ for all $f\in C(F)$. Finally, the existence and uniqueness of the character $\varepsilon_\GG$ satisfying the properties given in $(2)$ is a direct consequence of the universal property of $C(\GG)$. The fact that it is the counit of $\GG$ follows from the definition of $\Delta$ and $\varepsilon_\GG$.\end{proof}

\begin{definition} The CQG $\GG$ constructed above is called \textit{the generalized free wreath product} of $G$ and $(H,\beta)$, amalgamated on $F$, and denoted by $G\wr_{*,\beta,F}H$. When $F$ is the trivial group, we call it the generalized free wreath product of $G$ and $(H,\beta)$, denoted $G\wr_{*,\beta}H$.
\end{definition}

\noindent Recall that the C*-algebra $B$ is decomposed as a finite sum of matrix algebras: $$ B \simeq \bigoplus_{\kappa=1}^K M_{N_\kappa}(\C). $$ We also denote by $1_\kappa$ the unit of the block $\kappa$ so $1_B = \sum_\kappa 1_\kappa$.

\begin{remark}\label{RmkFwp}The following statements are easy to check.
\begin{enumerate}
\item If $F$ is equal to $G$ i.e. the inclusion $C(F)\subset C(G)$ is surjective, then $G\wr_{*,\beta,F}H \simeq G \times H$.
    \item The canonical surjection $\pi_H:=(\id\ot\varepsilon_G)\pi\,:\,C(\GG)\rightarrow C(H)$, where $\pi$ is constructed in Remark \ref{RmkIota} intertwines the comultiplications so $H$ is a compact quantum subgroup of $\GG$. Hence $H$ is both a quantum subgroup and a dual quantum subgroup of $\GG$.
    \item There exists a unique unital $*$-homomorphism $\pi_G\,:\, C(\GG)\rightarrow C(G)^{\ast K}$ such that, for all $a\in C(G)$, $(\id\ot\pi_G)\rho(a)= \sum_{\kappa=1}^K 1_\kappa \ot a$ and $\pi_G\vert_{C(H)}=\varepsilon_H(\cdot)1$. Moreover, $\pi_G$ is easily seen to be surjective and intertwines the comultiplications. Hence $G^{\ast K}$ is a compact quantum subgroup of $\GG$.
    \item If $\beta$ is the trivial action $B\rightarrow B\ot C(H)$ i.e. $\beta(b)=b\ot 1$ then: $$G\wr_{*,\beta,F}H=(G^{\ast_F K}) \underset{F}{*} (H\times F).$$
    \item  Suppose that $\beta_k\,:\, B_k\rightarrow B_k\ot C(H_k)$, $k=1,2$, two $\psi_k$-preserving actions, are \textit{isomorphic} in the sense that there exist two isomorphisms $\pi\,:\, B_1\rightarrow B_2$ and $\varphi\,:\, H_1\rightarrow H_2$ such that $\psi_2\circ\pi=\psi_1$ and $\beta_2\circ\pi=(\pi\ot\varphi)\beta_1$. Then, $$G\wr_{*,\beta_1,F}H_1\simeq G\wr_{*,\beta_2,F}H_2.$$
\item If we have a pair of dual quantum subgroups of $G$ , $\C 1_G \subset C(F_1) \subset C(F_2) \subset C(G)$, then we get surjective unital $*$-homomorphisms: $$ C(G\wr_{\ast,\beta}H) \rightarrow C(G\wr_{\ast,\beta,F_1} H) \rightarrow C(G\wr_{\ast,\beta,F_2} H) \rightarrow C(G\wr_{\ast,\beta,G} H) \simeq C(G)\otm C(H),$$ all of these morphisms intertwining the comultiplications. The last isomorphism is the one from $(1)$, in the case of the amalgamation on $G$ itself.
\end{enumerate}
\end{remark}

\noindent The following Proposition, which can be viewed as a generalization of Remark \ref{RmkFwp} $(4)$, shows that, up to an amalgamated free product, the action $\beta$ can always be assumed to be faithful.

\begin{proposition}\label{Propfaithful}
Denoting by $H_0$ the dual quantum subgroup of $H$ generated by the coefficients of $u$, there is an isomorphism $G\wr_{*,\beta,F}H \simeq (G\wr_{*,\beta,F}H_0) \underset{H_0 \times F}{\ast}(H\times F)$.
\end{proposition}

\begin{proof}
Let us denote by $\rho_0$ the map defining $\GG_0:=G\wr_{*,\beta,F}H_0$, $\rho$ the one for $\GG:=G\wr_{*,\beta,F}H$ and view $C(H_0\times F)\subset C(\GG_0)$ and $C(H\times F)\subset C(\GG)$. Define $\Sigma:=\GG_0\underset{H_0 \times F}{\ast}(H\times F)$ and view $C(\GG_0),C(H\times F)\subset C(\Sigma)$. By the universal property of $C(\GG)$, there is a unique unital $*$-homomorphism $\pi\,:\, C(\GG)\rightarrow C(\Sigma)$ such that:
$$(\id\ot\pi)\rho=\rho_0\,:\, C(G)\rightarrow B\ot C(\GG_0)\subset B\ot C(\Sigma)\text{ and }\pi\vert_{C(H\times F)}\text{ is the inclusion}.$$
It is easy to check that $\pi$ intertwines the comultiplication. Moreover, the universal property of $C(\GG_0)$ gives a unique unital $*$-homomorphism $\pi_0\,:\,C(\GG_0)\rightarrow C(\GG)$ such that $(\id\ot\pi_0)\rho_0=\rho$ and $\pi_0\vert_{C(H_0\times F)}$ is the inclusion $C(H_0\times F)\subset C(H\times F)\subset C(\GG)$. Now, by the universal property of the amalgamated free product $C(\Sigma)=C(\GG_0)\underset{C(H_0\times F)}{*}C(H\times F)$, there is a unique unital $*$-homomorphism $\pi^{-1}\,:\, C(\Sigma)\rightarrow C(\GG)$ such that $\pi^{-1}\vert_{C(\GG_0)}=\pi_0$ and $\pi\vert_{C(H\times F)}$ is the inclusion $C(H\times F)\subset C(\GG)$. It is clear that $\pi^{-1}$ is the inverse of $\pi$.\end{proof}

\subsection{Links with previous constructions} In this subsection we explore some more explicit examples and we make links with previously known constructions.

\subsubsection{Generalizations of wreath products}

The semi-direct product quantum group is defined and studied in details in \cite{Wa19}. Let us recall below the basic facts about this construction.

\vspace{0.2cm}

\noindent Let $G$ be a compact quantum group and $\Lambda$ a finite group acting on $C(G)$ by automorphisms of $G$, meaning that we have a group homomorphism $\alpha\,:\,\Lambda\rightarrow{\rm Aut}(C(G))$ such that $\Delta_G\circ\alpha_g=(\alpha_g\ot\alpha_g)\circ\Delta_G$ for all $g\in\Lambda$. Define the C*-algebra $C(G\rtimes\Lambda)=C(G)\otimes C(\Lambda)$ with the comultiplication $\Delta\,:\,C(G\rtimes\Lambda)\rightarrow C(G\rtimes\Lambda)\otimes C(G\rtimes\Lambda)$ such that:
$$\Delta(a\ot\delta_r)=\sum_{s\in\Lambda}\left[(\id\ot\alpha_{s^{-1}})(\Delta_G(a))\right]_{13}(1\ot\delta_s\ot 1\ot\delta_{s^{-1}r}).$$
In particular, the inclusion $C(\Lambda)\subset C(G\rtimes\Lambda)\,:\,x\mapsto 1\ot x$ preserves the comultiplications. It is shown in \cite{Wa19} that the pair $(C(G\rtimes\Lambda),\Delta)$ is a compact quantum group (in its maximal version).

\vspace{0.2cm}

\noindent Let $\Lambda\curvearrowright X$ be any action of a finite group $\Lambda$ on a finite set $X$, denoted by $\gamma\cdot x$, for $\gamma\in\Lambda$ and $x\in X$. We get an action $\alpha\,:\,\Lambda\rightarrow{\rm Aut}(G^{*_FX})$ of $\Lambda$ on $G^{*_FX}$ by quantum automorphisms of $G^{*_FX}$ defined by $\alpha_\gamma\circ\nu_x=\nu_{\gamma\cdot x}$, where $\nu_x\,:\, C(G)\rightarrow C(G^{*_F X})=\underset{x\in X, C(F)}{*}C(G)$ is the $x^{th}$ copy of $C(G)$ and we can consider the semi-direct product quantum group $G^{*_F X}\rtimes\Lambda$.

\vspace{0.2cm}

\noindent Fix a finite dimensional C*-algebra $B$ with a faithful state $\omega\in B^*$ and denote $$B^X:=B\oplus\dots\oplus B\quad \vert X\vert\text{-times}.$$ Define the faithful state $\psi\in (B^X)^*$ by $\psi(b):=\frac{1}{\vert X\vert}\sum_{x\in X}\omega(\chi_x(b))$, for all $b\in B^X$, where $\chi_x\,:\, B^X\rightarrow B$ is the canonical projection on the $x$-copy of $B$ in $B^X$. Consider now the unique action $\beta\,:\,B^X\rightarrow B^X\ot C(\Lambda)$ of the quantum group $C(\Lambda)$ on the finite dimensional unital C*-algebra $B^X$ satisfying, for all $x\in X$ and $b\in B^X$, $$(\chi_x\ot\id)\beta(b)=\sum_{\gamma\in\Lambda}\chi_{\gamma^{-1}\cdot x}(b)\ot\delta_\gamma.$$
Note that $\beta$ preserves the faithful state $\psi$ on $B^X$. In the following Proposition, $K$ denotes the number of matrix blocks in the decomposition $B=\bigoplus_{\kappa=1}^K M_{N_\kappa}(\C)$. Consider the matrix units $e^\kappa_{ij}\in B$ and write $e^{x,\kappa}_{ij}\in B^X$ for the matrix unit $e^\kappa_{ij}\in B$ viewed in the $x$-copy of $B$ in $B^X$. Define $e_{x,\kappa}:=\sum_re^{x,\kappa}_{rr}\in B^X$.

\begin{proposition}\label{PropSemiDirect}
If $\Lambda\curvearrowright X$ is free then there is a canonical isomorphism $$G\wr_{*,\beta,F}\Lambda\simeq G^{*_F X\times K}\rtimes\Lambda,$$
where $X\times K:=X\times\{1,\dots,K\}$ and the action $\Lambda\curvearrowright X\times K$ is given by $\gamma\cdot(x,\kappa):=(\gamma\cdot x,\kappa)$.
\end{proposition}

\begin{proof}
Let $\GG:=G\wr_{*,\beta,F}\Lambda$. Recall that $C(G^{*_F X}\rtimes\Lambda)=C(G^{*_F X})\ot C(\Lambda)$ and define $\rho\,:\,C(G)\rightarrow B^X\ot C(G^{*_F X\times K})\subset B^X\ot C(G^{*_F X\times K}\rtimes\Lambda)$ by $\rho(a):=\sum_{x\in X,1\leq\kappa\leq K}e_{x,\kappa}\ot\nu_{x,\kappa}(a)$ and both $\iota_F\,:\,C(F)\rightarrow C(G^{*_F X\times K})\subset C(G^{*_F X\times K}\rtimes\Lambda)$ and $\iota_\Lambda\,:\, C(\Lambda)\rightarrow C(G^{*_F X}\rtimes\Lambda)$ are the inclusions. It is clear that they satisfy the relations of $C(\GG)$. Moreover, with these maps, $C(G^{*_F X\times F}\rtimes\Lambda)$ satisfies the universal property of $C(\GG)$. Indeed, if $\pi\,:\, C(G)\rightarrow B^X\ot C$, $\iota'_F\,:\, C(F)\rightarrow C$ and $\iota'_\Lambda\,:\, C(\Lambda)\rightarrow C$ are unital $*$-homomorphisms such that
$$\pi(f)=1_{B^X}\ot\iota'_F(f), \,\,[\pi(a),(\id\ot\iota'_\Lambda)(\beta(b))]=0, \,\,[\iota'_\Lambda(c),\iota'_F(f)]=0$$
for all $a\in C(G)$, $b\in B^X$, $c\in C(\Lambda)$ and $f\in C(F)$. Writing $\pi(a)=\sum_{x,\kappa,i,j} e^{x,\kappa}_{ij}\ot\pi_{ij}^{x,\kappa}(a)$ and applying $(\chi_x\ot \id)$ to the relation $[\pi(a),(\id\ot\iota'_\Lambda)(\beta(b))]=0$ we find, for all $x\in X$, $a\in C(G)$ and $b\in B^X$,
$$\sum_{\kappa,i,j,\gamma}e^\kappa_{ij}\chi_{\gamma^{-1}\cdot x}(b)\ot\pi^{x,\kappa}_{ij}(a)\iota_\Lambda'(\delta_\gamma)=\sum_{\kappa,i,j,\gamma}\chi_{\gamma^{-1}\cdot x}(b)e^\kappa_{ij}\ot\iota_\Lambda'(\delta_\gamma)\pi^{x,\kappa}_{ij}(a).$$
Fix any $\kappa, r,s$ and $y\in X$ and apply this equation to $b:=e^{\kappa,y}_{rs}\in B^X$ to get:
$$\sum_{i}e^\kappa_{is}\ot\pi^{x,\kappa}_{ir}(a)\iota'_\Lambda(u_{xy})=\sum_{j}e^\kappa_{rj}\ot\iota'_\Lambda(u_{xy})\pi_{sj}^{x,\kappa}(a),$$
where $u_{xy}:=1_{\{\gamma\in\Lambda\,:\,\gamma\cdot y=x\}}$. This implies that, for all $\kappa,i,r$, all $a\in C(G)$ and all $x,y\in X$,
\begin{equation}\label{EqCummutation}\pi^{x,\kappa}_{ir}(a)\iota'_\Lambda(u_{xy})=\delta_{ir}\iota'_\Lambda(u_{xy})\pi^{x,\kappa}_{ss}(a).\end{equation}
Since $1=\sum_yu_{xy}\in C(\Lambda)$, for all $x\in X$, we deduce in particular that $\pi^{x,\kappa}_{ir}(a)=\delta_{ir}\pi^{x,\kappa}_{ss}(a)$ for all $x,\kappa,i,r,s,a$. For $a\in C(G)$ let $\pi_{x,\kappa}(a)\in C$ be the element satisfying $\pi_{x,\kappa}(a)=\pi^{x,\kappa}_{rr}(a)$ for all $1\leq r\leq N_\kappa$ and note that $\pi$ is of the form:
$$\pi(a)=\sum_{x,\kappa,r}e^{x,\kappa}_{rr}\ot\pi_{x,\kappa}(a)=\sum_{x,\kappa}e_{x,\kappa}\ot\pi_{x,\kappa}(a),$$
Since the projections $e_{x,\kappa}$ are pairwise orthogonal and $\pi$ is a unital $*$-homomorphism, each $\pi_{x,\kappa}\,:\,C(G)\rightarrow C$ is a unital $*$-homomorphism and, since $\pi(f)=1_{B^X}\ot\iota'_F(f)$, we have $\pi_{x,\kappa}(f)=\iota'_F(f)$ for all $f\in F$. Moreover, Equation $(\ref{EqCummutation})$ also implies that, for all  $a\in C(G)$ and all $x,y\in X$ and all $1\leq\kappa\leq K$, $\pi_{x,\kappa}(a)\iota'_\Lambda(u_{xy})=\iota'_\Lambda(u_{xy})\pi_{x,\kappa}(a)$. Hence, for all $a\in C(G)$ and $(x,\kappa)\in X\times K$, $\pi_{x,\kappa}(a)$ commutes with $\iota'_\Lambda(c)$ for all $c\in\Acal_x$, where $\Acal_x\subset C(\Lambda)$ is the unital $*$-algebra generated by $\{u_{xy}\,:\, y\in X\}$. Since the action $\Lambda\curvearrowright X$ is free, $\Acal_x$ separates the points of $\Lambda$ and so $\Acal_x=C(\Lambda)$, for all $x\in X$. This implies that for all $a\in A$, $(x,\kappa)\in X\times K$ and $c\in C(\Lambda)$ $[\pi_{x,\kappa}(a),\iota'_\Lambda(c)]=0$. Hence, there exists a unique unital $*$-homomorphism $\widetilde{\pi}\,:\, C(G^{*_FX\times K})\ot C(\Lambda)\rightarrow C$ such that $\widetilde{\pi}\circ\nu_{x,\kappa}=\pi_{x,\kappa}$ and $\widetilde{\pi}\circ\iota'_\Lambda=\iota_\Lambda$. This shows the universal property hence we may and will assume that $C(\GG)=C(G^{*_F X\times K})\ot C(\Lambda)$ with maps $\rho$, $\iota_\Lambda$ and $\iota_F$ given above. Let us denote by $\Delta_1$ the comultiplication on $C(\GG)$ given by the free wreath product construction and by $\Delta_2$ the one given by the semi-direct product construction. It suffices to show that $\Delta_1=\Delta_2$. Since by construction, they both coincide with $\Delta_\Lambda$ on $C(\Lambda)$, it suffices to show that $\Delta_1\circ\nu_{x,\kappa}=\Delta_2\circ\nu_{x,\kappa}$ for all $(x,\kappa)\in X\times K$. Note that, for all $a\in\Pol(G)$,  $\Delta_2(\nu_{x,\kappa}(a))=\sum_{\gamma\in\Lambda}\nu_{x,\kappa}(a_{(1)})\iota_\Lambda(\delta_\gamma)\ot\nu_{\gamma^{-1}\cdot x,\kappa}(a_{(2)})$ and,
\begin{eqnarray*}
(\id\ot\Delta_1)(\rho(a))&=&\sum_{x,\kappa} e_{x,\kappa}\ot\Delta_1(\nu_{x,\kappa}(a))=(\beta_G\ot\id)(\id\ot\rho)\Delta_G(a)
\\
&=&\sum\beta_G(a_{(1)}\ot e_{y,\zeta})\ot\nu_{y,\zeta}(a_{(2)})=\sum\rho(a_{(1)})\beta(e_{y,\zeta})\ot\nu_{y,\zeta}(a_{(2)})\\
&=&\sum e_{x,\kappa}e_{\gamma\cdot y,\zeta}\ot\nu_{x,\kappa}(a_{(1)})\iota_\Lambda(\delta_\gamma)\ot\nu_{y,\zeta}(a_{(2)})\\
&=&\sum e_{x,\kappa}\ot\nu_{x,\kappa}(a_{(1)})\iota_\Lambda(\delta_\gamma)\ot\nu_{\gamma^{-1}\cdot x,\kappa}(a_{(2)})
\end{eqnarray*}
It follows that $\Delta_1\circ\nu_{x,\kappa}=\Delta_2\circ\nu_{x,\kappa}$ for all $(x,\kappa)\in X\times K$.\end{proof}

\begin{remark}
Gromada's wreath product \cite{Gr23} is a particular case of generalized free wreath products: this follows from Proposition \ref{PropSemiDirect} (with $B=\C$) and \cite[Proposition 2.27]{FT24}.
\end{remark}

\subsubsection{The amalgamated free wreath product by $S_N^+$} This construction is a slight variation on Bichon's construction by adding an amalgamation on a dual quantum subgroup. This variation on Bichon's definition is due to Freslon \cite{Fr22}. The computation of the Haar measure and the study of its operator algebras have been done in \cite{FT24}. Arguing as in the proof of Proposition \ref{PropMotivation} (and checking that the comultiplications coincide) we can easily prove the following. We leave the proof to the reader.

\begin{proposition}\label{PropBichon}
Let $G$ be a compact quantum group with a dual quantum subgroup $F$. For the universal action $\beta\,:\, S_N^+\curvearrowright\C^N$ there is a canonical isomorphism $G\wr_{*,\beta,F} S_N^+\simeq G\wr_{*,F}S_N^+$, where $G\wr_{*,F}S_N^+$ denotes the free wreath product from \cite{Fr22}.
\end{proposition}

\subsubsection{The free wreath product by a quantum automorphism group} When $\beta\,:\, \GB\curvearrowright B$ is the universal action of the quantum automorphism group $\GB$, we identify our generalized free wreath product $G\wr_{*,\beta}\GB$ with the free wreath product $G\wr_{*}\GB$ constructed in \cite{FP16} in the following Proposition. Let's begin by recalling the free wreath product construction $G\wr_* {\rm Aut}^+(B,\psi)$ from \cite{FP16}. The C*-algebra $C(G\wr_*{\rm Aut}^+(B,\psi))$ is the universal unital C*-algebra generated by coefficients of $r(x)\in\mathcal{L}(H_x\ot B)\ot C(G\wr_*{\rm Aut}^+(B,\psi))$, for all $x\in\Irr(G)$, with relations:

\begin{itemize}
\item $r(x)$ is unitary for all $x\in\Irr(G)$
\item $(S\ot m)\Sigma_{23}\in\Mor(r(x)\ot r(y),r(z))$ for all $S\in\Mor_G(x\ot y,z)$ and all $x,y,z\in\Irr(G)$.
\item $\eta\in\Mor(1,r(1))$, where $1$ is the trivial representation.
\end{itemize}

\noindent The comultiplication $\Delta\,:\, C(G\wr_*{\rm Aut}^+(B,\psi))\rightarrow C(G\wr_*{\rm Aut}^+(B,\psi))\ot C(G\wr_*{\rm Aut}^+(B,\psi))$ is the unique unital $*$-homomorphism such that $r(x)$ is a representation for all $x\in\Irr(G)$ i.e. $$(\id\ot\Delta)(r(x))=r(x)_{12} r(x)_{13}\text{ for all }x\in\Irr(G).$$

\begin{proposition}
There is a unique $*$-isomorphism $\pi\,:\, C(G\wr_*\GB)\rightarrow C(G\wr_{*,\beta}\GB)$ such that $(\id\ot\pi)(r(x))=a(u^x)\in\mathcal{L}(H_x)\ot\Lcal(B)\ot C(\GG)$ for all $x\in\Irr(G)$. Moreover, $\pi$ intertwines the comultiplications.
\end{proposition}

\begin{proof}
The uniqueness is clear and the existence of $\pi$ is a direct consequence of the universal property of $C(G\wr_*\GB)$ and assertion $(6)$ of Theorem \ref{ThmDefFreeWr}. It is clear that $\pi$ intertwines the comultiplications since $a(u^x)$ is a unitary representation. Let us construct the inverse of $\pi$. By the universal property of $C(G)$ and \cite[Proposition 8.1]{FP16} there exists a unique unital $*$-homomorphism $\varphi_0\,:\,C(G)\rightarrow B\ot C(G\wr_*\GB)$ such that $$(\id\ot\varphi_0)(v)=r'(v):=r(v)(\id\ot 1_B\ot 1)\text{ for all }v\in\Rep(G),$$
where $r(v)$, for non irreducible $v\in\Rep(G)$, is defined just before \cite[Proposition 8.1]{FP16} (denoted there $a(v)$). By the universal property of $C(\GB)$ there is a unique unital $*$-homomorphism  $\iota_B\,:\,C(\GB)\rightarrow C(G\wr_*\GB)$ such that $(\id\ot\iota_B)(u)=r(1)$. From \cite[Proposition 8.1]{FP16} we know that $(\id_{H_v}\ot m)\in\Mor_{G\wr_*\GB}(r(v)\ot r(1),r(v))$ and $(\id_{H_v}\ot m)\Sigma_{12}\in\Mor_{G\wr_*\GB}(r(1)\ot r(v),r(v))$. Hence,

\begin{eqnarray*}
(\id_{H_v}\ot m\ot 1)r(v)_{124}r(1)_{34}&=&r(v)(\id_{H_v}\ot m\ot 1)\\
&=&((\id_{H_v}\ot m)\Sigma_{12}\ot 1)r(1)_{14}r(v)_{234}(\Sigma_{12}\ot 1)\\
&=&(\id_{H_v}\ot m\ot 1)r(1)_{24} r(v)_{134}
\end{eqnarray*}

\noindent Take an orthonormal basis $(e_r)_r$ of $H_v$ with matrix units $(e_{rs})_{rs}$ in $\Lcal(H_v)$, write $v=\sum_{r,s}e_{rs}\ot v_{rs}$ and:
$$r(v)=\sum_{r,s,i,j,\kappa,k,l,\zeta}e_{rs}\ot b_{ij}^\kappa\circ (b_{kl}^\zeta)^\star\ot r(v)^{ij,\kappa}_{kl,\zeta,rs}\quad r(1)=\sum_{i,j,\kappa,k,l,\zeta} b_{ij}^\kappa\circ (b_{kl}^\zeta)^\star\ot r(1)^{ij,\kappa}_{kl,\zeta,rs}.$$
Note that $\sum_{\gamma,r}\lambda_{\gamma,r}^{\frac{1}{2}}r(1)^{ij,\kappa}_{rr,\gamma}=\delta_{ij}\lambda_{\kappa,i}^{\frac{1}{2}}$, since $\iota_B$ is a unital $*$-homomorphism. From the relation $b^\kappa_{ip}b^{\zeta}_{qj}=\delta_{\kappa\zeta}\delta_{pq}\lambda^{-\frac{1}{2}}_{\kappa,p}b^\kappa_{ij}$ we find $m\circ (b^\kappa_{ip}\circ (b^\zeta_{kl})^\star\ot b^{\kappa'}_{qj}\circ (b^{\zeta'}_{k'l'})^\star)=\delta_{\kappa,\kappa'}\delta_{p,q}\lambda_{\kappa,p}^{-1/2}b^\kappa_{ij}\circ  ((b^\zeta_{kl})^\star\ot (b^{\zeta'}_{k'l'})^\star)$.

\vspace{0.2cm}

\noindent It follows that $(\id_{H_v}\ot m\ot 1)r(v)_{124}r(1)_{34}$ is equal to:
$$\sum e_{rs}\ot b^\kappa_{ij}\circ((b^\zeta_{kl})^\star\ot(b^{\zeta'}_{k'l'})^\star)\ot\left(\sum_p\lambda_{\kappa,p}^{-1/2}r(v)^{ip,\kappa}_{kl,\zeta,rs}r(1)^{pj,\kappa}_{k'l',\zeta'}\right),$$
and $(\id_{H_v}\ot m\ot 1)r(1)_{24}r(v)_{134}$ is equal to:
$$\sum e_{rs}\ot b^\kappa_{ij}\circ((b^\zeta_{kl})^\star\ot(b^{\zeta'}_{k'l'})^\star)\ot\left(\sum_p\lambda_{\kappa,p}^{-1/2}r(1)^{ip,\kappa}_{kl,\zeta}r(v)^{pj,\kappa}_{k'l',\zeta',rs}\right).$$

\noindent Then, for all $r,s,\kappa,i,j,\zeta,k,l,\zeta',k',l'$ one has

\begin{equation}\label{Eqfwp1}
\sum_p\lambda_{\kappa,p}^{-1/2}r(v)^{ip,\kappa}_{kl,\zeta,rs}r(1)^{pj,\kappa}_{k'l',\zeta'}=\sum_p\lambda_{\kappa,p}^{-1/2}r(1)^{ip,\kappa}_{kl,\zeta}r(v)^{pj,\kappa}_{k'l',\zeta',rs}.
\end{equation}

\noindent Note that $(b_{kl}^\zeta)^\star(1_B)=\psi((b_{kl}^\zeta)^*)=\delta_{kl}\lambda_{\zeta,k}^{1/2}$ hence,
$$r'(v)=\sum e_{rs}\ot b_{ij}^\kappa\circ (b_{kl}^\zeta)^\star(1_B)\ot r(v)^{ij,\kappa}_{kl,\zeta,rs}=\sum e_{rs}\ot b_{ij}^\kappa\ot  \lambda_{\zeta,k}^{1/2} r(v)^{ij,\kappa}_{kk,\zeta,rs}.$$
\noindent It follows that $\varphi_0(v_{rs})=\sum b_{ij}^\kappa\ot  \lambda_{\zeta,q}^{1/2} r(v)^{ij,\kappa}_{qq,\zeta,rs}$. Writing $\beta(b^\alpha_{kl})=\sum_{\kappa,i,j}b^\kappa_{ij}\ot u^{ij,\kappa}_{kl,\alpha}$ and using Equation $(\ref{Eqfwp1})$ we deduce that, with $X:=(\id\ot\iota_B)(\beta(b^\alpha_{kl}))\varphi_0(v_{rs})$,

\begin{eqnarray*}
X&=&\sum b^\kappa_{ij}\ot\sum_{p,\zeta,q}\left(\frac{\lambda_{\zeta,q}}{\lambda_{\kappa,p}}\right)^{1/2}r(1)^{ip,\kappa}_{kl,\alpha}r(v)^{pj,\kappa}_{qq,\zeta,rs}
=\sum b^\kappa_{ij}\ot\sum_{p,\zeta,q}\left(\frac{\lambda_{\zeta,q}}{\lambda_{\kappa,p}}\right)^{1/2}r(v)^{ip,\kappa}_{kl,\alpha,rs}r(1)^{pj,\kappa}_{qq,\zeta}\\
&=&\sum b^\kappa_{ij}\ot\sum_{p}\lambda_{\kappa,p}^{-1/2}r(v)^{ip,\kappa}_{kl,\alpha,rs}\left(\sum_{\zeta,q}\lambda_{\zeta,q}^{1/2}r(1)^{pj,\kappa}_{qq,\zeta}\right)=\sum b^\kappa_{ij}\ot r(v)^{ij,\kappa}_{kl,\alpha,rs}.
\end{eqnarray*}

\noindent The same computation with $Y:=\varphi_0(v_{rs})(\id\ot\iota_B)(\beta(b^\alpha_{kl}))$ gives:

\begin{eqnarray*}
Y&=&\sum b^\kappa_{ij}\ot\sum_{p,\zeta,q}\left(\frac{\lambda_{\zeta,q}}{\lambda_{\kappa,p}}\right)^{1/2}r(v)^{ip,\kappa}_{qq,\zeta,rs}r(1)^{pj,\kappa}_{kl,\alpha}
=\sum b^\kappa_{ij}\ot\sum_{p,\zeta,q}\left(\frac{\lambda_{\zeta,q}}{\lambda_{\kappa,p}}\right)^{1/2}r(1)^{ip,\kappa}_{qq,\zeta}r(v)^{pj,\kappa}_{kl,\alpha,rs}\\
&=&\sum b^\kappa_{ij}\ot\sum_{p}\lambda_{\kappa,p}^{-1/2}\left(\sum_{\zeta,q}\lambda_{\zeta,q}^{1/2}r(1)^{ip,\kappa}_{qq,\zeta}\right)r(v)^{pj,\kappa}_{kl,\alpha,rs}=\sum b^\kappa_{ij}\ot r(v)^{ij,\kappa}_{kl,\alpha,rs}.
\end{eqnarray*}

\noindent Hence, $(\id\ot\iota_B)(\beta(b^\alpha_{kl}))\varphi_0(v_{rs})=\varphi_0(v_{rs})(\id\ot\iota_B)(\beta(b^\alpha_{kl}))$ and since this holds for any $v\in\Rep(G)$ and any $\zeta,k,l$, we deduce that $[\varphi_0(a),(\id\ot\iota_B)(\beta(b))]=0$ for all $a\in C(G)$ and all $b\in B$. By the universal property of $C(G\wr_{*,\beta}\GB)$, there exists a unique unital $*$-homomorphism $\varphi\,:\,C(G\wr_{*,\beta}\GB)\rightarrow C(G\wr_{*}\GB)$ such that $\varphi\vert_{C(\GB)}=\iota_B$ and  $(\id\ot\varphi)\rho=\varphi_0$.
\end{proof}

\subsubsection{A new example}\label{SectionNewExample} Even starting from classical groups, our construction of generalized free wreath product produces new CQG, as shown in the following example.

\vspace{0.2cm}

\noindent Let $\Lambda$ be a finite group and $\Gamma$ a discrete group. Consider the $\widehat{\Lambda}$-action on the finite dimensional C*-algebra $B:=C^*(\Lambda)$ by left translation i.e. $\beta\,:\, C^*(\Lambda)\rightarrow C^*(\Lambda)\ot C^*(\Lambda)$, $\beta(\gamma)=\gamma\ot\gamma$ for all $\gamma\in\Lambda$ and note that $\beta$ is $\tau$-preserving, where $\tau$ is the canonical trace of $C^*(\Lambda)$ (i.e. the Haar state of $\widehat{\Lambda}$). As observed in Example \ref{ExAction1}, this action is ergodic but not $2$-ergodic as soon as $\vert\Lambda\vert\neq 2$. Then it is easy to check that $C(\widehat{\Gamma}\wr_{*,\beta}\widehat{\Lambda})$ is the universal unital C*-algebra generated by elements $\nu_\gamma(g)$, for $\gamma\in \Lambda$ and $g\in\Gamma$ with relations,
\begin{itemize}
\item $(\nu_\gamma(g))^*=\nu_{\gamma^{-1}}(g^{-1})$ and $\nu_\gamma(1)=\delta_{\gamma,1}1$ for all $g\in\Gamma$ and $\gamma\in\Lambda$.
    \item $\nu_\gamma(gh)=\sum_{r,s\in\Lambda,\,rs=\gamma}\nu_r(g)\nu_s(h)$, for all $g,h\in\Gamma$, $\gamma\in\Lambda$,
    \item $s\nu_{rs}(g)=\nu_{sr}(g)s$, for all $g\in\Gamma$, $r,s\in\Lambda$,
\end{itemize}
The comultiplication is:
$$\Delta(\gamma)=\gamma\ot\gamma\quad\text{and}\quad\Delta(\nu_\gamma(g))=\sum_{r,s\in\Lambda,\,rs=\gamma} \nu_r(g)s\ot\nu_s(g)\quad\text{for all }\gamma\in\Lambda,\,g\in\Gamma.$$

\vspace{0.2cm}

\noindent While, to the best of the authors knowledge, this type of quantum groups has never been studied before, the case of an abelian group $\Lambda$ is actually well known as shown in the following remark.

\begin{remark}\label{RemarkNewExample} If $\Lambda$ is abelian then we have a canonical isomorphism $\widehat{\Gamma}\wr_{*,\beta}\widehat{\Lambda}\simeq \widehat{\Gamma}^{*\Lambda}\rtimes\Lambda$, where the action $\Lambda\curvearrowright\Lambda$ is by left multiplication. Indeed, by Proposition \ref{PropSemiDirect} and Remark \ref{RmkFwp} $(5)$, it suffices to show that the two actions $\Lambda\curvearrowright \C^\Lambda=C(\Lambda)$ and $\widehat{\Lambda}\curvearrowright C^*(\Lambda)=C(\widehat{\Lambda})$ are isomorphic and it follows from the self Pontryagin's duality of finite abelian groups. 
\end{remark}

\subsection{The ergodic decomposition}

\noindent Whenever the state $\psi$ on the C*-algebra $B$ is not a $\delta$-form, we can decompose $B$ as a direct sum $(B,\psi)= \bigoplus_i (B_i,\psi_i)$, with $\psi_i$ a $\delta_i$-form for all $i$, for pairwise distinct $\delta_i$'s, see Remarks \ref{RmkSpectralDecomposition} and \ref{remadelta} $(1)$. We will call this decomposition \textit{the ergodic decomposition of $(B,\psi)$}.

\vspace{0.2cm}

\noindent A result by Brannan demonstrates that the quantum automorphism group of $(B, \psi)$ can be decomposed as a free product of the quantum automorphism groups of the blocks $(B_i, \psi_i)$. This is expressed through the ergodic decomposition $\GB \simeq \ast_{i}{\rm Aut}^+(B_i, \psi_i)$. Fima and Pittau extended this result to the free wreath product in \cite{FP16}, showing that for any compact quantum group $G$, if we denote $\GG_i = G \wr_* {\rm Aut}^+(B_i, \psi_i)$, the ergodic decomposition is given by $G \wr_* \GB \simeq \ast_i \GG_i$. However, this result does not generally hold for compact quantum groups acting on $B$ via $\psi$-preserving actions. Specifically, when considering compact subgroups of $\GB$ that are quantum automorphism groups of quantum graphs on $B$, \cite{Da22} provides an example where the quantum automorphism group of a quantum graph cannot be expressed as a free product of the quantum automorphism groups of its subgraphs. Consequently, we cannot extend the result to generalized free wreath products. Nevertheless, some results remain valid, and our perspective on free wreath products offers new insights into the ergodic decomposition proposed by Fima and Pittau.

\vspace{0.2cm}

\noindent We begin by recalling a basic construction of actions.

\begin{proposition}\label{PropFreeProdAction}
Let $I$ be a finite set and for all $i\in I$, $B_i$ a finite dimensional C*-algebra. Consider the unital C*-algebra $B:=\bigoplus B_i$ with canonical surjections $\chi_i\,:\, B\rightarrow B_i$. Suppose that $\beta_i\,:\, B_i\rightarrow B_i\ot C(H_i)$ is an action, for all $i\in I$. Let $H$ be a CQG containing all the $H_i$ as dual quantum subgroups. Then, there exists a unique action $\beta\,:\, B\rightarrow B\ot C(H)$ such that, viewing $C(H_i)\subset C(H)$ one has, for all $i\in I$, $\beta_i\circ\chi_i=(\chi_i\ot\id)\beta$. Moreover, if all the $\beta_i$ are $\psi_i$-preserving for some state $\psi\in B_i^*$ then, $\beta$ is $\psi$-preserving for $\psi:=\vert I\vert^{-1}\sum_{i\in I}\psi_i\circ\chi_i$.
\end{proposition}

\begin{proof}
View $\beta_i\circ\chi_i\,:\, B\rightarrow B_i\ot C(H_i)\subset B_i\ot C(H)$. By the universal property of direct sum there exists a unique unital $*$-homomorphism $\beta\,:\, B\rightarrow B\ot C(H)$ such that $(\chi_i\ot\id)\beta=\beta_i\circ\chi_i$. Since $H_i$ is a dual quantum subgroup of $H$ one has, for all $i\in I$, \begin{eqnarray*}
(\chi_i\ot\id\ot\id)(\id\ot\Delta_H)\beta&=&(\id\ot\Delta_H)\beta_i\circ\chi_i=(\id\ot\Delta_{H_i})\beta_i=(\beta_i\ot\id)\beta_i\circ\chi_i\\
&=&(\beta_i\circ\chi_i\ot\id)\beta=(\chi_i\ot\id\ot\id)(\beta\ot\id)\beta
\end{eqnarray*}
Hence, $(\id\ot\Delta_H)\beta=(\beta\ot\id)\beta$. Note that $(\chi_i\ot\varepsilon_H)\beta=(\id\ot\varepsilon_H)\beta_i\circ\chi_i=(\id\ot\varepsilon_{H_i})\beta_i\circ\chi_i=\chi_i(\cdot)1$ for all $i\in I$. Hence, $(\id\ot\varepsilon_H)\beta=\id_B$ and $\beta$ is an action. The uniqueness and the last statement are clear.\end{proof}

\noindent The action $\beta\,:\, H\curvearrowright B$ constructed in Proposition \ref{PropFreeProdAction} will be called \textit{the action associated with $(\beta_i)_{i\in I}$}.

\begin{example}\label{ExErgoDec}
Let $\beta_B\,:\,\GB\curvearrowright B$ be the universal action and assume that $\psi$ is not a $\delta$-form. Write the ergodic decomposition decomposition $(B,\psi)=\bigoplus (B_i,\psi_i)$. Brannan's result precisely means that $\GB=\underset{i\in I}{*}{\rm Aut}^+(B_i, \psi_i)$ and $\beta_B$ is the action associated with the universal actions $\beta_{B_i}\,:\,{\rm Aut}^+(B_i, \psi_i)\curvearrowright B_i$.
\end{example}

\noindent Let $H$ be the fundamental quantum group of a graph of CQG $(H_p,\Sigma_e,\G,\T)$ in the sense of \cite{FF14}. This means that:

\begin{itemize}
    \item $\G$ is connected graph (in the sense of \cite{Se77}) with vertex set $V(\G)$, edge set $E(\G)$, source and range maps $s,r\,:\,E(\G)\rightarrow V(\G)$ and $\T\subset\G$ is a maximal subtree.
    \item For every $p\in V(\G)$ and $e\in E(\G)$ $H_p$ and $\Sigma_e$ are CQG.
    \item For every $e\in E(\G)$, $\Sigma_e=\Sigma_{\ebar}$.
    \item For every $e\in E(\G)$, $s_e\,:\,C(\Sigma_e)\rightarrow C(H_{s(e)})$ is a faithful unital $*$-homomorphism intertwining the comultiplications (i.e. $\Sigma_e$ is a dual quantum subgroup of $H_{s(e)}$).
\end{itemize}

\noindent Define $r_e:=s_{\ebar}\,:\, C(\Sigma_e)\rightarrow C(H_{r(e)})$. $C(H)$ is defined as the universal unital C*-algebra generated by the algebras $C(H_p)$, for $p\in V(\G)$ and unitaries $u_e$, for $e\in E(\G)$ with relations:

\begin{itemize}
    \item For all $e\in E(\G)$, $u_e^*=u_{\ebar}$,
    \item For all $e\in E(\G)$ and all $b\in C(\Sigma_e)$ $u_{\ebar}s_e(b)u_e=r_e(b)$,
    \item $u_e=1$ for all $e\in E(\T)$.
\end{itemize}

\noindent By the universal property of $C(H)$, there exists a unique unital $*$-homomorphism $\Delta_H\,:\,C(H)\rightarrow C(H)\ot C(H)$ such that $\Delta_H(u_e)=u_e\ot u_e$ for all $e\in E(\G)$ and $\Delta_H\vert_{C(H_p)}=\Delta_{H_p}$ for all $p\in V(\G)$. It is shown in \cite{FF14} that the pair $H:=(C(H),\Delta_H)$ is a compact quantum group. $H$ is called the fundamental quantum group and is denoted by $H=\pi_1(H_p,\Sigma_e,\G,\T)$. Note that each $H_p$ is a dual quantum subgroup of $H$. In what follows we assume moreover that $\G$ is a finite graph and we write $s_e$ for the graph maps defining $H$.

\begin{theorem}\label{deltafree}
For each $p\in V(\G)$ let $B_p$ be a finite dimensional C*-algebra with a faithful state $\psi_p\in B_p^*$ and $\beta_p\,:\,H_p\curvearrowright B_p$ be a $\psi_p$-preserving action. Let $(B,\psi):=\bigoplus_{p\in V(\G)}(B_p,\psi_p)$ with the associated $\psi$-preserving action $\beta\,:\,H\curvearrowright B$. Then, for any compact quantum group $G$ with dual quantum subgroup $F$, there is a canonical isomorphism $$G\wr_{*,\beta,F}H\simeq\pi_1(G\wr_{*,\beta_p,F}H_p,\Sigma_e\times F,\G,\T),$$
where the graph maps on the right are given by
$$s'_e:=s_e\ot\id\,:\,C(\Sigma_e\times F)\rightarrow C(H_{s(e)}\times F)\subset C(G\wr_{*,\beta_{s(e)},F}H_{s(e)}),$$
which clearly intertwines the comultiplications.
\end{theorem}

\begin{proof}
Write $C(H)=\langle C(H_p),u_e\vert p\in V(\G),\,e\in E(\G)\rangle$. Let $\GG:=G\wr_{*,\beta,F}H$, $\GG_p:=G\wr_{*,\beta_p,F}H_p$ with canonical homomorphisms $\rho\,:\, C(G)\rightarrow B\ot C(\GG)$ and $\rho_p\,:\, C(G)\rightarrow B_p\ot C(\GG_p)$ and view $C(H\times F)\subset C(\GG)$ and $C(H_p\times F)\subset C(\GG_p)$. Define $\GG':=\pi_1(\GG_p,\Sigma_e\times F,\G,\T)$ and write $C(\GG')=\langle C(\GG_p), v_e\vert p\in V(\G),\,e\in E(\G)\rangle$. Note that, for all $e\in E(\G)$ and $b\in C(\Sigma_e)$ one has, viewing $C(\Sigma_e)\subset C(\Sigma_e\times F)$, $r'_e(b)=v_{\ebar}(s'_e(b))v_e$. By the universal property of $C(H)$ there exists a unique unital $*$-homomorphism $\nu_{H}\,:\, C(\GG)\rightarrow C(\GG')$ such that $\nu_{H}(u_e)=v_e$ and $\nu_{H}\vert_{C(H_p)}$ is the inclusion $C(H_p)\subset C(\GG_p)\subset C(\GG')$.

\vspace{0.2cm}

\noindent Note that, by definition, all the maps $s_e'\vert_{C(F)}\,:\, C(F)\subset C(\Sigma_e\times F)\rightarrow C(\GG_{s(e)})\subset C(\GG')$ coincide, for $e\in E(\G)$. Write $\nu_F\,:\, C(F)\rightarrow C(\GG')$ the unique map such that $\nu_F=s_e'\vert_{C(F)}$ for all $e\in E(\G)$. The fundamental quantum group relations give, for all $f\in C(F)$ and $e\in E(\G)$,
$$\nu_F(f)=s'_{\ebar}(f)=r'_e(f)=v_{\ebar}(s'_e(f))v_e=v_e^*\nu_F(f)v_e.$$
Hence, $\nu_F(C(F))$ commutes with $v_e$ for all $e\in E(\G)$. Also, by the free wreath product relations, $C(F)\subset C(\GG_p)$ and $C(H_p)\subset C(\GG_p)$ commute, for all $p\in V(\G)$. Hence, $\nu_F(C(F))$ commutes with the C*-algebra generated by $C(H_p)$, $p\in V(G)$ and $v_e$, $e\in E(\G)$ inside $C(\GG')$. This C*-algebra is exactly $\nu_H(C(H))$ so we have the relation $[\nu_H(x),\nu_F(f)]=0$ for all $a\in C(H)$ and $f\in C(F)$.

\vspace{0.2cm}

\noindent Viewing $C(\GG_p)\subset C(\GG')$, we may view each map $\rho_p$ as a map $ C(G)\rightarrow B_p\ot C(\GG')$. By universal property of the direct sum there exists a unique unital $*$-homomorphism $\rho'\,:\, C(G)\rightarrow B\ot C(\GG')$ such that $(\chi_p\ot\id)\rho'=\rho_p$ for all $p\in V(\G)$, where $\chi_p\,:\, B\rightarrow B_p$ is the canonical surjection. Note that, for all $p\in V(\G)$,
\begin{eqnarray*}
(\chi_p\ot\id)[\rho'(a)(\id\ot\nu_{H})(\beta(b))]&=&\rho_p(a)(\id\ot\nu_{H})(\beta_p(\chi_p(b)))=\rho_p(a)\beta_p(\chi_p(b))\\
&=&\beta_p(\chi_p(b))\rho_p(a)
=(\chi_p\ot\id)[(\id\ot\nu_{H})(\beta(b))\rho'(a)].\end{eqnarray*}
Hence, $[\rho'(a),(\id\ot\nu_{H})(\beta(b))]=0$ for all $a\in C(G)$, $b\in B$. Moreover, $$(\chi_p\ot\id)\rho'(f)=\rho_p(f)=1_{B_p}\ot f\text{ for all }f\in C(F).$$
Hence, $\rho'(f)=1_{B}\ot f$ for all $f\in C(F)$.

\vspace{0.2cm}

\noindent By the universal property of $C(\GG)$, there exists a unique unital $*$-homomorphism
$$\pi\,:\, C(\GG)\rightarrow C(\GG')\text{ such that }(\id\ot\pi)\rho=\rho'\text{ and }\pi\vert_{C(H)}=\nu_H,\,\, \pi\vert_{C(F)}=\nu_F.$$
Consider now the canonical maps:
$$\beta_G\,:\, C(G)\ot B\rightarrow B\ot C(\GG)\text{ and }\beta^p_{G}\,:\,C(G)\ot B_p\rightarrow B_p\ot C(\GG_p).$$
We claim that:
\begin{equation}\label{EqBetaG}
(\chi_p\ot\pi)\beta_G=\beta_G^p(\id\ot\chi_p)\quad\text{for all }p\in V(\G)
\end{equation}
Indeed, for $a\in C(G)$ and $b\in B$ one has,
\begin{eqnarray*}
(\chi_p\ot\pi)\beta_G(a\ot b)&=&(\chi_p\ot\pi)(\rho(a)\beta(b))=(\chi_p\ot\id)(\rho'(a))(\id\ot\pi)(\chi_p\ot\id)(\beta(b))\\
&=&\rho_p(a)\beta_p(\chi_p(b))=\beta^p_G(a\ot\chi_p(b)).
\end{eqnarray*}
Let us show that $\Delta_{\GG'}\circ\pi=(\pi\ot\pi)\circ\Delta_{\GG}$. Since $\nu_H \,:\, C(H)\rightarrow C(\GG')$ intertwines the comultiplications and $H$ is a dual subgroup of $\GG$, it suffices to show that
$$(\id\ot\Delta_{\GG'})(\id\ot\pi)\rho=(\id\ot\pi\ot\pi)(\id\ot\Delta_\GG)\rho.$$
Using Equation $(\ref{EqBetaG})$ we find, for all $p\in V(\G)$,
\begin{eqnarray*}
(\chi_p\ot\Delta_{\GG'})(\id\ot\pi)\rho&=&(\chi_p\ot\Delta_{\GG'})\rho'=(\id\ot\Delta_{\GG'})\rho_p=(\id\ot\Delta_{\GG_p})\rho_p\\
&=&(\beta_G^p\ot\id)(\id\ot\rho_p)\Delta_G
=(\beta_G^p\ot\id)(\id\ot\chi_p\ot\id)(\id\ot\rho')\Delta_G\\
&=&(\chi_p\ot\pi\ot\id)(\beta_G\ot\id)(\id\ot\id\ot\pi)(\id\ot\rho)\Delta_G\\
&=&(\chi_p\ot\pi\ot\pi)(\id\ot\Delta_\GG)\rho.
\end{eqnarray*}
This shows that $\pi$ intertwines the comultiplications. It remains to show that $\pi$ is invertible. To do so, we construct the inverse of $\pi$. Consider the unital $*$-homomorphism $$\rho'_p:=(\chi_p\ot\id)\rho\,:\, C(G)\rightarrow B_p\ot C(\GG)$$
and view $C(H_p)\subset C(H)\subset C(\GG)$ and $C(F)\subset C(\GG)$. It is clear that $C(F)$ commutes with $C(H_p)$, $\rho'_p(f)=1_{B_p}\ot f$ for $f\in C(F)$ and $[\rho_p'(a),\beta_p(b)]=0$ for all $a\in C(G)$ and $b\in B_p$. Hence there exists a unique unital $*$-homomorphism $\pi_p^{-1}\,:\, C(\GG_p)\rightarrow C(\GG)$ such that $(\id\ot\pi_p^{-1})\rho_p=\rho_p'$ and $\pi_p^{-1}\vert_{C(H_p)}$ is the inclusion of $C(H_p)$ in $C(\GG)$. In particular, for each vertex $p\in V(\G)$, $\pi_p^{-1}\vert_{C(F)}$ is the inclusion of $C(F)$ in $C(\GG)$. Consider the unitaries $u_e\in C(H)\subset C(\GG)$. By the universal property of the fundamental quantum group algebra, there exists a unique morphism $\pi^{-1}\,:\, C(\GG')\rightarrow C(\GG)$ such that $\pi^{-1}\vert_{C(\GG_p)}=\pi_p^{-1}$ and $\pi^{-1}(v_e)=u_e$, for all $p\in V(\G)$ and $e\in E(\G)$. It is now easy to check that $\pi^{-1}$ is the inverse of $\pi$.\end{proof}

\noindent Theorem \ref{deltafree} is a generalization of the ergodic decomposition in the following sense.

\begin{remark} When $(B,\psi)=\bigoplus(B_i,\psi)$ is the ergodic decomposition of $(B,\psi)$ and $H=\GB$ we recover the generalization to the amalgamated case of the ergodic decomposition \cite[Proposition 5.4]{FP16} (by using example \ref{ExErgoDec}): $$G\wr_{\ast,F} \GB\simeq \ast_{i\in I,F}\left(G\wr_{\ast,F} {\rm Aut}^+(B_i, \psi_i)\right).$$
\end{remark}

\section{The block extended C*-algebra and Haar measure}\label{sectionblock}
In the entire section, we fix an action $\beta\,:\,B\curvearrowright H$ on a finite dimensional C*-algebra $B$ with faithful invariant state and we use the notations introduced in Section \ref{SectionQAut}.

\subsection{Block decomposition}
Given $1\leq\kappa\leq K$, let $C_\kappa(\GG)\subset C(\GG)$ be the subalgebra generated by $C(H), C(F)$ and the coefficients of $\rho_\kappa:=(\chi_\kappa\ot\id)\rho\,:\,C(G)\rightarrow M_{N_\kappa}(\C)\ot C(\GG)$. Hence, $\rho_\kappa(C(G))\subset M_{N_\kappa}(\C)\ot C_\kappa(\GG)$ and the relations $[\rho_\kappa(a),\beta_\kappa(b)]=0$, $[x,f]=0$ and  $\rho_\kappa(f)=1\ot f$ are satisfied in $M_{N_\kappa}(\C)\ot C_\kappa(\GG)$ for all $a\in C(G)$, $b\in B$, $x\in C(H)$ and $f\in C(F)$. $C_\kappa(\GG)$ is called a \textit{block subalgebra}. Recall that we view $C(H)\otm C(F)\subset C_\kappa(\GG)$. The block C*-algebra satisfies the following universal property.

\begin{proposition}\label{PropUnivAlphaAction}
Fix $1\leq \kappa\leq K$. For any unital C*-algebra $D_\kappa$, any unital $*$-morphism $\pi_\kappa\,:\,C(G)\rightarrow M_{N_\kappa}(\C)\ot D_\kappa$ and any {\rm faithful} unital $*$-morphism $\iota_\kappa \,:\, C(H)\otm C(F)\rightarrow D_\kappa$ such that, for all $a\in C(G)$, $b\in B$, $f\in C(F)$, $$\pi_\kappa(a)(\id\ot\iota_\kappa)(\beta_\kappa(b)\ot 1_F)=(\id\ot\iota_\kappa)(\beta_\kappa(b)\ot 1_F) \pi_\kappa(a)\text{ and }\pi_\kappa(f) = 1_\kappa\ot \iota_\kappa(1\ot f),$$
there exists a unique unital $*$-homomorphism $\widetilde{\pi}_\kappa\,:\, C_\kappa(\GG)\rightarrow D$ such that $$(\id\ot\widetilde{\pi}_\kappa)\rho_\kappa=\pi_\kappa\text{ and }\widetilde{\pi}_\kappa(x f)=\iota_\kappa(x\ot f)\text{ for all }x\in C(H), f\in C(F).$$
\end{proposition}

\begin{proof}
If such $D_\kappa$ is given with the corresponding $*$-morphisms we define $$\mathcal{D}_0 := \underset{\gamma\neq \kappa,\, C(H)\otm C(F)}{\ast} C_\gamma(\GG)$$ and $\mathcal{D}:= D_\kappa\underset{C(H)\otm C(F)}{\ast}\mathcal{D}_0$, where the amalgamation is with respect to the faithful unital $*$-homomorphisms $\iota_\kappa\,:\,C(H)\otm C(F)\rightarrow D_\kappa$ and the inclusion $C(H)\otm C(F)\rightarrow\mathcal{D}_0$. Composing $\pi_\kappa$ (resp. $\rho_\gamma$) with the inclusion of $D_\kappa$ (resp. $C_\gamma(\GG)$ for $\gamma\neq\kappa$) in $\mathcal{D}$ we may view $\pi_\kappa\,:\,C(G)\rightarrow M_{N_\kappa}(\C)\ot \mathcal{D}$ (resp. $\rho_\gamma\,:\,C(G)\rightarrow M_{N_\gamma}(\C)\ot\mathcal{D}$). Use the universal property of the direct sum to get a map $\rho' : C(G)\rightarrow B\ot\mathcal{D}$ such that $(\chi_\gamma \ot \id) \rho' = \rho_\gamma$ for $\gamma\neq \kappa$ and $(\chi_\kappa \ot \id) \rho' = \pi_\kappa$, we also have a map $\iota\,:\,C(H)\otm C(F)\rightarrow \mathcal{D}$ which is the canonical inclusion. Note that, if $a\in C(G)$ and $b\in B$, the commutation relation $\rho'(a)((\id\ot \iota)(\beta(b)\ot 1_F)) =  ((\id\ot \iota)(\beta(b)\ot 1_F))\rho'(a)$ can be checked when composed with any $\chi_\gamma\ot \id$. For $f\in C(F)$, the relation $\rho'(f) = 1_B\ot \iota(1\ot f)$ also follows from the fact that we have the relations when applying $(\chi_\gamma\ot \id)$ and the fact that we amalgamated the different copies of $C(H)\otm C(F)$.
Applying the universal property of $C(\GG)$ gives a $*$-morphism $\widetilde{\pi} : C(\GG)\rightarrow \mathcal{D},$ such that $\rho' = (\id\ot\widetilde{\pi})\circ \rho$ and $\widetilde{\pi}_{\vert C(H)\otm C(F)} =\iota$. Define $\widetilde{\pi}_\kappa :=\widetilde{\pi}_{\vert C_\kappa(\GG)}\,:\, C_\kappa(\GG) \rightarrow \mathcal{D}$. Note that $\pi_\kappa=(\chi_\kappa\ot\id)\rho'=(\chi_\kappa\ot \widetilde{\pi})\rho=(\id\ot\widetilde{\pi})\rho_\kappa=(\id\ot \widetilde{\pi}_\kappa)\rho_\kappa$. In particular, $\widetilde{\pi}_\kappa(C_\kappa(\GG))\subset D_\kappa$ and we may view $\widetilde{\pi}_\kappa$ as a map from $C_\kappa(\GG)$ to $D_\kappa$. Hence, $\widetilde{\pi}_\kappa\vert_{C(H)\otm C(F)}=\iota_\kappa$ and this concludes the proof.\end{proof}

\begin{theorem}\label{ThmBlockAction}
The following holds.
\begin{enumerate}
\item There exists a unique unital $*$-isomorphism
$$\pi\,:\,C_1(\GG)\underset{C(H)\otm C(F)}{*}\dots\underset{C(H)\otm C(F)}{*}C_K(\GG)\rightarrow C(\GG)$$
such that, for all $1\leq\kappa\leq K$, $\pi\vert_{C_\kappa(\GG)}\,:\, C_\kappa(\GG)\rightarrow C(\GG)$ is the inclusion.
\item For all $1\leq\kappa\leq K$, there exists a unique unital $*$-isomorphism
$$\pi_\kappa\,:\,(B\ot C(G))\underset{B\ot C(F)}{*}(M_{N_\kappa}(\C)\ot C(H)\otm C(F))\rightarrow M_{N_\kappa}(\C)\ot C_\kappa(\GG)$$
such that $\pi_\kappa(b\ot a)=\beta_\kappa(b)\rho_\kappa(a)$ for $b\ot a\in B\ot C(G)$ and $\pi_\kappa\vert_{M_{N_\kappa}(\C)\ot C(H)\otm C(F)}$ is the inclusion. The amalgamated free product is done with respect to the inclusion map $B\ot C(F)\rightarrow B\ot C(G)$, and the map $\beta_\kappa\ot \id\,:\, B\ot C(F)\rightarrow M_{N_\kappa}(\C)\ot C(H)\otm C(F)$. The inverse isomorphism $\varphi_\kappa:=\pi_\kappa^{-1}$ satisfies $\varphi_\kappa\circ\rho_\kappa(a)=1\ot a\in B\ot C(G)$ for all $a\in C(G)$ and $\varphi_\kappa\vert_{M_{N_\kappa}(\C)\ot C(H)\otm C(F)}$ is the inclusion.
\end{enumerate}
\end{theorem}

\begin{proof}
$(1)$. The existence of $\pi$ follows from the universal property of the amalgamated free product. Let us construct the inverse of $\pi$. Denote by $C$ the free product of the C*-algebras $C_\kappa(\GG)$, amalgamated over $C(H)\otm C(F)$ and view $C(H)\otm C(F)\subset C$. Note that we have $\rho_\kappa\,:\,C(G)\rightarrow M_{N_\kappa}(\C)\ot C_\kappa(\GG)\subset M_{N_\kappa}(\C)\ot C$ such that $\rho_\kappa(f)=1_\kappa\ot f$ for all $f\in C(F)$ and $1\leq\kappa\leq K$. Hence, there exists a unique $\rho'\,:\,C(G)\rightarrow B\ot C$ such that $(\chi_\kappa\ot\id)\rho'=\rho_\kappa$ for all $\kappa$ and $\rho'(f) = 1_B\ot f$ for all $f\in C(F)$. Since the images of $\rho_\kappa$ and $\beta_\kappa$ commute in $M_{N_\kappa}(\C)\ot C_\kappa(\GG)$ for all $\kappa$, it follows that the images of $\rho'$ and $\beta$ commute in $B\ot C$. By the universal property of $C(\GG)$, there exists a unique $\varphi\,:\,C(\GG)\rightarrow C$ such that $\varphi\vert_{C(H)\otm C(F)}$ is the inclusion and $(\id\ot\varphi)\rho=\rho'$. It is clear that $\varphi$ is the inverse of $\pi$.

\vspace{0.2cm}

\noindent$(2)$. During this proof we use the notation: $$\Acal_\kappa:=(B\ot C(G))\underset{B\ot C(F)}{*}(M_{N_\kappa}(\C)\ot C(H)\otm C(F))$$
and we write $\nu_\kappa\,:\, B\ot C(G)\rightarrow \Acal_\kappa$ and $\mu_\kappa\,:\, M_{N_\kappa}(\C)\ot C(H)\otm C(F)\rightarrow\Acal_\kappa$ the canonical faithful unital $*$-homomorphisms. They satisfy $\mu_\kappa(\beta_\kappa(b)\ot f)=\nu_\kappa(b\ot f)$ for all $b\in B$ and $f\in C(F)$. We fix once and for all a basis of $M_{N_\kappa}(\C)$ consisting of matrix units $(e^\kappa_{ij})_{ij}$.

\vspace{0.2cm}

\noindent The existence of $\pi_\kappa$ is a consequence of the universal property of the amalgamated free product $\Acal_\kappa$. Let us construct its inverse $\varphi_\kappa\,:\,M_{N_\kappa}(\C)\ot C_\kappa(\GG)\rightarrow \Acal_\kappa$. We define the morphism $\iota\,:\,C(H)\otm C(F)\rightarrow\Acal_\kappa$, $x\ot f\mapsto\mu_\kappa(1\ot x\ot f)$ and we consider the elements
$$\varphi^\kappa_{ij}(a):=\sum_{s=1}^{N_\kappa}\mu_\kappa(e^\kappa_{si}\ot 1\ot 1)\nu_\kappa(1\ot a)\mu_\kappa(e^\kappa_{js}\ot 1\ot 1)\in\Acal_\kappa\text{ for }a\in C(G)\text{ and }1\leq i,j\leq N_\kappa.$$ A direct computation gives the relations:
\begin{enumerate}[(i)]
    \item $\varphi^\kappa_{ij}(ab)=\sum_k\varphi^\kappa_{ik}(a)\varphi^\kappa_{kj}(b)$.
    \item $(\varphi^\kappa_{ij}(a))^*=\varphi^\kappa_{ji}(a^*)$.
    \item $\varphi^\kappa_{ij}(1)=\delta_{ij}1_{\Acal_\kappa}$.
    \item $\varphi^\kappa_{ij}(f)= \delta_{ij}\iota(1\ot f)$ if $f\in C(F)$.
\end{enumerate}
Note that $(ii)$, $(iii)$ and $(iv)$ are clear so we only prove $(i)$. Let $a,b\in C(G)$ then,
\begin{eqnarray*}
    \sum_{k=1}^{N_\kappa} \varphi^{\kappa}_{ik}(a)\varphi^{\kappa}_{kj}(b)&=&\sum_{k,s}\mu_\kappa(e^\kappa_{si}\ot 1\ot 1)\nu_\kappa(1\ot a)\mu_\kappa(e^\kappa_{kk}\ot 1\ot 1)\nu_\kappa(1\ot b)\mu_\kappa(e^\kappa_{js}\ot 1\ot 1)\\
    &=&\sum_{s}\mu_\kappa(e^\kappa_{si}\ot 1\ot 1)\nu_\kappa(1\ot ab)\mu_\kappa(e^\kappa_{js}\ot 1\ot 1)=\varphi^\kappa_{ij}(ab).
\end{eqnarray*}

\noindent This relations mean that $\varphi_\kappa\,:\, C(G)\rightarrow M_{N_\kappa}(\C)\ot \Acal_\kappa$ defined by $\varphi_\kappa(a):=\sum_{i,j=1}^{N_\kappa}e^\kappa_{ij}\ot\varphi^\kappa_{ij}(a)$ is a unital $*$-homomorphism. Let us check that $\varphi_\kappa(a)$ and $(\id\ot\iota)(\beta_\kappa(b)\ot 1)$ commutes, for all $a\in C(G)$ and $b\in B$. We write $\beta_\kappa\,:\,B\rightarrow M_{N_\kappa}(\C)\ot C(H)$ with coefficients $\beta_\kappa(b)=\sum_{i,j}e^\kappa_{ij}\ot\beta^\kappa_{ij}(b)$. Then, for all $b\in B$,
$$(e^\kappa_{si}\ot 1)\beta_\kappa(b)=\sum_re^\kappa_{sr}\ot\beta^\kappa_{ir}(b)\quad\text{and}\quad\beta_\kappa(b)(e^\kappa_{js}\ot 1)=\sum_re^\kappa_{rs}\ot\beta^\kappa_{rj}(b).$$
Hence, writing $(\id\ot\iota)(\beta_\kappa(b)\ot 1)\varphi_\kappa(a)=\sum_{i,j}e^\kappa_{ij}\ot X^\kappa_{ij}$,
\begin{eqnarray*}
X^\kappa_{ij}&=&\sum_{r,s}\iota(\beta^\kappa_{ir}(b)\ot 1)\varphi^\kappa_{rj}(a)=\sum_{r,s}\mu_\kappa(e^\kappa_{sr}\ot\beta^\kappa_{ir}(b)\ot 1)\nu_\kappa(1\ot a)\mu_\kappa(e^\kappa_{js}\ot 1\ot 1)\\
&=&\sum_s\mu_\kappa((e^\kappa_{si}\ot 1)\beta_\kappa(b)\ot 1)\nu_\kappa(1\ot a)\mu_\kappa(e^\kappa_{js}\ot 1\ot 1)\\
&=&\sum_s\mu_\kappa(e^\kappa_{si}\ot 1\ot 1)\nu_\kappa(b\ot a)\mu_\kappa(e^\kappa_{js}\ot 1\ot 1)\\
&=&\sum_s\mu_\kappa((e^\kappa_{si}\ot 1\ot 1)\nu_\kappa(1\ot a)\mu_\kappa(\beta_\kappa(b)(e^\kappa_{js}\ot 1)\ot 1)\\
&=&\sum_{r,s}\mu_\kappa(e^\kappa_{si}\ot 1\ot 1)\nu_\kappa(1\ot a)\mu_\kappa(e^\kappa_{rs}\ot\beta^\kappa_{rj}(b)\ot 1)=\sum_r\varphi^\kappa_{ir}(a)\iota(\beta^\kappa_{rj}(b)\ot 1).
\end{eqnarray*}

\noindent This shows that $(\id\ot\iota)(\beta_\kappa(b)\ot 1)\varphi_\kappa(a)=\varphi_\kappa(a)(\id\ot\iota)(\beta_\kappa(b)\ot 1)$ for all $a\in C(G)$ and $b\in B$. Note that, for all $f\in C(F)$, $\varphi_\kappa(f) = \sum_{i,j} e_{ij}^\kappa \ot \delta_{ij}\iota(1\ot f) = 1_\kappa \ot \iota(1\ot f)$. The universal property of $C_\kappa(\GG)$ (Proposition \ref{PropUnivAlphaAction}) gives a unital $*$-homomorphism, still denoted $\varphi_\kappa\,:\,C_\kappa(\GG)\rightarrow \Acal_\kappa$ such that $\varphi_\kappa(\rho^\kappa_{ij}(a))=\varphi^\kappa_{ij}(a)$ for all $1\leq i,j\leq N_\kappa$ and all $a\in C(G)$ and $\varphi_\kappa(x\ot f)=\mu_\kappa(1\ot x\ot f)$ for all $x\in C(H)$ and $f\in C(F)$. Now, we note that $\varphi_\kappa(C_\kappa(\GG))\subseteq\mu_\kappa(M_{N_\kappa}(\C)\ot 1\ot 1)'\cap\Acal_\kappa$. Indeed, it is clear that, for all $x\in C(H)$ and $f\in C(F)$, $\varphi_\kappa(x\ot f)=\mu_\kappa(1\ot x\ot f)$ commutes with $\mu_\kappa(M_{N_\kappa}(\C)\ot 1\ot 1)$ and, if $x=\rho^\kappa_{ij}(a)$ is a coefficient of $\rho_\kappa$, then:
\begin{eqnarray*}
\varphi_\kappa(x)\mu_\kappa(e^\kappa_{rs}\ot 1\ot 1)&=&\sum_t\mu_\kappa(e^\kappa_{ti}\ot 1\ot 1)\nu_\kappa(1\ot a)\mu_\kappa(e^\kappa_{jt}e^\kappa_{rs}\ot 1\ot 1)\\ &=&\mu_\kappa(e^\kappa_{ri}\ot 1\ot 1)\nu_\kappa(1\ot a)\mu_\kappa(e^\kappa_{js}\ot 1\ot 1)\\
&=&\mu_\kappa(e^\kappa_{rs}\ot 1\ot 1)\varphi_\kappa(x).
\end{eqnarray*}
Hence, $\varphi_\kappa(x)$ commutes with $\mu_\kappa(M_{N_\kappa}(\C)\ot 1\ot 1)$ for any $x$ which is a coefficients of $\rho_\kappa$. Since $C_\kappa(\GG)$ is generated by $C(H)$ and coefficients of $\rho_\kappa$, it follows that
$$\varphi_\kappa(C_\kappa(\GG))\subseteq\mu_\kappa(M_{N_\kappa}(\C)\ot 1\ot 1)'\cap\Acal_\kappa.$$
Hence, there exists a unique unital $*$-homomorphism, still denoted $\varphi_\kappa\,:\, M_{N_\kappa}(\C)\ot C_\kappa(\GG)\rightarrow \Acal_\kappa$ such that $\varphi_\kappa(b\ot x)=\mu_\kappa(b\ot 1\ot 1)\varphi_\kappa(x)$, for all $b\in M_{N_\kappa}(\C)$ and $x\in C_\kappa(\GG)$. By construction we have $\varphi_\kappa\circ\rho_\kappa(a)=\nu_\kappa(1\ot a)$ for all $a\in C(G)$ and $\varphi_\kappa\vert_{M_{N_\kappa}(\C)\ot C(H)\otm C(F)}=\mu_\kappa$. Since $M_{N_\kappa}(\C)\ot C_\kappa(\GG)$ is generated, as a C*-algebra, by $\rho_\kappa(C(G))\cup M_{N_\kappa}(\C)\ot C(H)\otm C(F)$ this easily implies that $\varphi_\kappa$ is the inverse of $\pi_\kappa$.\end{proof}

\subsection{Reduced operators}

\noindent From now on we suppose that the action $\beta\,:\, H\curvearrowright B$ is ergodic with unique invariant state $\psi\in B^*$ and associated representation $u\in\Lcal(B)\ot C(H)$. Let $h_H$ be the Haar measure of $H$. 

\vspace{0.2cm}
\noindent Let $G$ be a compact quantum group with a dual quantum subgroup $F$ and associated ucp map $E_F:C(G)\rightarrow C(F)$. For all $a\in\Pol(G)$, we may write $\Delta(a)=\sum a_{(1)}\ot a_{(2)}$ with $a_{(1)},a_{(2)}\in\Pol(G)$. If moreover $E
_F(a)=0$ then we may choose $a_{(1)}$ and $a_{(2)}$ such that $E_F(a_{(1)})=E_F(a_{(2)})=0$ for all $(1)$, $(2)$ (by Remark \ref{RmkEF}).

\vspace{0.2cm}

\noindent Let $\GG=G\wr_{*,\beta,F}H$ the amalgamated free wreath product quantum group with generic morphism $\rho\,:\,C(G)\rightarrow B\ot C(\GG)$ and $C_\kappa(\GG)\subseteq C(\GG)$ be a block algebra.

\vspace{0.2cm}

\noindent Let $\mathcal{C}_\kappa\subset M_{N_\kappa}(\C)\ot C_\kappa(\GG)$ be the linear span of \textit{reduced operators} in $M_{N_\kappa}(\C)\ot C_\kappa(\GG)$ i.e. operators $x\in M_{N_\kappa}(\C)\ot C_\kappa(\GG)$ which can be written as products alternating from operators of the form $X\in M_{N_\kappa}(\C)\ot \Pol(H)\ot \Pol(F)\subset M_{N_\kappa}(\C)\ot C_\kappa(\GG)$ with $(E_\kappa\ot \id_F) (X)=0$ and operators of the form $\rho_\kappa(a):=(\chi_\kappa\ot\id)\rho(a)$, with $a\in \Pol(G)$ and $E_F(a)=0$. More precisely, $x$ is of the form
\begin{itemize}
    \item $x=X_1\rho_\kappa(a_1)\dots X_n\rho_\kappa(a_n)$ or,
    \item $x=\rho_\kappa(a_1)X_1\dots \rho_\kappa(a_n)X_n$ or,
    \item $x=X_0\rho_\kappa(a_1)X_1\dots \rho_\kappa(a_n)X_n$ or,
    \item $x=\rho_\kappa(a_0)X_1\rho_\kappa(a_1)\dots X_n\rho_\kappa(a_n)$ or,
    \item $x=\rho_\kappa(a_1)(\beta_\kappa(b)\ot 1_{C(F)})$ for some $b\in B$.
\end{itemize}
where $n\geq 1$, $X_s\in M_{N_\kappa}(\C)\ot \Pol(H)\ot \Pol(F)$, $a_s\in\Pol(G)$, $E_F(a_s)=0$, $(E_\kappa\ot \id_F)(X_s)=0$.

\vspace{0.2cm}

\begin{lemma}\label{LemReducedAction}
The following holds.
\begin{enumerate}
    \item $X_\kappa:=\mathcal{C}_\kappa +M_{N_\kappa}(\C)\ot \Pol(H)\ot \Pol(F)$ is a dense subalgebra of $M_{N_\kappa}(\C)\ot C_\kappa(\GG)$.
    \item There exists a unique conditional expectation $$\E_\kappa\,:\,M_{N_\kappa}(\C)\ot C_\kappa(\GG)\rightarrow M_{N_\kappa}(\C)\ot C(H)\otm C(F)\text{ such that }\ker(\E_\kappa)=\overline{\mathcal{C}_\kappa}.$$
    \item $\Delta(C_\kappa(\GG))\subset C_\kappa(\GG)\ot C(\GG)$ and $(\id_{M_{N_\kappa}}\ot\Delta)(\mathcal{C}_\kappa)\subset\mathcal{C}_\kappa\otimes  C(\GG)$.
    \item For all $X\in M_{N_\kappa}(\C)\otimes C_\kappa(\GG)$ one has:
    $$( \E_\kappa\ot\id_{C(\GG)})(\id_{M_{N_\kappa}}\ot\Delta)(X)=(\id_{M_{N_\kappa}}\ot \Delta)\E_\kappa(X).$$
    \item For all $X\in M_{N_\kappa}(\C)\otimes C_\kappa(\GG)$ one has:
    $$((\id_{M_{N_\kappa}}\ot h_H\ot h_F)\circ\E_\kappa\ot\id_{C(\GG)})(\id_{M_{N_\kappa}}\ot\Delta)(X)=(\id_{M_{N_\kappa}}\ot h_H\ot h_F)\circ\E_\kappa(X)\ot 1_{C(\GG)}.$$
\end{enumerate}
\end{lemma}

\begin{proof}

$(1)$. It is easy to see that $X_\kappa$ is a unital $*$-subalgebra of $M_{N_\kappa}(\C)\ot C_\kappa(\GG)$ containing $M_{N_\kappa}(\C)\ot \Pol(H)\ot \Pol(F)$. Note that it also contains $1\ot\rho^\kappa_{ij}(a)$ for all $a\in \Pol(G)$. Indeed, writing $1\ot\rho^\kappa_{ij}(a)=\sum_r(e^\kappa_{ri}\ot 1)\rho_\kappa(a)(e^\kappa_{jr}\ot 1)$ and $e^{\kappa}_{rs}=(e^\kappa_{rs})^\circ+\frac{\delta_{rs}}{\delta\lambda_{\kappa,r}}1_B$, where $E_\kappa((e^\kappa_{rs})^\circ\ot 1)=0$ we see that $1\ot\rho^\kappa_{ij}(a)\in X$ when $E_F(a)=0$ and, for any $a\in \Pol(G)$, writing $a=a^\circ+E_F(a)$ we have $\rho_{ij}^\kappa(a)=\rho^\kappa_{ij}(a^\circ)+\rho^\kappa_{ij}(E_F(a)) = \rho^\kappa_{ij}(a^\circ) + \delta_{ij}(1\ot 1\ot E_F(a))$ which implies that $1\ot\rho^\kappa_{ij}(a)\in X$.

\vspace{0.2cm}

\noindent$(2)$. \noindent As in the proof of Theorem \ref{ThmBlockAction}, we write 
$$\Acal_\kappa:=(B\ot C(G))\underset{B\ot C(F)}{*}(M_{N_\kappa}(\C)\ot C(H)\otm C(F)),$$
where the amalgamation is with respect to the inclusion $B\ot C(F)\rightarrow B\ot C(G)$ and the unital $*$-homomorphism $\beta_\kappa\ot\id_F\,:\, B\ot C(F)\rightarrow M_{N_\kappa}(\C)\ot C(H)\otm C(F)$ and we write $\nu_\kappa\,:\, B\ot C(G)\rightarrow \Acal_\kappa$ and $\mu_\kappa\,:\, M_{N_\kappa}(\C)\ot C(H)\otm C(F)\rightarrow\Acal_\kappa$ the canonical faithful unital $*$-homomorphism. Consider the unique ucp map $E_2\,:\,\Acal_\kappa\rightarrow M_{N_\kappa}(\C)\ot C(H)\otm C(F)$ such that $E_2\circ\mu_\kappa=\id$, $E_2\circ\nu_\kappa=\beta_\kappa\ot E_F$ and $E_2(x)=0$ for all $x$ reduced of length $\geq 2$ in $\Acal_\kappa$ in the amalgamated free product sense and with respect to the ucp maps $(\id\ot E_F)\,:\,B\ot C(G)\rightarrow B\ot C(F)$ and
$$E_\kappa\ot \id_F\,:\, M_{N_\kappa}(\C)\ot C(H)\otm C(F)\rightarrow B\ot C(F).$$
Recall that $\varphi_\kappa:=\pi_\kappa^{-1}\,:\, M_{N_\kappa}(\C)\ot C_\kappa(\GG)\rightarrow \Acal_\kappa$ is constructed in Theorem \ref{ThmBlockAction}. Define the ucp map $\E_\kappa:=E_2\circ\varphi_\kappa\,:\,M_{N_\kappa}(\C)\ot C_\kappa(\GG)\rightarrow M_{N_\kappa}(\C)\ot C(H)\otm C(F)$. 

\vspace{0.2cm}

\noindent If $x=X_0\rho_\kappa(a_1)X_1\dots \rho_\kappa(a_n)X_n$ ($n\geq 1$) is reduced then $$\varphi_\kappa(x)=\mu_\kappa(X_0)\nu_\kappa(1\ot a_1)\dots\nu_\kappa(1\ot a_n)\mu_\kappa(X_n)$$
is reduced in $\Acal_\kappa$ of length at least $2$ so $E_2(y)=0$. It follows that $\E_\kappa(x)=E_2(y)=0$. We prove in the same way that $\E_\kappa(x)=0$ for all types of reduced operators. Hence, $\E_\kappa(\mathcal{C}_\kappa)=\{0\}$. Recall that $\varphi_\kappa\vert_{M_{N_\kappa}(\C)\ot C(H)\ot C(F)}=\mu_\kappa$ and $E_2\circ\mu_\kappa=\id$. Hence, $\E_\kappa\vert_{M_{N_\kappa}(\C)\ot C(H)\otm C(F)}=\id$. It remains to show that $\ker(\E_\kappa)\subset \overline{\Ccal_\kappa}$ and this follows from $(1)$ and the fact that $\E_\kappa$ is the identity on the subalgebra $M_{N_\kappa}(\C)\ot C(H)\otm C(F)$.

\vspace{0.2cm}

\noindent$(3)$ By definition $\Delta$ we have, writing $\Delta_G(a)=\sum a_{(1)}\ot a_{(2)}$, with $a_{(1)},a_{(2)}\in\Pol(G)$,
\begin{eqnarray*}
(\id\ot\Delta)(\rho_\kappa(a))&=&\sum(\chi_\kappa\ot\id\ot\id)(\beta_G\ot\id)(a_{(1)}\ot \rho(a_{(2)}))\\
&=&\sum_{\gamma,k,l}(\chi_\kappa\ot\id\ot\id)(\beta_G(a_{(1)}\ot e^\gamma_{kl})\ot\rho^\gamma_{kl}(a_{(2)})\\
&=&\sum_{\gamma,k,l}\left[\rho_\kappa(a_{(1)})(\beta_\kappa(e^\gamma_{kl})\ot 1_{C(F)})\right]\ot\rho^\gamma_{kl}(a_{(2)})
\end{eqnarray*}
Since $\Delta(C(H)\otm C(F))=\Delta_{H\times F}(C(H\times F))\subset (C(H)\otm C(F)) \ot (C(H)\otm C(F))$ we conclude easily that $\Delta(C_\kappa(\GG))\subset C_\kappa(\GG)\ot C(\GG)$.

\vspace{0.2cm}

\noindent Let $x$ be a reduced operator and let us show that $(\id\ot\Delta)(x)\in\Ccal_\kappa\ot C(\GG)$. We prove our statement for $x$ a reduced operator of the form $x=\rho_\kappa(a_1)X_1\dots \rho_\kappa(a_n)X_n$, the proof for other types of reduced operator is the same.

\vspace{0.2cm}

\noindent\textbf{Claim.} \textit{If for all $s$, $X_s\in M_{N_\kappa}(\C)\ot\Pol(H)\ot\C1_{C(F)}$ then $(\id\ot\Delta)(x)\in\Ccal_\kappa\ot C(\GG)$.}

\vspace{0.2cm}

\noindent\textit{Proof of the Claim.} Write $x=\rho_\kappa(a_1)(X_1\ot 1)\dots \rho_\kappa(a_n)(X_n\ot 1)$, where $X_s\in M_{N_\kappa}(\C)\ot\Pol(H)$ is such that $E_\kappa(X_s)=0$. Writing $X_s\in M_{N_\kappa}(\C)\ot \Pol(H)$ as $X_s=\sum_{i_s,j_s}b^\kappa_{i_sj_s}\ot x_{i_sj_s}$ and $(\id\ot\Delta)(X_s\ot 1)=(\id\ot\Delta_H)(X_s)\ot 1_{C(F)}=\sum (X_s)_{(1)}\ot (X_s)_{(2)}$ where $$(X_s)_{(1)}=b^\kappa_{i_sj_s}\ot (x_{i_sj_s})_{(1)}\ot 1\in M_{N_\kappa}(\C)\ot\Pol(H)\ot\Pol(F)\text{ and},$$
$$(X_s)_{(2)}=(x_{i_sj_s})_{(2)}\ot 1\in \Pol(H)\ot\Pol(F)$$ with $\Delta_{H}(x_{i_sj_s})=\sum(x_{i_sj_s})_{(1)}\ot (x_{i_sj_s})_{(2)}$, we obtain:
$$
(\id\ot\Delta)(x)=\sum\rho_\kappa((a_1)_{(1)})(\beta_\kappa(e^{\gamma_1}_{k_1l_1})\ot 1)(X_1)_{(1)}\dots\rho_\kappa((a_n)_{(1)})(\beta_\kappa(e^{\gamma_n}_{k_nl_n})\ot 1)(X_n)_{(1)}$$
$$\ot\rho^{\gamma_1}_{k_1l_1}((a_1)_{(2)})(X_1)_{(2)}\dots \rho^{\gamma_n}_{k_nl_n}((a_n)_{(2)})(X_n)_{(2)}.$$
Where $\Delta_G(a_k)=\sum (a_k)_{(1)}\ot (a_k)_{(2)}$ and $E_F((a_k)_{(1)})=0$. If  $(E_\kappa\ot \id_F)((X_s)_{(1)})=0$, for all $s$, then $(\id\ot\Delta)(x)\in \Ccal_\kappa\ot C(\GG)$, since $$(E_\kappa\ot \id_F)((\beta_\kappa(e^{\gamma_s}_{k_sl_s})\ot 1)(X_s)_{(1)})=e^{\gamma_s}_{k_sl_s}(E_\kappa\ot \id_F)((X_s)_{(1)})=0.$$ Otherwise, we write $(X_s)_{(1)}=(X_s)_{(1)}^\circ+(\beta_\kappa E_\kappa\ot\id_F)((X_s)_{(1)})$, where $(E_\kappa\ot\id_F)((X_s)_{(1)}^\circ)=0$ so that we may decompose $(\id\ot\Delta)(x)$ as the sum of two terms $Y_s+Z_s$ where in $Y_s$ is the same as $(\id\ot\Delta)(x)$ but replacing $(X_s)_{(1)}$ by $(X_s)_{(1)}^\circ$ and the second term is:
$$Z_s=\sum\rho_\kappa((a_1)_{(1)})\dots(\beta_\kappa(e^{\gamma_s}_{k_sl_s})\ot 1)(\beta_\kappa E_\kappa\ot\id_F)(b^\kappa_{i_sj_s}\ot (x_{i_sj_s})_{(1)}\ot 1)\dots (X_n)_{(1)}$$
$$\ot\rho^{\gamma_1}_{k_1l_1}((a_1)_{(2)})(X_1)_{(2)}\dots \rho^{\gamma_n}_{k_nl_n}((a_n)_{(2)})(X_n)_{(2)}.$$
By Proposition \ref{PropUcpSection} we have:
$$E_\kappa(b^\kappa_{i_sj_s}\ot (x_{i_sj_s})_{(1)})=\delta^{-1}\lambda_{\kappa,i_s}^{-\frac{1}{2}}\lambda_{\kappa,j_s}^{-\frac{3}{2}}\sum_{\beta,r,t}\left(\frac{\lambda_{\beta,r}}{\lambda_{\beta,t}}\right)^{\frac{1}{2}}h_H(u^{j_si_s,\kappa}_{tr,\beta}(x_{i_sj_s})_{(1)})b^\beta_{rt}$$
Replacing this expression in $Z_s$ we get:
$$Z_s=\delta^{-1}\sum\left(\frac{\lambda_{\beta,r}}{\lambda_{\beta,t}}\right)^{\frac{1}{2}}\rho_\kappa((a_1)_{(1)})\dots(\beta_\kappa(e^{\gamma_s}_{k_sl_s}b^\beta_{rt})\ot 1)\rho_\kappa((a_{s+1})_{(1)})\dots $$
$$\ot(\rho)^{\gamma_1}_{k_1l_1}((a_1)_{(2)})\dots\left(\sum_{i_s,j_s,(1)',(2)'}\lambda_{\kappa,i_s}^{-\frac{1}{2}}\lambda_{\kappa,j_s}^{-\frac{3}{2}}h_H(u^{j_si_s,\kappa}_{tr,\beta}(x_{i_sj_s})_{(1)'})(x_{i_sj_s})_{(2)'}\right)\dots$$
Since $E_\kappa(X_s)=0$, we may now use Lemma \ref{LemE} to deduce that $Z_s=0$. Doing that for all $s$, we conclude that $(\id\ot\Delta)(x)\in\Ccal_\kappa\ot C(\GG)$.$\hfill$ $\qed$

\vspace{0.2cm}

\noindent\textit{End of the proof of $(3)$.} Fix a basis $(y_k)_k$ of $\Pol(F)$. If now $x$ is an arbitrary reduced operator of the form $x=\rho_\kappa(a_1)X_1\dots \rho_\kappa(a_n)X_n$, with $X_s\in M_{N_\kappa}(\C)\ot \Pol(H)\ot\Pol(F)$, we may write $X_s$ as a finite sum $X_s=\sum_kX_{s,k}\ot y_k$, where $X_{s,k}\in M_{N_\kappa}(\C)\ot\Pol(H)$. Since $x$ is reduced one has $\sum_k E_\kappa(X_{s,k})y_k=(E_\kappa\ot\id)(X_s)=0$. Since the family $(y_k)_k$ is linearly independent we obtain $E_\kappa(X_{s,k})=0$ for all $k$. Write $x$ as a finite sum:
\begin{eqnarray*}
x&=&\sum_{k_1,\dots k_n}\rho_\kappa(a_1)(X_{1,k_1}\ot y_{k_1})\dots \rho_\kappa(a_n)(X_{n,k_n}\ot y_{k_n})\\
&=&\sum_{k_1,\dots k_n}\rho_\kappa(a_1y_{k_1})(X_{1,k_1}\ot 1)\dots \rho_\kappa(a_ny_{k_n})(X_{n,k_n}\ot 1),
\end{eqnarray*}
where, in the last equality, we used the relations $\rho_\kappa(f)=1\ot1\ot f$, for all $f\in C(F)$. Since $E_F(a_sy_{k_s})=E_F(a_s)y_{k_s}=0$ for all $s$, we deduce from the Claim that $x\in\Ccal_\kappa$.

\vspace{0.2cm}

\noindent$(4)$. By $(1)$, it suffices to prove the formula for $X\in\mathcal{C}_\kappa$ and for $X\in M_{N_\kappa}(\C)\ot \Pol(H)\ot \Pol(F)$. If $X\in\mathcal{C}_\kappa$ we have $(\id\ot\Delta)(X)\in\Ccal_\kappa\ot C(\GG)$ by $(3)$. It now follows from $(2)$ that
$$(\E_\kappa\ot\id_{C(\GG)})(\id_{M_{N_\kappa}}\ot\Delta)(X)=0=\E_\kappa(X)\ot 1_{C(\GG)}.$$
If now $X\in M_{N_\kappa}(\C)\ot \Pol(H)\ot \Pol(F)$ then, since $\Delta\vert_{C(H)\otm C(F)}=\Delta_{H\times F}$,
$$(\id\ot\Delta)(X)\in M_{N_\kappa}(\C)\ot \Pol(H)\ot \Pol(F)\ot \Pol(H)\ot \Pol(F),$$
and therefore $(\E_\kappa\ot\id_{C(\GG)})(\id\ot\Delta)(X) = (\id\ot\Delta)(X) =(\id\ot\Delta)\E_\kappa$.

\vspace{0.2cm}

\noindent$(5)$. It suffices to apply $(\id\ot h_H\ot h_F\ot\id)$ to the equation proved in $(4)$ and use the fact that the Haar state of $H\times F$ is $h_H\ot h_F$.\end{proof}

\subsection{Block-extended C*-algebra}\label{SectionBlockExtended}
\noindent We use the map $(x\mapsto 1\ot x)$ to view $$C(H)\otm C(F)\subset M_{N_\kappa}(\C)\ot C_\kappa(\GG)$$ for all $1\leq\kappa\leq K$. Define:
$$C(\Gtilde)=(M_{N_1}(\C)\ot C_1(\GG))\underset{C(H)\otm C(F)}{*}\dots\underset{C(H)\otm C(F)}{*}(M_{N_K}(\C)\ot C_K(\GG)).$$
$C(\Gtilde)$ is called  the \textit{block-extended C*-algebra}. We will view $M_{N_\kappa}(\C)\ot C_\kappa(\GG)\subset C(\Gtilde)$. Note that, viewing $C(\GG)=C_1(\GG)\underset{C(H)\otm C(F)}{*}\dots\underset{C(H)\otm C(F)}{*}C_K(\GG)$ (by Theorem \ref{ThmBlockAction}) and using the universal property of the amalgamated free product $C(\GG)$, there exists a unique unital $*$-homomorphism $\iota\,:\, C(\GG) \rightarrow C(\Gtilde)$ such that $\iota(x) =1\ot x$, for all $x\in C_\kappa(\GG)$ and all $\kappa$.

\vspace{0.2cm}

\noindent Consider the conditional expectation 
$$\widetilde{\E}_\kappa:=(\psi_\kappa^{-1}\ot\id\ot\id)\E_\kappa\,:\,M_{N_\kappa}(\C)\ot C_\kappa(\GG)\rightarrow C(H)\otm C(F).$$
Note that $\Ccal_\kappa\subset\ker(\Etilde_\kappa)$. Consider the unique ucp map $E\,:\, C(\Gtilde)\rightarrow C(H)\otm C(F)$ such that $E(\iota(x))=x$ for all $x\in C(H)\ot C(F)$ and $E(x)=0$ for any reduced operator in $C(\Gtilde)$ i.e. of the form $x=x_1\dots x_n$ where $n\geq 1$, $\kappa_{s}\neq\kappa_{s+1}$, $x_s\in M_{N_{\kappa_s}}(\C)\ot C_{\kappa_s}(\GG)$ with $\Etilde_{\kappa_s}(x_s)=0$.

\vspace{0.2cm}

\noindent By Lemma \ref{LemReducedAction} we may consider:
$$\widetilde{\Delta}_\kappa:=\id_{M_{N_\kappa}}\ot\Delta\vert_{C_\kappa(\GG)}\,:\,M_{N_\kappa}(\C)\ot C_\kappa(\GG)\rightarrow  M_{N_\kappa}(\C)\ot C_\kappa(\GG)\ot C(\GG) \subset C(\Gtilde)\ot C(\GG)$$
By the universal property of the amalgamated free product $C(\Gtilde)$ there exists a unique unital $*$-homomorphism $\widetilde{\Delta}\,:C(\Gtilde)\rightarrow C(\Gtilde)\ot C(\GG)$ such that $\widetilde{\Delta}\vert_{M_{N_\kappa}(\C)\ot C_\kappa(\GG)}=\widetilde{\Delta}_\kappa$ for all $\kappa$.

\begin{remark}\label{RemarkErgodic}
We have $(\iota\ot\id)\circ\Delta=\widetilde{\Delta}\circ\iota$ and $(\Dtilde\ot\id)\widetilde{\Delta}=(\id\ot\Delta)\widetilde{\Delta}$. Indeed, $x\in C_\kappa(\GG)$, we have $\widetilde{\Delta}\circ\iota(x) = \widetilde{\Delta}(1\ot x) =(\id \ot \Delta)(1\ot x) = (\iota\ot \id)\Delta(x)$ and, since $\Delta(C_\kappa(\GG))\subset C_\kappa(\GG)\ot C(\GG)$ and $\Delta$ is co-associative, one has, for all $1\leq\kappa\leq K$,
\begin{eqnarray*}
    (\Dtilde \ot \id) \widetilde{\Delta}\vert_{M_{N_\kappa}(\C)\ot C_\kappa(\GG)} &=& (\Dtilde \ot \id)(\id\ot \Delta\vert_{C_\kappa(\GG)}) = (\id\ot \Delta\vert_{C_\kappa(\GG)}\ot \id)(\id \ot \Delta\vert_{C_\kappa(\GG)})\\&=& (\id\ot \id\ot \Delta\vert_{C_\kappa(\GG)})(\id \ot \Delta\vert_{C_\kappa(\GG)})= (\id\ot\Delta)\widetilde{\Delta}\vert_{M_{N_\kappa}(\C)\ot C_\kappa(\GG)}.
\end{eqnarray*}
\end{remark}

\begin{theorem}\label{TheoremErgodic}One has $(E\ot \id)\widetilde{\Delta}= \Delta_{H\times F} \circ E$. In particular:
\begin{enumerate}
    \item $((E\circ \iota)\ot\id)\Delta=\Delta_{H\times F}\circ (E\circ  \iota)$ so $h_\GG:= (h_H\ot h_F)\circ (E\circ\iota)$ and $E\circ\iota=E_{H\times F}$ is the conditional expectation $C(\GG)\rightarrow C(H\times F)$ of the dual quantum subgroup $H\times F$.
    \item $(\id_B \ot E\circ\iota)\rho = 1_B\ot 1_{C(H)} \ot E_F(\cdot)$ therefore $(\id_B\ot h_\GG)\rho=h_G(\cdot) 1_B$.
\end{enumerate}
\end{theorem}

\begin{proof}We start with the following Claim.

\vspace{0.2cm}

\noindent\textbf{Claim.} \textit{$\ker(\widetilde{\E}_\kappa)=\overline{\Ccal_\kappa+\ker(\psi_\kappa^{-1})\ot \Pol(H)\ot \Pol(F)}$.}

\vspace{0.2cm}

\noindent\textit{Proof of the Claim.} Let $x\in\ker(\Etilde_\kappa)$. By Lemma \ref{LemReducedAction} $(1)$, there are sequences $(x_n)$ and $(y_n)$ such that $x_n\in \Ccal_\kappa$ and $y_n\in M_{N_\kappa}(\C)\ot \Pol(H)\ot \Pol(F)$ such that $x_n+y_n\rightarrow x$. Hence, $$\Etilde_\kappa(x_n)+\Etilde_\kappa(y_n)=(\psi_\kappa^{-1}\ot\id\ot\id)(y_n)\rightarrow \Etilde_\kappa(x)=0.$$
Define $y_n^\circ:=y_n-1_{M_{N_\kappa}(\C)}\ot(\psi_\kappa^{-1}\ot\id\ot\id)(y_n)$ so that $(\psi_\kappa^{-1}\ot\id\ot\id)(y_n^\circ)=0$. Note that  $y_n^\circ\in\ker(\psi_\kappa^{-1})\ot\Pol(H)\ot\Pol(F)$ and $x_n+y_n^\circ\rightarrow x$.$\hfill$ $\qed$

\vspace{0.2cm}

\noindent\textit{End of the proof of Theorem.} Suppose that $x\in M_{N_\kappa}(\C)\ot C_\kappa(\GG)$ then, since $\widetilde{\Delta}(x)=(\id\ot\Delta)(x)\in M_{N_\kappa}(\C)\ot C_\kappa(\GG)\ot C(\GG)$ one has, by Lemma \ref{LemReducedAction},
\begin{eqnarray*}
(E\ot \id)\widetilde{\Delta}(x)&=&(E\ot \id)(\id\ot\Delta)(x)=(\Etilde_\kappa\ot \id)(\id\ot\Delta)(x)\\
&=&(\psi_\kappa^{-1}\ot\id\ot\id)(\E_\kappa\ot \id)(\id\ot\Delta)(x)\\
&=&(\psi_\kappa^{-1}\ot\Delta)\E_\kappa(x)=\Delta\circ E(x)=\Delta_{H\times F}\circ E(x).
\end{eqnarray*}
In particular, $(E\ot \id)\widetilde{\Delta}(x)=\Delta_{H\times F}\circ E(x)$ for $x\in C(H)\otm C(F)\subset C(\Gtilde)$ and it remains to show that the formula holds for any reduced operator $x$ in $C(\Gtilde)$. Write such a reduced operator $x=x_1\dots x_n\in C(\Gtilde)$, $x_s\in M_{N_{\kappa_s}}(\C)\ot C_{\kappa_s}(\GG)$ with $\Etilde_{\kappa_s}(x_s)=0$ and $\kappa_s\neq\kappa_{s+1}$ so that $E(x)=0$. It suffices to show that $(E\ot \id)\widetilde{\Delta}(x)=0$. By linearity, continuity and the Claim, we may and will assume that, for all $s$, $x_s\in\Ccal_{\kappa_s}\cup\ker(\psi_\kappa^{-1})\ot \Pol(H)\ot\Pol(F)$. Since
$$\widetilde{\Delta}(x)=(\id\ot\Delta)(x_1)\dots(\id\ot\Delta)(x_s),$$
$\widetilde{\Delta}(x)$ is an alternating product from $(\id\ot\Delta)(\Ccal_{\kappa_s})\subset\Ccal_{\kappa_s}\ot C(\GG)\subset \ker(\Etilde_{\kappa_s})\ot C(\GG)$ (Lemma \ref{LemReducedAction}) and we deduce from:
$$(\id\ot\Delta)(\ker(\psi_\kappa^{-1})\ot\Pol(H\times F))\subset \ker(\psi_\kappa^{-1})\ot\Pol(H\times F)\ot C(\GG)\subset\ker(\Etilde_{\kappa_s})\ot C(\GG)$$
that $(E\ot\id)\widetilde{\Delta}(x)=0$. It proves the first statement of the Theorem.

\vspace{0.2cm}

\noindent$(1)$. We use Remark \ref{RemarkErgodic} and the previous statement to deduce the first equation. Then, applying $h_H\ot h_F\ot\id$ shows that $(h_H\ot h_F)\circ E\circ\iota$ is $\Delta$-invariant hence, $h_\GG=(h_H\ot h_F)\circ E\circ\iota$. We conclude the proof of $(1)$ by Remark \ref{RmkEF} $(3)$.

\vspace{0.2cm}

\noindent$(2)$. Let $a\in C(G)$ and write $\rho(a)=\sum_{\kappa,i,j} e^\kappa_{ij}\ot\rho^\kappa_{ij}(a)$. We use the notations of Theorem \ref{ThmBlockAction}, Lemma \ref{LemReducedAction} and their proofs. By construction of $E$ we have: \begin{eqnarray*}
E\circ\iota(\rho^\kappa_{ij}(a))&=&E(1\ot \rho^\kappa_{ij}(a))=(\psi_\kappa^{-1}\ot\id\ot\id)\E_\kappa(1\ot\rho^\kappa_{ij}(a))\\
&=&(\psi_\kappa^{-1}\ot\id\ot\id)E_2\circ\varphi_\kappa(1\ot\rho^\kappa_{ij}(a))\\
&=&\sum_{s=1}^{N_\kappa}(\psi_\kappa^{-1}\ot\id\ot\id)E_2\left(\mu_\kappa(e^\kappa_{si}\ot 1\ot 1)\nu_\kappa(1\ot a)\mu_\kappa(e^\kappa_{js}\ot 1\ot 1)\right)\\
&=&\sum_{s=1}^{N_\kappa}(\psi_\kappa^{-1}\ot\id\ot\id)((e^\kappa_{si}\ot 1\ot 1)(\beta_\kappa\ot E_F) (1\ot a)(e^\kappa_{js}\ot 1\ot 1))\\
&=&\delta_{ij}1_{C(H)}\ot E_F(a).
\end{eqnarray*}
Hence, $(\id_B\ot E\circ\iota)\rho(a)=\sum_{\kappa,i}e^\kappa_{ii}\ot 1_{C(H)}\ot E_F(a)=1_B\ot 1_{C(H)}\ot E_F(a)$.\end{proof}

\subsection{Relations with fundamental quantum groups}

\noindent Let us now construct a dual quantum subgroup of $G\wr_{*,\beta,F} H$ from a dual quantum subgroup of $G$. Assume that $\Lambda$ is a CQG such that $C(F)\subset C(\Lambda)\subset C(G)$ with the inclusion intertwining the comultiplications. Define $\GG_\Lambda:=\Lambda\wr_{*,\beta,F}H$ with canonical homomorphism $\rho_\Lambda\,:\, C(\Lambda)\rightarrow B\ot C(\GG_\Lambda)$ and view $C(H),C(F)\subset C(\GG_\Lambda)$. Consider, by the universal property of $C(\GG_\Lambda)$, the unique unital $*$-homomorphism $\iota_\Lambda\,:\,C(\GG_\Lambda)\rightarrow C(\GG)$ such that $(\id\ot \iota_\Lambda)\rho_\Lambda=\rho\vert_{C(\Lambda)}$ and $\iota_{\Lambda}\vert_{C(H)\otm C(F)}$ is the inclusion.

\begin{proposition}\label{PropDualSUbGroups}
The map $\iota_\Lambda$ is faithful and intertwines the comultiplications so that $\GG_\Lambda$ is a dual quantum subgroup of $\GG$.
\end{proposition}

\begin{proof}
Consider the two canonical maps $\beta_G\,:\, C(G)\ot B\rightarrow B\ot C(\GG)$ and $\beta_\Lambda\,:\, C(\Lambda)\ot B\rightarrow B\ot C(\GG_\Lambda)$ and note that, by definition, $(\id\ot\iota_\Lambda)\beta_\Lambda=\beta_G\vert_{C(\Lambda)\ot B}$. It follows easily that $$(\id\ot\Delta_\GG\circ\iota_\Lambda)\rho_\Lambda=(\id\ot(\iota_\Lambda\ot\iota_\Lambda)\circ\Delta_{\GG_\Lambda})\rho_\Lambda$$ which implies that $\iota_\Lambda$ intertwines the comultiplications. Let us now show that $\iota_\Lambda$ is faithful. By Proposition \ref{dualqsg}, it suffices to show that $\iota_\Lambda$ intertwines the Haar states (since it will imply that it is faithful at the reduced level). Note that, for $1\leq\kappa\leq K$, $\iota_\Lambda(C_\kappa(\GG_\Lambda))\subset C_\kappa(\GG)$. Hence, there exists a unique unital $*$-homomorphism $\widetilde{\iota}_\Lambda\,:\, C(\Gtilde_\Lambda)\rightarrow C(\Gtilde)$ such that:
$$\widetilde{\iota}_\Lambda\vert_{M_{N_\kappa}(\C)\ot C_\kappa(\GG_\Lambda)}=\id_{M_{N_\kappa}(\C)}\ot\iota_\Lambda\vert_{C_\kappa(\GG)},\text{ for all }1\leq\kappa\leq K.$$
By construction, $\widetilde{\iota}_\Lambda$ maps reduced operators in $C(\Gtilde_\Lambda)$ to reduced operators in $C(\Gtilde)$ and is the identity on $C(H)\otm C(F)$. By uniqueness property, we have $E\circ\iota_\Lambda=E_\Lambda$, where $E\,:\, C(\Gtilde)\rightarrow C(H)\otm C(F)$ and $E_\Lambda\,:\, C(\Gtilde_\Lambda)\rightarrow C(H)\otm C(F)$ are the canonical ucp maps constructed in Section \ref{SectionBlockExtended}. It follows from Theorem \ref{TheoremErgodic} that $\iota_\Lambda$ intertwines the Haar states.\end{proof}

\noindent Using this dual quantum subgroup construction, we can now study the decomposition of a free wreath product $G\wr_{*,\beta,F} H$ as a fundamental quantum group when $G$ itself is a fundamental quantum group.

\vspace{0.2cm}

\noindent Let $G:=\pi_1(G_p,\Lambda_e,\G,\T)$ be a fundamental quantum group with graph maps $s_e\,:\,C(\Lambda_e)\rightarrow C(G_{s(e)})$ so that $C(G)=\langle C(G_p),u_e\,\vert\, p\in V(\G),e\in E(\G)\rangle$. Following \cite[Definition 6.1]{FT24}, the \textit{loop subgroup} of $G$ is the subgroup $\Gamma\subset\mathcal{U}(C(G))$ generated by $\{u_e\,:\,e\in E(\G)\}$. It is shown in \cite[Proposition 6.2]{FT24} that the unique unital $*$-homomorphism $\pi\,:\,C^*(\Gamma)\rightarrow C(G)$ extending the inclusion $\Gamma\subset\mathcal{U}(C(G))$ is faithful and intertwines the comultiplications so that $\widehat{\Gamma}$ is a dual quantum subgroup of $G$.

\vspace{0.2cm}

\noindent Define, for $p\in V(\G)$ and $e\in E(\G)$, $\GG_p:=G_p\wr_{*,\beta} H$ and $\GG_e:=\Lambda_e\wr_{*,\beta} H$. Write $\rho_p\,:\,C(G_p)\rightarrow B\ot C(\GG_p)$ and $\rho_e\,:\,C(\Lambda_e)\rightarrow B\ot C(\GG_e)$ the canonical morphisms and view $C(H)\subset C(\GG_p)$ and $C(H)\subset C(\GG_e)$. As in Proposition \ref{PropDualSUbGroups} we consider, for each $e\in E(\G)$, the unique map $s'_e\,:\,C(\GG_e)\rightarrow C(\GG_{s(e)})$ such that $(\id\ot s_e')\rho_e=\rho_{s(e)}\circ s_e$ and $s_e'\vert_{C(H)}$ is the inclusion, which is, by Proposition \ref{PropDualSUbGroups}, faithful and intertwines the comultiplications. We obtain a graph of quantum groups $(\GG_p,\GG_e,\G,\T)$ with graph maps $s'_e$.

\vspace{0.2cm}

\noindent The next result should be compared with Theorem \ref{deltafree} where we studied such a decomposition when $H$ is a fundamental quantum group. It is also a generalization of \cite[Theorem 6.3]{FT24}.

\begin{theorem}
There is a canonical isomorphism $G\wr_{*,\beta,\widehat{\Gamma}} H\simeq\pi_1(\GG_p,\GG_e,\G,\T)$.
\end{theorem}

\begin{proof}
Write $\GG:=G\wr_{*,\beta,\widehat{\Gamma}} H$ with canonical morphism $\rho\,:\, C(G)\rightarrow B\ot C(\GG)$ and view $C(H),C^*(\Gamma)\subset C(\GG)$. We also write $\GG':=\pi_1(\GG_p,\GG_e,\G,\T)$, we view $C(H)\subset C(\GG_p)\subset C(\GG')$ and we denote $$C(\GG')=\langle C(\GG_p),v_e\,\vert\,p\in V(\G),e\in E(\G)\rangle\text{ and }C(G)=\langle C(G_p),u_e\,\vert\,p\in V(\G),e\in E(\G)\rangle.$$
View, for all $p\in V(\G)$, $\rho_p\,:\,C(G_p)\rightarrow B\ot C(\GG_p)\subset B\ot C(\GG')$ and note that, for all $b\in C(\Lambda_e)$, $$\rho_{s(e)}(s_e(b))=(\id\ot s_e')\rho_e(b)=(1_B\ot v_{e}^*)(\id\ot r_e')(\rho_e(b))(1_B\ot v_e)=(1_B\ot v_{e}^*)\rho_{r(e)}\circ r_e(b)(1_B\ot v_e).$$
The universal property of $C(G)$ gives a unique unital $*$-homomorphism
$$\rho'\,:\, C(G)\rightarrow B\ot C(\GG')\text{ such that }\rho'\vert_{C(G_p)}=\rho_p,\,\,\rho'(u_e)=1_B\ot v_e\text{ for all }p\in V(\G),e\in E(\G).$$ By construction of $\rho'$, one has $\rho'(C^*(\Gamma))\subset \C 1_B\ot C(\GG')$ hence, there exists a unique unital $*$-homomorphism $\iota_{\widehat{\Gamma}}\,:\,C^*(\Gamma)\rightarrow C(\GG')$ such that $\rho'(f)=1_B\ot\iota_{\widehat{\Gamma}}(f)$ for all $f\in C^*(\Gamma)$. Note that for all $p\in V(\G)$, $a\in C(G_p)$ and $b\in B$, $\rho'(a)\beta(b)=\rho_p(a)\beta(b)=\beta(b)\rho_p(a)=\beta(b)\rho'(a)$. Moreover, since each $s_e'$ is the identity on $C(H)$, each $v_e$ commutes with $C(H)$ inside $C(\GG')$ hence, $\rho'(u_e)\beta(b)=(1_B\ot v_e)\beta(b)=\beta(b)(1_B\ot v_e)$ for all $b\in B$ and $e\in E(\G)$. Also, $C(H)$ commutes with $\iota_{\widehat{\Gamma}}(C^*(\Gamma))$. By the universal property of the free wreath product C*-algebra $C(\GG)$, there exists a unique unital $*$-homomorphism $\pi\,:\, C(\GG)\rightarrow C(\GG')$ such that $(\id\ot\pi)\rho=\rho'$, $\pi\vert_{C(H)}$ is the inclusion and $\pi\vert_{C^*(\Gamma)}=\iota_{\widehat{\Gamma}}$. Writing $\beta_G\,:\, C(G)\ot B\rightarrow B\ot C(\GG)$ and $\beta_{G_p}\,:\, C(G_p)\ot B\rightarrow B\ot C(\GG_p)\subset B\ot C(\GG')$ the canonical morphisms we have, by construction, $\beta_{G_p}=(\id\ot\pi)\circ\beta_G\vert_{C(G_p)\ot B}$. For $a\in C(G_p)\subset C(G)$ one has:
\begin{eqnarray*}
(\id\ot\Delta_{\GG'}\pi)\rho(a)&=&(\id\ot\Delta_{\GG'})\rho'(a)=(\id\ot\Delta_{\GG_p})\rho_p(a)=(\beta_{G_p}\ot\id)(\id\ot\rho_p)\Delta_{G_p}(a)\\
&=&(\id\ot\pi\ot\id)(\beta_G\ot\id)(\id\ot\rho')\Delta_{G_p}(a)\\
&=&(\id\ot\pi\ot\pi)(\beta_G\ot\id)(\id\ot\rho)\Delta_G(a)
=(\id\ot(\pi\ot\pi)\Delta_\GG)\rho.
\end{eqnarray*}
Since $\Delta_{\GG'}\pi(u_e)=\Delta_{\GG'}(v_e)=v_e\ot v_e=(\pi\ot\pi)\Delta_{\GG}(u_e)$ and,
$$\Delta_{\GG'}\circ\pi\vert_{C(H)}=\Delta_{\GG'}\vert_{C(H)}=\Delta_H=(\pi\ot\pi)\Delta_{\GG}\vert_{C(H)}$$
we deduce that $\pi$ intertwines the comultiplications.

\vspace{0.2cm}

\noindent We now show that $\pi$ is an isomorphism by constructing its inverse morphism $\pi^{-1}\,:\, C(\GG')\rightarrow C(\GG)$. We first note that, by the universal property of the free wreath product C*-algebra $C(\GG_p)$, there exists a unique $\pi_p^{-1}\,:\, C(\GG_p)\rightarrow C(\GG)$ such that $(\id\ot\pi_p^{-1})\rho_p=\rho\vert_{C(\GG_p)}$ and $\pi^{-1}_p\vert_{C(H)}$ is the inclusion. Viewing $C^*(\Gamma)\subset C(\GG)$, we have the relation $\rho(u_e)=1_B\ot u_e$, for all $e\in E(\G)$. Note that, for all $e\in E(\G)$ and $b\in C(\Lambda_e)$ one has,
\begin{eqnarray*}
(\id\ot\pi^{-1}_{s(e)}\circ s_e')\rho_e(b)&=&(\id\ot\pi^{-1}_{s(e)})\rho_{s(e)}\circ s_e(b)=\rho\circ s_e(b)=\rho(u_e)^*\rho(r_e(b))\rho(u_e)\\
&=&\rho(u_e)^*(\id\ot\pi^{-1}_{r(e)})\rho_{r(e)}\circ r_e(b)\rho(u_e)\\
&=&\rho(u_e)^*(\id\ot\pi^{-1}_{r(e)}\circ r_e')\rho_e(b)\rho(u_e)\\
&=&(1\ot u_e^*)(\id\ot\pi^{-1}_{r(e)}\circ r_e')\rho_e(b)(1\ot u_e).
\end{eqnarray*}
Moreover, $\pi_{s(e)}^{-1}\circ s_e'\vert_{C(H)}$ is the inclusion of $C(H)\subset C(\GG)$ and since $C(H)$ commutes with $C^*(\Gamma)$ in  $C(\GG)$, the universal property of $C(\GG')$ gives a unique unital $*$-homomorphism $\pi^{-1}\,:\, C(\GG')\rightarrow C(\GG)$ such that $\pi^{-1}\vert_{C(\GG_p)}=\pi_p^{-1}$ and $\pi^{-1}(v_e)=u_e$. It is easy to check that $\pi^{-1}$ is the inverse of $\pi$.\end{proof}

\section{The reduced C*-algebra and the von Neumann algebra}\label{section reduced}

\noindent We fix, for the entire section, an ergodic action $\beta\,:\,H\curvearrowright B$ on the finite dimensional C*-algebra $B$ with unique invariant state $\psi\in B^*$ and we write  $\GG:=G\wr_{*,\beta,F}H$.

\subsection{The reduced C*-algebra}

\noindent For $1\leq\kappa\leq K$, we use the following notation:
$$C_\kappa(\Gtilde):=(B\ot C(G))\underset{B\ot C(F)}{*} (M_{N_\kappa}(\C)\ot C(H)\otm C(F))\text{ and,}$$
$$ C_{r,\kappa}(\Gtilde):= (B\ot C_r(G), \id_B\ot E^r_{F} ) \underset{B\ot C_r(F)}{*} (M_{N_\kappa}(\C)\ot C_r(H)\ot C_r(F),E_\kappa\ot\id_{C_r(F)}),$$
relatively to the inclusion map on the left and $\beta_\kappa\ot \id_{C_r(F)}$ on the right. Note that the conditional expectations involved in the Voiculescu's reduced amalgamated free product are faithful which implies that the canonical conditional expectation
$$E_{\kappa,2}\,:\, C_{r,\kappa}(\Gtilde)\rightarrow M_{N_\kappa}(\C)\ot C_r(H)\ot C_r(F)$$
is faithful. Consider the unique surjective homomorphism $\widetilde{\lambda}_\kappa\,:\, C_{\kappa}(\Gtilde)\rightarrow C_{r,\kappa}(\Gtilde)$ such that:
$$\widetilde{\lambda}_\kappa\vert_{B\ot C(G)}=\id_B\ot\lambda_G\text{ and }\widetilde{\lambda}_\kappa\vert_{M_{N_\kappa}(\C)\ot C(H)\otm C(F)}=\id_{M_{N_\kappa}(\C)}\ot\lambda_{H\times F}.$$
Note that $C_r(H\times F)=C_r(H)\ot C_r(F)$ and $\lambda_{H\times F}=(\lambda_H\ot \lambda_F)\circ p$, where
$$p\,:\, C(H)\otm C(F)\rightarrow C(H)\ot C(F)$$
is the canonical surjection. Recall that we have, from Theorem \ref{ThmBlockAction}, an isomorphism $\pi_\kappa\,:\, C_\kappa(\Gtilde)\rightarrow M_{N_\kappa}(\C)\ot C_\kappa(\GG)$
with inverse $\varphi_\kappa\,:\,M_{N_\kappa}(\C)\ot C_\kappa(\GG)\rightarrow C_\kappa(\Gtilde)$.

\vspace{0.2cm}

\noindent By universal property, there exists a unique surjective homomorphism
$$\widetilde{\lambda}\,:\,C(\Gtilde)\rightarrow C_r(\Gtilde):=\underset{C_r(H)\ot C_r(F)}{*}\left(C_{r,\kappa}(\Gtilde),(\psi_\kappa^{-1}\ot\id\ot\id)\circ E_{2,\kappa}\right)$$
such that $\widetilde{\lambda}\vert_{M_{N_\kappa}(\C)\ot C_\kappa(\GG)}=\widetilde{\lambda}_\kappa\circ\varphi_\kappa$, for all $1\leq\kappa\leq K$. Note that that the conditional expectations $(\psi_\kappa^{-1}\ot\id\ot\id)\circ E_{\kappa,2}$ are faithful which implies that the canonical conditional expectation $E_r\,:\, C_r(\Gtilde)\rightarrow C_r(H)\ot C_r(F)$ is also faithful.
Recall that we have an inclusion $\iota\,:\, C(\GG)\rightarrow C(\Gtilde)$ and define $C_{r,\kappa}(\GG):=\widetilde{\lambda}\circ\iota(C_\kappa(\GG))=\widetilde{\lambda}_\kappa\circ\varphi_\kappa\circ\iota(C_\kappa(\GG))\subset C_{r,\kappa}(\Gtilde)$.

\begin{theorem}\label{ThmReduced}
The following holds.
\begin{enumerate}
\item One has $C_r(\GG)=\widetilde{\lambda}\circ\iota(C(\GG))$ with GNS map $\lambda_\GG=\widetilde{\lambda}\circ\iota$. Moreover, $E_r\vert_{C_r(\GG)}=E^r_{H\times F}$ is the reduced conditional expectation associated to the dual quantum subgroup $H\times F$.
\item One has $C_r(\GG)=\underset{C_r(H\times F)}{*}\left(C_{r,\kappa}(\GG),E_r\vert_{C_{r,\kappa}(\GG)}\right)$.
\item For all $1\leq\kappa\leq K$, there is a unique unital $*$-isomorphism:
$$\pi_{r,\kappa}\,:\,C_{r,\kappa}(\Gtilde)\rightarrow M_{N_\kappa}(\C)\ot C_{r,\kappa}(\GG)$$
such that $\pi_{r,\kappa}\circ\widetilde{\lambda}_\kappa=(\id\ot(\widetilde{\lambda}\circ\iota)\vert_{C_\kappa(\GG)})\pi_\kappa$.
\end{enumerate}
\end{theorem}

\begin{proof}
$(1)$. By construction, $\lambda_{H\times F}\circ E=E_r\circ\widetilde{\lambda}$. Since $E_r$ is faithful and $\widetilde{\lambda}$ is surjective, we deduce that $C_r(\Gtilde)$ is the reduced C*-algebra of $\lambda_{H\times F}\circ E$ and it follows from Remark \ref{RmkReducedUcp} that $\widetilde{\lambda}\circ\iota(C(\GG))$ is the reduced C*-algebra of $\lambda_{H\times F}\circ E\circ\iota=\lambda_{H\times F}\circ E_{H\times F}$. We conclude the proof of $(1)$ using Remark \ref{RmkEF} $(1)$.

\vspace{0.2cm}

\noindent$(2)$. Since $C(\GG)$ is generated by the algebras $C_\kappa(\GG)$, for $1\leq\kappa\leq K$, we conclude that $C_r(\GG)=\widetilde{\lambda}\circ\iota(C(\GG))$ is generated by $C_{r,\kappa}(\GG)=\widetilde{\lambda}\circ\iota(C_\kappa(\GG))$  ($1\leq\kappa\leq K$). Moreover, since the $C_{r,\kappa}(\Gtilde)$ are free with amalgamation with respect to $E_r\vert_{C_{r,\kappa}(\Gtilde)}=(\psi_\kappa^{-1}\ot\id\ot\id)\circ E_{2,\kappa}$ and since $C_{r,\kappa}(\GG)\subset C_{r,\kappa}(\Gtilde)$ the result follows.

\vspace{0.2cm}

\noindent$(3)$. We construct the inverse isomorphism $\varphi_{r,\kappa}\,:\, M_{N_\kappa(\C)}\ot C_{r,\kappa}(\GG)\rightarrow C_{r,\kappa}(\Gtilde)$ satisfying $\varphi_{r,\kappa}\circ(\id\ot(\widetilde{\lambda}\circ\iota)\vert_{C_\kappa(\GG)})=\widetilde{\lambda}_\kappa\circ\varphi_\kappa$. We view both $M_{N_\kappa}(\C)\subset C_\kappa(\Gtilde)$ and $M_{N_\kappa}(\C)\subset C_{r,\kappa}(\Gtilde)$ in the obvious way. Note that, since $\varphi_\kappa\circ\iota(C_\kappa(\GG))=M_{N_\kappa}(\C)'\cap C_\kappa(\Gtilde)$ we deduce that $C_{r,\kappa}(\GG)\subseteq M_{N_\kappa}(\C)'\cap C_{r,\kappa}(\Gtilde)$. Hence, there exists a unique faithful unital $*$-homomorphism $\varphi_{r,\kappa}\,:\,M_{N_\kappa}(\C)\ot C_{r,\kappa}(\GG)\rightarrow C_{r,\kappa}(\Gtilde)$, such that $\varphi_{r,\kappa}(b\ot x)=bx$ for all $b\in M_{N_\kappa}(\C)$ and $x\in C_{r,\kappa}(\GG).$
Since $C_\kappa(\Gtilde)$ is generated by $M_{N_\kappa}(\C)$ and $\varphi_\kappa\circ\iota(C_\kappa(\GG))$ and since $\widetilde{\lambda}_\kappa$ is surjective we deduce that $C_{r,\kappa}(\Gtilde)$ is generated by $M_{N_\kappa}(\C)$ and $\widetilde{\lambda}_\kappa(\varphi_\kappa\circ\iota (C_\kappa(\GG)))=C_{r,\kappa}(\GG)$. Both these C*-subalgebras are in the image of $\varphi_{r,\kappa}$ so $\varphi_{r,\kappa}$ is surjective.\end{proof}

\noindent We obtain the following interesting Corollary.

\begin{corollary}\label{CorExact}
The following are equivalent.
\begin{enumerate}
    \item $\widehat{G}$ and $\widehat{H}$ are exact.
    \item $\widehat{\GG}$ is exact.
\end{enumerate}
\end{corollary}

\begin{proof}
$(1)\Rightarrow(2)$. Suppose that $C_r(G)$ and $C_r(H)$ are exact then, since $C_r(F)\subset C_r(G)$ is a C*-subalgebra, $C_r(F)$ is also exact. Theorem \ref{ThmReduced} and \cite[Theorem 3.2]{Dy04} implies that $C_r(\GG)$ is exact.

\vspace{0.2cm}

\noindent$(2)\Rightarrow(1)$. Assume that $C_r(\GG)$ is exact and fix $1\leq \kappa\leq K$. Since $C_{r,\kappa}(\GG)\subset C_r(\GG)$ is a C*-subalgebra, we deduce that $C_{r,\kappa}(\GG)$ is exact so $C_{r,\kappa}(\Gtilde)\simeq M_{N_\kappa}(\C)\ot C_{r,\kappa}(\GG)$ is also exact. Finally, since both $C_r(G), C_r(H)\subset C_{r,\kappa}(\Gtilde)$ are C*-subalgebras, it concludes the proof.\end{proof}

\subsection{The von Neumann algebra}\label{sectionvna}
For $1\leq\kappa\leq K$ define:
$$L_\kappa(\Gtilde)= (B\ot \Linf(G),\id\ot E_F'')\underset{B\ot \Linf(F)}{\ast}(M_{N_\kappa}(\C)\ot \Linf(H)\ot \Linf(F),E_\kappa\ot\id_{\Linf(F)}).$$
Let $E_{B\ot\Linf(F)}\,:\, L_\kappa(\Gtilde)\rightarrow B\ot\Linf(F)$ be the canonical normal faithful conditional expectation. By uniqueness of the reduced C*-algebra of the conditional expectation
$$E_{B\ot C_r(F)}\,:\,C_{r,\kappa}(\Gtilde)\rightarrow B\ot C_r(F)$$
we may and will view $C_{r,\kappa}(\Gtilde)\subset L_\kappa(\Gtilde)$ in the canonical way. Note that $C_{r,\kappa}(\Gtilde)''= L_\kappa(\Gtilde)$. Let $E''_{\kappa,2}\,:\, L_\kappa(\Gtilde)\rightarrow M_{N_\kappa}(\C)\ot\Linf(H)\ot\Linf(F)$ be the canonical normal faithful conditional expectation and consider the faithful normal conditional expectation $$\Etilde''_\kappa:=(\psi_\kappa^{-1}\ot\id\ot\id)\circ E''_{\kappa,2}\,:\, L_\kappa(\Gtilde)\rightarrow \Linf(H)\ot\Linf(F).$$ Define
$$\Linf(\Gtilde):= \underset{\Linf(H)\ot \Linf(F)}{\ast} (L_\kappa(\Gtilde),\Etilde''_\kappa).$$
By uniqueness of the reduced C*-algebra of $\lambda_{H\times F}\circ E$ we may and will view $C_r(\Gtilde)\subset\Linf(\Gtilde)$ in the obvious way and we have $C_r(\Gtilde)''=\Linf(\Gtilde)$. Write $\widetilde{E}''\,:\,\Linf(\Gtilde)\rightarrow \Linf(H)\ot\Linf(F)$ and $\widetilde{E}_\kappa''\,:\,\Linf(\Gtilde)\rightarrow L_\kappa(\Gtilde)$ the canonical normal faithful conditional expectations.

\vspace{0.2cm}

\noindent Define $L_\kappa(\GG):=C_{r,\kappa}(\GG)''\subset C_{r,\kappa}(\Gtilde)''=L_\kappa(\Gtilde)$ so that $\Linf(H)\ot\Linf(F)\subset L_\kappa(\GG)$.

\begin{theorem}
The following holds.
\begin{enumerate}
    \item $\Linf(\GG)=\underset{\Linf(H)\ot \Linf(F)}{\ast} (L_\kappa(\GG),E''\vert_{L_\kappa(\GG)})\subset\Linf(\Gtilde)$ and $E''_{H\times F}:=\widetilde{E}''\vert_{\Linf(\GG)}$ is the conditional expectation of the dual quantum subgroup $H\times F$.
    \item For all $1\leq\kappa\leq K$, there is a unique unital normal $*$-isomorphism:
$$\pi''_{\kappa}\,:\,L_{\kappa}(\Gtilde)\rightarrow M_{N_\kappa}(\C)\ot L_{\kappa}(\GG)$$
such that $\pi''_{\kappa}\vert_{C_{r,\kappa}(\Gtilde)}=\pi_{r,\kappa}$.
\end{enumerate}
\end{theorem}

\begin{proof}
$(1).$ Let $M:=\underset{\Linf(H)\ot \Linf(F)}{\ast} (L_\kappa(\GG),E''\vert_{L_\kappa(\GG)})\subset\Linf(\Gtilde)$. By uniqueness of the reduced C*-algebra of $E^r_{H\times F}$ we may and will view $C_r(\GG)\subset M$. Note that $C_r(\GG)''=M$ and $ E''\vert_{C_r(\GG)}= E^r_{H\times F}$. $(1)$ follows.

\vspace{0.2cm}

\noindent$(2)$. Can be proved as statement $(3)$ in Theorem \ref{ThmReduced}.
\end{proof}

\noindent Define the state $h_{\Gtilde}$ on $\Linf(\Gtilde)$ by $h_{\Gtilde}:=(h_H\ot h_F)\circ \widetilde{E}''$ and note that the Haar state on $\Linf(\GG)$ is $h_\GG=(h_H\ot h_F)\circ E_{H\times F}''=h_{\Gtilde}\vert_{\Linf(\GG)}$. We also consider the states
$$\varphi_\kappa:=(\psi\ot h_F)\circ E_{B\ot\Linf(F)}\in L_\kappa(\Gtilde)_{*}\text{ for all }1\leq \kappa\leq K,$$
and $h_\kappa:=\varphi_\kappa\vert_{L_\kappa(\GG)}$.

\begin{proposition}\label{PropModular}
The following holds.
\begin{enumerate}
\item For all $1\leq\kappa\leq K$ one has $\varphi_\kappa=(h_H\ot h_F)\circ\widetilde{\E}''_\kappa$ and, for all $t\in\R$,
$$\sigma_t^{\varphi_\kappa}\vert_{B\ot\Linf(G)}=\sigma_t^{\psi}\ot \sigma_t^{h_G}\text{ and }\sigma_t^{\varphi_\kappa}\vert_{M_{N_\kappa}(\C)\ot\Linf(H)\ot\Linf(F)}=\sigma_{-t}^{\psi_\kappa}\ot\sigma_t^{h_H}\ot\sigma_t^{h_F}.$$

\item For all $1\leq\kappa\leq K$ one has $h_{\Gtilde}=\varphi_\kappa\circ\widetilde{E}_\kappa''$. In particular, $\sigma_t^{h_{\Gtilde}}\vert_{L_\kappa(\Gtilde)}=\sigma_t^{\varphi_\kappa}$ for all $t\in\R$.
\item For all $1\leq\kappa\leq K$ one has $\varphi_\kappa=(\psi_\kappa^{-1}\ot h_\kappa)\circ\pi_\kappa''$ and $\sigma_t^{\varphi_\kappa}\vert_{L_\kappa(\GG)}=\sigma_t^{h_\kappa}$
    \item For all $t\in\R$, $\sigma_t^{h_{\Gtilde}}\vert_{\Linf(\GG)}=\sigma_t^{h_\GG}$. In particular, there exists a unique normal faithful conditional expectation $E_\GG\,:\, \Linf(\Gtilde)\rightarrow\Linf(\GG)$ such that $h_{\GG}\circ E_{\GG}=h_{\Gtilde}$.
\end{enumerate}
\end{proposition}

\begin{proof}
$(1)$ and $(2)$. The first equality follows from $\psi\circ E_\kappa=\psi_\kappa^{-1}\ot h_H$ (Proposition \ref{PropUcpSection}). The equality about modular groups follows from Remark \ref{RmkModularFree} as well as statement $(2)$.

\vspace{0.2cm}

\noindent$(3)$. Fix a system of matrix units $(e^\kappa_{ij})_{ij}$ in $M_{N_\kappa}(\C)$ diagonalizing $Q_\kappa$ i.e. $Q_\kappa e^\kappa_{ij}=\lambda_{k,i}e^\kappa_{ij}$. We note that $L_\kappa(\GG)=M_{N_\kappa}(\C)'\cap L_\kappa(\Gtilde)$ is generated by $\Linf(H)\ot\Linf(F)\subset L_\kappa(\Gtilde)$ and $$\{\rho^\kappa_{ij}(a)\,:\,1\leq i,j\leq N_\kappa,\, a\in\Linf(G)\},\text{ where}$$ $\rho^\kappa_{ij}(a):=\sum_r(e^\kappa_{ri}\ot 1\ot 1)(1_B\ot a)(e^\kappa_{jr}\ot 1\ot 1)\in L_\kappa(\Gtilde)$. In particular, the linear span of $$\{x_0\rho^\kappa_{i_1j_1}(a_1)\dots\rho^\kappa_{i_nj_n}(a_n)x_n\,:\,n\geq 1, 1\leq i_s,j_s\leq N_\kappa,a_s\in\Linf(G),\,x_s\in\Linf(H)\ot\Linf(F)\}$$ is a unital weakly dense $*$-subalgebra of $L_\kappa(\GG)$. Hence, it suffices to show that, for any $X$ of the form $X=x_0\rho^\kappa_{i_1j_1}(a_1)\dots\rho^\kappa_{i_nj_n}(a_n)x_n$ as above, one has $\varphi_\kappa(e^\kappa_{st}X)=\psi_\kappa^{-1}(e^\kappa_{st})\varphi_\kappa(X)$, for all $1\leq s,t\leq N_\kappa$. Note that $e^\kappa_{st}X=(e^\kappa_{si_1}\ot x_0)(1\ot a_1)(e^\kappa_{j_1i_2}\ot x_2)\dots(1\ot a_n)(e^\kappa_{j_nt}\ot x_n)$. Then,
\begin{eqnarray*}
\varphi_\kappa(e^\kappa_{st} X)&=&(\psi_\kappa^{-1}\ot h_H\ot h_F)\circ E_{\kappa,2}''(e^\kappa_{st}X)\\
&=&(\psi_\kappa^{-1}\ot h_H\ot h_F)\left( (e^\kappa_{si_1}\ot x_0)E_{\kappa,2}''(X')(e^\kappa_{j_nt}\ot x_n)\right)\\
&=&\psi_\kappa^{-1}(e^\kappa_{si_1}Ye^\kappa_{j_nt})=\frac{{\rm Tr}(Q_\kappa^{-1}e^\kappa_{si_1}Ye^\kappa_{j_nt})}{{\rm Tr}(Q_\kappa^{-1})}
=\frac{{\rm Tr}(e^\kappa_{j_nt}Q_\kappa^{-1}e^\kappa_{si_1}Y)}{{\rm Tr}(Q_\kappa^{-1})}\\
&=&\frac{\delta_{s,t}\lambda_{\kappa,s}^{-1}{\rm Tr}(e^\kappa_{j_ni_1}Y)}{{\rm Tr}(Q_\kappa^{-1})}=\psi_\kappa^{-1}(e^\kappa_{st}){\rm Tr}(e^\kappa_{j_ni_1}Y)
\end{eqnarray*}
where $X':=(1\ot a_1)(e^\kappa_{j_1i_2}\ot x_2)\dots(1\ot a_n)$ and $$Y:=(\id\ot h_H\ot h_F)((1\ot x_0)E_{\kappa,2}''(X')(1\ot x_n)).$$
Note that $\psi_\kappa^{-1}\left(\sum_re^\kappa_{ri_1}Ye^\kappa_{j_nr}\right)=(\psi_\kappa^{-1}\ot h_H\ot h_F)E''_{\kappa,2}(X)=\varphi_\kappa(X)$.

\vspace{0.2cm}

\noindent Writing $e^\kappa_{j_ni_1}=\frac{\sum_r\lambda_{\kappa,r}^{-1}e^\kappa_{j_nr}e^\kappa_{ri_1}}{{\rm Tr}(Q_\kappa^{-1})}$ we get:
\begin{eqnarray*}
\varphi_\kappa(e^\kappa_{st}X)&=&\frac{\psi_\kappa^{-1}(e^\kappa_{st})}{{\rm Tr}(Q_\kappa^{-1})}\sum_r\lambda_{\kappa,r}^{-1}{\rm Tr}(e^\kappa_{j_nr}e^\kappa_{ri_1}Y)=\frac{\psi_\kappa^{-1}(e^\kappa_{st})}{{\rm Tr}(Q_\kappa^{-1})}\sum_r\lambda_{\kappa,r}^{-1}{\rm Tr}(e^\kappa_{ri_1}Ye^\kappa_{j_nr})\\
&=&\frac{\psi_\kappa^{-1}(e^\kappa_{st})}{{\rm Tr}(Q_\kappa^{-1})}\sum_r{\rm Tr}(Q_\kappa^{-1}e^\kappa_{ri_1}Ye^\kappa_{j_nr})=\psi_\kappa^{-1}(e^\kappa_{st})\psi_\kappa^{-1}\left(\sum_re^\kappa_{ri_1}Ye^\kappa_{j_nr}\right)\\
&=&\psi_\kappa^{-1}(e^\kappa_{st})\varphi_\kappa(X).
\end{eqnarray*}

\noindent For the last part of $(3)$ we note that, under the identification $L_\kappa(\Gtilde)=M_{N_\kappa}(\C)\ot L_\kappa(\GG)$, $\sigma_t^{\varphi_\kappa}=\sigma_{-t}^{\psi_\kappa}\ot\sigma_t^{h_\kappa}$ so, for all $x\in L_\kappa(\GG)$, $\sigma_t^{h_{\Gtilde}}(x)=\sigma_t^{\varphi_\kappa}(x)=(\sigma_{-t}^{\psi_\kappa}\ot\sigma_t^{h_\kappa})(1\ot x)=\sigma_t^{h_\kappa}(x)$.

\vspace{0.2cm}

\noindent$(4)$. Since for all $x\in L_\kappa(\GG)$, $(h_H\ot h_F)\circ E''(x)=\varphi_\kappa(x)=h_\kappa(x)$, we deduce from general properties of the amalgamated free product that $\sigma_t^{h_\GG}(x)=\sigma_t^{h_\kappa}(x)$. Since $\Linf(\GG)$ is generated by the $L_\kappa(\GG)$, for $1\leq\kappa\leq K$, it concludes the proof.\end{proof}

\noindent The following corollary is a direct consequence of Proposition \ref{PropModular}.

\begin{corollary}\label{CorKac}
$\GG$ is Kac of and only if both $G$ and $H$ are Kac.
\end{corollary}

\noindent We now compute the scaling group of $\GG$. Note that the scaling group of $G$, $\tau_t^G$, preserves the Haar state of $G$ and satisfies $\tau^G_t\vert_{\Linf(F)}=\tau_t^F$. Given $1\leq\kappa\leq K$, Proposition \ref{PropUcpSection} $(2)$ allows to use the universal property of the amalgamated free product $L_\kappa(\Gtilde)$ which implies that there exists a unique $\varphi_\kappa$-preserving automorphism $\widetilde{\tau}^\kappa_t$ of $L_\kappa(\Gtilde)$ such that $\widetilde{\tau}^\kappa_t\vert_{B\ot\Linf(G)}=\sigma_{-t}^\psi\ot\tau_t^G$ and $\widetilde{\tau}^\kappa_t\vert_{M_{N_\kappa}(\C)\ot\Linf(H)\ot\Linf(F)}=\sigma_{-t}^{\psi_\kappa}\ot\tau_t^H\ot\tau_t^F$. Then, since $\widetilde{\tau}^\kappa_t$ is $\varphi_\kappa$-preserving and $\widetilde{\tau}^\kappa_t\vert_{\Linf(F)}=\tau_t^F$ for all $1\leq\kappa\leq K$, using the universal property of the amalgamated free product $\Linf(\Gtilde)$, there is a unique $h_{\Gtilde}$-preserving automorphism $\tau_t^{\Gtilde}$ of $\Linf(\Gtilde)$ such that, for all $1\leq\kappa\leq K$, $\tau_t^{\Gtilde}\vert_{L_\kappa(\Gtilde)}=\widetilde{\tau}_t^\kappa$. It is not difficult to check that $(t\mapsto\tau_t^{\Gtilde})$ is a one parameter automorphism group of $\Linf(\Gtilde)$ which is strongly continuous, by Lemma \ref{LemOneParameter}.

\vspace{0.2cm}

\noindent In the sequel, we will also denote by $\rho\,:\,\Linf(G)\rightarrow B\ot\Linf(\GG)$ the unital normal $*$-homomorphism defined by $\rho(a)=\sum_{\kappa,i,j}e^\kappa_{ij}\ot\rho^\kappa_{ij}(a)$, where:
$$\rho^\kappa_{ij}(a):=\sum_r(e^\kappa_{ri}\ot 1\ot 1)(1_B\ot a)(e^\kappa_{jr}\ot 1\ot 1)\in L_\kappa(\GG).$$

\begin{proposition}\label{PropScaling}
For all $t\in\R$, $\tau_t^{\Gtilde}(\Linf(\GG))=\Linf(\GG)$, the scaling group of $\GG$ is $\tau_t^\GG=\tau_t^{\Gtilde}\vert_{\Linf(\GG)}$ and, $(\id\ot\tau_t^{\GG})\rho=(\sigma_{t}^\psi\ot\id)\rho\circ\tau_t^G$ and $(\id\ot\sigma_t^{h_\GG})\rho=(\sigma_{t}^\psi\ot\id)\rho\circ\sigma_t^{h_G}$.
\end{proposition}

\begin{proof} As in the proof of Proposition \ref{PropModular}, we note that $\Linf(\GG)$ is generated by $\Linf(H)\ot\Linf(F)$ and $\{\rho^\kappa_{ij}(a)\,:\,1\le\kappa\leq K,\,1\leq i,j\leq N_\kappa,\,a\in \Linf(G)\}$. 
Since $H\times F$ is a dual quantum subgroup of $\GG$ one has, by definition of $\tau_t^{\Gtilde}$, $\tau_t^{\Gtilde}(x)=(\tau_t^H\ot\tau_t^F)(x)=\tau_t^{\GG}(x)$ for all $x\in\Linf(H)\ot\Linf(F)$. Moreover,
\begin{eqnarray*}
\tau_t^{\Gtilde}(\rho^\kappa_{ij}(a))&=&\widetilde{\tau}^\kappa_t(\rho^\kappa_{ij}(a))=\sum_r(\sigma_{-t}^{\psi_\kappa}(e^\kappa_{ri})\ot 1\ot 1)(1_B\ot \tau_t^G(a))(\sigma_{-t}^{\psi_\kappa}(e^\kappa_{jr})\ot 1\ot 1)\\
&=&\sum_r\left(\frac{\lambda_{\kappa,i}}{\lambda_{\kappa,j}}\right)^{it}(e^\kappa_{ri}\ot 1\ot 1)(1_B\ot \tau_t^G(a))(e^\kappa_{jr}\ot 1\ot 1)\\
&=&\left(\frac{\lambda_{\kappa,i}}{\lambda_{\kappa,j}}\right)^{it}\rho^\kappa_{ij}(\tau_t^G(a)).
\end{eqnarray*}
If follows that $\tau_t^{\Gtilde}(\Linf(\GG))=\Linf(\GG)$, for all $t\in\R$. A similar computation gives
$$\sigma_t^{h_\GG}(\rho^\kappa_{ij}(a))=\left(\frac{\lambda_{\kappa,i}}{\lambda_{\kappa,j}}\right)^{it}\rho^\kappa_{ij}(\sigma_t^{h_G}(a)).$$
It follows that $(\id\ot\tau_t^{\Gtilde})\rho=(\sigma_{t}^\psi\ot\id)\rho\circ\tau_t^G$ and $(\id\ot\sigma_t^{h_\GG})\rho=(\sigma_{t}^\psi\ot\id)\rho\circ\sigma_t^{h_G}$.

\vspace{0.2cm}

\noindent\textbf{Claim.} \textit{For all $t\in\R$ and all $X\in \Pol(G)\ot B$ one has $\beta_G\circ(\tau_t^G\ot\sigma_{-t}^\psi)(X)=(\sigma_{-t}^\psi\ot\tau_t^{\Gtilde})\circ\beta_G(X)$.}

\vspace{0.2cm}

\begin{proof}[Proof of the Claim.] We may and will assume that $X=a\ot b$, with $a\in\Pol(G)$ and $b\in B$. Then, using the computation above and Proposition \ref{PropQAut} $(2.c)$ we find:
\begin{align*}
\beta_G\circ(\tau_t^G\ot\sigma_{-t}^\psi)(X)&=\rho(\tau_t^G(a))\beta(\sigma_{-t}^\psi(b))=(\sigma_{-t}^\psi\ot\tau_t^{\Gtilde})(\rho(a))(\sigma_{-t}^\psi\ot\tau_{t}^H)(\beta(b))\\
&=(\sigma^\psi_{-t}\ot\tau_t^{\Gtilde})(\rho(a)\beta(b))=(\sigma^\psi_{-t}\ot\tau_t^{\Gtilde})\circ\beta_G(a\ot b).\qedhere
\end{align*}
\end{proof}

\vspace{0.2cm}

\noindent\textit{End of the proof of Proposition \ref{PropScaling}.} To finish the proof, it suffices to show that, for all $t\in\R$, one has $\Delta_{\GG}\circ\sigma_t^{h_\GG}(x)=(\tau_t^{\Gtilde}\ot\sigma_t^{h_\GG})\circ\Delta_{\GG}(x)$, for all $x\in\Linf(H)\ot\Linf(F)$ and all $x\in L_\kappa(\GG)$, $1\leq\kappa\leq K$. If $x\in\Linf(H)\ot\Linf(F)$ the equality is clear, since $H\times F$ is a dual quantum subgroup of $\GG$. Moreover, using the claim we find, for all $x\in\Pol(G)$,
\begin{eqnarray*}
(\id\ot\Delta_{\GG}\circ\sigma_t^{h_\GG})\rho(x)&=&(\id\ot\Delta_{\GG})(\sigma_{t}^\psi\ot\id)\rho\circ\sigma_t^{h_G}(x)=(\sigma_{t}^\psi\ot\id\ot\id)(\id\ot\Delta_{\GG})\rho(\sigma_t^{h_G}(x))\\
&=&(\sigma_{t}^\psi\ot\id\ot\id)(\beta_G\ot\id)(\id\ot\rho)\Delta_G(\sigma_t^{h_G}(x))\\
&=&(\sigma_{t}^\psi\ot\id\ot\id)(\beta_G\ot\id)(\id\ot\rho)(\tau_t^G\ot\sigma_t^{h_G})\Delta_G(x)\\
&=&(\sigma_{t}^\psi\ot\id\ot\id)(\beta_G\ot\id)(\tau_t^G\ot\sigma_{-t}^\psi\ot\sigma_t^{h_{\GG}})(\id\ot\rho)\Delta_G(x)\\
&=&(\id\ot\tau_t^{\Gtilde}\ot\sigma_t^{h_\GG})(\beta_G\ot\id)(\id\ot\rho)\Delta_G\\
&=&(\id\ot(\id\ot\tau_t^{\Gtilde}\ot\sigma_t^{h_\GG})\Delta_{\GG})\rho(x).
\end{eqnarray*}
It finishes the proof.\end{proof}

\noindent We can now compute the invariant $T^\tau(\GG):=\{t\in\R\,:\,\tau^\GG_t=\id\}$ of $\GG$ introduced in \cite{KS24}.

\begin{lemma}\label{rhofaith}
The map $\rho : \Linf(G)\rightarrow B\ot \Linf(\GG)$ is faithful.
\end{lemma}

\begin{proof}
We first remark that the map restricted to $\Pol(G)$ is faithful by the formula proved in Theorem \ref{TheoremErgodic}, $$ (\psi\ot h_\GG) \rho = h_G, $$ and the fact that all the states are faithful. Moreover, on $\Pol(G)$, it is defined as $$ \rho(a) = \sum_{\kappa,ij} e_{ij}^\kappa \ot \sum_{s=1}^{N_\kappa}\mu_\kappa(e^\kappa_{si}\ot 1\ot 1)\nu_\kappa(1\ot a)\mu_\kappa(e^\kappa_{js}\ot 1\ot 1)\in B\ot \Pol(\GG),$$ therefore, the extension of $\rho$ to $\Linf(G)$ is also normal, and therefore the intertwining of the states still holds, and implies that $\rho$ is still faithful at the von Neumann algebraic level.
\end{proof}

\begin{corollary}
One has $T^\tau(\GG) = T^\tau(G)\cap T^\tau(H)$.
\end{corollary}

\begin{proof}
Since $H$ is a dual quantum subgroup of $\GG$, we get $T^\tau(\GG)\subset T^\tau(H)$ and, using the formulas $\tau^H_t(u^{ij,\kappa}_{kk,\gamma}) = \left(\dfrac{\lambda_{\kappa,i}}{\lambda_{\kappa,j}} \right)^{it}u_{kk,\gamma}^{ij,\kappa}$ and $\sigma_t^\psi(b^\kappa_{ij})=\left(\frac{\lambda_{\kappa,i}}{\lambda_{\kappa,j}}\right)^{it}$ we deduce that $T^\tau(H)\subset B^\psi$, where $B^\psi$ denotes the centralizer of $\psi$. Fix $t\in T^\tau(\GG)$ then $\sigma_t^\psi=\id$ and,
$$\rho =(\id\ot \tau^\GG_t)\rho = (\sigma_t^\psi\ot \id)\rho\circ \tau_t^G = \rho\circ\tau_t^G. $$ From this we get that $\tau_t^G=\id$ by faithfulness of the map $\rho$. Therefore $T^\tau(\GG) \subset T^\tau(G)\cap T^\tau(H)$.
Assume now that $t \in T^\tau(G)\cap T^\tau(H)$, then $t\in B^\psi$ and, using that $(\id\ot \tau^\GG_t)\rho = (\sigma_t^\psi\ot \id)\rho\circ \tau_t^G = \rho$, we get that $\tau_t^\GG$ is trivial on the coefficients of $\rho$. Since $t\in T^\tau(H)$ we deduce that $\tau_t^\GG=\id$.
\end{proof}

\noindent We can now explore various approximation properties of $\GG$. To facilitate this study, the following remark will be instrumental.

\begin{remark}\label{rmk fd amal}
View $\Linf(\Gtilde)=\underset{\Linf(H)\ot\Linf(F)}{*}(B\ot\Linf(G))\underset{B\ot\Linf(F)}{*}(M_{N_\kappa}(\C)\ot\Linf(H)\ot\Linf(F))$, where the inclusion maps and the conditional expectations are the ones described before. Let us denote by $\nu_\kappa\,:\, B\ot\Linf(G)\rightarrow\Linf(\Gtilde)$ and $\mu_\kappa\,:\,M_{N_\kappa}(\C)\ot\Linf(H)\ot\Linf(F)\rightarrow\Linf(\Gtilde)$, $1\leq\kappa\leq K$, the canonical inclusions. We also write $\mu\,:\,\Linf(H)\ot\Linf(F)\rightarrow\Linf(\Gtilde)$, $\mu(x):=\mu_\kappa(1\ot x)$ for all $1\leq\kappa\leq K$ and all $x\in\Linf(H)\ot\Linf(F)$. Let $M\subset\Linf(\Gtilde)$ be the von Neumann subalgebra defined by:
$$M_0:=\left\{\mu_\kappa(M_{N_\kappa}(\C)\ot\Linf(H)\ot\Linf(F))\,:\,1\leq\kappa\leq K\right\}''.$$
Remark that $$M_0 \simeq (\ast_\kappa M_{N_\kappa}(\C),\psi_\kappa^{-1})\ot \Linf(H)\ot \Linf(F).$$
Then we define inductively the algebras $$M_{\kappa} := \left( M_{\kappa-1} \cup \nu_\kappa(B\ot\Linf(G)) \right)''\subset\Linf(\Gtilde),$$
for any $1\leq \kappa \leq K$.
Note that, by uniqueness of amalgamated free product of von Neumann algebras, one has that $$ M_{\kappa} \simeq M_{\kappa-1} \ast_{B\ot \Linf(F)} \nu_\kappa(B\ot\Linf(G)), $$
where the amalgamated free products are with respect to the maps $\beta_\kappa\ot \id_F$ and the obvious corresponding conditional expectations. This also shows that we have
$$\Linf(\Gtilde)\simeq M_K.$$ 
Observe that $\Linf(\GG)\subset\Linf(\Gtilde)$ is the von Neumann subalgebra generated by
$$\{\rho^\kappa_{ij}(a)\,:\,1\leq\kappa\leq K,1\leq i,j\leq N_\kappa,a\in\Linf(G)\}\cup\mu(\Linf(H)\ot\Linf(F)),$$where $\rho^\kappa_{ij}(a):=\sum_{r=1}^{N_\kappa}\mu_\kappa(e^\kappa_{ri}\ot 1\ot 1)\nu_\kappa(1\ot a)\mu_\kappa(e^\kappa_{jr}\ot 1\ot 1)$.
\end{remark}

\begin{theorem}\label{thmAtxt}
The following holds.
\begin{enumerate}
    \item If $\widehat{\GG}$ has the Haagerup property then both $\widehat{G}$ and $\widehat{H}$ have the Haagerup property. If moreover $\Irr(F)$ is finite, then $\widehat{\GG}$ has the Haagerup property if and only if both $\widehat{G}$ and $\widehat{H}$ have the Haagerup property.
    \item Assume that $\GG$ is Kac (so that both $G$ and $H$ are Kac). If $\widehat{\GG}$ is hyperlinear then both $\widehat{G}$ and $\widehat{H}$ are hyperlinear. Moreover, if $\Linf(F)$ is amenable, then $\widehat{\GG}$ is hyperlinear if and only if both $\widehat{G}$ and $\widehat{H}$ are hyperlinear.
\end{enumerate}
\end{theorem}

\begin{proof}
\noindent$(1)$. Assume that $\Linf(\GG)$ has the Haagerup property. Since $H$ is a dual quantum subgroup of $\GG$, $\Linf(H)$ is a subalgebra with expectation, by Lemma \ref{dualqsg}. Therefore $\Linf(H)$ has the Haagerup approximation property. By the decomposition of $\Linf(\GG)$ as a free product, we get for any $\kappa$ a faithful conditional expectation onto the algebra $L_\kappa(\GG)$, which implies that the algebra $L_\kappa(\GG)$ and therefore $M_{N_\kappa}(\C) \ot L_\kappa(\GG)=L_\kappa(\Gtilde)$ have the Haagerup approximation property. Since there is a conditional expectation $L_\kappa(\Gtilde) \rightarrow (B\ot \Linf(G))$, $B\ot \Linf(G)$ has the Haagerup property, and therefore $\widehat{G}$ has it too.
For the converse, assume that $\Irr(F)$ is finite and that $\widehat{G}$ and $\widehat{H}$ have the Haagerup property, then we can apply \cite[Corollary 8.2]{CKSVW21} to the decomposition of $\Linf(\Gtilde)$ as an inductive amalgamated free product over $B\ot \Linf(F)$ as described in Remark \ref{rmk fd amal}. First, we remark that by the result of \cite[Corollary 8.2]{CKSVW21}, the algebra $\ast_\kappa( M_{N_\kappa}(\C),\psi_\kappa^{-1})$ has the Haagerup property and therefore, if $\widehat{H}$ has the Haagerup approximation property, then so does $M_0$. Moreover if $\widehat{G}$ has the Haagerup property, then so does $B\ot \Linf(G)$, and an induction shows that $M_K \simeq \Linf(\Gtilde)$ has the Haagerup property. Finally, using the conditional expectation $ \Linf(\Gtilde) \rightarrow \Linf(\GG)$, we get that $\Linf(\GG)$ has the Haagerup property.

\vspace{0.2cm}

\noindent$(2)$. We already know thanks to Proposition \ref{CorKac} that $\GG$ is Kac if and only $G$ and $H$ are Kac. Under these condition, $h_G, h_H,\psi,h_\GG$ are all traces.  Assume first that $\widehat{\GG}$ is hyperlinear then, as $H$ is a dual quantum subgroup of $\GG$, $\widehat{H}$ is also hyperlinear. Note that $B$ being finite dimensional $(B,\psi)$ is Connes embeddable (CE). Hence, $(B\ot\Linf(\GG),\psi\ot h_\GG)$ also is, by \cite{Pop95}. Since the map $\rho\,:\,\Linf(G)\rightarrow B\ot\Linf(\GG)$ is faithful and satisfies $(\psi\ot h_\GG)\rho=h_G$, $\widehat{G}$ is hyperlinear. Conversely, assume that both $\widehat{G}$ and $\widehat{H}$ are hyperlinear and $\Linf(F)$ is amenable. Then, \cite[Corollary 4.5]{BDJ08} implies that $*_\kappa(M_{N_\kappa}(\C),\psi_\kappa^{-1})$ is CE hence, $M_0$ is CE. Applying \cite[Corollary 4.5]{BDJ08} (and since $\Linf(F)$ is amenable) inductively in Remark \ref{rmk fd amal}, we deduce that $(\Linf(\Gtilde),h_{\Gtilde})$ is CE. Since the inclusion $\Linf(\GG)\subset\Linf(\Gtilde)$ is state preserving, we deduce that $\widehat{\GG}$ is hyperlinear.
\end{proof}

\noindent We now study coamenability. We need the following elementary Lemma.

\begin{lemma}\label{LemZ3}
Let $H$ be a CQG. If $\vert\Irr(H)\vert=3$ then $H\simeq\Z_3$ or $H\simeq S_3$.
\end{lemma}

\begin{proof}
Write $\Irr(H)=\{1,x_1,x_2\}$, where $x_1\neq x_2$ are non-trivial irreducibles. Since $x_1$ is non-trivial we have either $\overline{x}_1=x_1$ or $\overline{x}_1=x_2$. Assume that $\overline{x}_1=x_1$, so that $\overline{x}_2=x_2$. Hence, $\Mor(1,x_1\ot x_2)=0$ and it follows that there exists $n_1,n_2\geq 0$ such that $x_1\ot x_2=n_1x_1+n_2x_2$. Also, $\dim\Mor(1,x_k\ot x_k)=1$, $\dim\Mor(x_1,x_2\ot x_2)=\dim\Mor(x_2,x_1\ot x_2)=n_2$ and $\dim\Mor(x_2,x_1\ot x_1)=\dim\Mor(x_1,x_1\ot x_2)=n_1$. Hence, there exists $s,t\geq 0$ such that $x_1\ot x_1=1+ sx_1+n_1x_2$ and $x_2\ot x_2=1+n_2x_1+tx_2$. It follows that $d_1^2=1+sd_1+n_1d_2$ and $d_2^2=1+n_2d_1+td_2$, where $d_k:=\dim(x_k)$. In particular, ${\rm gcd}(d_1,d_2)=1$. Note that $(n_1,n_2)$ is a solution to the Diophantine equation $d_1d_2=n_1d_1+n_2d_2$ and since ${\rm gcd}(d_1,d_2)=1$ and $n_1,n_2\geq 0$ we deduce that $(n_1,n_2)\in\{(d_2,0),(0,d_1)\}$. Assume that $n_1=d_2$ and $n_2=0$. Then, $d_1^2=1+sd_1+d_2^2$ and $d_2^2=1+td_2$. From the equation $d_2^2-td_2-1=0$ we deduce that $\Delta_2:=t^2+4$ is a perfect square, which happens if and only if $t=0$. We conclude that $d_2=1$. From the equation $d_1^2-sd_1-2=0$ we deduce that $\Delta_1:=s^2+8$ is a perfect square, which happens if and only if $s=1$. Hence, $d_1=2$. It follows that $\dim(C(H))=1+d_1^2+d_2^2=6$. Assuming now that $n_1=0$ and $n_2=d_1$, the same proof, exchanging the role of $x_1$ and $x_2$ also shows that $\dim(C(H))=6$. However, the only compact quantum groups whose associated C*-algebras have dimension $6$ are $\Z_6$, $S_3 \simeq \Z_3\rtimes\Z_2$ and its dual $\widehat{S_3}$. $\Z_6$ and $\widehat{S_3}$ both have $6$ irreducibles, each of dimension $1$, and therefore $H \simeq S_3$. It remains to check the other case, that is $\overline{x}_1=x_2$ and $\overline{x}_2=x_1$. Let $s,t\geq 0$ be such that $x_2\ot x_1=1+tx_1+sx_2$ and write $d=\dim(x_1)=\dim(x_2)$. 
Hence, $d^2=1+(t+s)d$ and, this implies that $1=0$ modulo $d$ so $d=1$ and $x_2\ot x_1=1$. Letting $u\in\mathcal{U}(C(H))$ the group like unitary $u:=x_1\neq 1$ we have $u^*=x_2\neq u$ so $u^3=1$ and we deduce that $H=\Z_3$.
\end{proof}

\noindent We will also use the following  well known Lemma.

\begin{lemma}\label{LemZ2} Let $p\geq 2$ be a integer. The following are equivalent.
\begin{enumerate}
\item $p$ is prime.
\item $\Z_p$ has, up to isomorphism, exactly 2 ergodic actions on a finite dimensional C*-algebra $B$:
\begin{itemize}
    \item $B\simeq\C$ and $\beta$ is the trivial action,
    \item $B\simeq\C^p$ and $\beta$ is isomorphic to the action $\Z_p\curvearrowright C^*(\Z_p)$ by left translation.
    \end{itemize}
    \end{enumerate}
\end{lemma}

\begin{proof}
$(1)\Rightarrow(2)$. View $\beta\,:\,\Z_p\rightarrow{\rm Aut}(B)\subset\Lcal(B)$ and write $\Z_p=\langle\sigma\rangle$ where $\sigma^p=1$. Hence $\beta(\sigma)^p=\id_B$ so ${\rm Sp}(\beta(\sigma))\subseteq\{\xi^k\,:\, 0\leq k\leq p-1\}$, where $\xi=e^{2i\pi/p}$ and we may write $B=\bigoplus_{k=0}^{p-1}B_k$ as vector spaces, where $B_k=\{b\in B:\beta(\sigma)b=\xi^kb\}$. Since $\beta$ is ergodic we have $B_0=\C 1_B$. Note that $B_k^*=B_{-k}$ and $B_kB_l\subset B_{k+l}$, where the sum is modulo $p$. If $B_{1}\neq\{0\}$, take a nonzero $a\in B_{1}$. Then, $a^*a\in B_0=\C1$ and is non-zero. Let $\lambda\neq 0$ be such that $a^*a=\lambda1$. Take $\lambda^{-1/2}$ be any square root of $\lambda^{-1}$ and define $u:=\lambda^{-1/2}a$. Then $u$ satisfies $u^*u=1_B$ and, since $B$ is finite dimensional, we also have $uu^*=1$. Note that, for any $1\leq k\leq p-1$, any $b\in B_{k}$, one has $b(u^{k})^*\in B_kB_{-k}\subset B_0=\C1_B$ which implies that $B_{k}=\C u^{k}$. We also have $u^p\in (B_1)^p\subset B_0=\C1$. Hence, up to replacing $u$ by a modulus one multiple of $u$, we may and will assume that $u^p=1$. Hence, the unique map $C^*(\Z_p)\rightarrow B=\bigoplus_{k=0}^{p-1}\C u^k$ which map $\sigma$ onto $u$ is an isomorphism intertwining the left translation on $C^*(\Z_p)$ with the action $\beta$. Assume now that $B_1=\{0\}$. Fix $2\leq k\leq p-1$ and assume that $B_k\neq 0$. Let $a\in B_k$ be non-zero and, as before, we may further assume by ergodicity that $a^*a=aa^*=1$. Then, by Fermat's Little Theorem, $a^{k^{p-1}-1}\in B_k^{k^{p-1}-1}\subset B_{k^{p-1}}=B_1=\{0\}$. Hence, $a^{p^{k-1}-1}=0$ and $a^*=(a^*)^{p^{k-1}}a^{p^{k-1}-1}=0$ so $a=0$ a contradiction. We deduce that $B_k=0$ for all $1\leq k\leq p-1$. So $B=\C1$ and $\beta$ is the trivial action.

\vspace{0.2cm}

\noindent$(2)\Rightarrow(1)$. if $p>2$ is not prime write $p=nm$ with $n,m\geq 2$ and consider the ergodic action $\Z_n\curvearrowright \C^n=C^*(\Z_n)$ by left translation. Then, composing with the canonical surjection $\Z_p\rightarrow\Z_n$ sending the generator of $\Z_p$ on the generator of $\Z_n$, we get an ergodic action $\Z_p\curvearrowright\C^n$. However, this action is not faithful so it can not be isomorphic to the left translation of $\Z_p$, which is faithful.
\end{proof}

\noindent In the sequel we write $S_3=\langle\sigma,\tau\vert\sigma^3=1,\tau^2=1,\tau\sigma\tau^{-1}=\sigma^{-1}\rangle$. We view $\Z_3=\langle\sigma\rangle<S_3$ so $C(\Z_3)=C^*(\Z_3)\subset C^*(S_3)$ and $\Z_2=\langle\tau\rangle<S_3$ so $C(\Z_2)=C^*(\Z_2)\subset C^*(S_3)$.

\begin{lemma}\label{LemS3}
There are, up to isomorphism, exactly $4$ ergodic actions of $\widehat{S_3}$ on a finite dimensional C*-algebra $B$:
\begin{itemize}
    \item The trivial action $\beta_1$ on $B=\C$;
    \item The ergodic action $\beta_2\,:\,\C^2\rightarrow \C^2\ot C(\Z_2)\subset \C^2\ot C^*(S_3)$ on $B=\C^2$ coming from the unique ergodic action of $\Z_2$ on $\C^2$.
    \item The ergodic action $\beta_3\,:\,\C^3\rightarrow \C^3\ot C(\Z_3)\subset \C^3\ot C^*(S_3)$ on $B=\C^3$ coming from the unique ergodic action of $\Z_3$ on $\C^3$.
    \item The ergodic action $\beta_4:=\Delta_{\widehat{S_3}}\,:\,C^*(S_3)\rightarrow C^*(S_3)\ot C^*(S_3)$ on $B=C^*(S_3)$ given by the comultiplication.
\end{itemize}
\end{lemma}

\begin{proof}
Write $B=B_1\oplus B_\sigma\oplus B_{\sigma^2}\oplus B_\tau\oplus B_{\tau\sigma}\oplus B_{\sigma\tau}$, where $B_g:=\{b\in B\,:\,\beta(b)=b\ot g\}$. The order $2$ elements are $\{\tau,\tau\sigma,\sigma\tau\}$. Note that $A_\sigma:=B_1\oplus B_\sigma\oplus B_{\sigma^2}$ is a subalgebra of $B$ and $\beta(A_\sigma)\subset A_\sigma\ot C^*(\sigma)$ so we get an ergodic action $\Z_3\curvearrowright A_\sigma$. Similarly, for any order two element $\omega\in\{\sigma,\tau\sigma,\sigma\tau\}$, $A_\omega:=B_1\oplus B_\omega$ is a subalgebra with an ergodic action $\Z_2\curvearrowright A_\omega$.

\vspace{0.2cm}

\noindent Assume first that $\Z_3\curvearrowright A_\sigma$ is ergodic then there is an element $u\in B_\sigma$ such that $u^3=1$ and $B_\sigma=\C u$, $B_{\sigma^2}=\C u^2$. Now, if for some $\omega\in\{\tau,\sigma\tau,\tau\sigma\}$ the action $\Z_2\curvearrowright A_\omega$ is also ergodic then we have $B_\tau=\C v$ with some $v$ satisfying $v^2=1$. Since $B_gB_h\subset B_{gh}$ and $\langle\sigma,\omega\rangle=S_3$ we deduce that $B_g\neq\{0\}$ for all $g\in S_3$. It follows that $\dim(B)=6$ and $\beta$ is isomorphic to $\beta_4$. If now $\Z_2\curvearrowright A_\omega$ is trivial for all $\omega\in \{\tau,\tau\sigma,\sigma\tau\}$. It follows that $\dim(B)=3$ and $\beta$ is isomorphic to $\beta_3$.

\vspace{0.2cm}

\noindent Assume now that $\Z_3\curvearrowright A_\sigma$ is the trivial action. Since for any $\omega_1,\omega_2\in\{\tau,\tau\sigma,\sigma\tau\}$ with $\omega_1\neq\omega_2$ we have $\langle\omega_1,\omega_2\rangle=S_3$ and since $B_\sigma=\{0\}$, there at most one $\omega\in\{\tau,\sigma\tau,\tau\sigma\}$ such that $\Z^3\curvearrowright B_\omega$ is nontrivial. If there is such a $\omega$ then $\beta$ is isomorphic to $\beta_2$ and if there is no such $\omega$ then $\beta$ is the trivial action $\beta_1$ on $B=\C$.
\end{proof}

\noindent Note that if $G$ is trivial then $\GG=H$ and if $H$ is trivial then, since the action is ergodic, we must have $B=\C$ so $\GG=G$. In the non-trivial cases we obtain the following result concerning coamenability. Recall that $H\curvearrowright B$ is always assumed to be ergodic with unique invariant state $\psi$. 

\begin{theorem}\label{PropCoamenable}
Assume that $F$ is trivial and both $G$ and $H$ are non-trivial then the following are equivalent.
\begin{enumerate}
    \item $\GG$ is coamenable.
    \item $\Linf(\GG)$ is amenable.
    \item $G\simeq H\simeq \Z_2$ and $B\in\{\C,\C^2\}$.
\end{enumerate}
Moreover, in this case we have either $\GG=\widehat{\Z_2^{*2}}$ or $\GG=\widehat{\Z_2^{*2}}\rtimes\Z_2$.
\end{theorem}

\begin{proof}
$(1)\Rightarrow(2)$ follows from the general theory \cite{To06}.

\vspace{0.2cm}

\noindent $(2)\Rightarrow(3)$. Since $\Linf(\GG)$ is amenable,  each  $L_\kappa(\Gtilde)=M_{N_\kappa}(\C)\ot L_\kappa(\GG)$ is amenable. Write $C=B\ot\Linf(G)$ and $D=M_{N_\kappa}(\C)\ot\Linf(H)$ so that $L_\kappa(\Gtilde)=C\underset{B}{*}D$. Let $p:=e^\gamma_{jj}\in B$ be a minimal projection and note that $\sigma_t^\psi(p)=p$ for all $t\in\R$. Let us show that $pCp=\Linf(G)$ and $pDp=\beta_\kappa(p)(M_{N_\kappa}(\C)\ot\Linf(H))\beta_\kappa(p)$ are $*$-free in $pL_\kappa(\Gtilde)p$ with respect to the state $\omega:=(\varphi_\kappa(p))^{-1}\varphi_\kappa(p\cdot p)$. Let $x\in pL_\kappa(\Gtilde) p$ be a word alternating from $pCp$ and $pDp$ of the form $x=(pc_1p)(pd_1p)\dots (pc_np)(pd_np)$, with $c_k\in C$, $d_k\in D$ and $\omega(pc_kp)=\omega(pd_kp)=0$. Since $p\in B$ is a minimal projection, we have for $e_k\in\{c_k,d_k\}$, $E_B(pe_k p)=pE_B(e_k)p\in\C1_B$. Hence, $E_B(pe_k p)=\psi(E_B(pe_kp))1_B=\varphi_\kappa(pe_kp)1_B=0$. It follows that $x$ is a reduced operator in the amalgamated free product $L_\kappa(\Gtilde)=C\underset{B}{*}D$. Hence, $\varphi_\kappa(x)=0=\psi(p)\omega(x)$ so $\omega(x)=0$. One proves similarly that $\omega(x)=0$ for other alternating products $x$. This shows that $pCp*pDp\simeq\langle pCp,pDp\rangle\subset p L_\kappa(\Gtilde) p $, where the free product is with respect to the state $\omega$. Note that, since $p\in B^\psi\subset (L_\kappa(\Gtilde))^{\varphi_\kappa}$, one has $\sigma_t^\omega(pxp)=p\sigma_t^{\varphi_\kappa}(x)p$, for all $t\in\R$, $x\in L_\kappa(\Gtilde)$. In particular, $\sigma^\omega_t(pCp)=pCp$ and $\sigma^\omega_t(pDp)=pDp$. Hence, $\sigma_t^\omega(\langle pCp,pDp\rangle)=\langle pCp,pDp\rangle$ and, by Takesaki's Theorem, the subalgebra $\langle pCp,pDp\rangle\subset p L_\kappa(\Gtilde)p$ has a faithful normal conditional expectation. Hence, $pCp*pDp=\Linf(G)*\beta_\kappa(p)(M_{N_\kappa}(\C)\ot\Linf(H))\beta_\kappa(p)$ is amenable. Note that if $\Linf(H)$ is infinite dimensional, then $\Linf(H)$ is diffuse by \cite[Lemma 2.4]{FT24}. Hence, $M_{N_\kappa}(\C)\ot\Linf(H)$ is also diffuse and $\beta_\kappa(p)(M_{N_\kappa}(\C)\ot\Linf(H))\beta_\kappa(p)$ must be infinite dimensional. Since $\Linf(G)\neq\C$, the results of \cite{Ue10} imply that $pCp*pDp$ is not amenable. Hence, $\Linf(H)$ is finite dimensional. Let $q_x$ be the minimal central projections of $\Linf(H)$, which are parameterized by $x\in\Irr(\widehat{H})$ and note that $\varepsilon_H(a)q_1=aq_1$ for all $a\in C(H)$, where $1\in\Irr(\widehat{H})$ is the trivial representation. Note that $$\dim(\beta_\kappa(p)(M_{N_\kappa}(\C)\ot\Linf(H))\beta_\kappa(p))\geq 2.$$
Indeed, if the dimension is $1$, then $\beta_\kappa(p)$ is a minimal projection and it implies that $\beta_\kappa(p)(1_\kappa\ot p_1)=\beta_\kappa(p)$. Hence,
$$\beta_\kappa(p)=\beta_\kappa(p)(1_\kappa\ot p_1)=(\id\ot\varepsilon_H)(\beta_\kappa(p)(1_\kappa\ot p_1))=\chi_\kappa(p)\ot p_1=\delta_{\gamma,\kappa}(p\ot p_1).$$
Since $\beta_\kappa$ is faithful, we deduce that $K=1$ and $\beta(p)=p\ot p_1$. Applying $\psi\ot\id$ we get $\psi(p)1=\psi(p)p_1$ so $p_1=1$ and $H$ is trivial, a contradiction.

\vspace{0.2cm}

\noindent Since $pCp,pDp\neq\C$ and $pCp*pDp$ is amenable, the results of \cite{Ue10} imply that:
$$\dim(\Linf(G))\leq 2\text{ and }\dim(\beta_\kappa(p)(M_{N_\kappa}(\C)\ot\Linf(H))\beta_\kappa(p))\leq 2.$$
We deduce that $\dim(\Linf(G))= 2=\dim(\beta_\kappa(p)(M_{N_\kappa}(\C)\ot\Linf(H))\beta_\kappa(p))$ which implies that $G\simeq\Z_2$. Hence both $G$ and $H$ are Kac, by \cite[Theorem 2.2]{VD97}. Hence, $\GG$ is Kac and, since $\Linf(\GG)$ is amenable and $\GG$ is Kac, \cite{To06} implies that $\GG$ is amenable so $G^{*K}$, as a compact quantum subgroup of $\GG$, is also amenable. \cite[Proposition 2.8]{FT24} implies that $K\leq 2$. Write $$\beta(p)=\sum_{x\in\Irr(\widehat{H})}\beta(p)(1\ot q_x).$$
Write, for $1\leq\kappa\leq K$, $I_\kappa:=\{x\in\Irr(\widehat{H}):\beta_\kappa(p)(1\ot q_x)\neq 0\}$ and note that, since $(q_x)$ are central and pairwise orthogonal, $\{\beta_\kappa(p)(1\ot q_x):x\in I_\kappa\}$ is a family of non-zero pairwise orthogonal projections in $\beta_\kappa(p)(M_{N_\kappa}(\C)\ot \Linf(H))\beta_\kappa(p)$. Hence,
$$\vert I_\kappa\vert\leq\dim(\beta_\kappa(p)(M_{N_\kappa}(\C)\ot \Linf(H))\beta_\kappa(p))= 2\text{ for all }1\leq\kappa\leq K.$$
Note that $\Irr(\widehat{H})= \cup_{\kappa=1}^K I_\kappa$. Indeed, writing $I:=\{x\in\Irr(\widehat{H}):\beta(p)(1\ot p_x)\neq 0\}$ one has:
 $$\psi(p)1=(\psi\ot\id)(\beta(p))=\sum_{x\in\Irr(\widehat{H})}(\psi\ot\id)(\beta(p)(1\ot q_x))=\sum_{x\in I}\psi(p)q_x.$$
 Since $\psi(p)>0$ it follows that $\sum_{x\in I}p_x=1$ so $I=\Irr(\widehat{H})$. Let $x\in\Irr(\widehat{H})$, since $\beta(p)(1\ot q_x)\neq 0$, there exists $\kappa$ such that $x\in I_\kappa$ and this shows that $\Irr(\widehat{H})= \cup_{\kappa=1}^K I_\kappa$.
 
\vspace{0.2cm}
 
\noindent If $K=1$ this shows that $\vert\Irr(\widehat{H})\vert=2$. Arguing as in the proof of \cite[Proposition 2.8]{FT24}, we deduce that $\widehat{H}\simeq \Z_2$. Hence, $H\simeq\widehat{\Z_2}\simeq\Z_2$.
 
\vspace{0.2cm}
 
\noindent If $K=2$. Since $(\id\ot\varepsilon_H)\beta=\id_B$, the trivial representation $1$ is in $I_1\cap I_2$. Hence, $\vert \Irr(\widehat{H})\vert=\vert I_1\vert+\vert I_2\vert-\vert I_1\cap I_2\vert\leq \vert I_1\vert+\vert I_2\vert -1\leq 3$ so either $\vert \Irr(\widehat{H})\vert=2$ in which case we deduce again that $H=\Z_2$ or $\vert \Irr(\widehat{H})\vert=3$ and Lemma \ref{LemZ3} implies that $\widehat{H}\in\{\Z_3,S_3\}$ so $H\in\{\Z_3,\widehat{S}_3\}$. In conclusion, we have $G=\Z_2$ and either $K=1$ and $H=\Z_2$ or $K=2$ and $H\in\{\Z_2,\Z_3,\widehat{S_3}\}$. However, in the case $K=2$ and $H=\Z_3$, there is no ergodic action of $\Z_3$ on $B$, by Lemma \ref{LemZ2}. If $K=2$ and $H=\widehat{S_3}$ then, by Lemma \ref{LemS3} we have $B=\C^2$ and $\beta$ is the action $\C^2\rightarrow\C^2\ot C(\Z_2)=\C^2\ot C^*(\omega)\subset\C^2\ot C^*(S_3)$ coming from the unique ergodic action of $\Z_2$ on $\C^2$ and an order two element $\omega\in S_3$. We deduce from Proposition \ref{Propfaithful} that $\GG\simeq (\Z_2\wr_{*}\Z_2)\underset{\Z_2}{*}\widehat{S_3}$. By Proposition \ref{PropBichon} and \cite[Proposition 2.23]{FT24} we have $\Z_2\wr_{*}\Z_2\simeq \widehat{\Z_2^{*2}}\rtimes \Z_2$. Hence, $\Linf(\GG)=({\rm L}(\Z_2^{*2})\ot{\rm L}(\Z_2))\underset{{\rm L}(\Z_2)}{*}{\rm L}(S_3)$, where the amalgamated free product is with respect to the conditional expectations on the dual quantum subgroup $\Z_2$. Write $S_3=\langle\sigma,\tau\vert\sigma^3=1,\tau^2=1,\tau\sigma\tau=\sigma^{-1}\rangle$ and let $p:=\frac{1}{2}(1+\tau)\in{\rm L}(\Z_2)$ and note that $p=p^*=p^2$ and $p\tau=p$ so $p$ is a minimal projection in ${\rm L}(\Z_2)$. However, $p$ is not minimal in ${\rm L}(S_3)$ since $p{\rm L}(S_3)p=\C p\oplus\C p\sigma p$. It is easy to check that ${\rm L}(\Z_2^{*2})=p({\rm L}(\Z_2^{*2})\ot{\rm L}(\Z_2))p$ and $p{\rm L}(S_3)p$ are $*$-free in $p\Linf(\GG)p$ with respect to $h_\GG(p)^{-1}h_\GG\vert_{p\Linf(\GG)p}$. Hence, ${\rm L}(\Z_2^{*2})*p{\rm L}(S_3)p$ is isomorphic to a von Neumann algebra of $p\Linf(\GG)p$ which is impossible since ${\rm L}(\Z_2^{*2})$ is diffuse and $\dim(p{\rm L}(S_3)p)\geq 2$ so ${\rm L}(\Z_2^{*2})*p{\rm L}(S_3)p$ is not amenable. It follows that the only possible case is $G=H=\Z_2$ and Lemma \ref{LemZ2} implies that $B\in\{\C,\C^2\}$.
 
\vspace{0.2cm}
 
\noindent $(3)\Rightarrow(1)$. If $G=H=\Z_2$ and $B=\C$ then $\beta$ is the trivial action on $\C$ and it follows from Remark  \ref{RmkFwp} that $\GG\simeq\widehat{\Z_2^{*2}}$ which is coamenable. If $B=\C^2$ then Lemma \ref{LemZ2} implies that $\beta$ is the unique ergodic action on $B=\C^2$ so that it is isomorphic to the universal action $\Z_2=S_2^+\curvearrowright\C^2$. Proposition \ref{PropBichon} and \cite[Proposition 2.23]{FT24} imply that $\GG\simeq\widehat{\Z_2^{*2}}\rtimes \Z_2$, which is coamenable by \cite[Proposition 2.17]{FT24}.
\end{proof}

\begin{remark} The fact that, for $G$ a CQG of Kac type, the coamenability of $G$ is equivalent to the amenability of the von Neumann algebra $\Linf(G)$ is well known (see Corollary 3.17 in \cite{To06} and the discussion after its proof) and was used in the previous proof. However, the equivalence for general CQG is open. Theorem \ref{PropCoamenable} shows that in the class of CQG which are non-amalgamated generalized free wreath product of non-trivial CQG with ergodic action, the equivalence between coamenability and amenability of the von Neumann algebra is true. Since this class of CQG contains in particular the class of free products of non-trivial CQG, it extends \cite[Remark 2.9]{FT24}.\end{remark}

\begin{remark}
The ergodicity is necessary to obtain $(1)\Rightarrow (2)$ in Theorem \ref{PropCoamenable}. Indeed, consider the free action $\Z_2\curvearrowright\Z_2$ and let $B=M_N(\C)$ with $N\geq 2$ and $B^{\Z_2}=M_N(\C)\oplus M_N(\C)$ with the action $\beta\,:\,\Z_2\curvearrowright B^{\Z_2}$ defined before Proposition \ref{PropSemiDirect}. This action is non-ergodic and preserves the faithful state $\psi$ defined from any faithful state $\omega\in M_N(\C)^*$ before Proposition \ref{PropSemiDirect}. Let $\GG$ be the free wreath product $\GG:=G\wr_{*,\beta}\Z_2$. Note that, by Proposition \ref{PropSemiDirect}, $\GG\simeq G^{*2}\rtimes \Z_2$. Hence, $\GG$ is coamenable if and only if $G$ is trivial or $G=\Z_2$, by \cite[Propositions 2.8 and 2.17]{FT24}. Hence, taking $G=H=\Z_2$, we obtain a coamenable free wreath product $\GG$ but with $B=M_N(\C)\oplus M_N(\C)$.
\end{remark}

\section{Qualitative properties of the von Neumann algebra}\label{section quali}
\noindent We will now study the main properties of the von Neumann algebra of a generalized free wreath product: factoriality, type, Connes' invariant $T$, primeness and absence of Cartan subalgebra. We start with a preliminary section concerning intertwining theory.

\subsection{Intertwining theory}
We use \cite{HI17} as a reference concerning basics on intertwining theory. Let $1_A\in M$ be a projection. A von Neumann subalgebra $A\subset 1_AM1_A$ is called \textit{with expectation} if there is a faithful normal conditional expectation $1_AM1_A\rightarrow A$.

\begin{remark}\label{RmkExp}
If $A\subset M$ is with expectation then its relative commutant $A'\cap M$, the von Neumann algebra $A\vee (A'\cap M)\subset M$ generated by $A$ and $A'\cap M$ and its normalizer algebra $\mathcal{N}_M(A)''\subset M$, where $\mathcal{N}_M(A):=\{v\in\mathcal{U}(M)\,:\,vAv^*=A\}$, are all with expectation. Indeed, fix a faithful normal conditional expectation $E_A\,:\, M\rightarrow A$, fix a f.n.s. $\psi\in A_*$ and define the f.n.s. $\psi\in M_*$ by $\psi:=\psi\circ E_A$. Since $A$ is globally invariant under the modular group $\sigma_t^\psi$, it is easy to see that $A'\cap M$ also is. It implies that $A\vee (A'\cap M)$ is globally invariant under $\sigma_t^\psi$. Moreover, it is shown in \cite[Lemma 4.1]{BHV18} that, for all $u\in\mathcal{N}_M(A)$ and all $t\in\R$, one has $$u^*\sigma_t(u)\in A\vee (A'\cap M)\subset\mathcal{N}_M(A)'',$$
In particular, $\mathcal{N}_M(A)''$ is globally invariant under $\sigma_t^\psi$. It follows from Takesaki's Theorem that there exists a unique $\psi$-preserving conditional expectation from $M$ to $A'\cap M$, from $M$ to $A\vee (A'\cap M)$ and from $M$ to $\mathcal{N}_M(A)''$, these conditional expectations are moreover normal and faithful.\end{remark}

\begin{definition}\label{DefEmbedds}
Let $A\subset 1_AM1_A$ and $B\subset 1_BM1_B$ be von Neumann subalgebras with expectation. We say that $A$ \textit{embeds in} $B$ \textit{inside} $M$ and we write $A\prec_M B$ if there exist projections $e\in A$, $f\in B$ a non-zero element $v\in eMf$ and a unital, normal $*$-morphism $\theta : eAe \rightarrow fBf$ such that the inclusion $\theta(eAe)\subset fBf$ is with expectation and $av = v\theta(a)$ for all $a\in eAe$.
\end{definition}

\begin{remark}\label{RmkEmbedds} Let $A\subset 1_AM1_A$ and $B\subset 1_BM1_B$ be von Neumann subalgebras with expectations.
\begin{enumerate}
\item If $A$ is diffuse and $B$ is finite dimensional then $A\nprec_M B$. Indeed, assume that $A\prec_M B$, it follows from \cite[Remark 4.2 (3)]{HI17} that there exist projections $e\in A$, $f\in B$, a non-zero element $v\in eMf$ and a \textit{faithful} unital normal $*$-homomorphism $\theta\,:\,eAe\rightarrow fBf$ such that $av=v\theta(a)$ for all $a\in eAe$. If $B=\C1_B$ then, since $\theta$ is faithful, $eAe=\C e$ and $e$ is minimal. In general, since $fBf$ is finite dimensional and $v\in eMf$, there exists a minimal (and central) projection $q\in fBf$ such that $vq\neq 0$. Consider $w=vq\in eMq$ and $\psi\,:\, eAe\rightarrow qBq=\C q$, $\psi(a)=\theta(a)q=q\theta(a)$. Then, $w\psi(a)=vq\theta(a)=v\theta(a)q=avq=aw$. Hence, $A\prec_M\C$ and we can apply the first part of the proof.
\item Let $p\in\mathcal{Z}(M)$ be a projection. Then, $Ap\subset 1_AMp1_A$ and $Bp\subset 1_BMp1_B$ are with expectation and if $Ap\prec_{Mp}Bp$ then $A\prec_MB$ (see \cite[Remark 4.2 (2)]{HI17}). 

\item $A\prec_M B$ if and only if $A\ot M_k(\C)\prec_{M\ot M_k(\C)} B\ot M_k(\C)$ for some $k\geq 1$ (or for any $k\geq 1$). Indeed, assume that for some $k\geq 1$ one has $\mathcal{A}\prec_\mathcal{M} \mathcal{B}$, where $\mathcal{A}:=A\ot M_k(\C)$, $\mathcal{B}:=B\ot M_k(\C)$ and $\mathcal{M}:=M\ot M_k(\C)$ and consider the projection $p:=1_A\ot e_{11}\in \mathcal{A}$. Note that its central support in $\mathcal{A}$ is $z_{\mathcal{A}}(1_A\ot e_{11})=1_A\ot 1$. It follows from \cite[Remark 4.2 (4)]{HI17} that $p\mathcal{A}p=A\ot\C e_{11}\prec_{\mathcal{M}}\mathcal{B}$. The same argument, by using this time \cite[Remark 4.5]{HI17}, shows that $A\ot\C e_{11}\prec_{\mathcal{M}} B\ot\C e_{11}$. Hence, there exist projections $e\in A$, $f\in B$, a non-zero $w\in (e\ot e_{11})(M\ot M_k(\C))(f\ot e_{11})=eMf\ot \C e_{11}$ and a unital normal $*$-homomorphism $\psi\,:\,eAe\ot\C e_{11}\rightarrow fBf\ot\C e_{11}$ such that $\psi(eAe\ot\C e_{11})\subset fBf\ot\C e_{11}$ is with expectation and $(a\ot e_{11})w=w\psi(a\ot e_{11})$ for all $a\in eAe$. Let $\theta\,:\,eAe\rightarrow fBf$ be the unique unital normal $*$-homomorphism such that $\psi(a\ot e_{11})=\theta(a)\ot e_{11}$ for all $a\in eAe$ and $v\in eMf$ be the unique (non-zero) element such that $w=v\ot e_{11}$. Then $\theta(eAe)\subset fBf$ is with expectation and $av=v\theta(a)$ for all $a\in eAe$. For the converse, it is clear from Definition \ref{DefEmbedds} that $A\prec_MB$ implies $A\ot M_k(\C)\prec_{M\ot M_k(\C)} B\ot M_k(\C)$ for all $k\geq 1$.
\item If $A\prec_MB$ then, for any \textit{unital} $D\subset A$ with expectation one has $D\prec_M B$ by \cite[Lemma 4.8]{HI17}.
\end{enumerate}
\end{remark}

\noindent We will freely use the following result.

\begin{proposition}\cite[Theorem 4.3 and Lemma 4.10]{HI17}\label{HI17thm}
Let $A\subset 1_AM1_A$ and $B\subset 1_BM1_B$ with expectations and suppose that $A$ is finite or type ${\rm III}$. Then the following are equivalent:
\begin{enumerate}
    \item $A\prec_M B$.
    \item There exist $n\geq 1$, a projection $q\in M_n(\C)\ot B$, a non-zero partial isometry $w\in (e_{11}\ot 1_A)(M_n(\C)\ot M)q$ and a unital normal $*$-homomorphism $\pi\,:\, A\rightarrow q(M_n(\C)\ot B)q$ such that $\pi(A)\subset q(M_n(\C)\ot B)q$ is with expectation and $(1\ot a)w=w\pi(a)$ for all $a\in A$.
    \end{enumerate}
If $A$ is finite $(1)$ and $(2)$ are also equivalent to:
\begin{enumerate}\setcounter{enumi}{2}
    \item There is no net of unitaries $(w_n)_n$ in $\mathcal{U}(A)$ such that $E_B(x^*w_ny)\rightarrow 0$ in the $\sigma$-strong topology, for all $x,y\in 1_AM1_B$.
\end{enumerate}
Moreover, for any $A$ (not necessarily finite or type ${\rm III}$), we have $(2)\Rightarrow (3)$.
\end{proposition}

\begin{remark}\label{RmkNprec}
Let $A\subset 1_AM1_A$ and $B\subset 1_BM1_B$ with expectations and $\varphi$ a f.n.s. on $A$. If there exists a net of unitaries $(w_n)_n$ in $\mathcal{U}(A^\varphi)$ such that $E_B(x^*w_ny)\rightarrow 0$ in the $\sigma$-strong topology, for all $x,y\in 1_AM1_B$ then, $A\nprec_M B$. Indeed, by Proposition \ref{HI17thm} applied to the finite von Neumann algebra $A^\varphi$, we deduce that $A^\varphi\nprec_M B$. Note that $A^\varphi\subset A$ is unital and with expectation (by Takesaki's Theorem). Hence, Remark \ref{RmkEmbedds} $(4)$ implies that $A\nprec_M B$.
\end{remark}

\begin{lemma}\label{diffusefd}
Let $M$ be a von Neumann algebra and $(u_n)_n$ a sequence in $M$ going to $0$ $\sigma$-weakly. Then, for any finite dimensional subalgebra $N\subset M$ and any normal conditional expectation $E\,:\, M\rightarrow N$, $ E(xu_n y) \rightarrow 0$ in the operator norm topology, for any $x,y\in M$.
\end{lemma}

\begin{proof}
 Let $\varphi\in N_*$ be a faithful state and fix an orthonormal basis $(e_i)_i$ of $N$ for the $\varphi$ scalar product. Then, $E(a) = \sum_{i} \varphi(e_i^* E(a))e_i=\sum_{i} \varphi\circ E(e_i^* a)e_i$. For all $x,y\in M$,
\begin{eqnarray*}
 \Vert E(xu_ny)\Vert &=& \left\Vert \sum_{i} \varphi\circ E(e_i^* xu_ny)e_i\right\Vert
 \leq  \sum_{i} \vert  \varphi\circ E(e_i^* xu_ny)\vert\cdot\Vert e_i\Vert,
\end{eqnarray*}
and this goes to $0$, because $u_n$ goes to $0$ weakly and the sum is finite. The same holds for the conjugate, because $a\mapsto a^*$ is weakly continuous. It concludes the proof.\end{proof}

\begin{remark}\label{RmkDiffuse}
The previous Lemma applies in particular when $M$ is diffuse, $\varphi$ is almost periodic. In that case $M^\varphi$ is diffuse and the Lemma shows that there exists a sequence of unitaries $(u_n)_n$ in $M^\varphi$ such that, for any finite dimensional subalgebra $N$ of $M$ and any normal conditional expectation $E\,:\, M\rightarrow N$, $E(xu_ny)\rightarrow 0$ in norm.
\end{remark}

\noindent The following Lemma is \cite[Lemma 2.6]{HU16b} and \cite[Lemma 2.7]{HU16a}.

\begin{lemma}{\cite[Lemma 2.6]{HU16b} \cite[Lemma 2.7]{HU16a}}\label{LemmaNets}
Let $M=(M_1,E_1)\ast_N (M_2,E_2)$ be an amalgamated free product and $\psi\in M_*$ a f.n.s. such that $\psi\circ E_{M_k}=\psi$ for $k=1,2$. If $(w_n)_n$ is a net in the unit ball of $(M_1)^\psi$ such that $E_1(xw_ny)\rightarrow 0$ $\sigma$-strongly for all $x,y\in M_1$ then,
\begin{enumerate}
    \item $E_{M_1}(xw_ny)\rightarrow 0$ $\sigma$-strongly for all $x,y\in M\ominus M_1$.
    \item $E_{M_2}(xw_ny)\rightarrow 0$ $\sigma$-strongly for all $x,y\in M$.
\end{enumerate}
\end{lemma}

\noindent The following Lemma is \cite[Proposition 2.7]{HU16b} but we need it in the amalgamated case, with a finite dimensional amalgam. A more general statement could be given but we choose to include all the hypotheses we will have later so that we can write a shorter proof. Recall that, by Remark \ref{RmkExp}, if $Q\subset M$ is with expectation then, $\mathcal{N}_M(Q)''\subset M$ also is.

\noindent\begin{lemma}\label{LemmaNormalizer}
Let $M=M_1\ast_N M_2$ with $N$ is finite dimensional. Let $Q\subset M$ be any finite or type ${\rm III}$ diffuse von Neumann subalgebra with expectation. If $Q\prec_M M_1$ then $\mathcal{N}_M(Q)''\prec_M M_1$.
\end{lemma}

\begin{proof}
Let $n\geq 1$, a projection $q\in M_n(\C)\ot M_1$, a non-zero partial isometry $w\in (M_{1,n}(\C)\ot M)q$ and a unital normal $*$-homomorphism $\pi\,:\, Q\rightarrow q(M_n(\C)\ot M_1)q$ such that $\pi(Q)\subset q(M_n(\C)\ot M_1)q$ is with expectation and $xw=w\pi(x)$ for all $x\in Q$.

\vspace{0.2cm}

\noindent\textbf{Claim:} \textit{One has $w^*w\in qM_n( M_1)q$ and $w^*\mathcal{N}_M(Q)''w\subset w^*wM_n(M_1)w^*w$.}

\vspace{0.2cm}

\noindent\textit{Proof of the Claim.} We work in the amalgamated free product:
$$M_n(\C)\ot M=M_n(\C)\ot M_1\underset{M_n(\C)\ot N}{*}M_n(\C)\ot M_2,$$
with respect to the conditional expectation $\widetilde{E}_k:=\id_{M_{n}(\C)}\ot E_k$, $k=1,2$.
Since $\pi(Q)$ is diffuse we may apply \cite[Lemma 2.1]{HU16b} to get a f.n.s. $\varphi$ on $\pi(Q)$ with diffuse centralizer $\pi(Q)^\varphi$ and note that $q=\pi(1)\in\pi(Q)^\varphi$. Define the normal state $\psi_1:=\varphi\circ E(q\,\cdot\,q)\in (M_n(\C)\ot M_1)_*$, where $E\,:\, q (M_n(\C)\ot M_1)q\rightarrow Q$ is any faithful normal conditional expectation. Fix any f.n.s. $\omega\in [(1-q)M(1-q)]_*$ and define the f.n.s. $\psi_2:=\frac{1}{2}\psi_1(\cdot)+\frac{1}{2}\omega((1-q)\cdot(1-q))\in (M_n(\C)\ot M_1)_*$. Then, the f.n.s. $\psi:=\psi_2\circ(\id\ot E_{M_1})\in (M_n(\C)\ot M)_*$ satisfies the following: $\psi=\psi\circ(\id\ot E_{M_1})$, $q\in (M_n(\C)\ot M_1)^\psi$ and $\pi(Q)\cap q(M_n(\C)\ot M_1)^\psi q$ is diffuse. By Lemma \ref{diffusefd} there is a sequence of unitaries $(u_k)_k\in \pi(Q)\cap q(M_n(\C)\ot M_1)^\psi q$ such that $\widetilde{E}_1(xu_ky)\rightarrow 0$ in norm, for all $x,y\in q(M_n(\C)\ot M_1)q$. Viewing the sequence $(u_k)_k$ in the unit ball of $(M_n(\C)\ot M_1)^\psi$ we can apply Lemma \ref{LemmaNets} and deduce that $(\id\ot E_{M_1})(xu_ky)\rightarrow 0$ $\sigma$-strongly, for all $x,y\in q(M_n(M)\ominus M_n(M_1))q$. Fix a normalizing unitary $x\in\mathcal{N}_M(Q)$ and observe that, for all $a\in Q$, one has $w^*xw\pi(a)=w^*xaw=\pi(xax^*)w^*xw$. Define the element $v:=w^*xw-(\id\ot E_{M_1})(w^*xw)\in qM_n(M)q$ and note that $(\id\ot E_{M_1})(v)=0$. Moreover, for all $a\in Q$, one has $v\pi(a)=\pi(xax^*)v$. For any $k\in\N$, fix a unitary $a_k\in Q$ such that $\pi(a_k)=u_k$. Then, for all $k\in\N$, since $\pi(xa_kx^*)$ is a unitary in $\pi(Q)$,
\begin{eqnarray*}
\Vert(\id\ot E_{M_1})(vv^*)\Vert_\psi&=&\Vert\pi(xa_kx^*)(\id\ot E_{M_1})(vv^*)\Vert_\psi
=\Vert(\id\ot E_{M_1})(vu_kv^*)\Vert_\psi\rightarrow_k0.
\end{eqnarray*}
It follows that $vv^*=0$ so $w^*xw\in M_n(M_1)$ for all $x\in \mathcal{N}_M(Q)$. This proves the Claim.\hfill\qedsymbol

\vspace{0.2cm}

\noindent\textit{End of the proof of Lemma \ref{LemmaNormalizer}.} By Remark \ref{RmkEmbedds} $(2)$, it suffices to show that $$M_n(\C)\ot\mathcal{N}_M(Q)''\prec_{M_n(\C)\ot M} M_{n}(\C)\ot M_1.$$
Since $xw=w\pi(x)$ for all $x\in Q$, we deduce that $ww^*\in Q'\cap M\subset\mathcal{N}_M(Q)''$. Define the projections $e:=e_{11}\ot ww^*\in M_n(\C)\ot\mathcal{N}_M(Q)''$, $f:=w^*w\in M_n(\C)\ot M_1$ and note that, since $w$ is a partial isometry we have, viewing $w\in (e_{11}\ot 1)(M_n(\C)\ot M)$, $$ew=(e_{11}\ot ww^*)w=ww^*w=w\text{ and }wf=ww^*w=w$$
so $w\in e(M_n(\C)\ot M)f$. By the Claim, we have a unital normal $*$-homomorphism $$\theta\,:\,e(M_n(\C)\ot\mathcal{N}_N(Q)'')e=\C e_{11}\ot ww^*\mathcal{N}_N(Q)''ww^*\rightarrow w^*w(M_n(\C)\ot M_1)w^*w$$
defined by $\theta(e_{11}\ot x)=w^*xw$ for all $x\in ww^*\mathcal{N}_M(Q)''ww^*$. Moreover, $w\theta(x)=ww^*xw=xw$ for all $x\in ww^*\mathcal{N}_M(Q)''ww^*$. Finally, note that the image of $\theta$ is exactly $w^*\mathcal{N}_M(Q)''w$ and, given any faithful normal conditional expectation $E\,:\, M\rightarrow\mathcal{N}_M(Q)''$, the map $$\widetilde{E}\,:\,w^*w(M_n(\C)\ot M_1)w^*w\rightarrow w^*\mathcal{N}_M(Q)''w,\quad\widetilde{E}(x):=w^*E(wxw^*)w$$
is obviously a normal faithful conditional expectation. \end{proof}

\noindent We also record the following result for later use.

\begin{lemma}{Special case of \cite[Theorem 5.2]{HI20}}\label{LemmaNormalizer1} Let $M=M_1\underset{N}{*}M_2$ and $Q\subset M_1$ a unital subalgebra with expectation. If $Q\nprec_{M_1}N$ then $\mathcal{N}_M(Q)''\subset M_1$.
\end{lemma}

\begin{proof}
Since $Q\subset M_1$ is unital with expectation and $Q\nprec_{M_1}N$ we deduce from Remark \ref{RmkEmbedds} $(2)$ that $M_1\nprec_{M_1} N$ so that we can directly apply \cite[Theorem 5.2]{HI20}.
\end{proof}

\begin{lemma}\label{normalizerfp}
Let $M= M_1\underset{N}{*}M_2$ with $N$ finite dimensional. Then, for any finite diffuse and amenable subalgebra $A\subset M$ with expectation, one of the following statements holds:
\begin{itemize}
    \item There is a $k\in\{1,2\}$ such that $\mathcal{N}_M(A)''\prec_M M_k$,
    \item $\mathcal{N}_M(A)''$ is amenable.
\end{itemize}
\end{lemma}

\begin{proof}

Assume that $P:=\mathcal{N}_M(A)''\nprec_MM_k$ for $k=1,2$ and let us show that $P$ is amenable. By Lemma \ref{LemmaNormalizer} we have $A\nprec M_k$ for $k=1,2$. Hence, we may apply \cite[Theorem 4.4]{HU16a} and deduce that $A'\cap M$ is amenable. Let $\mathcal{A}:=A\vee(A'\cap M)\subset M$ and note that $\mathcal{A}$ is amenable and diffuse. Also, by Remark \ref{RmkExp}, $\mathcal{A}\subset M$ is with expectation. Define $\mathcal{P}:=\mathcal{N}_M(\mathcal{A})''$ and note that $P\subset\mathcal{P}$. By Remark \ref{RmkExp} $P\subset M$ is with expectation so $P\subset\mathcal{P}$ is with expectation and it suffices to show that $\mathcal{P}$ is amenable. Note that $\mathcal{A}'\cap M\subset A\subset \mathcal{A}$ and since $\mathcal{A}$ is diffuse and $N$ is finite dimensional, $\mathcal{A}\nprec_M N$. Also, since the inclusion $P\subset\mathcal{P}$ is unital with expectation, Remark \ref{RmkEmbedds} $(2)$ implies that $\mathcal{P}\nprec_M M_k$ for $k=1,2$. Note that the statement of \cite[Theorem G]{IS22} is still valid, using normalizers instead of stable normalizers (and the proof is even easier, by using \cite[Theorem A.4]{HU16a} instead of \cite[Lemma 6.1]{IS22}). Hence, we deduce that $\mathcal{P}$ is amenable.\end{proof}

\noindent The following Theorem is \cite[Theorem A]{HI20}. We state a slightly more general version for our purpose. An inclusion of von Neumann algebras $N\subset M$ is called \textit{entirely non-trivial} if for any non-zero projection $z\in\mathcal{Z}(N)$ one has $Nz\neq Mz$ (as sets). Note that if $N\subset M$ has expectation and if $M\nprec_M N$ then the inclusion $N\subset M$ is entirely non-trivial.

\begin{theorem}\label{HI20}
Let $M=\underset{N,1\leq i\leq K}{*}M_i$ be an amalgamated free product. Suppose that there exists $1\leq i\neq j\leq K$ such that $M_i\nprec_{M_i}N$ and the inclusion $N\subset M_j$ is entirely non trivial. Then, the following holds.
\begin{enumerate}
\item $\mathcal{Z}(M)=M'\cap N$.
\item For all f.n.s. $\varphi\in M_*$ such that $\varphi=\varphi\circ E_N$, we have
$$T(M):=\{t\in\R\,:\,\exists u\in\mathcal{U}(N)\text{ such that }\sigma_t^\varphi={\rm Ad}(u)\}.$$
\item $M$ is semifinite if and only if there exists a faithful normal trace ${\rm Tr}_N$ on $N$ such that, for all $1\leq k\leq K$, ${\rm Tr}\circ E_k$ is a trace on $M_k$.
\item For every non-zero projection $z\in\mathcal{Z}(M)$, $Mz$ is not amenable relative to $M_i$ inside $M$. In particular, $Mz$ is not amenable.
\end{enumerate}
\end{theorem}

\begin{proof}
We write $M=M_i\underset{N}{*}M_i'$. Let us show that the inclusion $N\subset M_i'$ is entirely non trivial. If there is a non-zero projection $z\in\mathcal{Z}(N)$ such that $M_i'z=Nz$ then $M_jz\subset M_i'z=Nz$ which implies that $M_jz=Nz$. Since $N\subset M_j$ is entirely non-trivial, we deduce that $N\subset M_i'$ also is. Hence, we may apply \cite[Theorem A]{HI20} to deduce $(1)$, $(2)$ and $(4)$. It remains to show $(3)$ for which we can repeat the proof of \cite[Theorem 4.3 (3)]{Ue13}: fix any faithful normal state $\omega\in N_*$ and consider the state $\varphi:=\omega\circ E_N$. Then, $M$ is semifinite if and only if there exists one-parameter group of unitaries $u_t\in M$ such that $\sigma_t^\varphi={\rm Ad}(u_t)$ for all $t\in \R$. Fix $i$ such that $M_i\nprec_{M_i} N$. Applying Lemma \ref{LemmaNormalizer1} we deduce that $u_t\in M_i$, for all $t\in\R$. Let now $j\neq i$ be such that $N\subset M_j$ is entirely non-trivial. Let $y\in M_j$ such that $E_j(y)=0$ and note that $\sigma_t^\varphi(y)(u_t-E_i(u_t))+ \sigma_t^\varphi(y)E_i(u_t)=\sigma_t^\varphi(y)u_t=u_ty=E_i(u_t)y+(u_t-E_i(u_t))y$. Writing $v:=u_t-E_i(u_t)$ we have, by freeness of $M_i$ and $M_j$ and since $\sigma_t^\varphi(y)\in\ker(E_j)$ and $v\in\ker(E_i)$,
$$
E_N(vyy^*v^*)=E_N(\sigma_t^\varphi(y)vy^*v^*)+E_N(\sigma_t^\varphi(y)E_i(u_t)y^*v^*)-E_i(u_t)E_N(yy^*v^*)=0.
$$
It follows that $E_N(vyy^*v^*)=E_N(vE_j(yy^*)v^*)=0$ and since $E_N$ is faithful, $vE_j(yy^*)=0$, for all $y\in\ker(E_j)$. Hence, $vs(E_j(yy^*))=0$, for all $y\in\ker(E_j)$, where $s(z)$ denotes the support projection of the self-adjoint operator $z$. By \cite[Lemma 4.2 (3)]{Ue13}, there is a family $(y_i)_{i\in I}$ of elements in $\ker(E_j)$ such that $\sum_{i\in i}s(E_j(y_iy_i^*))=1$. It implies that $v=0$ i.e. $u_t=E_i(u_t)\in N$. Write $u_t=h^{it}$ with $h$ a positive non-singular operator affiliated with the centralizer $N^\omega\subset N$ of $\sigma_t^\omega$. Then, the semifinite faithful normal weight $\omega_{h^{-1}}$ on $N$ is the desired trace.
\end{proof}

\subsection{Factoriality and type}

\noindent We fix an ergodic action $\beta\,:\, H\curvearrowright B$ with unique invariant faithful state $\psi\in B^*$. Let's begin with the following Lemma.

\begin{lemma}\label{LemmaNprec}
Suppose that $\Irr(G)$ is infinite and $\Irr(F)$ is finite then, for all $1\leq\kappa\leq K$, and all projection $q\in\Linf(G)'\cap\Linf(F)$ one has: $$L_\kappa(\GG)(1_{\Linf(H)}\ot q)\nprec_{L_\kappa(\GG)(1_{\Linf(H)}\ot q)}\Linf(H)\ot\Linf(F)q.$$
\end{lemma}

\begin{proof}
Since $\Irr(G)$ is infinite, it follows from \cite[Lemma 2.4]{FT24} that $\Linf(G)$ is diffuse. Hence, $\Linf(G)q$ is also diffuse. Moreover, since $\Linf(F)q\subset\Linf(G)q$ is a finite dimensional subalgebra and $h_G':=h_G(q)^{-1}h_G\vert_{\Linf(G)q}$ is an almost periodic state on $\Linf(G)q$ (note that $q$ is in the centralizer of $h_G$), Remark \ref{RmkDiffuse} implies that there exists a sequence of unitaries $(u_n)_n$ in $(\Linf(G)q)^{h'_G}$ such that $E_F(xu_ny)\rightarrow 0$ in norm, for all $x,y\in\Linf(G)q$. Considering the amalgamated free product:
\begin{eqnarray*}
    M&:=&(M_{N_\kappa}(\C)\ot L_\kappa(\GG))(1\ot q)\simeq (B\ot\Linf(G)q)\underset{B\ot\Linf(F)q}{*}(M_{N_\kappa}(\C)\ot\Linf(H)\ot\Linf(F)q)\\
    &=:&M_1\underset{N}{*}M_2,
\end{eqnarray*}
Lemma \ref{LemmaNets} $(2)$ implies that $E_{M_2}(x(1\ot u_n)y)\rightarrow 0$ $\sigma$-strong for all $x,y\in M$. We conclude from Remark \ref{RmkNprec} that $M\nprec_{M}M_2$. Remark \ref{RmkEmbedds} $(3)$ implies that
\begin{equation*}
L_\kappa(\GG)(1\ot q)\nprec_{L_\kappa(\GG)(1\ot q)}\Linf(H)\ot\Linf(F)q.\qedhere
\end{equation*}
\end{proof}

\begin{lemma}\label{LemmaVNBlock}
Suppose that $\beta$ is $2$-ergodic, $\Irr(G)$ and $\Irr(H)$ are infinite and $\Irr(F)$ is finite. For any $1\leq\kappa\leq K$ such that $N_\kappa\geq 2$ one has:
\begin{enumerate}
\item $L_\kappa(\GG)$ is has no amenable direct summand and:
$$\mathcal{Z}( L_\kappa(\GG))=\C1_{\Linf(H)}\ot(\mathcal{Z}(\Linf(G))\cap\Linf(F)).$$
\item For all non-zero projection $p\in\mathcal{Z}(L_\kappa(\GG))$ one has:
$$T(L_\kappa(\GG)p)=\{t\in\R\,:\,\exists u\in\mathcal{U}(\mathcal{Z}(\Linf(F)q))\,:\,\sigma_t^{h_G}\vert_{\Linf(G)q}={\rm Ad}(u)\}\cap\{t\in\R\,:\,\sigma_t^{h_H}=\id\},$$
where $q\in\mathcal{Z}(\Linf(G))\cap\Linf(F)$ is the unique projection such that $p=1\ot q$.
\end{enumerate}
\end{lemma}

\begin{proof}$(1)$. It suffices to prove the statement for $L_\kappa(\Gtilde)=M_{N_\kappa}(\C)\ot L_{\kappa}(\GG)$ instead of $L_\kappa(\GG)$. Since $\Irr(G)$ and $\Irr(H)$ are infinite, it follows from \cite[Lemma 2.4]{FT24} that both $\Linf(G)$ and $\Linf(H)$ are diffuse. Hence, since $B\ot\Linf(F)$ is finite dimensional, $$M_{N_\kappa}(\C)\ot\Linf(H)\ot\Linf(F)\nprec_{M_{N_\kappa}(\C)\ot\Linf(H)\ot\Linf(F)}\beta_\kappa(B)\ot\Linf(F)$$
and, $B\ot\Linf(F)\subset B\ot\Linf(G)$ is entirely non-trivial.

\vspace{0.2cm}

\noindent From the discussion above, Theorem \ref{HI20} $(1)$ and Lemma \ref{LemmaT} $(2)$ we find that $$\mathcal{U}(\mathcal{Z}(L_\kappa(\Gtilde)))\subseteq\{u\in\mathcal{U}(\mathcal{Z}(B\ot\Linf(F)))\,:\,{\rm Ad}((\beta_\kappa\ot\id)(u))=\id\}\subset\C1\ot\mathcal{Z}(\Linf(F)).$$
Let $u\in\mathcal{Z}(L_\kappa(\Gtilde))$ and write $u=1\ot v$, with $v\in\mathcal{Z}(\Linf(F))$. Then, for all $x\in\Linf(G)$, $u(1_B\ot x)u^*=1_B\ot vxv^*=1_B\ot x$, which implies that $v\in\mathcal{Z}(\Linf(G))\ot\Linf(F)$. The converse inclusion being trivial, this gives the formula for the center. Note that $L_\kappa(\Gtilde)$ has no amenable direct summand by Theorem \ref{HI20} $(4)$.

\vspace{0.2cm}

\noindent$(2)$. Let $p\in\mathcal{Z}(L_\kappa(\GG))$ be a non-zero projection. Write $p=1\ot q$ with $q\in\mathcal{Z}(\Linf(G))\cap\Linf(F)$. Note that:
$$L_\kappa(\Gtilde)(1\ot p)=M_{N_\kappa}(\C)\ot L_\kappa(\GG)p=(B\ot\Linf(G)q)\underset{B\ot\Linf(F)q}{*}(M_{N_\kappa}(\C)\ot\Linf(H)\ot\Linf(F)q).$$
Since $\Linf(G)q$ and $\Linf(H)$ are diffuse and $B\ot\Linf(F)q$ is finite dimensional, we can apply Theorem \ref{HI20} $(2)$ and we find:
$$T(L_\kappa(\GG)p)=T(L_\kappa(\Gtilde)(1\ot p))=\{t\in\R\,:\,\exists v\in\mathcal{U}(B\ot\Linf(F)q)\,:\,\sigma_t^{\varphi_\kappa}\vert_{L_\kappa(\Gtilde)(1\ot p)}={\rm Ad}(v)\},$$
Fix $t\in T(L_\kappa(\Gtilde)(1\ot p))$ and let $v\in\mathcal{U}(B\ot\Linf(F)q)$ be such that $\sigma_t^{\varphi_\kappa}\vert_{L_\kappa(\Gtilde)(1\ot p)}={\rm Ad}(v)$. Note that, since $F$ is a finite quantum group, the Haar state of $F$ is tracial by \cite[Theorem 2.2]{VD97} so $\sigma_t^{h_F}=\id$ hence, restricting the modular group of $\varphi_\kappa$ to $B\ot\Linf(H)\ot\Linf(F)q$, we find that $\sigma^{\psi_\kappa}_{-t}\ot\sigma_t^{h_H}\ot\id_{\Linf(F)q}={\rm Ad}((\beta_\kappa\ot\id)(v))$. Lemma \ref{LemmaT} implies that $v=1\ot u$ for some unitary $u\in\mathcal{Z}(\Linf(F)q)$ and $\sigma_t^{h_H}=\id$. In particular, $\sigma_t^\psi=\id$ and, restricting this time the modular group of $\varphi_\kappa$ to  $B\ot\Linf(G)q$ gives $\sigma_t^\psi\ot\sigma_t^{h_G}\vert_{\Linf(G)q}=\id\ot\sigma_t^{h_G}\vert_{\Linf(G)q}={\rm Ad}(1\ot u)$. Hence, $T(L_\kappa(\Gtilde)(1\ot p))\subseteq T'$, where:
$$T':=\{t\in\R\,:\,\exists u\in\mathcal{U}(\mathcal{Z}(\Linf(F)q))\,:\,\sigma_t^{h_G}\vert_{\Linf(G)q}={\rm Ad}(u)\}\cap\{t\in\R\,:\,\sigma_t^{h_H}=\id\}.$$
Let now $t\in T'$. Then $\sigma_t^{h_H}=\id$ implies that $\sigma_t^\psi=\id_B$ and $\sigma_{-t}^{\psi_\kappa}=\id_{M_{N_\kappa}(\C)}$. Let $u\in\mathcal{U}(\mathcal{Z}(\Linf(F)q))$ be such that $\sigma_t^{h_G}\vert_{\Linf(G)q}={\rm Ad}(u)$. Then, $$\sigma_t^{\varphi_\kappa}\vert_{B\ot\Linf(G)q}=\sigma_t^\psi\ot\sigma_t^{h_G}\vert_{\Linf(G)q}={\rm Ad}(1\ot u)\text{ and,}$$
$$\sigma_t^{\varphi_\kappa}\vert_{M_{N_\kappa}(\C)\ot \Linf(H)\ot\Linf(F)q}=\sigma_{-t}^{\psi_\kappa}\ot\sigma_t^{h_H}\ot\sigma_t^{h_F}\vert_{\Linf(F)q}=\id={\rm Ad}((\beta_\kappa\ot\id)(1\ot u))$$
since $u\in\mathcal{U}(\mathcal{Z}(\Linf(F)q))$. It follows that $t\in T(L_\kappa(\Gtilde)(1\ot p))$.
\end{proof}

\begin{theorem}\label{ThmFactor}
Suppose that $\beta$ be is a $2$-ergodic action, ${\rm dim}(B)\geq 3$, $\Irr(G)$ and $\Irr(H)$ are infinite and $\Irr(F)$ is finite. The following holds.
\begin{enumerate}
    \item $\Linf(\GG)$ has no amenable direct summand and:
    $$\mathcal{Z}(\Linf(\GG))=
        \C1_{\Linf(H)}\ot\left( \Linf(F)\cap\mathcal{Z}(\Linf(G))\right).$$
        \item For all non-zero projection $p\in\mathcal{Z}(\Linf(\GG))$ one has:
$$T(\Linf(\GG)p)=\{t\in\R\,:\,\exists u\in\mathcal{U}(\mathcal{Z}(\Linf(F)q))\,:\,\sigma_t^{h_G}\vert_{\Linf(G)q}={\rm Ad}(u)\}\cap\{t\in\R\,:\,\sigma_t^{h_H}=\id\},$$
where $q\in\mathcal{Z}(\Linf(G))\cap\Linf(F)$ is the unique projection such that $p=1\ot q$.
    \item For all non-zero projection  $p\in\mathcal{Z}(\Linf(\GG))$ the following are equivalent:
    \begin{enumerate}
    \item $H$ is Kac and there exists a positive and invertible element $Q\in\mathcal{Z}(\Linf(F)q)$ 
    such that $\sigma_t^{h_G}\vert_{\Linf(G)q}={\rm Ad}(Q^{it})$ for all $t\in\R$, where $q\in\mathcal{Z}(\Linf(G))\cap\Linf(F)$ is the unique projection such that $p=1\ot q$.
    \item $\Linf(\GG)p$ is finite.
    \item $\Linf(\GG)p$ is semifinite.
    \end{enumerate}
  \end{enumerate}
\end{theorem}

\begin{proof} We write $\Linf(\GG)=\underset{\Linf(H)\ot\Linf(F),1\leq\gamma\leq K}{*}L_\gamma(\GG)$ and since $L_\gamma(\GG)\nprec_{L_\gamma(\GG)}\Linf(H)\ot\Linf(F)$ (Lemma \ref{LemmaNprec}) we have in particular that the inclusion $\Linf(H)\ot\Linf(F)\subset L_\gamma(\GG)$ is entirely non-trivial, for $1\leq\gamma\leq K$.
\vspace{0.2cm}

\noindent$(1)$. Note that we always have $\C1_{\Linf(H)}\ot (\Linf(F)\cap\mathcal{Z}(\Linf(G)))\subseteq\mathcal{Z}(\Linf(\GG))$. If $K=1$ statement $(1)$ follows directly from Lemma \ref{LemmaVNBlock} (since ${\rm dim}(B)\geq 3$). Hence we may and will assume that $K\geq 2$ so that we can apply Theorem \ref{HI20} to deduce that $\Linf(\GG)$ has no amenable direct summand. Let $u\in\mathcal{U}(\mathcal{Z}(\Linf(\GG)))$. If there exists $1\leq\kappa\leq K$ such that $N_\kappa\geq 2$ then, since $u\in \mathcal{N}_{\Linf(\GG)}(L_\kappa(\GG))$, it follows from Lemma \ref{LemmaNormalizer1} that $u$ is in the center of $L_\kappa(\GG)$. We deduce that $u\in\C1_{\Linf(H)}\ot (\Linf(F)\cap\mathcal{Z}(\Linf(G)))$ by Lemma \ref{LemmaVNBlock}. Hence, we may and will assume that $N_\kappa=1$ for all $\kappa$. Hence, $B=\C^K$ and $K\geq 3$.

\vspace{0.2cm}

\noindent Let $u\in\mathcal{U}(\mathcal{Z}(\Linf(\GG)))$. By Theorem \ref{HI20}, $u\in\Linf(\GG)'\cap\Linf(H)\ot\Linf(F)$. In particular, $u\in\mathcal{Z}(L_1(\GG))\cap\mathcal{Z}(L_2(\GG))$. For each $1\leq\kappa\leq K$, we can again apply Theorem \ref{HI20} to the amalgamated free product $L_\kappa(\GG)=(\C^K\ot\Linf(G))\underset{\C^K\ot\Linf(F)}{*}(\Linf(H)\ot\Linf(F))$ to deduce:
$$\mathcal{Z}(L_\kappa(\GG))\subseteq\mathcal{Z}(\Linf(H)\ot\Linf(F))\cap\beta_\kappa(\C^K)\ot\Linf(F).$$
In particular, $u\in(\beta_1(\C^K)\ot\Linf(F))\cap(\beta_2(\C^K)\ot\Linf(F))=\C1\ot\Linf(F)$, by Lemma \ref{LemmaTdim1}. Then, there exists $v\in\Linf(F)$ such that $u=1\ot v$ and since $u$ is central in $\C^K\ot\Linf(G)$, we have $v\in\mathcal{Z}(\Linf(G))$. Hence, $\mathcal{Z}(\Linf(\GG))\subseteq \C1\ot\mathcal{Z}(\Linf(G))\cap\Linf(F)$.

\vspace{0.2cm}

\noindent$(2)$. Once again, we may and will assume that $K\geq 2$, By Lemma \ref{LemmaVNBlock}. Let $p=1\ot q$ be a projection in $\mathcal{Z}(\Linf(\GG))$, with $q\in\mathcal{Z}(\Linf(G))\cap\Linf(F)$. Note that $p$ is in the centralizer of $h_{\GG}$ and $q$ is in the centralizer of $h_G$ and:
$$\Linf(\GG)p\simeq\underset{\Linf(H)\ot\Linf(F)q,1\leq\gamma\leq K}{*}L_\gamma(\GG)p.$$
Since $L_\kappa(\GG)p\nprec_{L_\kappa(\GG)p}\Linf(H)\ot\Linf(F)q$ (Lemma \ref{LemmaNprec}) we have $\mathcal{N}_{\Linf(\GG)p}( L_\gamma(\GG)p)\subset L_\gamma(\GG)p$ for all $1\leq\gamma\leq K$, by Lemma \ref{LemmaNormalizer1}. Define:
$$T':=\{t\in\R\,:\,\exists u\in\mathcal{U}(\mathcal{Z}(\Linf(F)q))\,:\,\sigma_t^{h_G}\vert_{\Linf(G)q}={\rm Ad}(u)\}\cap\{t\in\R\,:\,\sigma_t^{h_H}=\id\}.$$
Let $t\in T'$ and let $u\in\mathcal{U}(\mathcal{Z}(\Linf(F)q))$ such that $\sigma_t^{h_G}\vert_{\Linf(G)q}={\rm Ad}(u)$. Since $\sigma_t^{h_H}=\id$ one has $\sigma_t^\psi=\id$ and $\sigma_{-t}^{\psi_\gamma}=\id$ for all $1\leq\gamma\leq K$. Arguing as in the end of the proof of Lemma \ref{LemmaVNBlock}, we deduce that $\sigma_t^{\varphi_\gamma}\vert_{L(\Gtilde_\gamma)(1\ot p)}={\rm Ad}(v)$, where
$$v:=(1\ot u)=(\beta_\gamma\ot\id)(1\ot u)\in\mathcal{U}(L_\gamma(\Gtilde)(1\ot p))\subset\mathcal{U}(\Linf(\Gtilde)(1\ot p)),\text{ for all }\gamma\in\{1,\dots,K\}.$$ Hence $\sigma_t^{\varphi_\gamma}\vert_{L(\Gtilde_\gamma)(1\ot p)}$ is actually implemented by the same unitary in $\Linf(F)q$, for all $1\leq\gamma\leq K$. Since the isomorphism $M_{N_\gamma}(\C)\ot L_\gamma(\GG)p\simeq L_\gamma(\Gtilde)(1\ot p)$ intertwines $\psi_\gamma^{-1}\ot h_\gamma\vert_{L_\gamma(\GG)p}$, it intertwines the modular automorphisms $\sigma_t^{\psi_\gamma^{-1}}\ot\sigma_t^{h_\gamma}\vert_{L_\gamma(\GG)p}=\id_{M_{N_\gamma}(\C)}\ot\sigma_t^{h_\gamma}\vert_{L_\gamma(\GG)p}$ and $\sigma_t^{\varphi_\gamma}\vert_{L_\gamma(\Gtilde)(1\ot p)}$ and since the isomorphism is the inclusion on $\Linf(H)\ot\Linf(F)q$, $\id_{M_{N_\gamma}(\C)}\ot\sigma_t^{h_\gamma}\vert_{L_\gamma(\GG)p}$ and hence $\sigma_t^{h_\gamma}\vert_{L_\gamma(\GG)p}$ is implemented by the same unitary in $\Linf(F)q$, for all $1\leq\gamma\leq K$. Since $\sigma_t^{h_\GG}\vert_{\Linf(\GG)p}=\underset{1\leq \gamma\leq K}{*}\sigma_t^{h_\gamma}\vert_{L_\gamma(\GG)p}$, we deduce that $t\in T(\Linf(\GG)p)$.

\vspace{0.2cm}

\noindent Let us prove the converse inclusion. Let $t\in T(\Linf(\GG)p)$. Then, there exists $v\in\mathcal{U}(\Linf(\GG)p)$ such that $\sigma_t^{h_\GG}\vert_{\Linf(\GG)p}={\rm Ad}(v)$. Since $\sigma_t^{h_\GG}({L_\gamma(\GG)})=\sigma_t^{h_\gamma}(L_\gamma(\GG))=L_\gamma(\GG)$, we have $v\in\mathcal{N}_{\Linf(\GG)p}(L_\gamma(\GG)p)\subset L_\gamma(\GG)p$, for all $1\leq\gamma\leq K$. If there exists $1\leq\kappa\leq K$ such that $N_\kappa\geq 2$ then it follows that $t\in T(L_\kappa(\GG)p)$ so $t\in T'$, by Lemma \ref{LemmaVNBlock}. Hence, we may and will assume that $N_\kappa=1$, for all $1\leq\kappa\leq K$. In that case, $B=\C^K$, $K\geq 3$ and $\psi$ is the uniform trace since it is the unique $\delta$-form on $\C^K$. Viewing, for all $1\leq\gamma\leq K$,
$$L_\gamma(\GG)p=(\C^K\ot\Linf(G)q)\underset{\C^K\ot\Linf(F)q}{*}(\Linf(H)\ot\Linf(F)q)$$
and observing that:
$$\Linf(H)\ot\Linf(F)q\nprec_{\Linf(H)\ot\Linf(F)q}\beta_\kappa(\C^K)\ot\Linf(F)q\text{ and,}$$
$$\C^K\ot\Linf(G)q\nprec_{\C^K\ot\Linf(G)q}\C^K\ot\Linf(F),$$
we may apply Lemma \ref{LemmaNormalizer1} to deduce that $v\in (\C^K\ot\Linf(G)q)\cap(\Linf(H)\ot\Linf(F)q)\subset L_\gamma(\GG)q$, for all $1\leq \gamma\leq K$. In particular, $v\in\beta_1(\C^K)\ot\Linf(F)q\cap\beta_2(\C^K)\ot\Linf(F)q=\C1\ot\Linf(F)q$, by Lemma \ref{LemmaTdim1}. Write $v=1\ot u$, where $u\in\mathcal{U}(\Linf(F)q)$. Since, $\sigma_t^\psi\ot\sigma_t^{h_G}=\id\ot\sigma_t^{h_G}={\rm Ad}(1\ot u)$ we deduce that $\sigma_t^{h_G}={\rm Ad}(u)$ and $u\in\mathcal{Z}(\Linf(F))$ (since $F$ is Kac). We also have $\sigma_t^{h_H}\ot\sigma_t^{h_F}={\rm Ad}((\beta_1\ot\id)(1\ot u))=\id$, which implies that $\sigma_t^{h_H}=\id$.

\vspace{0.2cm}

\noindent$(3)$. We fix a projection $p\in\mathcal{Z}(\Linf(\GG))$ and we write $p=1\ot q$, where $q\in \mathcal{Z}(\Linf(G))\cap\Linf(F)$ is a projection. 
\vspace{0.2cm}

\noindent$(a)\Rightarrow(b)$. Assume that $H$ is Kac and there exists $Q\in\mathcal{Z}(\Linf(F)q)$ positive and invertible such that $\sigma_t^{h_G}\vert_{\Linf(G)q}={\rm Ad}(Q^{it})$. Then $\psi$ and $\psi_\kappa^{-1}$ are traces so $$\sigma_t^{\varphi_\kappa}\vert_{M_{N_\kappa}(\C)\ot \Linf(H)\ot\Linf(F)q}=\id={\rm Ad}((\beta_\kappa\ot\id)(1\ot Q^{it}))$$
since $Q\in\mathcal{Z}(\Linf(F)q)$. Also, $\sigma_t^{\varphi_\kappa}\vert_{B\ot\Linf(G)q}=\id\ot\sigma_t^{h_G}\vert_{\Linf(G)q}={\rm Ad}(1\ot Q^{it})$. Hence, $\sigma_t^{\varphi_\kappa}\vert_{L_\kappa(\GG)p}={\rm Ad}(Q_1^{it})$, where $Q_1:=(1\ot Q)=(\beta_\kappa\ot\id)(1\ot Q)\in L_\kappa(\Gtilde)p$, for all $1\leq\kappa\leq K$. Arguing as in the end of the proof of $(2)$, we deduce that $\sigma_t^{h_\kappa}\vert_{L_\kappa(\GG)p}$ is implemented by the same unitary $Q_1^{it}\in\Linf(F)q$, for all $1\leq\kappa\leq K$ and  so $\sigma_t^{h_\GG}\vert_{\Linf(\GG)p}$ is also implemented by $Q_1^{it}\in\Linf(F)q$, for all $t\in\R$. It follows that the faithful state $\tau=h_{\GG}(p)^{-1}h_\GG(Q_1^{-1}\cdot)\vert_{\Linf(\GG)p}$ is a trace on $\Linf(\GG)p$.

\vspace{0.2cm}\noindent$(b)\Rightarrow(c)$ is obvious.

\vspace{0.2cm}

\noindent$(c)\Rightarrow(a)$. Assume that $\Linf(\GG)p$ is semifinite. Hence, $T(\Linf(\GG)p)=\R$ and it follows from $(2)$ that $H$ is Kac. In particular, $\psi$ is a tracial $\delta$-form so it is the uniform trace on $B$ i.e. $\psi=\frac{1}{K}\sum_\kappa\frac{1}{N_\kappa}{\rm Tr}_\kappa$ (and $\psi_\gamma$ is tracial, for all $1\leq\gamma\leq K$). Moreover, since $L_\gamma(\GG)p\subset\Linf(\GG)p$ has expectation, $L_\gamma(\GG)p$ is semifinite, for all $1\leq\gamma\leq K$, so $L_\gamma(\Gtilde)(1\ot p)=M_{N_\gamma}(\C)\ot L_\gamma(\GG)p$ also is. Write $L_\gamma(\Gtilde)(1\ot p)=(B\ot\Linf(G)q)\underset{B\ot\Linf(F)q}{*}(M_{N_\gamma}(\C)\ot\Linf(H)\ot\Linf(F)q)$ and apply Theorem \ref{HI20} $(3)$ to deduce that there exists a faithful normal semifinite trace $\omega$ on $B\ot\Linf(F)q$ such that $\omega\circ(\id\ot E_F)$ is tracial on $B\ot\Linf(G)q$. Note that, since  $B\ot\Linf(F)q$ is finite dimensional, $\omega$ is actually finite and up to renormalization, we may and will assume that $\omega$ is a faithful normal tracial state on $B\ot\Linf(F)q$. Then, there exists $Q_0\in\mathcal{Z}(B\ot\Linf(F)q)$ positive and invertible such that $\omega=(\psi\ot h_{F})(Q_0\cdot)$. Write $Q_0=\sum_{\gamma,r}e^\gamma_{rr}\ot Q_\gamma$, where $Q_\gamma\in\mathcal{Z}(\Linf(F)q)$ is positive and invertible. In particular, $Q:=\frac{1}{K}\sum_\gamma Q_\gamma\in\mathcal{Z}(\Linf(F)q)$ is positive and invertible and, for all $x\in \Linf(G)q$, one has:
\begin{eqnarray*}
\omega\circ(\id\ot E_F)(1\ot x)&=&\frac{1}{K}\sum_{\gamma,r}\frac{1}{N_\gamma}({\rm Tr}_\gamma\ot h_F)(e^\gamma_{rr}\ot Q_\gamma E_F(x))=\frac{1}{K}\sum_{\gamma}h_F(Q_\gamma E_F(x))\\
&=&\frac{1}{K}\sum_{\gamma}h_F( E_F(Q_\gamma x))=\frac{1}{K}\sum_{\gamma}h_G(Q_\gamma x)=h_G(Qx).
\end{eqnarray*}
Note that the modular group of $h_G(Q\cdot)\vert_{\Linf(G)q}$ is given by $Q^{it}\sigma_t^{h_G}\vert_{\Linf(G)q}(\cdot)Q^{-it}$. Since $\omega\circ(\id\ot E_F)$ is tracial, the faithful state $h_G(Q\cdot)\vert_{\Linf(G)q}$ also is hence $Q^{it}\sigma_t^{h_G}(\cdot)\vert_{\Linf(G)q}Q^{-it}=\id$ and $\sigma_t^{h_G}\vert_{\Linf(G)q}={\rm Ad}(Q^{-it})$, for all $t\in\R$.\end{proof}

\begin{remark} Note that condition $(a)$ in statement $(3)$ of the previous Theorem implies in particular that $\Linf(G)$ is a finite von Neumann algebra.
\end{remark}

\begin{example}
Under the assumptions of Theorem \ref{ThmFactor}, $\Linf(\GG)$ is a factor when the amalgam $F$ is trivial. However, coming back to the example $\GG:=\widehat{\Gamma}\wr_{*,\beta}\widehat{\Lambda}$ given in Section \ref{SectionNewExample}, Remark \ref{RemarkNewExample} implies that $\Linf(\GG)=L(\Gamma)^{*\vert\Lambda\vert}\ot {\rm L}^\infty(\Lambda)$ whenever $\Lambda$ is abelian. Hence, $\Linf(\GG)$ is never a factor, as soon as $\Lambda$ is a non-trivial abelian finite group.
\end{example}

\noindent Let us compute the invariant $T^\tau_\Int(\GG):=\{t\in\R\,:\,\exists u\in\mathcal{U}(\Linf(\GG)),\,\,\tau_t^\GG={\rm Ad}(u)\}$ as defined in \cite{KS24}. We also recall the following notation from \cite{KS24}, for any CQG $H$: $$T^\tau(H):=\{t\in\R\,:\,\tau_t^H=\id\}.$$

\begin{proposition}
If $\beta$ is $2$-ergodic, ${\dim}(B)\geq 3$, both $\Irr(G)$ and $\Irr(H)$ are infinite and $\Irr(F)$ is finite then $ T^\tau_\Int(\GG) = \{t\in\R\,:\,\exists u\in\mathcal{U}(\mathcal{Z}(\Linf(F))\,\,\tau_t^G={\rm Ad}(u)\} \cap T^\tau(H)$.
\end{proposition}

\begin{proof}
Let $t\in T^\tau_\Int(\GG)$ and $u_0\in \Ucal(\Linf(\GG))$ such that $\tau^\GG_{t} = \Ad(u_0)$. Since $\tau^\GG_t$ preserves the subalgebras $L_\kappa(\GG)$, and $\Ncal_{\Linf(\GG)}(L_\kappa(\GG))''\subset L_\kappa(\GG)$ we deduce that $u_0\in L_\kappa(\GG)$, for all $1\leq\kappa\leq K$.

\vspace{0.2cm}

\noindent First assume that, for some $1\le\kappa\leq K$, $N_\kappa\geq 2$. Identifying $L_\kappa(\Gtilde)$ with $M_{N_\kappa}(\C)\ot L_\kappa(\GG)$ and $\widetilde{\tau}^\kappa_t:=\tau_t^{\Gtilde}\vert_{L_\kappa(\Gtilde)}$ with $\sigma_{-t}^{\psi_\kappa}\ot\tau_t^{\GG}\vert_{L_\kappa(\GG)}$ we can write $\widetilde{\tau}^\kappa_t={\rm Ad}(w)$, where $w=v\ot u_0$ for some $v\in\mathcal{U}(M_{N_\kappa}(\C))$. Viewing $\Linf(\Gtilde)$ as an amalgamated free product and since, by construction, $\widetilde{\tau}^\kappa_t$ leaves globally invariant both subalgebras $B\ot\Linf(G)$ and $M_{N_\kappa}(\C)\ot\Linf(H)\ot\Linf(F)$ we deduce that there exists $u$ in the amalgam $M_{N_\kappa}(\C)\ot\Linf(F)$ such that $\widetilde{\tau}^\kappa_t={\rm Ad}(u)$. In particular, $\sigma_{-t}^{\psi_\kappa}\ot\tau_t^H\ot\id={\rm Ad}((\beta_\kappa\ot\id)(u))$ (since $\tau_t^F$ is trivial, because $\Linf(F)$ is finite dimensional). Lemma \ref{LemmaT} gives that $u\in \C1\ot \Zcal(\Linf(F))$ and $\tau_t^H=\id$. Hence, viewing $u\in \Zcal(\Linf(F))$ we have $\tau_t^G={\rm Ad}(u)$.

\vspace{0.2cm}

\noindent Now assume that $N_\kappa=1$, for all $1\leq\kappa\leq K$ so that $B=\C^K$, $\Linf(\Gtilde)=\Linf(\GG)$, $K\geq 3$ and $\psi$ is the uniform trace since it is the unique $\delta$-form on $\C^K$. Viewing, for all $1\leq\gamma\leq K$,
$$L_\gamma(\GG)=(\C^K\ot\Linf(G))\underset{\C^K\ot\Linf(F)}{*}(\Linf(H)\ot\Linf(F))$$
 and since $\tau_t^\GG$ leaves globally invariant both $\C^K\ot\Linf(G)$ and $\Linf(H)\ot\Linf(F)$, we may apply Lemma \ref{LemmaNormalizer1} to deduce that $u_0\in (\C^K\ot\Linf(G))\cap(\Linf(H)\ot\Linf(F))\subset L_\gamma(\GG)$, for all $1\leq \gamma\leq K$. In particular, $u_0\in\beta_1(\C^K)\ot\Linf(F)\cap\beta_2(\C^K)\ot\Linf(F)=\C1\ot\Linf(F)$, by Lemma \ref{LemmaTdim1}. Write $u_0=1\ot u$, where $u\in\mathcal{U}(\Linf(F))$. Since $$\tau_t^\GG\vert_{\C^K\ot \Linf(G)}=\sigma_{-t}^\psi\ot\tau_t^{h_G}={\rm Ad}(1\ot u)=\id\ot{\rm Ad}(u)$$ we deduce that $\tau_t^{h_G}={\rm Ad}(u)$ and $u\in\mathcal{Z}(\Linf(F))$ (since $\Linf(F)$ is finite dimensional). We also have:
 $$\tau_t^\GG\vert_{\Linf(H)\ot\Linf(F)}=\tau_t^{H}\ot\tau_t^{F}={\rm Ad}((\beta_1\ot\id)(1\ot u))=\id,$$
 which implies that $\tau_t^{H}=\id$. This shows the first inclusion.
 
 \vspace{0.2cm}
 
 \noindent Let now $t\in\R$ be such that $\tau_t^H=\id$ and $u\in\mathcal{U}(\mathcal{Z}(\Linf(F)))$ such that $\tau_t^G={\rm Ad}(u)$. By construction of $\tau_t^{\GG}$, we can argue as the in proof of the $T$-invariant computation in Theorem \ref{ThmFactor} to deduce that $\tau_t^{\GG}={\rm Ad}(u)$.\end{proof}

\begin{remark}
Under the hypothesis of the previous Proposition and if $F$ is moreover trivial then $T^\tau_\Int(\GG)=T^\tau(G)\cap T^\tau(H)=T^\tau(\GG)$. In particular, the Vaes $T$-invariant \cite{Va05}
$$T_V(\GG):=\{t\in\R\,:\,\exists u\in\mathcal{U}(\Linf(\GG))\,\,\tau_t^\GG={\rm Ad}(u)\text{ and }\Delta_\GG(u)=u\ot u\}$$ satisfies
$T_V(\GG)=T^\tau_\Int(\GG)=T^\tau(\GG)=T^\tau(G)\cap T^\tau(H)$.
\end{remark}

\subsection{Primeness and absence of Cartan subalgebras}

\noindent In this subsection we prove the following result, which is a more general form of Theorem \ref{ThmC}.

\begin{theorem}\label{ThmPrimeCartan}
Suppose that $\beta$ be is a $2$-ergodic action, ${\rm dim}(B)\geq 3$, $\Irr(G)$ and $\Irr(H)$ are infinite,  $\Irr(F)$ is finite and $\Linf(F)\cap\mathcal{Z}(\Linf(G))=\C1$. Then $\Linf(\GG)$ is a non-amenable factor of type ${\rm II}_1$ or ${\rm III}$ satisfying the following:
\begin{enumerate}
    \item $\Linf(\GG)$ is prime.
    \item $\Linf(\GG)$ does not have any Cartan subalgebra.
\end{enumerate}
\end{theorem}

\begin{proof}
The fact that $\Linf(\GG)$ is a non-amenable factor of type ${\rm II}_1$ or ${\rm III}$ follows from Theorem \ref{ThmFactor}. Using the notations of Remark \ref{rmk fd amal}, we write $\Linf(\Gtilde)$ as an amalgamated free product: $$ \Linf(\Gtilde) \simeq M_K \simeq M_{K-1} \underset{B\ot \Linf(F)}{\ast} \nu_K(B\ot \Linf(G)).$$ To simplify the notations we write $M$ for $M_{K-1}$. We then have the following Claim.

\vspace{0.2cm}

\noindent\textbf{Claim.} \textit{One has $\Linf(\GG)\nprec_{\Linf(\Gtilde)} M$ and $\Linf(\GG)\nprec_{\Linf(\Gtilde)} \nu_K(B\ot\Linf(G))$.}

\vspace{0.2cm}

\noindent\textit{Proof of the Claim.} Since $\Irr(G)$ is infinite, the von Neumann algebra $\Linf(G)$ is diffuse \cite[Lemma 2.4]{FT24} and since the Haar state of $G$ is almost periodic, its centralizer $\Linf(G)^{h_G}$ is still diffuse. Let $(u_n)_n$ be a sequence of unitaries in $\Linf(G)^{h_G}$ going to $0$ $\sigma$-weakly. Since $\Linf(\Gtilde)=\nu_K(B\ot{\rm L}^\infty(G))*_{B\ot\Linf(F)}M$ we may apply Lemma \ref{diffusefd} to deduce that $$(\id_B\ot E_F'')(x(1\ot u_n)y)\rightarrow 0\,\, \text{ in norm for all }x,y\in B\ot\Linf(G).$$
Since $\nu_K(1\ot u_n)$ is in the centralizer of $h_{\Gtilde}$ for all $n$, we may apply Lemma \ref{LemmaNets} to deduce that $E_{M}(x\nu_K(1\ot u_n)y)\rightarrow 0$ $\sigma$-strongly for all $x,y\in \Linf(\Gtilde)$. In particular, $E_{M}(x\rho^K_{ij}(u_n)y)\rightarrow 0$ $\sigma$-strongly for all $x,y\in \Linf(\Gtilde)$ and $1\leq i,j\leq N_\kappa$. Note that the map $$\rho_K\,:\, \Linf(G)\rightarrow M_{N_K}(\C)\ot\Linf(\GG),\quad a\mapsto\sum_{i,j}e^K_{ij}\ot\rho^K_{ij}(a)$$ is a unital $*$-homomorphism so that $\rho_K(u_n)$ is unitary in $M_{N_K}(\C)\ot\Linf(\GG)$. Moreover,
$$E_{M_{N_K}(\C)\ot M}(X\rho_K(u_n)Y)=\sum_{i,j,r,s}e^K_{ij}\ot E_{M}(X_{ir}\rho^\kappa_{rs}(u_n)Y_{sj})\rightarrow 0\,\,\sigma\text{-strongly}$$
for all $X,Y\in M_{N_\kappa}(\C)\ot \Linf(\Gtilde)$. Since $M_{N_K}(\C)\ot\Linf(\GG)$ is finite or of type ${\rm III}$ we have $$M_{N_K}(\C)\ot \Linf(\GG)\nprec_{M_{N_K}(\C)\ot\Linf(\Gtilde)} M_{N_K}(\C)\ot M.$$
Hence, $\Linf(\GG)\nprec_{\Linf(\Gtilde)}M$.

\vspace{0.2cm}

\noindent Now, since $\Irr(H)$ is infinite and arguing as in the beginning of the proof, we have a sequence of unitaries $(v_n)_n$ in $\Linf(H)^{h_H}$ going to $0$ $\sigma$-weakly so $1\ot v_n\ot 1\in M_{N_K}(\C)\ot\Linf(H)\ot\Linf(F)$ is going to $0$ $\sigma$-weakly and, by Lemma \ref{diffusefd}, we deduce that $(E_K\ot\id_F)(x(1\ot v_n\ot 1)y)\rightarrow 0$ in norm, for all $x,y\in M_{N_K}(\C)\ot\Linf(H)\ot\Linf(F)$. Next, since $\mu(v_n\ot 1)\in M$ is in the centralizer of $h_{\Gtilde}$ for all $n$, we can apply Lemma \ref{LemmaNets} and deduce that $E_{\nu_K(B\ot\Linf(G))}(x\mu(v_n\ot 1)y)\rightarrow 0$ $\sigma$-strongly for all $x,y\in\Linf(\Gtilde)$. Note that $\mu(v_n\ot 1)$ is actually a sequence of unitaries in $\Linf(\GG)$. Since $\Linf(\GG)$ is finite or type ${\rm III}$, we deduce that $\Linf(\GG)\nprec_{\Linf(\Gtilde)}\nu_K(B\ot\Linf(G))$.$\qed$

\vspace{0.2cm}

\noindent\textit{End of the proof of Theorem \ref{ThmPrimeCartan}.}$(1)$. By contradiction, assume that $\Linf(\GG)=P\ot Q$, where $P$ and $Q$ are two diffuse factors. Since $\Linf(\GG)$ is either finite or type ${\rm III}$ we may and will assume that both $P$ and $Q$ are either finite or of type ${\rm III}$. Moreover, since $\Linf(\GG)$ is not amenable one of them, let say $P$, is not amenable. Note that, since $\Linf(\GG)=P\ot Q$, there is a normal faithful conditional expectation from $\Linf(\GG)$ to $Q$ (and to $P$) and since $\Linf(\GG)\subset\Linf(\Gtilde)$ has expectation, we deduce that both $Q, P\subset\Linf(\Gtilde)$ have expectation. Since $Q$ is diffuse, we may apply \cite[Lemma 2.1]{HU16b} to get a f.n.s. $\omega\in Q_*$ with diffuse centralizer $A:=Q^\omega\subset Q$. Note that $A\subset Q$ has expectation by Takesaki's Theorem hence $A\subset\Linf(\Gtilde)$ also has expectation. Also, $P\subset Q'\cap\Linf(\GG)\subset A'\cap\Linf(\GG)\subset A'\cap\Linf(\Gtilde)$. Since $P$ is not amenable and has expectation, $A'\cap\Linf(\Gtilde)$ is not amenable either. Since $A$ is finite, diffuse and has expectation we may apply \cite[Theorem 4.4]{HU16a} to deduce that either $A\prec_{\Linf(\Gtilde)}\nu_K(B\ot\Linf(G))$ or $A\prec_{\Linf(\Gtilde)}M$. Applying now Lemma \ref{LemmaNormalizer} we deduce that $\mathcal{N}_{\Linf(\Gtilde)}(A)''\prec_{\Linf(\Gtilde)}\nu_K(B\ot\Linf(G))$ or $\mathcal{N}_{\Linf(\Gtilde)}(A)''\prec_{\Linf(\Gtilde)}M$. Since $P\subset A'\cap\Linf(\Gtilde)$, Remark \ref{RmkEmbedds} $(4)$ implies that $P\prec_{\Linf(\Gtilde)}\nu_K(B\ot\Linf(G))$ or $P\prec_{\Linf(\Gtilde)}M$. Applying again Lemma \ref{LemmaNormalizer} we get that $\mathcal{N}_{\Linf(\Gtilde)}(P)''\prec_{\Linf(\Gtilde)}\nu_K(B\ot\Linf(G))$ or $\mathcal{N}_{\Linf(\Gtilde)}(P)''\prec_{\Linf(\Gtilde)}M$. However, $P,Q\subset\mathcal{N}_{\Linf(\Gtilde)}(P)''$ which implies that $\Linf(\GG)\subset\mathcal{N}_{\Linf(\Gtilde)}(P)''$ so that, with Remark \ref{RmkEmbedds} $(4)$ and the Claim, we obtain a contradiction.

\vspace{0.2cm}

\noindent$(2)$. By contradiction, let $A\subset\Linf(\GG)$ be a Cartan subalgebra (with expectation). Since $\Linf(\GG)\subset\Linf(\Gtilde)$ has expectation, $A\subset\Linf(\Gtilde)$ has expectation. Note that $\Linf(\GG)=\mathcal{N}_{\Linf(\GG)}(A)''\subset \mathcal{N}_{\Linf(\Gtilde)}(A)''$. Hence, the Claim implies that
$$ \mathcal{N}_{\Linf(\Gtilde)}(A)''\nprec_{\Linf(\Gtilde)}\nu_K(B\ot\Linf(G))\text{ and }\mathcal{N}_{\Linf(\Gtilde)}(A)''\nprec_{\Linf(\Gtilde)}M.$$
Moreover, since $\Linf(\GG)$ is not amenable and has expectation, $\mathcal{N}_{\Linf(\Gtilde)}(A)''$ is not amenable either. We get a contradiction with Lemma \ref{normalizerfp}.\end{proof}

\section{KK-theory}\label{section KK}

\noindent In this section, we study $K$-amenability and we use our free product decomposition of the C*-algebra of a free wreath product as well as the one of the block-extended C*-algebra and the exact sequences from \cite{FG18,FG20} to compute $K$-theory.

\begin{proposition}\label{1nontor}
For a compact quantum group $G$, the class of the unit $[1]\in K_0(C(G))$ is never $0$ or of torsion. If $G$ is K-amenable, then this is also true for $[1]\in K_0(C_r(G))$.
\end{proposition}

\begin{proof}
We use the counit, which is a $*$-morphism $\varepsilon : C(G)\rightarrow \C$, which gives $\varepsilon_\ast : K_0(C(G)) \rightarrow K_0(\C)\simeq \Z$, sending $[1]$ to $[1]$, which is the generator of the K-theory group of $\C$. Therefore the class cannot be of torsion. If $G$ is K-amenable, the same holds for the reduced C*-algebra, because $\lambda_\ast$, which is an isomorphism in that case, sends the class $[1]$ to $[1]$ too.
\end{proof}

\noindent We will need the following Lemma about the Morita equivalence between a unital C*-algebra $A$ and its amplification $M_N(A)$. For any unital C*-algebra $A$, we define the KK-theory class
$$m_A\in KK(M_N(\C) \ot A, A) = [\C^N\ot A, \id, 0].$$
It is the inverse to the class of the non unital map $i : A\rightarrow M_N(\C) \ot A$ sending $a$ to $e_{11}\ot a$.

\begin{lemma}\label{Morita}
For any element $\gamma = [H,\eta, F]\in KK(A,B)$, the Kasparov product of $m_A$ and $\gamma$ is given by $m_A\ot\gamma=[\C^N\ot H, \id\ot \eta, 1\ot F]\in KK(M_N(\C)\ot A,B).$
\end{lemma}

\begin{proof}
The Kasparov product of $m_A$ and $\gamma$ can be taken as $[(\C^N \ot A)\ot_A H, \id \ot_\eta 1, \Tilde{F}]$, where $\Tilde{F}$ is an adequate $F$-connection making it a Kasparov triple. There is a unitary isomorphism 
\begin{align*}
U : (\C^N \ot A)\ot_A H&\rightarrow \C^N\ot H \\
(f\ot a) \ot_\eta \xi &\mapsto f\ot \eta(a)\xi.
\end{align*}
Under this isomorphism, the map $\id\ot_\eta 1$ corresponds to the map $U(\id\ot_\eta 1)U^*$ which sends any element $b\ot a\in M_N(\C)\ot A$ to $b\ot \eta(a)\in \mathcal{L}(\C^N\ot H)$. We then fix $\Tilde{F}$ to be the operator $U (1\ot F) U^*$, and we only have to check that $\Tilde{F}$ is a $F$-connection and that it defines a Kasparov module.
First, for any $b\ot a\in M_N(\C)\ot A$, we have 
$$[(\id\ot\eta)(b\ot a), 1\ot F] = b \ot [\eta(a), F],$$
which is a compact operator because $[\eta(a), F]$ is, by hypothesis, and $\C^N$ is finite dimensional. The same way, $(1\ot F)^2 (\id\ot \eta)(b\ot a) = b\ot F^2\eta(a)$ and
$$(1\ot F-1\ot F^*)(\id\ot \eta)(b\ot a) = b\ot (F-F^*)\eta(a)$$
are also compact operators. Lastly, we have to check the connection properties. For any $x\in (\C^N \ot A)$, we define the adjointable map $T_x \in \Lcal(H, (\C^N \ot A)\ot_A H), \xi\mapsto x\ot_\eta \xi$. Then if $x= f\ot a$, the map $UT_x$ becomes $\xi \mapsto f \ot \eta(a)\xi$ hence:
$$T_x F - \Tilde{F}T_x = T_xF-U^*(1\ot F)UT_x= U^*(UT_xF-(1\ot F)UT_x).$$
For any $\xi$, we have $(T_x F - \Tilde{F}T_x)\xi = U^*(f\ot (\eta(a)F\xi-F\eta(a)\xi))$, which is a compact operator by hypothesis. The other property can be shown the same way. Thus this indeed defines a Kasparov product for $m_A\ot \gamma$, which corresponds to $[\C^N\ot H, \id\ot \eta, 1\ot F]$.
\end{proof}

\noindent For the remainder of the section, we assume the action $\beta:H\curvearrowright B$ to be ergodic and we write $\GG=G\wr_{*,\beta,F}H$.

\vspace{0.2cm}

\noindent Concerning $K$-amenability, we prove the following. The symbol $\bullet$ is used to denote either the full or reduced C*-algebra of a CQG or amalgamated free product of C*-algebras.

\begin{theorem}\label{thmDtxt}
The following are equivalent:
\begin{enumerate}
    \item $G$ and $H$ are K-amenable.
    \item $\GG$ is K-amenable.
\end{enumerate}
\end{theorem}

\begin{proof}
$(1)\Rightarrow (2)$.
Assume first that $G$ and $H$ are K-amenable, then we can apply \cite[Theorem 5.1]{FG18} to the amalgamated free product $$ C_{\bullet,\kappa}(\Gtilde)\simeq (B\ot C_\bullet(G))\underset{B\ot C_\bullet(F)}{\ast_\bullet} (M_{N_\kappa}(\C)\ot C_\bullet(H\times F)),$$ to get that the full algebra $C_\kappa(\Gtilde)$ is K-equivalent to the reduced $C_{r,\kappa}(\Gtilde)$ for all $\kappa$, through the canonical map $\widetilde{\lambda_\kappa}$ and as $C_{\bullet,\kappa}(\Gtilde)\simeq M_{N_\kappa}(\C)\ot C_{\bullet,\kappa}(\GG)$, the same is true for the map $\lambda_\kappa$ such that $\widetilde{\lambda_\kappa}=\id \ot \lambda_\kappa$. We then apply it to $$ C_\bullet(\GG)\simeq \underset{C_\bullet(H\times F)}{\ast_\bullet} C_{\bullet,\kappa}(\GG),$$ to get that $\GG$ is K-amenable.

\noindent$(2)\Rightarrow (1).$
Assume that $\GG$ is K-amenable i.e. there exists $\gamma \in KK(C_r(\GG),\C)$ such that $[\lambda_\GG] \ot \gamma = [\varepsilon_\GG]\in KK(C(\GG),\C)$. Then, since $H$ is a dual quantum subgroup of $\GG$, it is also K-amenable. We then fix any $1\leq \kappa\leq K$, and we consider the morphism $\pi :C(G)\rightarrow M_{N_\kappa}(\C)\ot C(\GG)$ defined by $\pi : a \mapsto \beta_\kappa(e^\kappa_{11})\rho_\kappa(a)$. Note that $(\id\ot \varepsilon_\GG)\circ \pi = \varepsilon_G(\cdot)e_{11}^\kappa$. By Theorem \ref{TheoremErgodic} we have $(\id_B\ot h_\GG)\rho=h_G(\cdot)1_B$. In particular, $(\id\ot h_\GG)\rho_\kappa=h_G(\cdot)1$. It follows that there exists a unique unital $*$-homomorphism $\rho_{r,\kappa}\,:\,C_r(G)\rightarrow M_{N_\kappa}(\C)\ot C_r(\GG)$ such that $(\id\ot\lambda_\GG)\rho_\kappa=\rho_{r,\kappa}\lambda_G$. Viewing $\beta_\kappa(e^\kappa_{11})$ as a projection in $ (M_{N_\kappa}(\C)\ot\Pol(H))\cap\rho_{r,\kappa}(C_r(G))'$, the map $\rho_{r,\kappa}(C_r(G))\rightarrow M_{N_\kappa}(\C)\ot C_r(\GG)$, $x\mapsto\beta_\kappa(e^\kappa_{11})x$ is a (non-unital) $*$-homomorphism. Hence $\pi_r: C_r(G) \rightarrow M_{N_\kappa}(\C)\ot C_r(\GG)$ defined by $\pi_r(a):=\beta_\kappa(e^\kappa_{11})\rho_{r,\kappa}(a)$ is such that $\pi_r\circ \lambda_G = (\id \ot \lambda_\GG)\circ \pi$. Let us show that $\Tilde{\gamma} := [\pi_r] \ot m_{C_r(\GG)} \ot \gamma \in KK(C_r(G),\C)$ satisfies $[\lambda_G]\ot \Tilde{\gamma} =[\varepsilon_G]$. Using Lemma \ref{Morita} we find:
\begin{align*}
    [\lambda_G] \ot \Tilde{\gamma} &= [\lambda_G]\ot [\pi_r] \ot m_{C_r(\GG)} \ot \gamma= [\pi_r\circ \lambda_G]\ot m_{C_r(\GG)}\ot\gamma\\
                   &= [(\id\ot\lambda_\GG)\circ \pi]\ot [\C^N\ot H, \id\ot \eta, 1\ot F]\\
                   &= [\pi]\ot [\id\ot \lambda_\GG]\ot [\C^N\ot H, \id\ot \eta, 1\ot F]\\
                   &= [\pi]\ot[\C^N\ot H, (\id\ot \eta\lambda_\GG), 1\ot F]
                   = [\pi]\ot m_{C(\GG)}\ot [H,  \eta\lambda_\GG, F]\\
                   &= [\pi]\ot m_{C(\GG)}\ot [\lambda_\GG]\ot [H,  \eta, F]
                   = [\pi]\ot m_{C(\GG)}\ot [\lambda_\GG]\ot \gamma
                   = [\pi]\ot m_{C(\GG)}\ot [\varepsilon_\GG]\\
                   &= [\pi] \ot m_{C(\GG)}\ot [\C,\varepsilon_\GG,0]
                   = [\pi] \ot [\C^N,\id\ot\varepsilon_\GG,0]
                   = [\C^N,(\id\ot\varepsilon_\GG)\circ \pi,0]\\
                   &= [\C^N,\varepsilon_G(\cdot)e_{11}^\kappa,0]= [\C,\varepsilon_G,0]\oplus [\C^{N-1},0,0]=[\varepsilon_G],
\end{align*}
because the action of $e_{11}^\kappa$ is only on the first coordinate in $\C^{N}$, and the second summand is a degenerate module. Hence $[\lambda_G]\ot \Tilde{\gamma} =[\varepsilon_G]$ and $G$ is K-amenable.
\end{proof}

\begin{proposition}\label{exactseq1}
For a compact quantum group $G$ with a dual quantum subgroup $F$, and a compact quantum group $H$ acting on $(B,\psi)$, and there is an exact sequence of K-theory:
\begin{equation*}
\begin{tikzcd}[sep=small]
\scalemath{0.85}{K_0(B\ot C_\bullet(F)))} \arrow[r] & \scalemath{0.85}{K_0(B\ot C_\bullet(G))} \oplus \scalemath{0.85}{K_0(M_{N_\kappa}(\C)\ot C_\bullet(H\times F))} \arrow[r] & \scalemath{0.85}{K_0(M_{N_\kappa}(\C)\ot C_{\bullet,\kappa}(\GG))} \arrow[d]\\
 \scalemath{0.85}{K_1(M_{N_\kappa}(\C)\ot C_{\bullet,\kappa}(\GG))} \arrow[u] & \scalemath{0.85}{K_1(B\ot C_\bullet(G)) \oplus K_1(M_{N_\kappa}(\C)\ot C_\bullet(H\times F))} \arrow[l] & \scalemath{0.85}{K_1(B\ot C_\bullet(F))}. \arrow[l]
\end{tikzcd}
\end{equation*}
Moreover, the vertical arrow on the right is zero as soon as $\Irr(F)$ is finite, and the vertical arrow on the left is zero as soon as the map $\iota_\ast : K_0(C_\bullet(F))\rightarrow K_0(C_\bullet(G))$ is injective.
\end{proposition}

\begin{proof}
We write $$M_{N_\kappa}(\C)\ot C_{\bullet,\kappa}(\GG)\simeq (B\ot C_\bullet(G))\underset{B\ot C_\bullet(F)}{\ast_\bullet} (M_{N_\kappa}(\C)\ot C_\bullet(H\times F)),$$
and we can immediately apply the result of \cite{FG20} to get the exact sequences for the reduced and the full case. Now assume that $\Irr(F)$ is finite. Then $C(F)\simeq C_r(F)\simeq \Pol(F)$ is a finite dimensional C*-algebra, therefore we have immediately $K_1(B\ot C_\bullet(F))=0$ and so the vertical map on the right-hand side is $0$. We can also deduce from this that $$K_0( M_{N_\kappa}(\C)\ot C_\bullet(H) \ot C(F))\simeq K_0(M_{N_\kappa}(\C)\ot C_\bullet(H)) \ot K_0(C(F)),$$ and we also have in any case $$K_0(B\ot C_\bullet(F))\simeq K_0(B)\ot K_0(C_\bullet(F)). $$ 
Using that $\iota_\ast : K_0(C_\bullet(F))\rightarrow K_0(C_\bullet(G))$ is injective, we see that the first horizontal map must be injective in the full case. This also implies that the left vertical map is $0$.
\end{proof}

\begin{remark}
In the case where $F$ is trivial and we consider the maximal C*-algebras, Proposition \ref{1nontor} shows that the map $\iota_\ast : \C \rightarrow K_0(C(G))$ is injective, so the vertical arrow on the left is zero. This is also true for the reduced C*-algebras is $G$ has a K-amenable dual.
\end{remark}

\noindent We deduce the following proposition.

\begin{proposition}
If $\Irr(F)$ is finite and $\iota_*\,:\, K_0(C(F))\rightarrow K_0(C_\bullet(G))$ is injective then, we have an isomorphism $$ K_1(M_{N_\kappa}(\C)\ot C_{\bullet,\kappa}(\GG)) \simeq K_1(B\ot C_\bullet(G)) \oplus K_1(M_{N_\kappa}(\C)\ot C_\bullet(H)\ot C(F)), $$ and a short exact sequence $$\scalemath{0.88}{0\rightarrow K_0(B\ot C(F)) \rightarrow K_0(B\ot C_\bullet(G)) \oplus K_0(M_{N_\kappa}(\C)\ot C_\bullet(H)\ot C(F))\rightarrow K_0(M_{N_\kappa}(\C)\ot C_{\bullet,\kappa}(\GG))\rightarrow 0.}$$
This shows that the group $K_0(M_{N_\kappa}(\C)\ot C_{\bullet,\kappa}(\GG))$ is the quotient
$$\scalemath{0.85}{K_0(M_{N_\kappa}(\C)\ot C_{\bullet,\kappa}(\GG)) \simeq K_0(B\ot C_{\bullet}(G)) \oplus K_0(M_{N_\kappa}(\C)\ot C_{\bullet}(H)\ot C(F))/ <[e_{11}^\gamma\ot f]-[\beta_\kappa(e_{11}^\gamma)\ot f]>,}$$ where $f\in C(F)$ are projections corresponding to the generators of $K_0(C(F))$ and $1\leq \gamma\leq K$.
\end{proposition}

\noindent We use the other isomorphism $C_\bullet(\GG)\simeq \underset{C_\bullet(H)\ot C(F)}{\ast_\bullet}C_{\bullet,\kappa}(\GG)$ to get another exact sequence.

\begin{proposition}
There is a 6-term exact sequence
\begin{equation*}
\begin{tikzcd}
\bigoplus_{\kappa=1}^{K-1}K_0(C_\bullet(H\times F)) \arrow[r] & \bigoplus_{\kappa=1}^K K_0(C_{\bullet,\kappa}(\GG)) \arrow[r] & K_0(C_\bullet(\GG)) \arrow[d]\\
 K_1(C_\bullet(\GG)) \arrow[u] & \bigoplus_{\kappa=1}^K K_1(C_{\bullet,\kappa}(\GG)) \arrow[l] & \bigoplus_{\kappa=1}^{K-1}K_1(C_\bullet(H\times F)). \arrow[l]
\end{tikzcd}
\end{equation*}
Moreover, when $\Irr(F)$ is finite, the map on the right is $0$, and if $F$ is trivial, then the first map on the left is also $0$, in the maximal case.
\end{proposition}

\begin{proof}
The exact sequence is a direct application of \cite{FG20} to the amalgamated free product. 
The map $K_1(C_\bullet(H\times F)) \rightarrow K_1(C_{\bullet,\kappa}(\GG))$ corresponds to $\iota_\ast$, while the map in the part of the isomorphism $K_1(M_{N_\kappa}(\C)\ot C_\bullet(H\times F)) \hookrightarrow K_1(M_{N_\kappa}(\C)\ot C_{\bullet,\kappa}(\GG))$ corresponds to $(\id\ot \iota)_\ast$, which is injective as shown in Proposition \ref{exactseq1}. These maps are the same at the level of the K-theory, through the Morita equivalence. Therefore the map at the utmost right on the bottom line is injective, and the map on the right is $0$. Assume then that $F$ is trivial. We have a sequence of morphisms $$ C(H) \xrightarrow{\iota_\kappa} C_\kappa(\GG) \rightarrow C(\GG) \xrightarrow{\pi_H} C(H),$$ where $\pi_H$ is defined in Remark \ref{RmkFwp}. The composition is equal to $\id_H$. Therefore, at the K-theoretic level, we have that the maps $K_0(C(H))\rightarrow K_0(C_\kappa(\GG))$ must be injective, and that is enough to show that the map on the left side is $0$.
\end{proof}

\noindent We deduce the following.
\begin{proposition}
If $F$ is trivial then:
$$ K_1(C(\GG)) \simeq K_1(C(H))\oplus \bigoplus_{\kappa=1}^K K_1(B\ot C(G))\simeq K_1(C(H))\oplus K_1(C(G))^{\oplus K^2}. $$
and,
$$K_0(C(\GG))\simeq \bigoplus_{\kappa=1}^K K_0(C_\kappa(\GG)) / < \iota_{\gamma\ast}(h) - \iota_{\kappa\ast}(h) >, $$
for all $1\leq \kappa,\gamma\leq K$, and $g\in K_0(C(H))$.
\end{proposition}

\noindent The easiest cases are the ones where $H = S_N^+$ or $H= \Aut^+(M_N(\C),\psi)$, in which we know the generators of the $K_0$ group by the following proposition. We first recall the following Theorem by Voigt.

\begin{theorem}\cite{Vo17}
Let $(B,\psi)$ be a finite dimensional C*-algebra with a $\delta$-form. The CQG $H=\Aut^+(B,\psi)$ is K-amenable and the K-theory groups of $C_\bullet(H)$ are
$$ K_0(C_\bullet(H))\simeq \Z^{K^2-2K+2} \oplus \Z_d^{2K-1}, $$
where $d$ is the gcd of the sizes of the blocks and $K$ their number, and
$$ K_1(C_\bullet(H))\simeq \Z.$$
\end{theorem}

\noindent In the case $S_N^+$, it is also proved that the family of the following set of classes of projections $$\lbrace [u_{ij}],\, 1\leq i,j\leq N-1\rbrace\cup \lbrace1\rbrace$$ is a family of generators for $K_0(C_\bullet(S_N^+))$.

\begin{lemma}
If $H= \Aut^+(M_N(\C),\psi)$ then the class $[\beta(e_{11})]$ generates the free part of $K_0(C(H))$. That is to say $\Z \simeq \langle[\beta(e_{11})]\rangle \subset K_0(C(H)) \simeq \Z \oplus \Z_N$ is a direct factor.
\end{lemma}

\begin{proof}
We use the counit $\varepsilon_H$ which is a unital $*$-morphism from $C(H)$ to $\C$. We have that $(\id\ot \varepsilon_H)\beta = \id$. Therefore the map $\varepsilon_{H\ast}$ is a left-inverse to the map $\beta_\ast$. Hence: $$ \Z \simeq K_0(M_N(\C))\xrightarrow{\beta_\ast} K_0(M_N(\C)\ot C(H))\xrightarrow{\varepsilon_\ast} K_0(M_N(\C))\simeq \Z, $$ with the composition being the identity morphism. This implies that $[\beta(e_{11})]= \beta_\ast ([e_{11}])$ generates a direct factor.
\end{proof}

\noindent As a consequence we can compute the K-theory groups of any free wreath product where $G$ is a finite abelian group and $H$ is $\Aut^+(M_N(\C),\psi)$. We give the following proposition for quantum groups of the form $\Z_s \w \Aut^+(M_N(\C),\psi)$.

\begin{theorem}\label{thmEtxt}
For $s,N\in \N$, and any faithful state $\psi\in M_N(\C)^*$, the compact quantum group $\GG :=\Z_s \w \Aut^+(M_N(\C),\psi)$ is K-amenable and the K-theory groups of $C_\bullet(\GG)$ are:
$$K_0(C_\bullet(\GG))\simeq \Z^{s}\oplus \Z_N\quad\text{and}\quad K_1(C_\bullet(\GG))\simeq \Z.$$
\end{theorem}

\begin{proof}
In the case $B= M_N(\C)$, we have the isomorphism $M_N(\C) \ot C(\GG) \simeq C(\Gtilde)$. Therefore we can compute the K-theory groups of $C(\Gtilde)$ only. The K-theory groups of $C(\Z_s)\simeq \C^s$ are 
$$
    K_0(\C^s)\simeq \Z^{s}\quad\text{and}\quad
    K_1(\C^s)\simeq 0.
$$
A basis of generators for $K_0(\C^s)\simeq \Z^{s}$ is the classes $[e_i]$ for $1\leq i\leq s$ but we remark that we can replace the last generator by the class $[1]= \sum_{i=1}^s [e_i]$, so that $\Z\cdot [1]$ is a direct factor in it. We immediately have that
$$K_1(C(\GG))\simeq K_1(C(\Gtilde))\simeq K_1(M_N(\C)\ot C(H))\simeq \Z.$$
For the $K_0$ group we use the exact sequence
\begin{align*}
    0\rightarrow K_0(M_N(\C)) \rightarrow K_0(M_N(\C)\ot C(\Z_s))\oplus K_0(M_N(\C)\ot C(H)) \rightarrow K_0(C(\Gtilde))\rightarrow 0.
\end{align*}
This gives the isomorphism $$ K_0(C(\Gtilde))\simeq \Z^s\oplus \Z \oplus \Z_N / \langle [1]-[\beta(e_{11})] \rangle. $$
By the previous discussions on the generators, we know that $\Z[1]$ and $\Z[\beta(e_{11})]$ are direct factors, and therefore we get that $K_0(C(\GG))\simeq K_0(C(\Gtilde))\simeq \Z^s \oplus \Z_N$.
\end{proof}

\noindent We can also do the same for the dual of the free groups on $t$ generators, K-amenable thanks to the results of Cuntz \cite{Cun82}.
\begin{proposition}
For $t,N\in \N$, the compact quantum group $\GG :=\widehat{\F_t} \w \Aut^+(M_N(\C),\psi)$ is K-amenable and the K-theory groups of $C_\bullet(\GG)$ are 
$$
    K_0(C_\bullet(\GG))\simeq \Z\oplus \Z_N\quad\text{and}\quad
    K_1(C_\bullet(\GG))\simeq \Z^{t+1}.
$$
\end{proposition}

\begin{proof}
The proof is exactly the same as the previous case. Using that:
\begin{align*}
    K_0(C^\ast\F_t)&\simeq \Z, \text{ the generator being } [1],\\
    K_1(C^\ast\F_t)&\simeq \Z^t.
\end{align*}
We get that the two copies of $\Z$ in $K_0$ are identified, and the result.
\end{proof}

\bibliography{ref.bib}
\bibliographystyle{amsalpha}

\end{document}